%% file: main.tex
\documentclass[final,12pt]{colt2021} 

\title[Convergence and approximation for SGD]{Convergence rates and
  approximation results for SGD and its continuous-time counterpart}
\usepackage{times}
\usepackage{etoc}

\coltauthor{%
 \Name{Xavier Fontaine} \Email{\maillink{xavier.fontaine@polytechnique.edu}}\\
 \addr Centre Borelli, ENS Paris-Saclay
 \AND
 \Name{Valentin De Bortoli} \Email{\maillink{valentin.debortoli@gmail.com}}\\
 \addr Department of Statistics,  University of Oxford
 \AND
 \Name{Alain Durmus} \Email{\maillink{alain.durmus@ens-paris-saclay.fr}}\\
 \addr Centre Borelli, ENS Paris-Saclay
}

\input{header}
\input{def}

\usepackage[page,header]{appendix}
\usepackage{titletoc}

\makeatletter
 \let\Ginclude@graphics\@org@Ginclude@graphics 
\makeatother

\begin{document}

\maketitle

\defcitealias{bachmoulines2011}{BM'11}
\defcitealias{orvieto2019continuous}{OL'19}

\begin{abstract}
  This paper proposes a thorough theoretical analysis of Stochastic Gradient
  Descent (SGD) with non-increasing step sizes.  First, we show that the
  recursion defining SGD can be provably approximated by solutions of a time
  inhomogeneous Stochastic Differential Equation (SDE) using an appropriate
  coupling. In the specific case of a batch noise we refine our results using
  recent advances in Stein's method. Then, motivated by recent analyses of
  deterministic and stochastic optimization methods by their continuous
  counterpart, we study the long-time behavior of the continuous processes at
  hand and establish non-asymptotic bounds. To that purpose, we develop new
  comparison techniques which are of independent interest. Adapting these
  techniques to the discrete setting, we show that the same results hold for the
  corresponding SGD sequences.  In our analysis, we notably improve
  non-asymptotic bounds in the convex setting for SGD under weaker assumptions
  than the ones considered in previous works. Finally, we also establish
  finite-time convergence results under various conditions, including
  relaxations of the famous \L ojasiewicz inequality, which can be applied to a
  class of non-convex functions.
\end{abstract}


\addtocontents{toc}{\protect\setcounter{tocdepth}{0}}

\input{intro}
\input{setting}
\input{control_continuous}

\input{convergence_rate}
\input{graph_tikz}
\input{discussion}

\section*{Acknowledgments}

V. De Bortoli was partially supported by EPSRC grant EP/R034710/1.

\bibliography{../bibliography/bibliography}

\appendix

\input{appendix_control}

\input{appendix_strongly}
\input{appendix_fz_strongly}
\input{appendix_tech}
\input{appendix_convex}
\input{appendix_fz}
\input{appendix_weakly}

\end{document}

%% file: header.tex
\usepackage[utf8]{inputenc}   
\usepackage[T1]{fontenc}      

\usepackage{comment}
\usepackage{amssymb,mathrsfs}
\usepackage{amsfonts}
\usepackage{upgreek}
\usepackage{graphicx}
\usepackage{color}

\usepackage{stmaryrd}
\usepackage{enumitem}
\usepackage{url}

\usepackage{graphicx} 
\usepackage{tikz}
\usepackage{pgfplots}
\usepackage{xcolor}
\usepackage{bbm}

\usepackage{ifthen}
\usepackage{xargs}
\usepackage{wrapfig}

\usepackage{cleveref}
\usepackage{autonum}

 \usepackage{aliascnt}

\newaliascnt{lemma_colt}{theorem}
\newtheorem{lemma_colt}[lemma_colt]{Lemma}
\aliascntresetthe{lemma_colt}
\crefname{lemma}{lemma}{lemmas}
\Crefname{Lemma}{Lemma}{Lemmas}

\newaliascnt{corollary_colt}{theorem}
\newtheorem{corollary_colt}[corollary_colt]{Corollary}
\aliascntresetthe{corollary_colt}
\crefname{corollary}{corollary}{corollaries}
\Crefname{Corollary}{Corollary}{Corollaries}

\newaliascnt{proposition_colt}{theorem}
\newtheorem{proposition_colt}[proposition_colt]{Proposition}
\aliascntresetthe{proposition_colt}
\crefname{proposition}{proposition}{propositions}
\Crefname{Proposition}{Proposition}{Propositions}

\newaliascnt{definition_colt}{theorem}

\aliascntresetthe{definition_colt}
\crefname{definition}{definition}{definitions}
\Crefname{Definition}{Definition}{Definitions}

\newtheorem{assumption}{\textbf{A}\hspace{-3pt}}
\crefformat{assumption}{{\textbf{A}}#2#1#3}

\newtheorem{assumptionF}{\textbf{F}\hspace{-3pt}}
\crefformat{assumptionF}{{\textbf{F}}#2#1#3}

\newenvironment{assumptionbis}[1]
  {\renewcommand{\theassumptionF}{\ref*{#1}$\mathbf{b}$}%
   \addtocounter{assumptionF}{-1}%
   \begin{assumptionF}}
  {\end{assumptionF}}

\Crefname{assumptionB}{\textbf{B}\hspace{-3pt}}{\textbf{B}\hspace{-3pt}}
\crefname{assumptionB}{\textbf{B}}{\textbf{B}}

\Crefname{assumptionC}{\textbf{C}\hspace{-3pt}}{\textbf{C}\hspace{-3pt}}
\crefname{assumptionC}{\textbf{C}}{\textbf{C}}

\newtheorem{assumptionH}{\textbf{H}\hspace{-3pt}}
\Crefname{assumptionH}{\textbf{H}\hspace{-3pt}}{\textbf{H}\hspace{-3pt}}
\crefname{assumptionH}{\textbf{H}}{\textbf{H}}

\Crefname{assumptionT}{\textbf{T}\hspace{-3pt}}{\textbf{T}\hspace{-3pt}}
\crefname{assumptionT}{\textbf{T}}{\textbf{T}}

\Crefname{assumptionT}{\textbf{T}\hspace{-3pt}}{\textbf{T}\hspace{-3pt}}
\crefname{assumptionT}{\textbf{T}}{\textbf{T}}

\Crefname{assumptionL}{\textbf{L}\hspace{-3pt}}{\textbf{L}\hspace{-3pt}}
\crefname{assumptionL}{\textbf{L}}{\textbf{L}}

\Crefname{assumptionQ}{\textbf{Q}\hspace{-3pt}}{\textbf{Q}\hspace{-3pt}}
\crefname{assumptionQ}{\textbf{Q}}{\textbf{Q}}

\Crefname{assumptionAR}{\textbf{AR}\hspace{-3pt}}{\textbf{AR}\hspace{-3pt}}
\crefname{assumptionAR}{\textbf{AR}}{\textbf{AR}}

\usepackage{caption}

%% file: def.tex
\def\yo{Y}
\def\yoo{y_0}
\def\trcst{\Lip_{\mathtt{T}}}
\def\T{^{\top}}
\def\f{\tilde{f}}
\def\rmg{\mathrm{g}}
\def\rmf{\mathrm{f}}
\newcommand{\maillink}[1]{\href{mailto:#1}{\textcolor{black}{#1}}}

\newcommand{\Cweakapp}{\ttd}
\def\ttfp{\Cweakapp_{p}}
\def\ttfpun{\Cweakapp_{p,1}}
\def\ttfpdeux{\Cweakapp_{p,2}}

\def\ttfpquatre{\Cweakapp_{p,4}}
\def\ttamin{\mathtt{a}}
\def\ttfun{\Cweakapp_4}
\def\ttfdeux{\Cweakapp_5}
\def\btta{\bar{\mathtt{A}}}

\def\KuLo{Kurdyka-\L ojasiewicz}
\newcommand{\tta}{\mathtt{A}}
\newcommand{\ttb}{\mathtt{B}}
\newcommand{\ttc}{\mathtt{C}}
\newcommand{\ttd}{\mathtt{D}}

\def\boundLSig{\Lip\eta}

\newcommand{\Capprox}{\tta}
\newcommand{\Ctech}{\ttc}
\newcommand{\Cstrong}{\ttb}
\newcommand{\Cconv}{\ttc}
\newcommand{\Cweak}{C}

\def\conj{\varkappa}
\def\mtta{\mathtt{a}}
\def\explog{\vareps}

\def\Cbeta{\Cweak_{\beta, \explog}}
\def\Aar{\Capprox_{\alpha, r}}
\def\xo{Y}
\def\xoo{x_0}
\def\Db{\Ctech}
\def\intk{\int_{k\gua}^{(k+1)\gua}}
\newcommandx\ctun[1][1=T]{\Capprox_{#1,1}}

\def\cttun{\tilde{\Capprox}_{T,1}}

\def\ctdeux{\Capprox_{T,2}}
\def\cttrois{\Capprox_{T,3}}
\def\ctquatre{\Capprox_{T,4}}
\def\ctcinq{\Capprox_{T,5}}
\def\ctsix{\Capprox_{T,6}}
\def\bctsix{\bar{\Capprox}_{T,6}}
\def\ctsept{\Capprox_{T,7}}
\def\cthuit{\Capprox_{T,8}}
\def\ctneuf{\Capprox_{T,9}}
\def\gfun{\mathbb{G}}

\def\Cconvcontun{\Cconv_{1,\alpha}^{(c)}}
\def\Cconvcontdeux{\Cconv_{2,\alpha}^{(c)}}
\def\Cconvconttrois{\Cconv_{3,\alpha}^{(c)}}
\def\Cconvdiscun{\Cconv_{1,\alpha}^{(d)}}
\def\Cconvdiscdeux{\Cconv_{2,\alpha}^{(d)}}
\def\Cconvdisctrois{\Cconv_{3,\alpha}^{(d)}}
\def\Cconvcont{\Phibf_{\alpha}^{(c)}}
\def\Cconvdisc{\Phibf_{\alpha}^{(d)}}
\def\Csham{\Cconv_1}
\def\Cshamz{\tilde{\Cconv_1}}

\def\Cshama{\Cconv_{\alpha}}
\def\Cshamamoins{\Cshama^-}
\def\Cshamaplus{\Cshama^+}
\def\Ccont{\Cconv^{(c)}}
\def\Ccontz{\tilde{\Ccont}}
\def\Cdisc{\Cconv^{(d)}}

\def\Cconvdtrois{\Cconv^{(b)}}
\def\Cconvdun{(\gamma\eta/2)}
\def\Cconvddeux{(\gamma/2)}
\def\Cshamdisc{\Cconv_{0}}
\def\Cshamt{\tilde{\Cconv}_{\alpha}}
\def\Psial{\Psibf_{\alpha}}

\def\Cstrongdisc{\Cstrong_3}
\def\Cstrongdiscf{\Cstrong_4}
\def\Cstrongloj{\Cstrong_5}
\def\Cstronglojdisc{\Cstrong_6}
\def\Cstrongtilde{\tilde{\Cstrong}}
\def\maxnorm{C}

\newcommand{\pinv}{^{-1}}
\newcommand{\st}{^{\star}}

\newcommand{\rref}[1]{\tup{\Cref{#1}}}
\newcommand{\la}{\langle}
\newcommand{\ra}{\rangle}
\newcommand{\LL}{\L ojasiewicz~}
\newcommand{\gua}{\gamma_{\alpha}}
\newcommand{\bgua}{\bgamma_{\alpha}}

\newcommand{\sigb}{\eta}
\newcommand{\phe}{\varphi_{\varepsilon}}
\newcommand{\feps}{f_{\varepsilon}}
\newcommand{\nfeps}{\nabla f_{\varepsilon}}
\newcommand{\intd}{\int_{\bR^{\dim}}}
\newcommandx{\expec}[2]{\mathbb{E}\left[#1 \middle \vert #2  \right]} 
\newcommandx{\expecLigne}[2]{\mathbb{E}[#1 \vert #2]} 
\newcommand{\expek}[1]{\expec{#1}{\cF_k}}
\newcommand{\expen}[1]{\expec{#1}{\cF_n}}
\newcommand{\expenLigne}[1]{\expecLigne{#1}{\cF_n}}
\newcommand{\nn}{_{n+1}}
\newcommand{\kk}{_{k+1}}
\newcommand{\pal}{^{\alpha}}
\newcommand{\pmal}{^{-\alpha}}

\def\dim{d}

\newcommand{\Lip}{\mathtt{L}}
\newcommand{\Mtt}{\mathtt{M}}
\newcommand{\Mip}{\Mtt}
\newcommand{\Ktt}{\mathtt{K}}

\newcommand{\SDE}{SDE}

\newcommandx{\norm}[2][1=]{\ifthenelse{\equal{#1}{}}{\left\Vert #2 \right\Vert}{\left\Vert #2 \right\Vert^{#1}}}
\newcommandx{\normLigne}[2][1=]{\ifthenelse{\equal{#1}{}}{\Vert #2 \Vert}{\Vert #2\Vert^{#1}}}

\def\bfc{\mathbf{c}}
\def\bfY{\mathbf{Y}}
\def\bfX{\mathbf{X}}
\def\bfXt{\tilde{\mathbf{X}}}
\def\bfXd{\overline{\mathbf{X}}}

\def\bfB{\mathbf{B}}

\def\msa{\mathsf{A}}

\def\mss{\mathsf{S}}

\def\msb{\mathsf{B}}

\def\mse{\mathsf{E}}

\def\msz{\mathsf{Z}}
\def\msy{\mathsf{Y}}

\newcommand{\mcb}[1]{\mathcal{B}(#1)}

\def\mcz{\mathcal{Z}}
\def\mcy{\mathcal{Y}}

\def\mce{\mathcal{E}}

\def\mcf{\mathcal{F}}
\def\mcg{\mathcal{G}}
\def\mck{\mathcal{K}}
\def\mch{\mathcal{H}}

\def\rset{\mathbb{R}}

\def\nset{\mathbb{N}}
\def\nsets{\mathbb{N}^{\star}}

\def\bN{\mathbb{N}}
\def\bR{\mathbb{R}}
\def\bRd{\mathbb{R}^{\dim}}
\def\cF{\mathcal{F}}

\def\rmb{\mathrm{b}}

\def\rmw{\mathrm{w}}
\def\rmd{\mathrm{d}}

\def\rmc{\mathrm{C}}
\def\rmcc{\mathrm{c}}

\def\rmcc{\mathrm{c}}

\def\Pens{\mathscr{P}}
\def\Fens{\mathscr{F}}
\def\Hens{\mathscr{H}}

\newcommand{\R}{\mathbb R}

\newcommand{\dd}{\mathrm{d}}

\def\trace{\operatorname{Tr}}
\newcommandx{\functionspace}[2][1=+]{\mathbb{F}_{#1}(#2)}

\newcommand{\argmin}{\operatorname*{arg\,min}}

\newcommandx{\VarDeux}[3][3=]{\operatorname{Var}^{#3}_{#1}\left\{#2 \right\}}

\newcommand{\1}{\mathbbm{1}}

\newcommand{\LeftEqNo}{\let\veqno\@@leqno}

\newcommand{\floor}[1]{\left\lfloor #1 \right\rfloor}
\newcommand{\floorLigne}[1]{\lfloor #1 \rfloor}
\newcommand{\ceil}[1]{\left\lceil #1 \right\rceil}

\newcommand{\N}{\ensuremath{\mathbb{N}}}

\newcommand{\PE}{\mathbb{E}}
\newcommand{\PP}{\mathbb{P}}

\newcommand{\abs}[1]{\left\vert #1 \right\vert}
\newcommand{\absLigne}[1]{\vert #1 \vert}

\newcommandx{\Vnorm}[2][1=V]{\| #2 \|_{#1}}
\newcommandx{\VnormEq}[2][1=V]{\left\| #2 \right\|_{#1}}

\newcommand{\parenthese}[1]{\left(#1 \right)}
\newcommand{\parentheseLigne}[1]{(#1 )}
\newcommand{\parentheseDeux}[1]{\left[ #1 \right]}
\newcommand{\parentheseDeuxLigne}[1]{[ #1 ]}
\newcommand{\defEns}[1]{\left\lbrace #1 \right\rbrace }
\newcommand{\defEnsLigne}[1]{\lbrace #1 \rbrace }

\newcommand{\proba}[1]{\mathbb{P}\left( #1 \right)}

\newcommandx\probaMarkovTilde[2][2=]
{\ifthenelse{\equal{#2}{}}{{\widetilde{\mathbb{P}}_{#1}}}{\widetilde{\mathbb{P}}_{#1}\left[ #2\right]}}

\newcommand{\expe}[1]{\PE \left[ #1 \right]}
\newcommand{\expesq}[1]{\PE^{1/2} \left[ #1 \right]}

\newcommand{\expeLigne}[1]{\PE [ #1 ]}

\newcommand{\bigO}{\ensuremath{\mathcal O}}

\def\ie{\textit{i.e.}}

\def\eqsp{\;}
\newcommand{\coint}[1]{\left[#1\right)}
\newcommand{\ocint}[1]{\left(#1\right]}
\newcommand{\ooint}[1]{\left(#1\right)}
\newcommand{\ccint}[1]{\left[#1\right]}

\newcommand{\ccintLigne}[1]{[#1]}

\newcommandx{\weight}[2][2=n]{\omega_{#1,#2}^N}

\newcommandx\sequence[3][2=,3=]
{\ifthenelse{\equal{#3}{}}{\ensuremath{\{ #1_{#2}\}}}{\ensuremath{\{ #1_{#2}, \eqsp #2 \in #3 \}}}}

\newcommandx\sequenceD[3][2=,3=]
{\ifthenelse{\equal{#3}{}}{\ensuremath{\{ #1_{#2}\}}}{\ensuremath{( #1)_{ #2 \in #3} }}}

\newcommandx{\sequencen}[2][2=n\in\N]{\ensuremath{\{ #1_n, \eqsp #2 \}}}
\newcommandx\sequenceDouble[4][3=,4=]
{\ifthenelse{\equal{#3}{}}{\ensuremath{\{ (#1_{#3},#2_{#3}) \}}}{\ensuremath{\{  (#1_{#3},#2_{#3}), \eqsp #3 \in #4 \}}}}
\newcommandx{\sequencenDouble}[3][3=n\in\N]{\ensuremath{\{ (#1_{n},#2_{n}), \eqsp #3 \}}}

\def\iid{i.i.d.}

\def\eg{\textit{e.g.}}

\newcommand{\opnorm}[1]{{\left\vert\kern-0.25ex\left\vert\kern-0.25ex\left\vert #1
    \right\vert\kern-0.25ex\right\vert\kern-0.25ex\right\vert}}

\def\Id{\operatorname{Id}}

\newcommandx{\CPE}[3][1=]{{\mathbb E}_{#1}\left[#2 \middle \vert #3  \right]} 
\newcommandx{\CPELigne}[3][1=]{{\mathbb E}_{#1}[#2 \vert #3  ]} 
\newcommandx{\CPEsq}[3][1=]{{\mathbb{E}^{1/2}}_{#1}\left[#2 \middle \vert #3  \right]} 
\newcommandx{\CPVar}[3][1=]{\mathrm{Var}^{#3}_{#1}\left\{ #2 \right\}}
\newcommand{\CPP}[3][]
{\ifthenelse{\equal{#1}{}}{{\mathbb P}\left(\left. #2 \, \right| #3 \right)}{{\mathbb P}_{#1}\left(\left. #2 \, \right | #3 \right)}}

\newcommandx{\osc}[2][1=]{\mathrm{osc}_{#1}(#2)}

\def\Id{\operatorname{Id}}

\def\transpose{\top}

\def\bgamma{\bar{\gamma}}

\def\veps{\varepsilon}

\def\sphere{\mss}

\newcommand{\ensemble}[2]{\left\{#1\,:\eqsp #2\right\}}
\newcommand{\ensembleLigne}[2]{\{#1\,:\eqsp #2\}}

\newcommand\coupling[2]{\Gamma(\mu,\nu)}

\def\vareps{\varepsilon}

\def\Phibf{\mathbf{\Phi}}
\def\Psibf{\mathbf{\Psi}}

\newcommandx{\KL}[2]{\text{KL}\left( #1 | #2 \right)}
\newcommandx{\KLLigne}[2]{\text{KL}( #1 | #2 )}

\def\gaStep
\def\QKer{Q}

\def\distance{\mathbf{d}}
\newcommandx{\wasserstein}[3][1=\distance,3=]{\mathbf{W}_{#1}^{#3}\left(#2\right)}
\newcommandx{\wassersteinLigne}[3][1=\distance,3=]{\mathbf{W}_{#1}^{#3}(#2)}
\newcommandx{\wassersteinD}[1][1=\distance]{\mathbf{W}_{#1}}
\newcommandx{\wassersteinDLigne}[1][1=\distance]{\mathbf{W}_{#1}}

\def\Rcoupling{\mathrm{R}}

\def\sigmaD{\sigma^2}

\newcommandx{\phibfs}[1][1=]{\pmb{\varphi}_{\sigmaD_{#1}}}

\newcommandx\sequenceg[3][2=,3=]
{\ifthenelse{\equal{#3}{}}{\ensuremath{( #1_{#2})}}{\ensuremath{( #1_{#2})_{ #2 \geq #3}}}}

\def\Rker{\Rcoupling}

\def\Qker{\mathrm{Q}}

\newcommandx{\distV}[1][1=\bfc]{\mathbf{W}_{#1}}
\newcommandx{\distVdeux}[1][1=W_2]{\mathbf{d}_{#1}}

\def\mtt{\mathtt{m}}

\newcommand{\tup}[1]{\textup{#1}}

\def\loiz{\pi^Z}
\def\muz{\loiz}
\def\funH{H}

\renewcommand{\doteq}{=}



%% file: intro.tex

\section{Introduction}
\label{sec:introduction}

Recently, first-order optimization methods \citep{boydcandes} have been shown
to share similar long-time behavior with solutions of certain Ordinary
Differential Equations (ODE). One starting point of this analysis is to
remark that most of these algorithms can be regarded as discretization
schemes. 
For instance, gradient descent (GD) can be seen as the Euler discretization of
the gradient flow corresponding to the objective function $f$, \ie, the ODE
$\rmd x(t)/\rmd t=-\nabla f(x(t))$. The analysis of the long-time behavior of
solutions of this gradient flow equation provides fruitful insights on the
convergence of GD. This idea has been adapted to the Nesterov acceleration
scheme~\citep{nesterov83} by~\cite{boydcandes}, and in this case the limiting
continuous flow is associated with a second-order ODE. This result then allows
for a much more intuitive analysis of this scheme and the technique has been
subsequently extended to derive tighter estimates~\citep{jordanhighresolution}
or to analyze different settings~\citep{krichene, adr2018, apidopoulos}.

Following this approach this paper proposes a new analysis of the Stochastic
Gradient Descent (SGD) algorithm to optimize a continuously differentiable
function $f :\rset^d \to \rset $ given stochastic estimates of its gradient in
convex and non-convex settings.  
Using ODEs, and in particular the gradient flow
equation, to study SGD has already been applied in numerous papers~\citep{ljung:1977,kushner:clark:1978,metivier:priouret:1984,metivier:priouret:1987,benveniste:metivier:priouret:1990,benaim:1996,tadic:doucet:2017}.
However, to take into account more precisely the noisy nature of SGD, it has
been recently suggested to use Stochastic Differential Equations (SDE) as
continuous-time models for the analysis of SGD.  \citet{li2017sme} introduced
Stochastic Modified Equations and established weak approximations
theorems, gaining more intuition on SGD, in particular to obtain new
hyper-parameter adjustment policies. In another line of work,
\citet{feng2019uniform} derived uniform in time approximation bounds using
ergodic properties of SDEs.  

The first contribution of this paper is to show that SDEs can also be
used as continuous-times processes properly modeling SGD with non-increasing
stepsizes. In \Cref{sec:sgd_inhomogenous}, we show that SGD with non-increasing
stepsizes is a discretization of a certain class of stochastic continuous
processes $(\bfX_t)_{t \geq 0}$ solution of time inhomogeneous SDEs. More
precisely, we derive strong and weak approximation estimates between the two
processes. Our strong approximation results are new and rely on some appropriate
coupling between SGD and the associated SDE. These new estimates highlight the
advantages and limitations of the analysis of an SDE as a continuous-time proxy
for SGD. In the specific case of a batch noise we can sharpen our analysis using
recent advances in Stein's method. 

However, in general, these approximation bounds between solutions of SDEs and
recursions defined by SGD are derived under a finite time horizon $T \geq 0$ and
the error between the discrete and the continuous-time processes does not go to
zero as $T$ goes to infinity, which is a strong limitation to study the
long-time behavior of SGD, see \citep{li2017sme,li2019jmlr}. We emphasize that
our goal is not to address this problem here by showing uniform in time bounds
between the two processes. Instead, we highlight how the long-time behavior of
the continuous process related to SGD can be used to gain insight on the
convergence of SGD itself.  In that sense our work follows the same lines
as~\citep{boydcandes, krichene, adr2018} which use continuous-time approaches to
provide intuitive ways of deriving convergence results.  More precisely, in the
rest of the paper we first study the behavior of
$(t \mapsto \expeLigne{f(\bfX_t)}-\min_{\rset^{\dim}} f)$ which can be analysed
under different sets of assumptions on $f$, including a convex and weakly
quasi-convex setting. Then, we propose an adaptation of the main arguments of
this analysis to the discrete setting. This allows us to show, under the same
conditions, that $(\expeLigne{f(X_n)} - \min_{\rset^{\dim}} f)_{n \in \nset}$
also converges to $0$ with the same rates, where $(X_n)_{n \in \nset}$ is the
recursion defined by SGD.

Based on this interpretation, we provide much simpler proofs of existing results
and obtain sharper convergence rates for SGD than the ones derived in previous
works in the convex and the weakly quasi-convex settings
\citep{bachmoulines2011, taylorbach,orvieto2019continuous}.  In the convex
setting, we prove for the first time that the convergence rates of SGD match the
minimax lower-bounds \citep{agarwal2012lower} under the same assumptions as
~\citep{bachmoulines2011}. 
Finally, we consider a relaxation of the weakly quasi-convex
setting introduced in \citep{hardt2018gradient}.  
Recent works \citep{orvieto2019continuous} use SDEs to analyse SGD and derive
convergence rates in the weakly quasi-convex. However the rates they obtained
are not optimal and we show that our analysis leads to better rates under weaker
assumptions.  To summarize, our contributions are as follows:
\begin{enumerate}[labelwidth=!,wide,label=(\roman*),noitemsep,nolistsep]
\item We derive strong approximation results between the discrete-time and the
  continuous-time processes  in \Cref{sec:sgd_inhomogenous}. Our strong
  approximation results are new and rely on a specific coupling between SGD and
  the associated SDE. Contrary to other works our bounds cover the case of
  non-increasing stepsizes.
\item We introduce our main tools for the analysis of discrete and
  continuous-time processes and apply them in the context of strongly-convex
  functions to give intuition on our approach in
  \Cref{sec:conv-cont-discr}. Then, we use them to study SGD for the
  minimization of convex functionals in \Cref{sec:convex-case}. We show for the
  first time that the convergence rate is at least of order $\bigO(n^{-1/2})$,
  with stepsize $\gamma_n = \bigO(n^{-1/2})$, without bounded gradient
  assumptions (both for the continuous-time and discrete-time processes). This
  disproves a conjecture of \citep{bachmoulines2011}.
\item In \Cref{sec:non-convex-case}, we relax the convexity assumption (weakly
  quasi-convex assumption) and in this framework we improve on recent bounds by
  \citet{orvieto2019continuous} and derive new convergence results under general \L
  ojasiewicz-type assumptions.
\end{enumerate}


%% file: setting.tex

\section{SGD with Non-Increasing Stepsizes as a Time Inhomogeneous Diffusion Process}
\label{sec:sgd_inhomogenous}
\subsection{Problem Setting and Main Assumptions}

Throughout this paper we consider the problem of the unconstrained minimization
of $f\in \rmc^1(\rset^d, \rset)$, an objective function satisfying the following regularity
condition.
\begin{assumption}
	\label{assum:f_lip}
	For any $x, y \in \rset^{\dim}$,
        $\norm{ \nabla f(x) - \nabla f(y) } \leq \Lip \norm{x - y}$, with
        $\Lip \geq 0$, \ie, $f$ is $\Lip$-smooth.
      \end{assumption}
We consider the general case where we do not have access to $\nabla f$ but only to unbiased estimates.
There are classically two ways to handle this and we will treat both of them in
this
paper. 

\begin{assumption}
	\label{assum:grad_sto}
	There exists a Polish probability space $(\msz,\mcz,\muz)$ and $\eta
        \geq 0$ such that one of the following conditions holds:
	\begin{enumerate}[wide, labelwidth=!, labelindent=0pt,label=(\alph*),noitemsep,nolistsep]
	\item \label{item:approx_sto} There exists a function
          $\funH : \rset^d \times \msz \to \rset^d$ such that for
          any $x \in \rset^{\dim}$,
      \begin{equation}
			\int_{\msz} H(x, z) \rmd \muz(z) = \nabla f(x) \eqsp, \qquad \int_{\msz} \norm{H(x, z) - \nabla f(x)}^2 \rmd \muz(z)  \leq \eta \eqsp .
       \end{equation}
     \item \label{item:ml} There exists a function
       $\f:\bR^d \times \msz \to \bR$ such that for all $z\in\msz$,
       $\f(\cdot,z) \in \rmc^1(\rset^d, \rset)$ is $\Lip$-smooth. In addition,
       there exists $x^\star \in \rset^d$ such that for any $x \in \rset^d$
	\begin{equation}
		\int_{\msz} \f(x, z) \rmd \muz(z) = f(x) \eqsp, \quad \int_{\msz} \nabla \f(x, z) \rmd \muz(z) = \nabla f(x) \eqsp, \quad \int_{\msz} \normLigne{\nabla \f(x\st, z)}^2 \rmd \muz(z)  \leq \eta \eqsp .
	\end{equation}
	In this case, we define $H = \nabla \tilde{f}$.
	\end{enumerate}
\end{assumption}
The first setting \Cref{assum:grad_sto}-\ref{item:approx_sto} corresponds to the
stochastic approximation setting with a square-integrable noise term and has
been studied in
\citep{robbinsmonro,bachmoulines2011,orvieto2019continuous}. This is a weaker
assumption than the bounded gradient assumption considered in
\citep{kingma:ba:2014,shamirzhang,feng2019uniform,rakhlin2012}. The second
setting \Cref{assum:grad_sto}-\ref{item:ml} relaxes the square-integrability
condition, which is often not satisfied in classical machine learning problems
(logistic regression or smooth Support Vector Machines) at the cost of imposing
the Lipschitz regularity of $H(\cdot, z)$ for all $z \in \msz$. We also point
out that the Polish assumption (\ie, the space $\msz$ is metric, complete and
separable) is only used in the proof of \Cref{prop:strong_approx} and can be
avoided in the rest of the paper.

Under \Cref{assum:f_lip} and \Cref{assum:grad_sto}, we introduce the sequence
$(X_n)_{n \in \nset}$ starting from $X_0 \in \rset^d$ corresponding to SGD with
non-increasing stepsizes and defined for any $n \in \nset$ by
\begin{equation}
  \label{eq:sgd}
X_{n+1} = X_n - \gamma (n+1)^{-\alpha} H(X_n, Z_{n+1})  \eqsp ,
\end{equation}
where $\gamma >0$, $\alpha \in \ccint{0,1}$ and $(Z_n)_{n \in \nset}$ is a
sequence of independent random variables on a probability space
$(\Omega,\mcf,\PP)$ valued in $(\msz, \mcz)$ such that for any $n \in \nset$,
$Z_n$ is distributed according to $\muz$. We now turn to the continuous
counterpart of \eqref{eq:sgd}.  Define for any $x \in \rset^{\dim}$, the
semi-definite positive matrix
$\Sigma(x) = \muz(\{ H(x, \cdot) - \nabla f(x) \} \{ H(x, \cdot) - \nabla
f(x)\}^{\top})$ and, for $\alpha \in \coint{0, 1}$, consider the time
inhomogeneous SDE,
\begin{equation}
  \label{eq:sde}
\rmd \bfX_t = -(\gua + t)^{-\alpha} \{\nabla f(\bfX_t) \rmd t + \gua^{1/2} \Sigma(\bfX_t)^{1/2} \rmd \bfB_t\} \eqsp ,
\end{equation}
where $\gua = \gamma^{1/(1-\alpha)}$ and $(\bfB_t)_{t \geq 0}$ is a
$d$-dimensional Brownian motion.  For solutions of this SDE to exist in a strong
sense, we consider the following assumption on $x \mapsto \Sigma(x)^{1/2}$.
\begin{assumption}
  \label{assum:lip_sigma}
  There exists $\Mtt \geq 0$ such that for any $x,y \in \rset^{\dim}$,
  $\normLigne{\Sigma(x)^{1/2} - \Sigma(y)^{1/2}} \leq \Mtt
  \normLigne{x -y }$.
\end{assumption}
Indeed, using \citep[Chapter 5, Theorem 2.5]{karatzas1991brownian}, strong
solutions $(\bfX_t)_{t \geq 0}$ exist if \Cref{assum:f_lip} and
\Cref{assum:lip_sigma} hold. Condition \Cref{assum:lip_sigma} can be hard to
check in practice and can be replaced by the following stronger (but easier to
verify) assumption: $\Sigma \in \rmc^2(\rset^d, \rset)$ with bounded Hessian,
see \cite[Theorem 5.2.3]{stroock2007multidimensional}. In the sequel,
$(\bfX_t)_{t \geq 0}$ is referred to as the \textit{continuous} SGD process in
contrast to $(X_n)_{n \in \nset}$ which is referred to as the \textit{discrete}
SGD process.



%% file: control_continuous.tex

\subsection{Approximations Results}
In this section, we prove that $(\bfX_t)_{t \geq 0}$ solution of \eqref{eq:sde}
is indeed, under some conditions, a continuous counterpart of
$(X_n)_{n \in \nset}$ given by \eqref{eq:sgd}.  First, we informally derive the
form of \eqref{eq:sde}.  Let $(\bfXt_{t})_{t \geq 0}$ be the linear
interpolation of $(X_n)_{n \in \nset}$, \ie, for any
$t \in \ccint{n\gamma_{\alpha},(n+1)\gamma_{\alpha}}$, $n \in \nset$,
$\bfXt_t = ((t-n \gua)X_{n+1}+((n+1)\gua-t) X_{n})/\gua$, with
$\gua = \gamma^{1/(1-\alpha)}$.  Using a first-order Taylor expansion and
assuming that the noise is roughly Gaussian with zero-mean and covariance matrix
$\Sigma(\bfXt_{n\gua})$, we have the following approximation,
\begin{align}
  &\bfXt_{(n+1)\gua} - \bfXt_{n \gua} = X_{n+1}-X_n   \approx - \gamma(n+1)^{-\alpha} H(\bfXt_{n \gua}, Z_{n+1}) \\
  & \quad  \approx -\gua (n \gua + \gua)^{-\alpha} \{ \nabla f(\bfXt_{n\gua}) + \Sigma(\bfXt_{n \gua})^{1/2} G_{n+1} \} \\
  & \quad  \textstyle{\approx -\int_{n\gua}^{(n+1)\gua} (s+\gua)^{-\alpha} \nabla f(\bfXt_s) \rmd s - \gua^{1/2} \int_{n\gua}^{(n+1)\gua} (s+\gua)^{-\alpha} \Sigma(\bfXt_{s})^{1/2} \rmd \bfB_s \eqsp,}
  \label{eq:informal}
\end{align}
where for any $n \in \nset$, $G_n$ is a $d$-dimensional standard Gaussian random
variable. The next result justifies the \textit{ansatz} \eqref{eq:informal} and
establishes strong approximation bounds for SGD. We recall the definition of the
Wasserstein (extended) distance of order $2$, denoted
$\wassersteinD[2]: \ \Pens(\rset^d) \times \Pens(\rset^d) \to \ccint{0,+\infty}$
(where $\Pens(\rset^d)$ is the set of probability measures over
$(\rset^d, \mcb{\rset^d})$ and given for any $\mu_1, \mu_2 \in \Pens(\rset^d)$
by
$ \wassersteinD[2]^2(\mu_1, \mu_2) = \inf_{\Lambda \in \Gamma(\mu_1,
  \mu_2)}\int_{\rset^d \times \rset^d} \norm{v_1 - v_2}^2 \rmd
\Lambda(v_1,v_2)$, where $\Gamma(\mu_1, \mu_2) \subset \Pens(\rset^{2d})$ is the
set of transference plans between $\mu_1$ and $\mu_2$, \ie \
$\Lambda \in \Gamma(\mu_1, \mu_2)$ if for any $\msa \in \mcb{\rset^d}$,
$\Lambda(\msa \times \rset^d) = \mu_1(\msa)$ and
$\Lambda(\rset^d \times \msa) = \mu_2(\msa)$.

  \begin{theorem}    
  \label{prop:strong_approx}
  Let $\bgamma >0$ and $\alpha \in \coint{0,1}$. Assume
  \tup{\Cref{assum:f_lip}}, \tup{\rref{assum:grad_sto}-\ref{item:ml}} and
  \tup{\Cref{assum:lip_sigma}}.  Then there exists a random variable
  $((\bfB_t)_{t \geq 0}, (Z_n)_{n \in \nset})$ such that the following hold:
  \begin{enumerate}[wide, labelwidth=!, labelindent=0pt, label=(\alph*)]
  \item $(Z_n)_{n \in \nset}$ is a sequence of independent random variables such
    that for any $n \in \nset$, $Z_n$ is distributed according to $\pi^Z$ and
    $(\bfB_t)_{t \geq 0}$ is a $d$-dimensional Brownian motion.
  \item For any $T \geq 0$, there exists $C \geq 0$ such that for any
    $\gamma \in \ocint{0, \bgamma}$, $n \in \nset$ with
    $n \leq n_T = \floor{T / \gua}$ and $\gua = \gamma^{1/(1-\alpha)}$ we have
  \begin{equation}
    \label{eq:strong_approx}
      \expesq{\| \bfX_{n\gua} - X_{n} \|^{2}} \leq C (\gamma^{\delta} \vareps + \gamma) (1 + \log(\gamma^{-1}))  \eqsp , \quad \text{with $\delta = \min(1, (2-2\alpha)^{-1})$ ,}
    \end{equation}
    where $(\bfX_{t})_{t \geq 0}$ is solution of \eqref{eq:sde}, 
    $(X_n)_{n \in \nset}$ is defined by \eqref{eq:sgd} with
    $\bfX_0 = X_0 \in \rset^d$ and
    \begin{equation}
  \label{eq:vareps_def}
  \textstyle{\vareps^2 = \sup_{n \in \{0, \dots, n_T\} } \expeLigne{\wassersteinD[2]^2(\nu^{\rmd}(\bfX_{n \gua}), \nu^{\rmcc}(\bfX_{n \gua}))} \eqsp , }
\end{equation}
where for any $\tilde{x} \in \rset^d$, $\nu^{\rmd}(\tilde{x})$ is the
distribution of $H(\tilde{x}, Z_{0})$ and $\nu^{\rmcc}(\tilde{x})$ is the
distribution of $\nabla f(\tilde{x}) + \Sigma^{1/2}(\tilde{x}) G$, with $G$ a
standard Gaussian random variable.
  \end{enumerate}
\end{theorem}
The proof is postponed to \Cref{sec:mean-square-appr}. It relies on a coupling
argument which is made explicit in \Cref{sec:coupling} and uses tools from the
optimal transport theory. The rest of the proof extends approximation results
from \cite{milstein1995numerical} to our coupled setting. To the best of our
knowledge, this strong approximation result is new. A few remarks are in order:

\begin{enumerate}[wide, labelwidth=!,label=(\alph*), nolistsep]
\item This result illustrates the fundamental difference between SGD and
  discretization of SDEs such as the Euler-Maruyama (EM) discretization. In the
  fixed stepsize setting, \ie, $\alpha = 0$, consider $(\bfY_t)_{t \geq 0}$ and
  its EM discretization $(Y_n)_{n \in \nset}$ given by
  $\bfY_0 = Y_0 \in \rset^d$ and for any $n \in \nset$
\begin{equation}
  \label{eq:sde_homo_y}
  \rmd \bfY_t  = \mathrm{b}(\bfY_t) \rmd t +  \upsigma(\bfY_t) \rmd \bfB_t \eqsp, \qquad Y_{n+1} = Y_n + \gamma \rmb(Y_n) + \sqrt{\gamma} \upsigma(Y_n)
  G_{n+1} \eqsp ,
\end{equation}
with $b: \ \rset^d \to \rset^d$, $\upsigma: \ \rset^d \to \rset^{d \times d}$,
and $(G_n)_{n \in \nset}$ is a sequence of \iid \ random variables such that for
any $n \in \nset$, $\expeLigne{G_n} = 0$ and $\expeLigne{G_n G_n^\top} =
\Id$. Using \Cref{prop:strong_approx}, we have that in the Gaussian case the
strong approximation bound for SGD is at least of order $1$. For SDE, this
depends on the structure of $\upsigma$. If $\upsigma$ is constant then the
strong approximation is of order $1$, otherwise it is of order $1/2$, see \eg,
\citep{kloeden:platen:2011,milstein1995numerical}. In addition, it can be shown
that if $(G_n)_{n \in \nset}$ is no longer a sequence of Gaussian random
variables then for $\rmb= 0$, $\upsigma = \Id$, (but it holds under mild
conditions on $\rmb$ and $\upsigma$), there exists $C \geq 0$ such that for any
$T \geq 0$, $\gamma >0$, $n \in \nset$, $n\gamma \leq T$,
$\PE^{1/2}[\normLigne{\bfY_{n\gamma} - Y_n}^2] \geq C \sqrt{T} \eqsp ,$ \ie, no
strong approximation holds. The behavior is different for SGD for which we
obtain a strong approximation of order $\bigO(\gamma^{1/2} \vareps)$, regardless
the structure of the noise.
\item We remark that a strong approximation of order $\bigO(\gamma^{1/2})$ can
  also be derived for the error between SGD and the associated gradient flow
  ODE. Replacing the gradient flow by a stochastic continuous-time process
  improves this error bound up to $\bigO(\gamma^{1/2}\vareps)$, where $\vareps$
  is a measure of the distance between the noise and some Gaussian distribution
  in $\wassersteinD[2]$. This highlights the fact that the SDE \eqref{eq:sde} is
  well-suited to model SGD \eqref{eq:sgd} in the case of a noise which is close
  to a Gaussian, but might not be better than a classical ODE approach for a
  more general noise. In this case, we conjecture that an appropriate L\'{e}vy
  process would further improve these bounds.
\item Finally, we highlight that \Cref{prop:strong_approx} can be improved to
  obtain functional strong approximation bounds using Doob's inequality. Note
  also that we derive our results under the regularity assumption
  \rref{assum:grad_sto}-\ref{item:ml} which implies that for any $z \in \msz$,
  $x \mapsto H(x, z)$ is Lipschitz continuous. It is not clear if our results
  can be extended to \rref{assum:grad_sto}-\ref{item:approx_sto}. We postpone
  these investigations to future work.
\end{enumerate}

We now present a refinement of \Cref{prop:strong_approx} in the case of batch
noise. We begin by recalling the batch noise setting. Assume that
$(\msz, \mcz) = (\msy^{M}, \mcy^{\otimes M})$, $\pi^Z = \pi^{\otimes M}$ with
$M \in \nset$, $\pi$ a probability measure on the Polish space $(\msy, \mcy)$
and for any $x \in \rset^d$, $z = \{y_i\}_{i=1}^M$, let
\begin{equation}
  \label{eq:H-def}
 \textstyle{H(x, z) = (1/M) \sum_{i=1}^M \nabla \tilde{f}(x, y_i)  \eqsp .}
\end{equation}
Note that $\Sigma = (1/M) \Sigma_f$, where for any $x \in \rset^d$,
$\Sigma_f(x) = \pi[(\nabla \tilde{f} - \nabla f(x))(\nabla \tilde{f} - \nabla
f(x))^\top]$.
\begin{corollary_colt}
  \label{coro:batch_noise}
  Let $\bgamma >0$ and $\alpha \in \coint{0,1}$. Assume
  \tup{\Cref{assum:f_lip}}, \tup{\rref{assum:grad_sto}-\ref{item:ml}} and
  \tup{\Cref{assum:lip_sigma}} (with respect to $(\msy, \mcy, \pi)$). Let $H$ be
  given by \eqref{eq:H-def}.  Assume that there exists
  $x^\star \in \rset^d$, $C, p \geq 0$ such that for any $x \in \rset^d$ and
  $y \in \msy$
  \begin{equation}
    \textstyle{\int_{\msy} \normLigne{\nabla \tilde{f}(x^\star, y)}^4 \rmd \pi(y) <
  +\infty \eqsp, \quad \normLigne{\Sigma_f(x)^{-1/2}}\leq C (1 + \norm{x}^p) \eqsp .}
\end{equation}
Then, there exists a random variable
  $((\bfB_t)_{t \geq 0}, (Z_n)_{n \in \nset})$ such that for any $T \geq 0$,
  there exists $C \geq 0$ such that for any $\gamma \in \ocint{0, \bgamma}$,
  $n \in \nset$ with $n\gua \leq T$ $\gua = \gamma^{1/(1-\alpha)}$ we have
  \begin{equation}
      \expesq{\| \bfX_{n\gua} - X_{n} \|^{2}} \leq C (\gamma^{\delta}M^{-1}  + \gamma) (1 + \log(\gamma^{-1}))  \eqsp , \quad \text{with $\delta = \min(1, (2-2\alpha)^{-1})$ .}
    \end{equation}  
  \end{corollary_colt}
  The proof is postponed to \Cref{sec:case-batch-noise} and heavily relies on
  new quantitative bounds for the Central Limit Theorem established using
  Stein's method in \cite{bonis2020stein}.  \Cref{coro:batch_noise} shows that
  in the presence of batch noise (and in the fixed stepsize setting), choosing a
  batch size $M = \bigO(\gamma^{-1/2})$ is enough to obtain a linear
  approximation between the continuous-time process and SGD. In
  \Cref{sec:case-batch-noise}, we also show that a batch-size of order
  $M = \bigO(\gamma^{-1})$ is necessary to obtain a linear approximation between
  the deterministic gradient flow and SGD.  Finally, we also establish weak
  approximation errors between continuous and discrete versions of SGD but due
  to space constraints, they are stated and proved in
  \Cref{sec:weak-approximation}.




%% file: convergence_rate.tex
\section{Convergence of the Continuous and Discrete SGD Processes}
\label{sec:conv-cont-discr}


\input{strongly_convex}
\input{convex}
\input{weakly_convex}


%% file: strongly_convex.tex
\subsection{Two Basic Comparison Lemmas}
\label{sec:two-basic-comparison}
We now turn to the convergence of SGD.  Our general strategy is as follows: in
the continuous-time setting, in order to derive sharp convergence rates for
\eqref{eq:sde}, we consider appropriate energy functions
$\mathscr{V} : \rset_+ \times \rset^d \to \rset_+$ which depend on the
conditions imposed on the function $f$.  Then, we show that
$(t \mapsto v(t) = \PE[\mathscr{V}(t,\bfX_t)])$ satisfies an ODE and prove that
it is bounded using the following simple lemma.
\begin{lemma_colt}
	\label{lemma:ode}
	Let $F \in \rmc^1(\bR_+\times \bR, \bR)$ and $v \in \rmc^1(\bR_+,\bR_+)$
        such that for all $t \geq 0$, $\rmd v(t)/\rmd t  \leq F(t,v(t))$.  If there exists
        $t_0 > 0$ and $A > 0$ such that for all $t \geq t_0$ and for all
        $u \geq A$, $F(t,u)< 0$, then there exists $B >0$ such that for
        all $t\geq 0$, $v(t) \leq B$, with $B = \max(\max_{t\in \ccint{0,t_0}}v(t),A)$.
\end{lemma_colt}
\begin{proof}
  Assume that there exists $t \geq 0$ such that $v(t)>B$, and let $		t_1= \inf \defEns{t \geq 0 \, : \, v(t) > B }$.
	By definition of $B$, $t_1 \geq t_0$, and by continuity of
        $v$, $v(t_1)=B$. By assumption, $F(t_1,v(t_1))< 0$. Then
        $\rmd v(t_1)/\rmd t  < 0$ and there exists $t_2<t_1$ such that
        $v(t_2)>v(t_1)=B$, hence the contradiction.
\end{proof}
Considering discrete analogues of the energy functions and ODEs found in the
study of the continuous process solution of \eqref{eq:sde}, we derive explicit
convergence bounds for the discrete SGD process. To that purpose, we establish a
discrete analog of \Cref{lemma:ode} whose proof is postponed to
\Cref{sec:strongly-convex_appendix}. 

\begin{lemma_colt}
	\label{lemma:ode_discrete}
	Let $F : \nset \times \rset \to \rset$ satisfying for any $n \in \nset$,
        $F(n,\cdot) \in \rmc^1(\bR, \bR)$. Let $(u_n)_{n \in \bN}$ be a sequence
        of non-negative numbers satisfying for all $n \in \bN$,
        $u\nn-u_n \leq F(n,u_n)$. Assume that there exist $n_0 \in \bN$ and
        $A_1 > 0$ such that for all $n \geq n_0$ and for all $x \geq A_1$,
        $F(n,x)< 0$. In addition, assume that there exists $A_2>0$ such that
        for all $n \geq n_0$ and for all $x \geq 0$, $F(n,x) \leq A_2$.
        Then, there exists $B > 0$ such that for all $n \in \nset$, 
        $u_n \leq B$ with $B = \max(\max_{n \leq n_0+1} u_n,A_1)+A_2$.
\end{lemma_colt}

\subsection{Strongly-Convex Case}
\label{sec:strongly_convex}

First, we illustrate the simplicity and effectiveness of our approach by
recovering optimal convergence rates if the objective function is strongly
convex.  Due to the two settings associated with \Cref{assum:grad_sto}, we
consider two versions of the strong convexity hypothesis, either directly on $f$
if \rref{assum:grad_sto}-\ref{item:approx_sto} holds or on $\f$ if
\rref{assum:grad_sto}-\ref{item:ml} holds.
\begin{assumptionF}
  \label{assum:strongly_cvx}
	Either one of the following conditions holds:
	\begin{enumerate}[wide, labelwidth=!, labelindent=0pt,label=(\alph*),noitemsep,nolistsep]
	\item Case \tup{\rref{assum:grad_sto}-\ref{item:approx_sto}}: \label{item:f_str_conv}  $f$ is $\mu$-strongly convex with $\mu>0$, \ie, for any $x,y \in \bRd$,
        $\langle \nabla f(x) - \nabla f(y), x -y \rangle \geq \mu \norm{x-y}^2$.
      \item \label{item:ftilde_str_conv} Case
        \tup{\rref{assum:grad_sto}-\ref{item:ml}}: for all $z \in \msz$,
        $\f(\cdot, z)$ is $\mu$-strongly convex.
	\end{enumerate}
\end{assumptionF}
Note that \Cref{assum:strongly_cvx}-\ref{item:ftilde_str_conv} implies directly
the strong convexity of $f$.  The results presented below are not new, see
\citep{bachmoulines2011} for the discrete case and \citep{orvieto2019continuous}
for the continuous one, but they can be obtained very easily within our
framework. We only derive our results in the continuous-time setting for
pedagogical purposes, and gather their discrete counterparts in
\Cref{sec:strongly-convex_appendix}. First, we derive convergence rates on the
last iterates.  Denote by $x^\star$ the unique minimizer of $f$ (which exists
under
\Cref{assum:strongly_cvx}).
\begin{theorem}
	\label{thm:strong_continuous}
	Let $\alpha, \gamma \in \ooint{0,1}$ and $(\bfX_t)_{t\geq0}$ be given by
        \eqref{eq:sde}.  Assume \rref{assum:f_lip}, \rref{assum:grad_sto},
        \rref{assum:lip_sigma} and \rref{assum:strongly_cvx}. Then there exists
        $C\geq0$ (explicit in the proof) such that for any $T \geq 1$,
        $ \expeLigne{\norm{\bfX_T-x\st}^2} \leq C T^{-\alpha}$.
\end{theorem}
This result holds for both versions of \Cref{assum:strongly_cvx} and we present
below a proof under \Cref{assum:strongly_cvx}-\ref{item:f_str_conv}. The result
under \Cref{assum:strongly_cvx}-\ref{item:ftilde_str_conv} is stated and proved
in Appendix~\ref{app:fz_strongly}.
\begin{proof}
  Let $\alpha, \gamma \in \ooint{0,1}$ and consider $\mce : \rset_+ \to \rset_+$
  defined for $t \geq 0$ by
  $\mce(t) = \expeLigne{(t+\gua)^{\alpha} \normLigne{\bfX_t - x^{\star}}^2}$,
  with $\gua = \gamma^{1/(1-\alpha)}$.  Using Dynkin's formula, see
  \Cref{lemma:dynkin}, we have for any $t\geq 0$,
\begin{align}
  \mce(t)=\mce(0) + \alpha \int_0^t  \frac{\mce(s)}{s+\gua}\dd s   + \int_0^t \gua \frac{\expe{\trace (\Sigma(\bfX_s))}}{ (s+\gua)^{\alpha}} \dd s -2 \int_0^t \expe{\la \nabla f(\bfX_s), \bfX_s-x\st \ra} \dd s
           \eqsp . \label{eq:strong_ito}
\end{align}
We now differentiate this expression with  respect to $t$ and
using \Cref{assum:strongly_cvx}-\ref{item:f_str_conv} and \Cref{assum:grad_sto}-\ref{item:approx_sto}, we get for any $t > 0$,
\begin{align}
\rmd   \mce(t)/\rmd t &=\alpha \mce(t)(t+\gua)^{-1} -2 \expe{\la \nabla f(\bfX_t), \bfX_t-x\st \ra} + \gua \expe{\trace (\Sigma(\bfX_t))} (t+\gua)^{-\alpha} \\
  &\leq \alpha \mce(t)/(t+\gua) -2\mu \expeLigne{\normLigne{\bfX_t-x\st}^2} + \gua \eta/ (t+\gua)^{\alpha} \\
  &\leq F(t,\mce(t)) = \alpha \mce(t)(t+\gua)^{-1} -2\mu \mce(t)(t+\gua)^{-\alpha} + \gua \eta (t+\gua)^{-\alpha} \eqsp,
\end{align}
where we have used that $\trace(\Sigma(x)) \leq \eta$
for any $x \in \rset^d$ by \Cref{assum:grad_sto}-\ref{item:approx_sto}.  Hence,
since $F$ satisfies the conditions of \Cref{lemma:ode} with
$t_0=(\alpha/\mu)^{1/(1-\alpha)} $ and $A = 2 \gua \eta / \mu$, applying this
result we get, for any $t \geq 0$, $\mce(t) \leq B$ with
$B = \max(\max_{s \in [0,t_0]} \mce(s),A)$ which concludes the proof.
\end{proof}

Due to space constraints and to avoid over-complicated propositions, we do not
precise the dependency of $C$ with respect to the parameters $\mu$, $\eta$ and
the initial condition. However, in \Cref{prop:constante} we obtain that (i) the
constant in front of the asymptotic term $T^{-\alpha}$ scales as $\eta / \mu$
and (ii) the initial condition is forgotten exponentially fast.

In \Cref{thm:strongly_discrete}, we extend this result to the discrete setting
using \Cref{lemma:ode_discrete} and recover the rates obtained in \citep[Theorem
1]{bachmoulines2011} in the case where $\alpha \in \ocint{0, 1}$. In particular,
if $\alpha = 1$, we obtain a convergence rate of order $\mathcal{O}(T^{-1})$
which matches the minimax lower-bounds established in
\citep{nemirovsky1983problem,agarwal2012lower}. In \Cref{fig:fig_str_cvx_a} and
\Cref{fig:fig_str_cvx_b}, we experimentally verify that the results we obtain
are tight in the simple case where $f(x) =
\normLigne{x}^2$.
\begin{figure}
  \centering
  \begin{minipage}{0.45\linewidth}
    \centering
    \begin{tikzpicture}
      \node at (0,-2.3) {\scriptsize number of iterations};
      \node[rotate=90] at (-2.95, 0) {\scriptsize $\log(\expeLigne{f(X_n)} - \min_{\rset^{\dim}} f)$};
      \node[inner sep=0pt] (russell) at (0,0)
      {\includegraphics[width=0.9\linewidth]{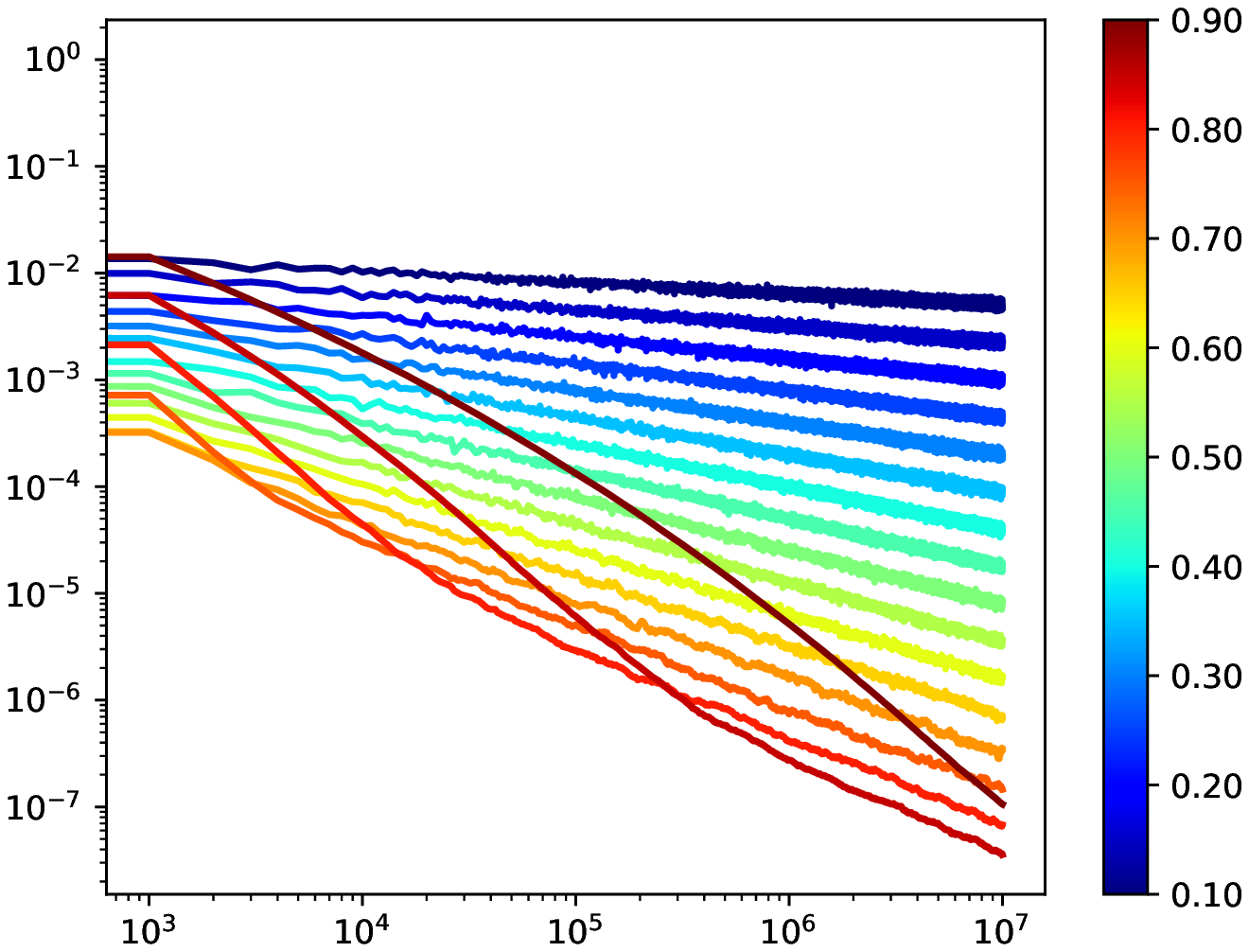}};
\end{tikzpicture}    
    \vspace{-0.3cm}
    \captionof{figure}{Evolution of $(\log(\expeLigne{f(X_n)} - \min_{\rset^{\dim}} f))_{n \in \nset}$}     \label{fig:fig_str_cvx_a}
\end{minipage}
\hfill
  \begin{minipage}{0.45\linewidth}
    \centering
        \begin{tikzpicture}
      \node at (0,-2.3) {\scriptsize value of $\alpha$};
      \node[rotate=90] at (-2.95, 0) {\scriptsize rate of convergence};
      \node[inner sep=0pt] (russell) at (0,0)
      {    \includegraphics[width=0.9\linewidth]{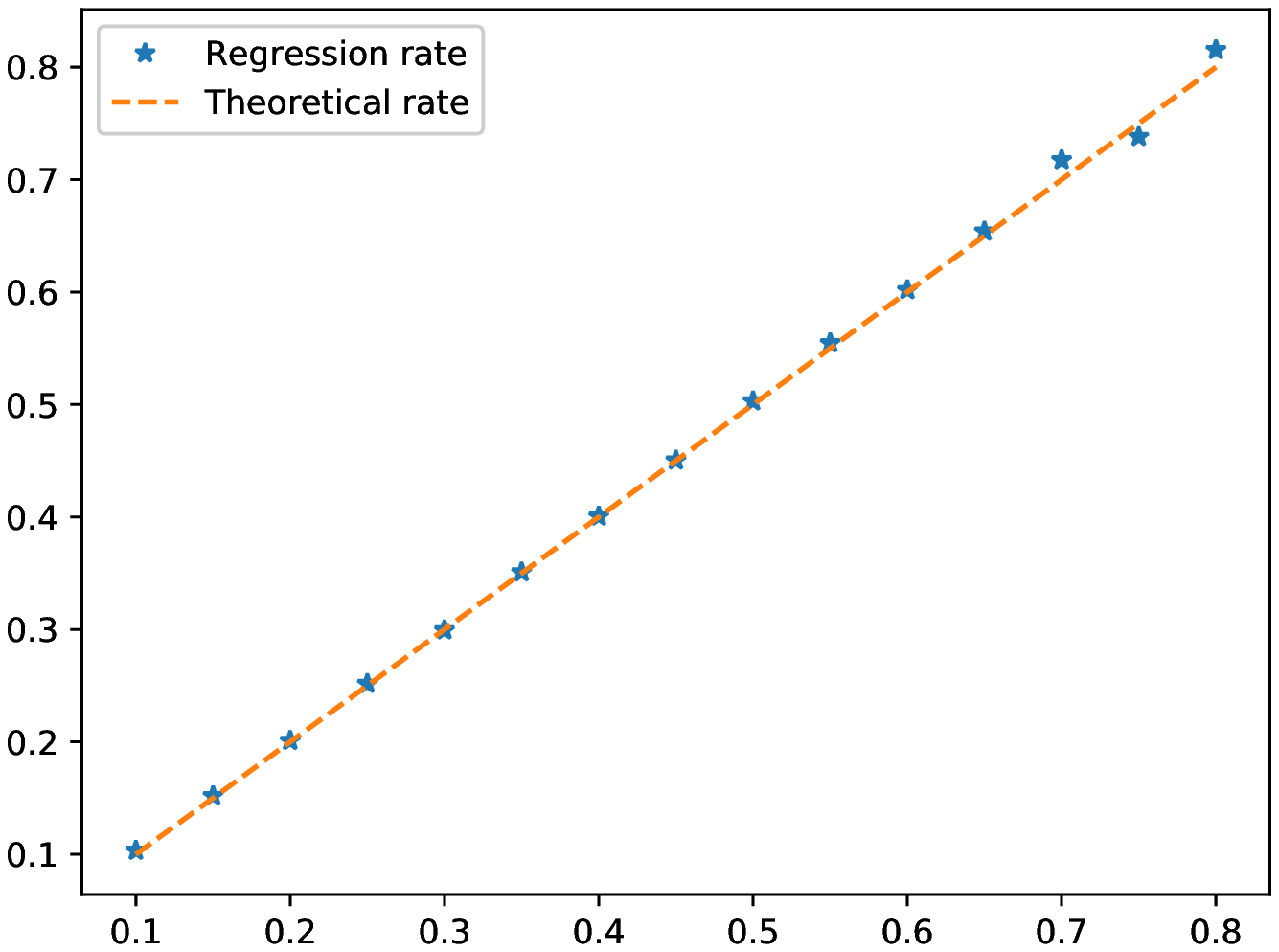}};
\end{tikzpicture}    
\vspace{-0.3cm}
\captionof{figure}{Empirical rates match theoretical rates
    for different values of $\alpha$.}
      \label{fig:fig_str_cvx_b}
    \end{minipage}
    \vspace{.1cm}
  \end{figure}

  We emphasize that the strong convexity assumption can be relaxed if we only
  assume that $f$ is weakly $\mu$-strongly convex, \ie, for any
  $x \in \rset^{\dim}$,
  $\langle \nabla f(x), x - x^{\star} \rangle \geq \mu \norm{x - x^{\star}}^2$.
  In \citep{kleinberg2018alternative} the authors experimentally show that
  modern neural networks satisfy a relaxation of this last condition and it was
  proved in \citep{li2017convergence} that two-layer neural networks with ReLU
  activation functions are weakly $\mu$-strongly convex if the inputs are
  Gaussian. Finally, we show in \Cref{thm:strong_continuous_f} and
  \Cref{thm:strongly_discrete_f} that \Cref{thm:strong_continuous} also implies
  convergence rates for the process
  $(\expe{f(\bfX_t)}-\min_{\rset^{\dim}} f)_{t \geq 0}$ and its discrete
  counterpart.



%% file: convex.tex

\section{Convex Case}
\label{sec:convex-case}

In this section, we relax the strong convexity condition. Again we need to
consider two different settings depending on the version of
\Cref{assum:grad_sto} we consider.
\begin{assumptionF}
	\label{assum:f}
	Either one of the following conditions holds:
	\begin{enumerate}[wide, labelwidth=!, labelindent=0pt,label=(\alph*),noitemsep,nolistsep]
        \item \label{item:f_conv} Case
          \tup{\rref{assum:grad_sto}-\ref{item:approx_sto}}: $f$ is convex, \ie, for any $x,y \in \bRd$,
          $\langle \nabla f(x) - \nabla f(y), x -y \rangle \geq 0$, and there
          exists a minimizer $x^{\star} \in \argmin_{\rset^{\dim}}f$.
		\item \label{item:ftilde_conv} Case
          \tup{\rref{assum:grad_sto}-\ref{item:ml}}: for all $z\in \msz$, $\f(\cdot,z)$ is convex and there exists a minimizer  $x^{\star} \in \argmin_{\rset^{\dim}}f$.
	\end{enumerate}
      \end{assumptionF}
      Similarly to the strongly-convex case, we start by studying the continuous
      process.  The discrete analog of the following result is given in
      \Cref{thm:shamir_discrete}.
\begin{theorem}
	\label{thm:shamir_continuous}
        Let $\alpha, \gamma \in \ooint{0,1}$ and $(\bfX_t)_{t \geq 0}$ be given by \eqref{eq:sde}. Assume $f \in \rmc^2(\rset^{\dim}, \rset)$, \rref{assum:f_lip}, \rref{assum:grad_sto},
         \rref{assum:lip_sigma} and \rref{assum:f}. Then, there exists $C \geq 0$ (explicit and given in the proof) such that for any $T \geq 1$
  \begin{equation}
    \expe{f(\bfX_T)}-{ \textstyle\min_{\rset^d}} f
    \leq C(1+\log(T))^2/T^{\alpha\wedge(1-\alpha)}
		\eqsp .
  \end{equation}
\end{theorem}
To the best of our knowledge, these non-asymptotic results are new for the
continuous process $(\bfX_t)_{t \geq 0}$ defined by \eqref{eq:sde}.  Note that
for $\alpha =1/2$ the convergence rate is of order $\bigO(T^{-1/2} \log^2(T))$
which matches (up to a logarithmic term) the minimax lower-bound for the
discrete-time process \citep{agarwal2012lower} and is in accordance with the
tight bounds derived in the discrete case under additional assumptions
\citep{shamirzhang}.  The general proof is postponed to
\Cref{app:shamir_cont}. The main strategy to prove \Cref{thm:shamir_continuous}
is to carefully analyze a continuous version of the suffix averaging \citep{shamirzhang,HarveyLPR19}, introduced in the discrete case by
\cite{zhang2004solving}. We can relax the assumption
$f \in \rmc^2(\rset^{\dim}, \rset)$ assuming that the set
$\argmin_{\rset^{\dim}} f$ is bounded.

\begin{corollary_colt}
	\label{thm:shamir_continuous_C1}
        Let $\alpha, \gamma \in \ooint{0,1}$ and $(\bfX_t)_{t \geq 0}$ be given
        by \eqref{eq:sde}.  Assume that $\argmin_{\rset^{\dim}} f$ is bounded,
        \rref{assum:f_lip}, \rref{assum:grad_sto}, \rref{assum:lip_sigma} and
        \rref{assum:f}. Then, there exists $C \geq 0$ (explicit and given in the
        proof) such that for any $T \geq 1$,
  \begin{equation}
    \expe{f(\bfX_T)}-{ \textstyle{ \textstyle\min_{\rset^d}}} f
    \leq C(1+\log(T))^2/T^{\alpha\wedge(1-\alpha)} \eqsp .
  \end{equation}
\end{corollary_colt}
The proof is postponed to \Cref{app:shamir_cont} and relies on the fact that if
$f$ is convex then for any $\vareps >0$, $f \ast g_{\vareps}$ is also convex,
where $(g_{\vareps})_{\vareps >0}$ is a family of non-negative mollifiers. We
now turn to the discrete counterpart of \Cref{thm:shamir_continuous}.

\begin{theorem}
	\label{thm:shamir_discrete}
	Let $\gamma, \alpha \in \ooint{0,1}$ and
        $(X_n)_{n\geq 0}$ be given by \eqref{eq:sgd}. Assume \rref{assum:f_lip}, \rref{assum:grad_sto} and \rref{assum:f}. Then, there exists $C \geq 0$ (explicit and given in the proof) such that for any $N \geq 1$,
		\begin{equation}
       \expe{f(X_N)}-{ \textstyle\min_{\rset^d}} f \leq C (1+\log(N+1))^2/(N+1)^{\alpha\wedge(1-\alpha)} \eqsp.
     \end{equation}
\end{theorem}
The proof is postponed to \Cref{app:shamir_discrete} and takes its inspiration
from the proof of the continuous counterpart \Cref{thm:shamir_continuous}. Note
that in the case $\alpha = 1/2$ we recover (up to a logarithmic term) the rate
$\bigO(N^{-1/2}\log(N+1))$ derived in \citep[Theorem 2]{shamirzhang} which
matches the minimax lower-bound \cite{agarwal2012lower}, up to a logarithmic
term. We also extend this result to the case $\alpha \neq 1/2$. Note however
that our setting differs from the one of \citep{shamirzhang}. Indeed,
\citep[Theorem 2]{shamirzhang} established the optimal convergence rate for a
projected version of SGD onto a convex compact set of $\rset^{\dim}$ under the
assumption that $f$ is convex (possibly non-smooth) and
$(\expeLigne{\normLigne{H(X_n, Z_{n+1})}^2})_{n \in \nset}$ is bounded. Our
result avoids the boundedness assumption and the projection step of
\citep{shamirzhang}, since in 
\Cref{thm:shamir_discrete}, we replace the boundedness condition by the
regularity condition \Cref{assum:f_lip} (actually our proof can be very easily
adapted to the setting of \citep{shamirzhang}, see
\Cref{cor:shamir_bounded}). 
Our main
contributions in the convex setting are summarized in \Cref{table:bachmoulines}
and \Cref{fig:rates}.
\vspace{-.5cm}
\begin{table}[ht]
\begin{minipage}[c]{0.4\linewidth}
  \centering
  \begin{tikzpicture}
      \node at (0,-2.2) {\scriptsize value of  $\alpha$};
      \node[rotate=90] at (-2.8, 0) {\scriptsize rate of convergence};
      \node[inner sep=0pt] (russell) at (0,0)
{\includegraphics[width=1.\linewidth]{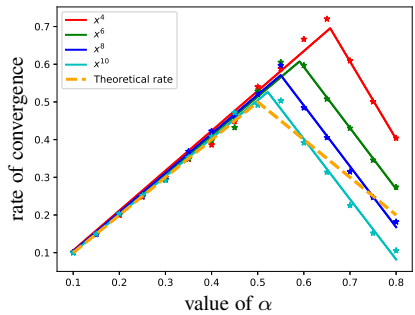}} ;
\end{tikzpicture}
\captionof{figure}{Convergence rates for $\varphi_p$ match the theoretical
  results of \Cref{thm:shamir_discrete} asymptotically.}           \label{fig:figure_cvx}
\end{minipage}
\begin{minipage}[c]{0.58\linewidth}
    \label{table:bachmoulines}
    \small
  \begin{tabular}{|c||c|c|c|}
    \hline
    Reference & Thm.\ref{thm:shamir_discrete} (L) &  \citepalias{bachmoulines2011} (B, L) & \citepalias{bachmoulines2011} (L) \\ \hline
    $\alpha \in \ooint{0,1/3}$ & $\alpha$ & $\times$ & $\times$ \\ \hline
    $\ooint{1/3,1/2}$ & $\alpha$ & $(3\alpha-1)/2$ & $\times$ \\ \hline
    $\ooint{1/2, 2/3}$ & $1-\alpha $ & $\alpha/2$ & $\alpha/2$ \\ \hline
    $\ooint{2/3, 1}$ & $1-\alpha$ & $1-\alpha$ & $1-\alpha$ \\ \hline
  \end{tabular}
  \caption{Convergence rates for convex SGD  (B: Bounded gradients, L: Lipschitz gradient).} \label{table:bachmoulines}
\end{minipage}
\end{table}


            

On the other hand, the setting we consider is the same as
\citep{bachmoulines2011}, but we always obtain better convergence
rates and in particular we get an optimal choice for $\alpha$ ($\alpha = 1/2$)
different from theirs $(\alpha = 2/3)$, see
\Cref{table:bachmoulines}. Hence, we disprove the conjecture
  formulated in \citep{bachmoulines2011} which asserts that the minimax
  rate for SGD in this setting is $1/3$.

In \Cref{fig:figure_cvx}, we experimentally assess the results of
\Cref{thm:shamir_discrete}. We apply SGD on the family of functions
$(\varphi_p)_{p \in \nsets}$, where for any $x \in \rset$, $p \in \nsets$,
\begin{equation}
  \label{eq:varphi_p}
  \varphi_p(x) = x^{2p} \eqsp, \text{if } x\in [-1, 1] \eqsp , \varphi_p(x) = 2p(|x|-1) + 1 \eqsp , \text{otherwise} \eqsp .
\end{equation}
For any $p \in \nset$, $\varphi_p$ satisfies and \Cref{assum:f_lip} and
\Cref{assum:f}. Denoting $\alpha^{\star}_p$ the non-increasing rate $\alpha$ for
which the convergence rate $r_p^{\star}$ is maximum, we experimentally check
that $\lim_{p \to +\infty} r_p^{\star} = 1/2$ and
$\lim_{p \to +\infty} \alpha^{\star}_p = 1/2$.  Note also that
$\alpha^{\star}_p$ decreases as $p$ grows, which is in accordance with the
deterministic setting where the optimal rate in this case is given by $p/(p-2)$,
see \citep{bolte2017error, frankel2015splitting}. As an immediate consequence of
\Cref{thm:shamir_discrete}, we can show that
$(\expeLigne{\normLigne{\nabla f(X_n)}^2})_{n \in \nset}$ enjoys the same rates
of convergence as $(\expeLigne{f(X_n)} - \min_{\rset^{\dim}} f)_{n \in \nset}$,
using that $f$ is smooth.
\begin{corollary_colt}
  \label{cor:bounded_grad_disc}
	Let $\gamma, \alpha \in \ooint{0,1}$ and
        $(X_n)_{n\geq 0}$ be given by \eqref{eq:sgd}. Assume \rref{assum:f_lip}, \rref{assum:grad_sto} and \rref{assum:f}. Then, there exists $C \geq 0$ (explicit and given in the proof) such that for any $N \geq 1$,
        \begin{equation}
        	\expeLigne{\norm{\nabla f(X_N)}^2} \leq C (1+\log(N+1))^2/(N+1)^{\alpha \wedge(1-\alpha)} \eqsp.
     \end{equation}
\end{corollary_colt}
In particular,
$(\expeLigne{\normLigne{\nabla f(X_n)}^2})_{n \in \nset}$ is bounded
which is often found as an assumption for the study of the convergence
of SGD in the convex
setting~\citep{shalev2011pegasos,nemirovski:2009,jmlrhazan,shamirzhang,recht2011hogwild}. Our
result shows that this assumption is unnecessary.


%% file: weakly_convex.tex
\section{Weakly Quasi-Convex Case}
\label{sec:non-convex-case}

In this section, we no longer consider that $f$ is convex but
a relaxation of this condition.
\begin{assumptionF}
  \label{assum:f_weak}
  There exist $r_1 \in \ooint{0,2}$,
  $r_2 \geq 0$, $\tau >0$ such that for any $x \in \rset^{\dim}$
  \begin{equation}
    \| \nabla f(x) \|^{r_1} \|x - x^{\star}\|^{r_2} \geq \tau(f(x) - f(x^{\star})) \eqsp, \quad \text{ where $\textstyle{x^{\star} \in \argmin_{\rset^{\dim}} f \neq \emptyset\eqsp .}$}
  \end{equation}
\end{assumptionF}
This setting is a generalization of the weakly quasi-convex assumption
considered in \citep{orvieto2019continuous} and introduced in
\citep{hardt2018gradient} as follows.
\begin{assumptionbis}{assum:f_weak}
  \label{assum:f_weak_quasi}
The function   $f$ is weakly quasi-convex if there $\tau >0$ such that for any  $x \in \rset^{\dim}$
      \begin{equation}
      \langle \nabla f(x), x - x^{\star} \rangle \geq \tau(f(x) - f(x^{\star})) \eqsp, \quad \text{ where $\textstyle{x^{\star} \in \argmin_{\rset^{\dim}} f \neq \emptyset\eqsp .}$}
    \end{equation}
\end{assumptionbis}
This last condition itself is a modification of the quasi-convexity assumption
\citep{hazan2015beyond}. It was shown in \citep{hardt2018gradient} that an
idealized risk for linear dynamical system identification is weakly
quasi-convex, and in \citep{yuan2019stagewise} the authors experimentally check
that a residual network (ResNet20) used on CIFAR-10 (with differentiable
activation units) satisfy the weakly quasi-convex assumption.

The assumption \Cref{assum:f_weak} also encompasses the setting where $f$
satisfies some Kurdyka-\LL condition \citep{bolte2017error}, \ie, if there exist
$r \in \ooint{0, 2}$ and $\tilde{\tau} >0$ such that for any
$x \in \rset^{\dim}$,
\begin{equation}
  \label{eq:kl}
  \norm{\nabla f(x)}^{r} \geq \tilde{\tau}(f(x) - f(x^{\star})) \eqsp, \quad \text{ where $\textstyle{x^{\star} \in \argmin_{\rset^{\dim}} f \neq \emptyset\eqsp .}$}
\end{equation}
In this case, \Cref{assum:f_weak} is satisfied with $r_1 = r$, $r_2=0$ and
$\tau = \tilde{\tau}$.  \KuLo \ conditions have been often used in the context of
non-convex minimization~\citep{attouch2010proximal, noll}.  Even though the case
$r_1 = 2$ and $r_2 = 0$ is not considered in \Cref{assum:f_weak}, one can still
derive convergence of order $\alpha$ for $\alpha \in \ooint{0, 1}$, see
\Cref{prop:kl2_discrete}, extending the
results obtained in the strongly convex setting.  We now state the main theorem
of this section.
\begin{theorem}
  \label{thm:f_weak}
  Let $\alpha, \gamma \in \ooint{0,1}$ and $(\bfX_t)_{t \geq 0}$ be given by
  \eqref{eq:sde}. Assume $f \in \rmc^2(\rset^{\dim}, \rset)$,
  \rref{assum:f_lip}, \tup{\rref{assum:grad_sto}-\ref{item:approx_sto}}, \rref{assum:lip_sigma} and
  \rref{assum:f_weak}.  In addition, assume that there exist
  $\beta, \explog \geq 0$ and $\Cbeta \geq 0$ such that for any $t \geq 0$,
  \begin{equation}
    \label{eq:theo_f_weak_condition_norm}
    \expeLigne{\| \bfX_t - x^{\star} \|^{r_2 r_3}} \leq \Cbeta(\gua +t)^{\beta}(1 +
  \log(1 + \gua^{-1}t))^{\explog} \eqsp,
  \end{equation}
  where  $\gua = \gamma^{1/(1-\alpha)}$ and $r_3 = (1 - r_1/2)^{-1}$.  Then, there exists $C \geq 0$ (explicit and given in the proof) such that
  for any $T \geq 1$
  \begin{equation}
        \label{eq:delta}
    \expe{f(\bfX_T)}-{\textstyle \min_{\rset^d}} f \leq C T^{-\delta } \parentheseDeuxLigne{1 + \log(1+\gua^{-1}T)}^{\vareps} \eqsp ,
  \end{equation}
  \begin{equation}
    \begin{aligned}
      &\text{where} \ \delta_1 \wedge \delta_2 \eqsp, \quad \delta_1 = (r_1/2)(1-r_1/2)^{-1}(1-\alpha) -\beta \quad \text{and} \quad \delta_2 = (r_1 /2)\alpha - \beta(1 - r_1/2) \eqsp.
    \end{aligned}
  \end{equation}
\end{theorem}
The proof is postponed to \Cref{sec:weakly-convex_appendix}.
First, note that if $f$ satisfies a \KuLo \ condition of type \eqref{eq:kl}
then \Cref{assum:f_weak} is satisfied with $r_1=r$ and $r_2= 0$ and the rates in \Cref{thm:f_weak} simplify and we obtain that
$\delta = \min((r/2)(1-r/2)^{-1}(1-\alpha), (r/2)\alpha)$. The rate is maximized
for $\alpha = (2 - r/2)^{-1}$ and in this case, $\delta = r/(4-r)$. Therefore, if $r \to 2$, then $\delta \to 1$ and we obtain at the limit the same convergence rate that the case where $f$ is strongly convex \Cref{assum:strongly_cvx}.

In the general case $r_2 \neq 0$, the convergence rates obtained in
\Cref{thm:f_weak} depend on $\beta$ where
$(\expeLigne{\normLigne{\bfX_t - x^{\star}}^{r_2r_3}}(\gua +
t)^{-\beta})_{ t\geq 0}$ has at most logarithmic growth.  If
$\beta \neq 0$, then the convergence rates deteriorate.  In what
follows, we shall consider different scenarios under which $\beta$ can
be explicitly controlled.  These estimates imply explicit convergence
rates for SGD using \Cref{thm:f_weak}.

\begin{corollary_colt}
  \label{cor:let-alpha-gamma}
  Let $\alpha, \gamma \in \ooint{0,1}$ and $(\bfX_t)_{t \geq 0}$ given by
  \eqref{eq:sde}. Assume $f \in \rmc^2(\rset^{\dim}, \rset)$,
  \rref{assum:f_lip}, \tup{\rref{assum:grad_sto}-\ref{item:approx_sto}},
  \rref{assum:lip_sigma}.
  \begin{enumerate}[wide, labelwidth=!, labelindent=0pt,label=(\alph*),noitemsep,nolistsep]
  \item \label{item:a_weak} If \rref{assum:f_weak_quasi} holds, then there exists $\Cweak \geq 0$ such that for any $T \geq 1$
  \begin{equation}
    \expe{f(\bfX_T)}-{\textstyle \min_{\rset^d}} f \leq \Cweak \parentheseDeuxLigne{T^{(1 - 3\alpha)/2} + T^{-\alpha/2} + T^{\alpha -1}} \eqsp .
  \end{equation}
\item \label{item:b_weak} If \rref{assum:f_weak_quasi} holds and there exist
  $R \geq 0$ and $c> 0$ such that for any $x \in \rset^{\dim}$ with
  $\| x -x^{\star} \| \geq R$, $f(x) - f(x^{\star}) \geq c \| x- x^{\star} \|$
  then there exists $\Cweak \geq 0$ such that for any $T \geq 1$
  \begin{equation}
    \label{eq:item:b_weak}
          \expe{f(\bfX_T)}-{\textstyle \min_{\rset^d}} f \leq \Cweak \parentheseDeuxLigne{T^{-\alpha/2} + T^{\alpha -1}} \eqsp .
    \end{equation}
  \item \label{item:c_weak} If \rref{assum:f_weak} holds and if there exist
    $R \geq 0$ and $\mtt > 0$ such that for any $x \in \rset^{\dim}$ with
    $\| x - x^\star\| \geq R$,
    $\langle \nabla f(x), x- x^{\star}\rangle \geq \mtt \norm{x-x^{\star}}^2$,
    then there exists $\Cweak \geq 0$ such that for any $T \geq 1$,
    \eqref{eq:item:b_weak} holds.
  \end{enumerate}
\end{corollary_colt}
The proof is postponed to \Cref{sec:weakly-convex_appendix}.  The
main ingredient of the proof is to control the growth of
$t \mapsto \expeLigne{\| \bfX_t - x^{\star} \|^{2}}$ using either the \SDE \
satisfied by $(\| \bfX_t - x^{\star} \|^{2})_{t \geq 0}$ in the case of (a) and
(c), or the \SDE \ satisfied by $(f(\bfX_t) - {\textstyle \min_{\rset^d}} f)_{t \geq 0}$
in the case of (b).

Under \Cref{assum:f_weak_quasi}, we compare the rates we obtain using
\Cref{cor:let-alpha-gamma}-\ref{item:a_weak} with the ones derived by
\citep{orvieto2019continuous} in \Cref{table:orvieto} and \Cref{fig:rates}. Note
that compared to \citep{orvieto2019continuous}, we establish that SGD converges
as soon as $\alpha > 1/3$ and not $\alpha >1/2$. In addition, the convergence
rates we obtain are always better than the ones of
\citep{orvieto2019continuous}. However, note that in both cases, the optimal
convergence rate is $1/3$ obtained using $\alpha = 2/3$. In addition, under
additional growth conditions on the function $f$, and using
\Cref{cor:let-alpha-gamma}-\ref{item:b_weak}-\ref{item:c_weak} we show that the
convergence of SGD in the weak quasi-convex case occurs as soon as $\alpha > 0$.
\begin{table}
    \vskip 0.15in
    \centering
    \small
  \begin{tabular}{|c||c|c|c|}
    \hline
    Reference & \Cref{cor:let-alpha-gamma}-\ref{item:a_weak} &  \Cref{cor:let-alpha-gamma}-\ref{item:b_weak} & \citepalias{orvieto2019continuous} \\ \hline
    $\alpha \in \ooint{0,1/3}$ & $\times$ & $\alpha/2$ &$\times$ \\ \hline
    $\alpha \in \ooint{1/3,1/2}$ & $(3\alpha -1)/2$ & $\alpha/2$ & $\times$  \\ \hline
    $\alpha = 1/2$ & $1/4 + \text{log.}$ & $1/4 + \text{log.}$ &$\times$ \\ \hline
    $\alpha \in \ooint{1/2, 2/3}$ & $\alpha /2 $ & $1-\alpha$&$2\alpha - 1$ \\ \hline
    $\alpha \in \ooint{2/3, 1}$ & $1-\alpha$ & $1-\alpha$ &$1-\alpha$ \\ \hline
  \end{tabular}
  	\caption{Rates for continuous SGD with non-convex assumptions \label{table:orvieto}}
\end{table}
Finally, as in the previous sections, we extend our main result to the discrete setting.
\begin{theorem}
  \label{thm:f_weak_discrete}
  Let $\alpha, \gamma \in \ooint{0,1}$ and $(X_n)_{n \in \nset}$ be given by
  \eqref{eq:sgd}. Assume \rref{assum:f_lip},
  \tup{\rref{assum:grad_sto}-\ref{item:approx_sto}} and \rref{assum:f_weak}.  In
  addition, assume that there exist $\beta, \explog, \Cbeta \geq 0$ such that
  for any $n \in \nset$,
  $\expe{\| X_n - x^{\star} \|^{r_2 r_3}} \leq \Cbeta(n+1)^{\beta}\{1 + \log(1 +
  n)\}^{\explog}$, where $r_3 = (1 - r_1/2)^{-1}$.  Then, there exists
  $\Cweak \geq 0$ (explicit and given in the proof) such that for any $N \geq 1$
  \begin{equation}
    \expe{f(X_N)}-{\textstyle \min_{\rset^d}} f \leq \Cweak N^{-\delta_1\wedge \delta_2} \parenthese{1 + \log(1+N))}^{\vareps} \eqsp ,
  \end{equation}
  where  $\delta_1,\delta_2$ are given in \eqref{eq:delta}.
\end{theorem}
The proof is postponed to \Cref{sec:weakly-convex_appendix}.  We can conduct the
same discussion as the one after \Cref{thm:f_weak}, and
\Cref{cor:let-alpha-gamma} can be extended to the discrete case, see
\Cref{cor:let-alpha-gamma_disc} in \Cref{sec:weakly-convex_appendix}.


%% file: graph_tikz.tex
\begin{figure}
	\begin{minipage}[b]{0.49\linewidth}
		\begin{center}
    \scalebox{.7}{
		\begin{tikzpicture}[domain=0:1]
			\begin{axis}
        [
        ,axis equal,
        ,axis x line=bottom
  			,axis y line=center
				,ylabel near ticks,
				,xlabel near ticks,
        ,width=9cm
        ,xlabel=$\text{value of $\alpha$}$
        ,ylabel=$\text{rate of convergence}$
        ,xmin=0
        ,xmax=1
        ,ymin=0
        ,ymax=1.1
        ,xtick={0,0.33,0.5,0.67,1}
        ,xticklabels={$0$,$1/3$,$1/2$,$2/3$,$1$},
        ,ytick={0,0.25,0.33,0.5,1},
        ,yticklabels={$0$,$1/4$,$1/3$,$1/2$,$1$}
        ,legend columns=2
        ,legend style={
        	at={(0.55,1)},
        	anchor=north,
        	/tikz/column 2/.style={
                column sep=5pt,
            }
        	}
        ]
        \addplot[densely dashdotted][domain=0.5:1]{x} ;
        \addlegendentry{strongly convex} ;
        \addplot[loosely dashed][domain=0:0.5]{1-x} ;
        \addlegendentry{deterministic} ;
        \addplot[line width=1.3pt][domain=0:0.5]{x} ;
        \addlegendentry{convex (Thm~\ref{thm:shamir_discrete})} ;
        \addplot[smooth][domain=0.5:0.67]{x/2} ;
        \addlegendentry{\citepalias[Table 1]{bachmoulines2011}} ;

        \addplot[smooth][domain=0.33:0.5]{3*x/2-1/2} ;
        \addplot[line width=1.3pt][domain=0.5:1]{1-x} ;
        \addplot[smooth][domain=0.5:0.67]{x/2} ;
    \end{axis}
  \end{tikzpicture}
  }
	\end{center}
	\end{minipage} \hfill
  	\begin{minipage}[b]{0.49\linewidth}
		\begin{center}
    \scalebox{.7}{
		\begin{tikzpicture}[domain=0:1]
			\begin{axis}
        [
        ,axis equal,
        ,axis x line=bottom
  			,axis y line=center
				,ylabel near ticks,
				,xlabel near ticks,
        ,width=9cm
        ,xlabel=$\text{value of $\alpha$}$
        ,ylabel=$\text{rate of convergence}$
        ,xmin=0
        ,xmax=1
        ,ymin=0
        ,ymax=1.1
        ,xtick={0,0.33,0.5,0.67,1}
        ,xticklabels={$0$,$1/3$,$1/2$,$2/3$,$1$},
        ,ytick={0,0.25,0.33,0.5,1},
        ,yticklabels={$0$,$1/4$,$1/3$,$1/2$,$1$}
        ,legend columns=2
        ,legend style={
        	at={(0.55,1)},
        	anchor=north,
        	/tikz/column 2/.style={
                column sep=5pt,
            }
        	}
        ]
        \addplot[densely dashdotted][domain=0:1]{x} ;
        \addlegendentry{strongly convex} ;
        \addplot[loosely dashed][domain=0:0.67]{1-x} ;
        \addlegendentry{deterministic} ;
        \addplot[line width=1.3pt][domain=0.67:1]{1-x} ;
        \addlegendentry{Cor.\ref{cor:let-alpha-gamma}-\ref{item:a_weak}} ;
        \addplot[dashed, line width=1.3pt][domain=0:0.5]{x/2} ;
        \addlegendentry{Cor.\ref{cor:let-alpha-gamma}-\ref{item:b_weak}} ;
        \addplot[solid][domain=0.5:0.67]{2*x-1} ;
        \addlegendentry{\citepalias[Table 1]{orvieto2019continuous}} ;

        \addplot[line width=1.3pt][domain=0.33:0.5]{3*x/2-1/2} ;
        \addplot[line width=1.3pt][domain=0.5:0.67]{x/2} ;
    \end{axis}
  \end{tikzpicture}
  }
	\end{center}
	\end{minipage}
	\caption{Comparison of convergence rates in convex (left) and weakly quasi-convex (right) settings. \label{fig:rates}}
\end{figure}
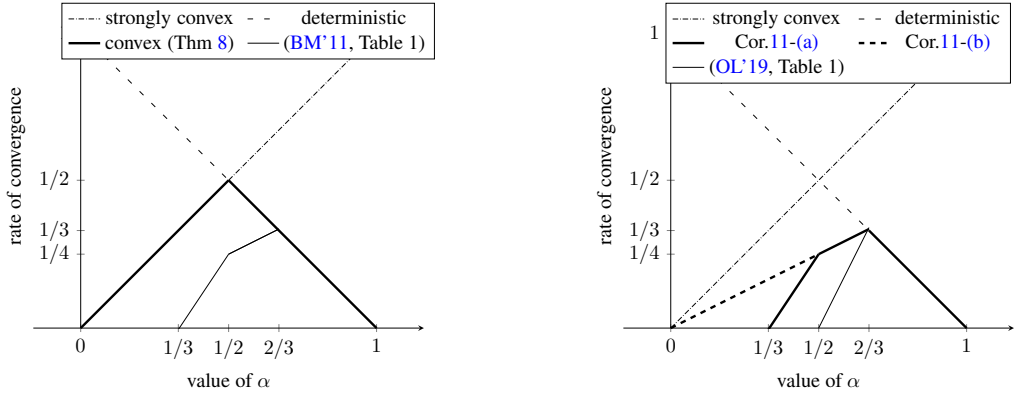


%% file: discussion.tex

\section{Conclusion}
\label{sec:discussion}

In this paper we investigated the connection between SGD and solutions of
appropriate time inhomogenuous SDEs.  We first proved approximation bounds
between these two processes motivating convergence analysis of continuous
SGD. Then, we turned to the convergence behavior of SGD and showed how the
continuous process can provide a better understanding of SGD using tools from
ODE analysis and stochastic calculus. In particular, we obtained optimal convergence
rates in the strongly convex case and new optimal convergence rates in the convex case. In the non-convex setting, we
considered a relaxation of the weakly quasi-convex condition and improved the
state-of-the art convergence rates in both the continuous and discrete-time
setting.


%% file: appendix_control.tex

\addtocontents{toc}{\protect\setcounter{tocdepth}{2}}

\section{Organization of the appendix}

In these appendices we gather the proofs of our results. We start by deriving
strong and weak approximation bounds in \Cref{sec:appr-results}. We then turn to
the study of the long-time behavior of SGD and its continuous-time counterpart
for the minimization of strongly convex functions in
\Cref{sec:strongly-convex_appendix} under
\Cref{assum:grad_sto}-\ref{item:approx_sto}.  The counterpart of these results
in the case where \Cref{assum:grad_sto}-\ref{item:ml} holds is presented in
\Cref{app:fz_strongly}. In \Cref{sec:convex-case_appendix}, we analyse the
convex setting under \Cref{assum:grad_sto}-\ref{item:approx_sto}. Again, the
counterpart of these results in the case where
\Cref{assum:grad_sto}-\ref{item:ml} holds is given in
\Cref{app:fz_convex}. We conclude with the proofs of the weakly quasi-convex
setting in \Cref{sec:weakly-convex_appendix}.

\tableofcontents

\section{Approximation Results}
\label{sec:appr-results}

In this section, we present the proof of our strong and weak approximation
results. In \Cref{sec:technical-lemmas}, we gather technical lemmas which will
be of use throughout the section. Our coupling construction between the
discrete-time and continuous processes is presented in \Cref{sec:coupling}. In
\Cref{sec:moment-bounds} we provide moment bounds which constitute the first
step towards deriving the strong approximation bounds in
\Cref{sec:mean-square-appr}. The refinement of our theorem in the presence of
batch-noise is considered in \Cref{sec:case-batch-noise}. We also derive weak
approximation bounds in \Cref{sec:weak-approximation}. Throughout this section
we will denote all the constants by the letter $\tta$ followed by some
subscript.


\subsection{Technical Lemmas}
\label{sec:technical-lemmas}

In order to derive the finite-time strong approximations from the one-step
approximations we will make use of the following version of the discrete
Grönwall's lemma.
\begin{lemma_colt}
  \label{lemma:gronwall}
  Let $(u_n)_{n \in \nset}$, $(v_n)_{n \in \nset}$ and $(w_n)_{n \in \nset}$
  such that for any $n \in \nset$, $u_n, v_n, w_n \geq 0$ and
  $u_{n+1} \leq (1 + v_n) u_n + w_n$. Then for any $n \in \nset$
  \begin{equation}
    u_n \leq \exp \parentheseDeux{\sum_{k=0}^{n-1} v_k} \parenthese{u_0 + \sum_{k=0}^{n-1} w_k} \eqsp .
  \end{equation}
\end{lemma_colt}
\begin{proof}
  The proof is a straightforward consequence of the discrete Grönwall's lemma.
\end{proof}
The sums appearing in \Cref{lemma:gronwall} will be controlled with the following lemma.
\begin{lemma_colt}
  \label{lemma:sum}
  Let $r > 0$, $\gamma >0$, $\alpha \in \coint{0,1}$ and
  $\gua = \gamma^{1/(1-\alpha)}$. Then for any $T \geq 0$, there exists
  $\Aar \geq 0$ such that for any $N \in \nset$ with $N\gua \leq T$ we have
  \begin{equation}
    \gamma^r \sum_{k=0}^{N-1} (k+1)^{-\alpha r} \leq \left\lbrace \begin{aligned}
        & \Aar \gamma^r (1 + \log(\gamma^{-1}))(1 + \log(T)) \eqsp , \qquad \text{if } \alpha \geq  1/r \eqsp ,\\
        & \Aar \gamma^r \gua^{\alpha r-1 } T^{1 - \alpha r }\eqsp , \qquad \qquad  \eqsp  \qquad  \qquad \text{otherwise} \eqsp .
      \end{aligned} \right.
  \end{equation}
\end{lemma_colt}
\begin{proof}
  Let $r > 0$, $\gamma >0$ and $\alpha \in \coint{0, 1}$. If $\alpha > 1/r$ then there exists $\Aar \geq 0$  such that
\begin{equation}
  \gamma^{r} \sum_{k =0}^{N-1} (k+1)^{-\alpha r} \leq \Aar \gamma^r \eqsp .
\end{equation}
If $\alpha < 1/r$ then there exists $\Aar \geq 0$  such that
\begin{equation}
  \gamma^{r} \sum_{k=0}^{N-1} (k+1)^{-\alpha r } \leq \Aar \gamma^r N^{-\alpha r + 1} \leq \Aar \gamma^r \gua^{\alpha r-1} T^{1 - \alpha r }  \eqsp .
\end{equation}
if $\alpha = 1/r$ then there exists $\Aar \geq 0$  such that
\begin{equation}
  \gamma^r \sum_{k=0}^{N-1}(k+1)^{-\alpha r} \leq \gamma^r(1+\log(N))\leq \Aar \gamma^r (1+\log(T))(1 + \log(\gamma^{-1})) \eqsp .
\end{equation}
\end{proof}
Note that if $r = 1$ then
$\gamma^r \sum_{k=0}^{N-1} (k+1)^{-\alpha r} \leq \Capprox_{\alpha,1}
T^{1 - \alpha}$.  Using a slight modification of \Cref{lemma:sum} we
also obtain that there exists $\tilde{\Capprox}$ such that if $r = 1$
then
$\gamma^r \sum_{k=0}^{N-1} (k+1)^{-\alpha r} \leq T^{1 - \alpha} + \tilde{\Capprox}$.

The following lemma derives upper-bound from the regularity assumption
\Cref{assum:f_lip} and \Cref{assum:lip_sigma}.

\begin{lemma_colt}
  \label{lemma:bound_f_sig}
  Assume \rref{assum:f_lip} and \rref{assum:lip_sigma}. Then there exists
  $C \geq 0$ such that for any $x \in \rset^d$,
  \begin{equation}
    \norm{\nabla f(x)} \leq C (1 + \norm{x}) \eqsp , \qquad \normLigne{\Sigma^{1/2}(x)} \leq C ( 1 + \norm{x}) \eqsp , \qquad \norm{\Sigma(x)} \leq C ( 1 + \norm{x}^2) \eqsp . 
  \end{equation}
\end{lemma_colt}

\begin{proof}
  First, we have for any $x \in \rset^d$ using \rref{assum:f_lip}
  \begin{equation}
    \norm{\nabla f(x)} \leq \norm{\nabla f(0)} + \Lip \norm{x} \leq (\norm{\nabla f(0)} + \Lip)(1 + \norm{x}) \eqsp . 
  \end{equation}
  Similarly, we have for any $x \in \rset^d$ using \rref{assum:lip_sigma},
  \begin{equation}
    \label{eq:sig_sqrt}
    \normLigne{\Sigma^{1/2}(x)} \leq (\norm{\nabla f(0)} + \Mip)(1 + \norm{x}) \eqsp . 
  \end{equation}
  Denote for any $x \in \rset^d$, $(a_{i,j}(x))_{1 \leq i,j \leq d} = \Sigma(x)$
  and $(b_{i,j}(x))_{1 \leq i,j \leq d} = \Sigma(x)^{1/2}$. Using the fact that for
  any $u, v\in \rset$, $2uv \leq u^2 + v^2$ we get that for any $x \in \rset^d$
  \begin{equation}
    \norm{\Sigma(x)} \leq \sum_{i,j=1}^d \abs{a_{i,j}(x)} \leq \sum_{i,j, k=1}^d \abs{b_{i,j}(x) b_{j,k}(x)} \leq 2d \normLigne{\Sigma^{1/2}(x)}^2 \eqsp .
  \end{equation}
  We conclude the proof upon combining this result and \eqref{eq:sig_sqrt}.
\end{proof}

\subsection{Construction of the coupling}
\label{sec:coupling}

In this section, we describe and prove the existence of an appropriate coupling
between the discrete-time and continuous-time process. In the following
sections, we always assume that $(\bfB_t)_{t \geq 0}$ and $(Z_n)_{n \in \nset}$
are given by \Cref{thm:coupling}. The proof of \Cref{thm:coupling} is based on
an abstract construction of an appropriate measure on a joint space. In order to
construct such a measure we use the gluing lemma \cite[Lemma
5.3.2]{ambrosio2008gradient} and tools from the optimal transport theory to
impose the desired properties on the marginals.

\begin{theorem}
  \label{thm:coupling}
  Assume \rref{assum:f_lip} and \rref{assum:lip_sigma}. Let
    $\alpha \in \coint{0,1}$, $\bgamma > 0$ and $\gamma \in \ocint{0,
      \bgamma}$. Then, there exists $(\bfB_t)_{t \geq 0}$ a $d$-dimensional
    Brownian motion and $(Z_n)_{n \in \nset}$ such that the following hold:
  \begin{enumerate}[wide, labelwidth=!, labelindent=0pt, label=(\alph*)]
  \item For any $k \in \nset$, $Z_{k+1}$ has distribution $\pi^Z$ and is
    independent from $\mck_k$, where for any $k \in \nset$
    \begin{equation}
    \mck_k = \sigma \parentheseLigne{\ensembleLigne{\bfB_t, Z_j}{t \in
        \ccint{0, k \gua}, \ j \in \{1, \dots, k\}}}   \eqsp ,    
  \end{equation}
  with $\mck_{0} = \{\emptyset, \Omega\}$. Similarly, for any $k \in \nset$,
  $(\bfB_t - \bfB_{k \gua})_{t \geq 0}$ is independent from $\mck_k$.
  \item For any $k \in \nset$, there exists $\msa_k \in \mcb{\rset^d}$ such that
    $\proba{\bfX_{k \gua} \in \msa_k} = 1$ and for any
    $\tilde{x} \in \msa_k$.
    \begin{equation}
      \wassersteinD[2]^2(\nu_k^{\rmd}(\tilde{x}), \nu_k^{\rmcc}(\tilde{x})) = \expe{\norm{H(\tilde{x}, Z_{k+1}) - \nabla f(\tilde{x}) - \gua^{-1/2} \Sigma^{1/2}(\tilde{x}) \int_{k \gua}^{(k+1)\gua} \rmd \bfB_s}^2}  \eqsp ,      
    \end{equation}
    where $(\bfX_t)_{t \geq 0}$ is a unique strong solution of \eqref{eq:sde}
    starting from $\bfX_0 = X_0 \in \rset^d$, $\nu^{\rmd}(\tilde{x})$ is the distribution
    of $H(\tilde{x}, Z_{0})$ and $\nu^{\rmcc}(\tilde{x})$ is the distribution of
    $\nabla f(\tilde{x}) + \Sigma^{1/2}(\tilde{x}) G$ with $G$ a Gaussian random
    variable with zero mean and identity covariance matrix.
  \end{enumerate}
\end{theorem}

\begin{proof}
  Let $\alpha \in \coint{0,1}$, $\bgamma > 0$ and
  $\gamma \in \ocint{0, \bgamma}$.  By recursion, we show that there exists
  $((\bfB_t^k)_{t \in \ccint{0, \gua}}, Z_k)_{k \in \nset}$ such that for any
  $k \in \nset$, the following assertion \Cref{assum:recu}($k$)
  is true.
  \begin{assumptionH}[$k$]
    \label{assum:recu}
    We have that $(\bfB_t^{k+1})_{t \in \ccint{0, \gua}}$ and $Z_{k+1}$ are independent
    from
    $\mch_{k} = \sigma \parentheseLigne{\ensembleLigne{\bfB_t^j, Z_j}{t \in
        \ccint{0, \gua}, \ j \in \{1, \dots, k-1\}}}$ (with
    $\mch_{0} = \{\emptyset, \Omega\}$) and there exists $\msa_k \in \mcb{\rset^d}$ such
    that $\proba{\bfY_{k \gua} \in \msa_k} = 1$ and for any $\tilde{x} \in \msa_k$.
  \begin{equation}
    \label{eq:coupling}
      \wassersteinD[2](\nu^{\rmd}(\tilde{x}), \nu^{\rmcc}(\tilde{x})) = \expe{\norm{H(\tilde{x}, Z_{k+1}) - \nabla f(\tilde{x}) - \gua^{-1/2} \Sigma^{1/2}(\tilde{x}) \int_{0}^{\gua} \rmd \bfB_s^{k+1}}^2}  \eqsp ,      
    \end{equation}
    where $\nu^{\rmd}(\tilde{x})$ is the distribution of $H(\tilde{x}, Z_{0})$
    and $\nu^{\rmcc}(\tilde{x})$ is the distribution of
    $\nabla f(\tilde{x}) + \Sigma^{1/2}(\tilde{x}) G$ with $G$ a Gaussian random
    variable with zero mean and identity covariance matrix, and for any
    $t \in \ccint{0 , k\gua}$
    \begin{multline}
      \label{eq:solution}
      \bfY_{t} = \bfX_0 + \sum_{j=0}^{n_t-1} \left( - \int_{0}^{\gua}
        ((j+1) \gua +s)^{-\alpha} \nabla f (\bfY_{j \gua +s}) \rmd s \right. \\ \left. -
        \gua^{1/2} \int_{0}^{\gua} ((j+1) \gua +s)^{-\alpha}
        \Sigma^{1/2}(\bfY_{j \gua + s}) \rmd \bfB_s^{j+1} \right) \\
       - \int_{0}^{t - n_t \gua}
        ((n_t+1) \gua +s)^{-\alpha} \nabla f (\bfY_{n_t \gua + s}) \rmd s  \\ -
        \gua^{1/2} \int_{0}^{t -n_t \gua} ((n_t+1) \gua +s)^{-\alpha}
        \Sigma^{1/2}(\bfY_{n_t \gua + s}) \rmd \bfB_s^{n_t+1}  \eqsp ,
      \end{multline}
      where $n_t = \floor{t / \gua}$. Denote $\mu_k$ the distribution of $((\bfB_t^j)_{t \in \ccint{0, \gua}}, Z_j)_{j \in \{0, \dots, k\}}$.
  \end{assumptionH}

  We denote $\pi^B \in \Pens(\rmc_{\gua}) = \Pens(\rmc(\ccint{0, \gua}, \rset^d))$ the
  distribution of the Brownian motion up to time $\gua$. For any $k \in \nset$
  we denote $\mse^k = (\rmc_{\gua} \times \msz)^k$ and $\mse =
  \mse^1$. Similarly, we denote $\pi^Z \in \Pens(\msz)$ the distribution of $Z$. For any
  $k \in \nset$, let $F: \ \rset^d \times \rmc_{\gua}$ such that for any
  $x \in \rset^d$ and $\pi^B$-almost every $\rmw \in \rmc_{\gua}$ we have
  \begin{equation}
    F_{k}(x, \rmw) = \nabla f(x) + \gua^{-1/2} \Sigma^{1/2}(x) \int_{0}^{\gua}  \rmd \rmw_s \eqsp . 
  \end{equation}
  Since $x \mapsto F(x, \rmw)$ is continuous for $\pi^B$-almost every
  $\rmw \in \rmc_{\gua}$ by \Cref{assum:f_lip} and \Cref{assum:lip_sigma}, and
  for any $x \in \rset^d$, $\rmw \mapsto F(x, \rmw)$ is measurable, we get that
  $F$ is measurable using \cite[Lemma 4.51]{charalambos2006infinite}. In
  addition, using \Cref{assum:f_lip}, \Cref{assum:lip_sigma} and \cite[Theorem
  10.4]{rogers2000diffusions}, for any $k \in \nset$, there exists a measurable
  mapping $\tilde{S}_{k+1}: \ \rset^d \times \rmc_{\gua} \to \rset^d$ such that
  for any Brownian motion $(\bfB_t)_{t \in \ccint{0, \gua}}$,
  $\tilde{S}_{k+1}(\tilde{x}, (\bfB_t)_{t \in \ccint{0,\gua}}) = \bfY_{\gua}^k$,
  where $(\bfY_{t}^k)_{t \in \ccint{0, \gua}}$ is the unique strong solution to
  the following SDE: for any $t \in \ccint{0, \gua}$
  \begin{equation}
    \bfY_t^k = \tilde{x} - \int_0^{t} ((k+1) \gua + s)^{-\alpha} \nabla f(\bfY_s^k) \rmd s - \gua^{1/2} \int_0^{t}((k+1) \gua +s)^{-\alpha} \Sigma^{1/2}(\bfY_s^k) \rmd \bfB_s \eqsp . 
  \end{equation}
  In addition, let $\tilde{S}_{0}: \ \rset^d \times \rmc_{\gua} \to \rset^d$
  such that for any $\rmw \in \rmc_{\gua}$,
  $\tilde{S}_{0}(\tilde{x}, \rmw) = \tilde{x}$.  For any $k \in \nset$, denote
  $S_{k} : \ \mse^{k+1} \to \rset^d$ such that for any
  $\{(\rmw_t^j)_{t \in \ccint{0, \gua}}, Z_j\}_{j=0}^{k} \in \mse^{k}$, we have
  \begin{multline}
    S_{k}(\{(\rmw_t^j)_{t \in \ccint{0, \gua}}, Z_j\}_{j=1}^{k}) \\=
    \tilde{S}_{k}(\tilde{S}_{k-1}(\dots(\tilde{S}_1(\tilde{S}_0(\bfX_0,
    (\rmw_t^0)_{t \in \ccint{0, \gua}}), (\rmw_t^1)_{t \in \ccint{0, \gua}}))
    \dots, (\rmw_t^{k-1})_{t \in \ccint{0, \gua}}), (\rmw_t^{k})_{t \in
      \ccint{0, \gua}})) \eqsp .
  \end{multline}
  Finally, for any $k \in \nset$, let
  $\Fens_{k}: \ \mse^{k+1} \times \rmc_{\gua} \to \rset^d$ and
  $\Hens_{k}: \ \mse^{k+1} \times \msz \to \rset^d$ such that for any $k \in \nset$,
  $\{(\rmw_t^j)_{t \in \ccint{0, \gua}}, Z_j\}_{j=0}^k \in \mse^{k+1}$,
  $(\rmw_t)_{t \in \ccint{0, \gua}} \in \rmc_{\gua}$ and $z \in \msz$
  \begin{align}
    \Fens_{k}(\{(\rmw_t^j)_{t \in \ccint{0, \gua}}, Z_j\}_{j=0}^{k}, (\rmw_t)_{t \in \ccint{0, \gua}} ) &= F(S_{k}(\{(\rmw_t^j)_{t \in \ccint{0, \gua}}, Z_j\}_{j=0}^k), (\rmw_t)_{t \in \ccint{0, \gua}}) \eqsp , \\
    \Hens_{k}(\{(\rmw_t^j)_{t \in \ccint{0, \gua}}, Z_j\}_{j=0}^{k}, z ) &= H(S_{k}(\{(\rmw_t^j)_{t \in \ccint{0, \gua}}, Z_j\}_{j=0}^k), z) \eqsp .
  \end{align}
  Note that for any $k \in \nset$, $\Fens_k$ and $\Hens_k$ are
  measurable. For any $k \in \nset$, let
  $\Qker_k: \ \mse \times \mcb{\rset^d} \to \ccint{0,1}$, the Markov kernel
  given for any $u \in \mse^{k+1}$ and $\msa \in \mcb{\rset^d \times \rset^d}$ by
  \begin{equation}
    \Qker_k(u, \msa) = \mathtt{Opt}(\Fens_k(u, \cdot)_{\#} \pi^B, \Hens_k(u, \cdot)_{\#} \pi^Z)(\msa)  \eqsp ,
  \end{equation}
  where for all $u \in \mse^{k+1}$,
  $\mathtt{Opt}(\Fens_k(u, \cdot)_{\#} \pi^B, \Hens_k(u, \cdot)_{\#} \pi^Z)$ is
  the optimal transference plan between $\Fens_k(u, \cdot)_{\#} \pi^B$ and
  $\Hens_k(u, \cdot)_{\#} \pi^Z$ w.r.t. to the $\wassersteinD[2]$, which exists
  by \cite[Theorem 4.1]{villani2009optimal}.  Note that for any $k \in \nset$,
  $\Qker_k$ is well-defined since
  $u \mapsto \mathtt{Opt}(\Fens_k(u, \cdot)_{\#} \pi^B, \Hens_k(u, \cdot)_{\#}
  \pi^Z)$ is measurable, \cite[Corollary 5.22]{villani2009optimal}.

  We divide the rest of the proof into two parts. First, we show by recursion
  that for any $k \in \nset$ the assertion \Cref{assum:recu}($k$) is
  true. Second, we show that we can construct a Brownian motion from the random
  variables introduced in \Cref{assum:recu}($k$) for any $k \in \nset$ such that
  the proposition holds.
  \begin{enumerate}[wide, labelwidth=!, labelindent=0pt, label=(\alph*)]
  \item We start by proving that \Cref{assum:recu}($0$) holds. Let
    $\mu_0 \in \Pens(\rmc_{\gua} \times \msz)$ be any coupling between
    $\pi^B$ and $\pi^Z$. Let
    $\eta^0 \in \Pens(\mse \times \rmc_{\gua} \times \rset^d \times \mse)$,
    $\eta^1 \in \Pens(\rset^d \times \mse \times \rset^d)$ and
    $\eta^2 \in \Pens(\mse \times \rset^d \times \mse \times \msz)$ such
    that
  \begin{equation}
    \eta^0 = (\Id, (\Fens_0, \Pi_1))_{\#} (\mu_0 \otimes \pi^B) \eqsp , \qquad \eta^2 = ((\Pi_1, \Hens_0), \Id)_{\#} (\mu_0 \otimes \pi^Z) \eqsp ,
  \end{equation}
  where $\Pi_1$ is the projection on the first variable. In addition, let
  $\eta^1 = \mu_0 \otimes \Qker_0$, \ie, for any $\msa \in \mcb{E}$ and $\msb_1, \msb_2 \in \mcb{\rset^d}$ we have
  \begin{equation}
    \eta^1(\msb_1 \times \msa \times \msb_2) = \int_{\msa} \Qker_0(x, \msb_1 \times \msb_2) \rmd \mu_0(x) \eqsp .  
  \end{equation}
  Note that $\eta^0_1 = (\Fens_0, \Id)_\# (\mu_0 \otimes \pi^B) =
  \eta^1_{12}$. Therefore, using the gluing lemma \cite[Lemma
  5.3.2]{ambrosio2008gradient} (which is valid since $\mse$, $\rmc_{\gua}$ and
  $\msz$ are Polish spaces), there exists a probability measure
  $\tilde{\eta}\in \Pens(\mse \times \rmc_{\gua} \times \rset^d \times \mse
  \times \rset^d)$ such that $\tilde{\eta}_{1234} = \eta^0$ and
  $\tilde{\eta}_{345} = \eta^1$. In addition note that
  $\tilde{\eta}_{45} = (\Id, \Hens_0)_{\#}(\mu_0 \otimes \pi^Z) =
  \eta^2_{12}$. Therefore, using the gluing lemma, there exists a probability
  measure
  $\eta \in \Pens(\mse \times \rmc_{\gua} \times \rset^d \times \mse \times
  \rset^d \times \msz \times \mse)$ such that $\eta_{12345} = \tilde{\eta}$ and
  $\eta_{4567} = \eta^1$. In particular, $\eta^{1234} = \eta^0$ and
  $\eta^{4567} = \eta^1$.  Let $(U_i)_{i \in \{1, \dots, 7\}}$ be a random
  variable with distribution $\eta$. Then, using that $\eta^{1234} = \eta^0$ and
  $\eta^{4567} = \eta^1$, we have almost surely
  \begin{equation}
    U_3 = \Fens_0(U_1, U_2) \eqsp , \qquad U_4 = U_1 \eqsp, \qquad U_4 = U_7 \eqsp , \qquad U_5 = \Hens_0(U_7, U_6) \eqsp .
  \end{equation}
  Therefore, we get that
  \begin{equation}
    (U_1, \dots, U_7) = (U_1, U_2, \Fens_0(U_1, U_2), U_1, \Hens_0(U_1, U_6), U_6, U_1) \eqsp .
  \end{equation}
  Since $U_2$ is independent from $U_1$ and $U_6$ is independent from $U_7$, we
  get that $U_6$ is independent from $U_1$.  Hence, there exists
  $\mu_1 \in \Pens(\mse \times \rmc_{\gua} \times \msz)$ such that
  \begin{equation}
    \eta = (\Pi_1, \Pi_2, \Fens_0(\Pi_1, \Pi_2), \Pi_1, \Hens_0(\Pi_1, \Pi_3), \Pi_3, \Pi_1)_{\#} \mu_1 \eqsp . 
  \end{equation}
  Let
  $((\bfB_t^0)_{t \in \ccint{0, \gua}}, Z_0, (\bfB_t^1)_{t \in \ccint{0, \gua}},
  Z_1)$ be a random variable with distribution $\mu_1$. Then
  $(\bfB_t^1)_{t \in \ccint{0, \gua}}$ and $Z_1$ are independent from
  $((\bfB_t^0)_{t \in \ccint{0, \gua}}, Z_0)$,
  $(\bfB_t^1)_{t \in \ccint{0, \gua}}$ has distribution $\pi^B$ and $Z_1$ has
  distribution $\pi^Z$. Hence, $(\bfB_t^1)_{t \in \ccint{0, \gua}}$ and $Z_1$
  are independent from $\mch_0$. Finally, we show that \eqref{eq:coupling} holds.
  Denote by $\Rker_0$ the Markov kernel
  given for any $u \in \mse$ and $\msa \in \mcb{\rset^d \times \rset^d}$ by
  \begin{equation}
    \Rker_0(u, \msa) = \int_{\rset^d} (\Fens_0(u, \Pi_2(\cdot)), \Hens_0(u, \Pi_3(\cdot))))_{\#} \mu_1  \rmd \mu_0(u) \eqsp ,
  \end{equation}
  Note that $\mu_1 = \mu_0 \otimes \Rker_0$. But by definition of $\mu_1$ we
  also have that $\mu_1 = \mu_0 \otimes \Qker_0$. Therefore, $\mu_0$ almost
  surely we have $\Rker_0(x, \cdot) = \Qker_0(x, \cdot)$, which concludes the
  proof of \eqref{eq:coupling}. Therefore \Cref{assum:recu}($0$) holds. Assume
  that \Cref{assum:recu}($k$) is true with $k \in \nset$. Then
  \Cref{assum:recu}($k+1$) holds. The proof is similar to the one for
  \Cref{assum:recu}($0$) upon replacing $\mu_0$ by $\mu_k$, $\Fens_0$ by
  $\Fens_k$, $\Hens_0$ by $\Hens_k$ and $\Qker_0$ by $\Qker_k$. We conclude by
  recursion.
\item Finally, it remains to define a Brownian motion $(\bfB_t)_{t \geq 0}$ such
  that for any $k \in \nset$, $\mck_k =\mch_k$ and
  $(\bfX_t)_{t \geq 0} = (\bfY_t)_{t \geq 0}$. For any $t \geq 0$, let
  $n_t = \floor{t / \gua}$ and $(\bfB_t)_{t \geq 0}$ such that for any
  $t \geq 0$
  \begin{equation}
    \bfB_t = \bfB^{n_t}_{(t - n_t)\gua} + \sum_{k=0}^{n_t - 1} \bfB_{\gua}^k \eqsp .
  \end{equation}
  Since $((\bfB_t^k)_{t \in \ccint{0, \gua}})_{k \in \nset}$ is a sequence of
  independent Brownian motion, we get that $(\bfB_t)_{t \geq 0}$ is a Brownian
  motion. In addition, there exists a measurable bijection mapping
  $(\bfB_t)_{t \geq 0}$ to $((\bfB_t^k)_{t \in \ccint{0, \gua}})_{k \in \nset}$
  and therefore for any $k \in \nset$, $\mck_k =\mch_k$. Finally, we have that
  $(\bfY_t)_{t \geq 0}$ solution to \eqref{eq:solution} is a solution to
  \eqref{eq:sde} with initial condition $\bfY_0 = \bfX_0$, \ie,
  $(\bfY_t)_{t \geq 0} = (\bfX_t)_{t \geq 0}$ which concludes the proof.
\end{enumerate}

\end{proof}

\subsection{Moment bounds and one-step approximation}
\label{sec:moment-bounds}

The following result is well-known in the field of \SDE \ but its proof is given
for completeness. For any $t \geq 0$ and $k \in \nset$, denote
$\mcf_t = \sigma(\ensembleLigne{\bfX_s}{s \in \ccint{0, t}})$ and
$\mcg_k = \sigma(\ensembleLigne{Z_j}{j \in \{0, \dots, k\}})$. We derive
classical moment bounds in \Cref{lemma:bound_moments}. This bounds are then used
in \Cref{lemma:mean_square_one} in order to provide one-step approximations. The
proof of these lemmas is based on the repeated application of the Gr\"{o}nwall's
lemma (both discrete and continuous) and It\^{o}'s formula.

\begin{lemma_colt}
  \label{lemma:bound_moments}
  Let $ p \in \nset$, $\bgamma >0$ and $\alpha \in \coint{0,1}$.  Assume
  \tup{\Cref{assum:f_lip}}, \tup{\Cref{assum:grad_sto}} and
  \rref{assum:lip_sigma}. Then for any $T \geq 0$, there exists $\ctun \geq 0$,
  such that for any $s \geq 0$ and $t \in \ccint{s, s + T}$,
  $\gamma \in \ocint{0, \bgamma}$, we have
  \begin{equation}
    \CPELigne{1 + \| \bfX_t \|^{2p}}{\mcf_s} \leq
\ctun (1+ \norm{\bfX_s}^{2p}) \eqsp ,
\end{equation}
where $(\bfX_t)_{t \geq 0}$ is the solution of \eqref{eq:sde}.  In addition, if
there exists $x^\star$ such that
$\int_{\msz} \norm{H(x^\star, z)}^{2p} \rmd \pi^Z(z) < +\infty$,
then 
for any $T \geq 0$, there exists $\cttun \geq 0$, such that for any $k_0 \geq 0$, $\gamma \in \ocint{0, \bgamma}$ and $k \in \{k_0, \dots, k_0 + N\}$ with $N \gua \leq T$, we have
  \begin{equation}
    \CPE{1 + \| X_k \|^{2p}}{\mcg_{k_0}} \leq
\cttun (1+ \norm{X_{k_0}}^{2p}) \eqsp ,
\end{equation}
where $(X_k)_{k \in \nset}$ satisfies the recursion \eqref{eq:sgd}.
\end{lemma_colt}

\begin{proof}
  We prove the result under \tup{\Cref{assum:grad_sto}-\ref{item:ml}}. The proof
  under \tup{\Cref{assum:grad_sto}-\ref{item:approx_sto}} is similar and left to
  the reader.  Let $p \in \nset$, $\alpha \in \coint{0,1}$,
  $s, T \in \coint{0, +\infty}$, $t \in \ccint{s, s +T}$, and
  $g_p \in \rmc^2(\rset^{\dim}, \coint{0, +\infty})$ such that for any
  $x \in \rset^{\dim}$, $g_p(x) = 1 + \norm{x}^{2p}$. Let $\bgamma>0$ and
  $\gamma \in \ocint{0,\bgamma}$.

We divide the proof into two parts.
  \begin{enumerate}[wide, labelwidth=!, labelindent=0pt, label=(\alph*)]
  \item Let $(\bfX_t)_{t \geq 0}$ be a solution to \eqref{eq:sde}. We have for any $x \in \rset^{\dim}$
\begin{equation}
  \label{eq:derivee_p}
  \nabla g_p(x) = 2p \norm{x}^{2(p-1)}x \eqsp , \qquad \nabla^2 g_p(x) = 4p(p-1) \norm{x}^{2(p-2)} x x^{\transpose} + 2p \norm{x}^{2(p-1)} \Id \eqsp .
\end{equation}
Let $n \in \nset$, and set $\tau_n = \inf \ensembleLigne{u \geq 0}{g_p(\bfX_u) > n}$. Applying Itô's lemma and using \eqref{eq:sde} and \eqref{eq:derivee_p} we get
\begin{multline}
  \CPE{g_p(\bfX_{t \wedge \tau_n})}{\mcf_s} - \CPE{g_p(\bfX_{s \wedge \tau_n})}{\mcf_s} \\ =  \CPE{\int_{s \wedge \tau_n}^{t \wedge \tau_n} -(\gua + u)^{-\alpha} \langle \nabla f(\bfX_u), \nabla g_p(\bfX_u) \rangle}{\mcf_s} \rmd u  \\ + (\gua /2) \CPE{ \int_{s \wedge \tau_n}^{t \wedge \tau_n} (\gua + u)^{-2\alpha} \langle \Sigma(\bfX_u), \nabla^2 g_p(\bfX_u) \rangle \rmd u }{\mcf_s} \eqsp .                                                                                                                                                  \label{eq:ito_approx}                                                                   
\end{multline}
Using \Cref{assum:f_lip}, \eqref{eq:derivee_p} and the Cauchy-Schwarz inequality we get that for any $u \in \ccintLigne{s, s + T}$
\begin{align}
  |\langle \nabla f(\bfX_u), \nabla g_p(\bfX_u) \rangle| &\leq 2p \| \bfX_u \|^{2(p-1)}\defEns{|\langle \nabla f(\bfX_u) - \nabla f(0) , \bfX_u\rangle| + \| \nabla f(0) \| \| \bfX_u \|} \\
                                                         &\leq 2p(\Lip + \| \nabla f(0) \|) g_p(\bfX_u) \eqsp .
  \label{eq:majo_un}                                                           
\end{align}
In addition, using \Cref{assum:f_lip}, \Cref{lemma:bound_f_sig},
\eqref{eq:derivee_p} and the Cauchy-Schwarz inequality we get that for any
$u \in \ccintLigne{s, s + T}$
\begin{align}
  \label{eq:majo_deux}
  \abs{\langle \Sigma(\bfX_u), \nabla^2 g_p(\bfX_u) \rangle} &\leq C (1 + \norm{\bfX_u}^2) \norm{\nabla^2 g_p(\bfX_u)} \\
                                                             &\leq C (1+\norm{\bfX_u}^2) (8p(p-1)d + 2p d)\norm{\bfX_u}^{2(p-1)} \\
  &\leq 4Cdp(4(p-1) + 1) g_p(\bfX_u) \eqsp .
\end{align}
Combining \eqref{eq:majo_un} and \eqref{eq:majo_deux} in \eqref{eq:ito_approx} we get for large enough $n \in \nset$
\begin{align}
  &\CPE{g_p(\bfX_{t \wedge \tau_n})}{\mcf_s} - g_p(\bfX_s) \\
  & \quad \leq  2p (\Lip + \| \nabla f(0) \|)  \CPE{\int_{s}^{t \wedge \tau_n} g_p(\bfX_u)  \rmd u}{\mcf_s}  + \bgua p(2p - 1)  \CPE{ \int_{s}^{t\wedge \tau_n}  g_p(\bfX_u)  \rmd u }{\mcf_s} \\
  & \quad \leq \defEns{2p (\Lip + \| \nabla f(0)\|) + 4\gua Cdp(4(p-1) + 1)} \int_{s}^{t}  \CPE{g_p(\bfX_{\wedge \tau_n})}{\mcf_s}  \rmd u \eqsp .
\end{align}
Using Grönwall's lemma we obtain
\begin{equation}
\CPE{g_p(\bfX_{t \wedge \tau_n})}{\mcf_s} \leq g_p(\bfX_s) \exp \parentheseDeux{ T \defEns{2p (\Lip + \| \nabla f(0) \|) + 4\gua Cdp(4(p-1) + 1)} } \eqsp .
 \end{equation}
 We conclude upon using Fatou's lemma and remarking that $\lim_n \tau_n = +\infty$, since $\bfX_t$ is well-defined for any $t \geq 0$.
\item Let $(X_k)_{k \in \nset}$ be a sequence which satisfies the recursion
  \eqref{eq:sgd}.  Let $A_k = X_k \gamma (k+1)^{-\alpha} $ and
  $B_k = - \gamma (k+1)^{-\alpha} \{\nabla f(X_k) - H(X_k, Z_{k+1})\}$.  We
  have, using Cauchy-Schwarz inequality and the binomial formula,
  \begin{align}
    \label{eq:first_bound}
    \norm{X_{k+1}}^{2p} &= \norm{ A_k + B_k}^{2p} = \defEns{\norm{A_k}^2 + 2\langle A_k, B_k \rangle + \norm{B_k}^2}^p \\
                        &\leq \sum_{i=0}^p \sum_{j=0}^i {p \choose i} {i \choose j} \norm{A_k}^{2(p-i)+j} \norm{B_k}^{2 i- j} \times 2^j \\
    &\leq \norm{A_k}^{2p} + 2^p \sum_{i=1}^p \sum_{j=0}^i {p \choose i} {i \choose j} \norm{A_k}^{2(p-i)+j} \norm{B_k}^{2 i- j} \eqsp .
  \end{align}
Using \Cref{assum:f_lip}, there exists $\cttun^{(a)}, \cttun^{(b)}, \cttun^{(c)} \geq 0$ such that for any $\ell \in \{0, \dots, 2p\}$
\begin{align}
  \norm{A_k}^{\ell}  &\leq \sum_{m=0}^{\ell} {\ell \choose m} (1 + \gamma (k+1)^{-\alpha}\Lip)^m \norm{X_k}^m \parenthese{\gamma (k+1)^{-\alpha} \norm{\nabla f(0) }}^{\ell -m} \\
                    &\leq (1 + \gamma (k+1)^{-\alpha} \cttun^{(a)} ) \norm{X_k}^{\ell} + \gamma (k+1)^{-\alpha} \cttun^{(b)} (1 + \norm{X_k}^{2p}) \\
                    &\leq (1 + \gamma (k+1)^{-\alpha} \cttun^{(c)}) (1 + \norm{X_k}^{2p}) \eqsp .
\end{align}
  In addition, we have that for any $\ell \in \{1, \dots, p\}$, $x \in \rset^d$
  and $z \in \msz$,
\begin{equation}
  \norm{H(x,z)}^{\ell} \leq (\norm{H(x, z)} + \Lip \norm{x})^{2\ell} \leq 2^{2\ell-1} \norm{H(x,z)}^{2\ell} + 2^{2\ell-1} \Lip^{2\ell} \norm{x}^{2\ell} \eqsp . 
\end{equation}
Therefore, there exists $\eta_{\ell} > 0$ such that for any $\ell \in \nset$,
$\CPE{\|B_k\|^{2\ell}}{\mcg_{k_0}} \leq \gamma^{2\ell} (k+1)^{-2\alpha \ell}
\eta_{\ell} (1 + \norm{X}^{2\ell})$.  Combining this result,
\eqref{eq:first_bound}, Jensen's inequality and that f we have
\begin{align}
  & \CPE{\|X_{k+1}\|^{2p}}{\mcg_{k}}  \leq  \norm{X_k}^{2p} + \gamma (k+1)^{-\alpha} (\cttun^{(a)} + \cttun^{(b)}) (1 + \norm{X_k}^{2p}) \\ & \qquad + 2^{p+1} (1 + \gamma (k+1)^{-\alpha} \cttun^{(c)}) (1 + \norm{X_k}^{2p}) \sum_{i=1}^p \sum_{j=0}^i {p \choose i} {i \choose j} \eta_{2i-j}^{1/2} \gamma^{2i - j}(k+1)^{-\alpha(2i-j)} \eqsp .
\end{align}
Therefore, there exists $\cttun^{(d)} \geq 0$ such that
\begin{equation}
  \CPE{1 + \norm{X_{k+1}}^{2p}}{\mcg_{k_0}} \leq (1 + \cttun^{(d)} \gamma (k+1)^{-\alpha}) \CPE{1 + \norm{X_{k}}^{2p}}{\mcg_{k_0}} + \cttun^{(d)} \gamma (k+1)^{-\alpha} \eqsp .
\end{equation}
We conclude combining this result, \Cref{lemma:gronwall} and \Cref{lemma:sum}.
 \end{enumerate}
\end{proof}

\begin{lemma_colt}
  \label{lemma:mean_square_one}
  Let $p \in \nset$, $\bgamma >0$ and $\alpha \in \coint{0,1}$. Assume
  \tup{\Cref{assum:f_lip}}, \tup{\Cref{assum:grad_sto}-\ref{item:ml}},
  \rref{assum:lip_sigma} and that there exists $x^\star$ such that
  $\int_{\msz} \norm{H(x^\star, z)}^{2p} \rmd \pi^Z(z) < +\infty$. Then for any
  $T \geq 0$, there exists $\ctdeux \geq 0$ such that for any
  $\gamma \in \ocint{0, \bgamma}$, $k \in \nset$ with $(k+1)\gua \leq T$,
  $t \in \ccintLigne{k \gua, (k+1)\gua}$, we have
  \begin{align}
    &\CPE{\| X_{k+1} - X_k\|^{2p}}{\mcg_k} \leq \ctdeux (k+1)^{-2\alpha p}\gamma^{2p} (1 + \norm{X_k}^{2p}) \eqsp , \\
    &\CPE{\| \bfX_{t} - \bfX_{k \gua} \|^{2p}}{\mcf_{k \gua}} \leq \ctdeux (k+1)^{-2\alpha p}\gamma^{2p} (1 + \norm{\bfX_{k \gua}}^{2p} \eqsp .
  \end{align}
  where $(X_k)_{k \in \nset}$ satisfies the recursion and  $(\bfX_t)_{t \geq 0}$ is the solution of \eqref{eq:sde}.
\end{lemma_colt}
\begin{proof}
  Let $p \in \nset$, $\alpha \in \coint{0,1}$, $\bgamma > 0$,
  $\gamma \in \ocint{0, \bgamma}$, $k \in \nset$,
  $t \in \ccintLigne{k\gua, (k+1)\gua}$. We divide the rest of the proof into
  two parts.
  \begin{enumerate}[wide, labelwidth=!, labelindent=0pt, label=(\alph*)]
  \item
    Let $(\bfX_s)_{s \geq 0}$ be a solution to \eqref{eq:sde}.
Using \Cref{assum:f_lip}, \Cref{assum:grad_sto}, Jensen's inequality, Burkholder-Davis-Gundy's inequality \citep[Theorem 42.1]{rogers2000diffusions} and \Cref{lemma:bound_moments} there exists $B_p \geq 0$ such that
\begin{align}
  \label{eq:numero_duo}
  & \CPE{\norm{\bfX_{t} - \bfX_{k \gua}}^{2p}}{\mcf_{k \gua}}  \\ & \leq 2^{2p -1} \CPE{\norm{\int_{k \gua}^{t}(\gua + s)^{-\alpha} \nabla f(\bfX_s) \rmd s}^{2p}}{\mcf_{k \gua}}
                                          \\  & \qquad + 2^{2p  -1} \gua^p \CPE{\norm{\int_{k \gua}^{t}(\gua + s)^{-\alpha} \Sigma(\bfX_s)^{1/2} \rmd \bfB_s}^{2p}}{\mcf_{k \gua}}  \\
  & \leq 2^{2p - 1} \gua^{2 p - 1} \int_{k \gua}^{t}(\gua + s)^{-2\alpha p} \CPE{\norm{\nabla f(\bfX_s)}^{2p}}{\mcf_{k \gua}}\rmd s  \\
  & \qquad + B_p 2^{2p - 1} \gua^{p } \parenthese{\int_{k \gua}^{t} (\gua + s)^{-2 \alpha } \CPE{\trace(\Sigma(\bfX_s))}{\mcf_{k \gua}} \rmd s}{\mcf_{k \gua}}^p \\
  & \leq 2^{2p -1}  \gua^{2p-1 - 2\alpha p} (k+1)^{-2\alpha p}(B_p+1)  \left\lbrace \int_{k \gua}^{t} \CPE{\norm{\nabla f(\bfX_s)}^{2p}}{\mcf_{k \gua}}\rmd s  \right. \\
  & \qquad \left. + \int_{k \gua}^{t} \CPE{\trace(\Sigma(\bfX_s))}{\mcf_{k \gua}}^p \rmd s \right\rbrace \\
  &\leq 2^{4p} (1+\Lip^{2p}) \gamma^{2p} \gua\pinv (k+1)^{-2\alpha p}  (B_p+1)  \int_{k \gua}^{t} C^{2p}\parenthese{1 + \CPE{\norm{\bfX_s}^{2p}}{\mcf_{k \gua}}} \rmd s    \\
  &\leq  2^{4p} (1+ \Lip^{2p})  \gamma^{2p} (k+1)^{-2\alpha p}  (B_p+1) \parenthese{1 +  \sup_{s \in \ccint{k\gua, t}}\CPE{\norm{\bfX_s}^{2p}}{\mcf_{k \gua}}} \\
  &\leq 2^{4p} (1+\Lip^{2p})  \gamma^{2p} (k+1)^{-2 \alpha p}(B_p+1)  \parenthese{1 +  \ctun} g_p(\bfX_{k \gua}) \eqsp .
\end{align}
\item Let $(X_n)_{n \in \nset}$ which satisfies the recursion
  \eqref{eq:sgd}. Using \Cref{assum:f_lip} and \Cref{assum:grad_sto}-\ref{item:ml} we get that
  \begin{align}
    \label{eq:inter_duo}
    &\CPE{\| X_{k+1} - X_k \|^{2p}}{\mcg_k} \\
    &\qquad= \CPE{\norm{-\gamma(k+1)^{-\alpha}(H(X_k, Z_{k+1}) - H(x^\star, Z_{k+1}) + H(x^\star, Z_{k+1}))}^{2p}}{\mcg_k} \\
    &\qquad \leq 2^{2p} \gamma^{2p} (k+1)^{-2\alpha p} \parenthese{\Lip^{2p} \norm{X_k - x^\star}^{2p} + \CPE{\norm{H(x^\star, Z_{k+1})}^{2p}}{\mcg_k}} \\
    &\qquad \leq 2^{4p} \gamma^{2p} (k+1)^{-2\alpha p} \parenthese{\Lip^{2p} \norm{X_k}^{2p} + \Lip^{2p}\norm{x^\star}^{2p}+ \CPE{\norm{H(x^\star, Z_{k+1})}^{2p}}{\mcg_k}} \eqsp ,
  \end{align}
  which concludes the proof.
\end{enumerate}
\end{proof}

\subsection{Mean-square approximation}
\label{sec:mean-square-appr}

In this section, we introduce an auxiliary process $(\bfXd_t)_{t \geq 0}$. This
process is a continuous interpolation of the discrete-time process
$(\bar{X}_k)_{k \in \nset}$ such that for any $k \in \nset$,
\begin{equation}
  \bar{X}_{k+1} = \bar{X}_k - \gamma(1+k)^{-\alpha} \defEns{\nabla f(\bar{X}_k) + \gua^{1/2} \Sigma^{1/2}(\bar{X}_k) G_{k+1}} \eqsp ,
\end{equation}
where $(G_k)_{k \in \nset}$ is a sequence of \iid \ Gaussian random variables
with zero mean. For any $x \in \rset^d$ and $k \in \nset$,
$\nabla f(x) + \gua^{1/2} \Sigma^{1/2}(x) G_{k+1}$ is a Gaussian approximation
of the true noise term $H(x, Z_{k+1})$. Using \Cref{thm:coupling}, $G_{k+1}$ and
$Z_{k+1}$ will be coupled in order to minimize the distance between the two
discrete-time processes.

We now introduce the continuous-time process $(\bfXd_t)_{t \geq 0}$. Consider
the stochastic process $(\bfXd_t)_{t \geq 0}$ defined by $\bfXd_0 = X_0$ and
solution of the following \SDE
\begin{equation}
  \label{eq:sde_disc}
  \rmd \bfXd_t = -\gua^{-1}\sum_{k=0}^{+\infty} \1_{\coint{k\gua, (k+1)\gua}}(t) (1 + k)^{-\alpha} \gamma \defEns{\nabla f(\bfXd_{k \gua}) \rmd t +  \gua^{1/2}\Sigma(\bfXd_{k \gua})^{1/2 }\rmd \bfB_t} \eqsp .
\end{equation}
Note that for any $k \in \nset$, we have
\begin{equation}
  \bfXd_{(k+1)\gua} = \bfXd_{k \gua} - \gamma (k+1)^{-\alpha} \defEns{\nabla f(\bfXd_{k \gua}) + \Sigma(\bfXd_{k \gua})^{1/2} G_k} \eqsp ,
\end{equation}
with $G_k = \gua^{-1/2} \intk \rmd \bfB_s$. Hence, for any $k \in \nset$,
$\bfXd_{k \gua}$ has the same distribution as $X_k$ given by \eqref{eq:sgd} with
$H(x,z) = \nabla f(x) + \Sigma(x)^{1/2}z$,
$(\msz, \mcz) = (\rset^{\dim}, \mcb{\rset^{\dim}})$ and $\muz$ the Gaussian
probability distribution with zero mean and covariance matrix identity. In our
proof, we will not consider this process but a similar version whose initial
point is given by the continuous-time process, see for instance. However, we
found that introducing $(\bfXd_t)_{t \geq 0}$ and its discrete-time counterpart
provides intuition for our derivation.

In \Cref{lemma:sgd_diff}, we bound the one-step difference between the
continuous-time auxiliary process $(\bfXd_t)_{t \geq 0}$ and the discrete-time
process $(X_k)_{k \in \nset}$. In \Cref{lemma:mean_square_diff}, we bound the
one-step difference between the continuous-time auxiliary process
$(\bfXd_t)_{t \geq 0}$ and the continuous-time process $(\bfX_t)_{t \geq 0}$. We
combine these estimates in \Cref{prop:mean_square_diff_fin}. We conclude by
proving \Cref{prop:strong_approx_appendix} which is a restatement of
\Cref{prop:strong_approx}.

Recall that $n_T = \floorLigne{T / \gua}$. In what follows, we denote
\begin{equation}
  \label{eq:vareps_def}
  \vareps^2 = \sup_{k \in \{0, \dots, n_T\}} \expe{\wassersteinD[2]^2(\nu_k^{\rmd}(\bfX_{k \gua}), \nu_k^{\rmcc}(\bfX_{k \gua}))} \eqsp , 
\end{equation}
where for any $\tilde{x} \in \rset^d$, $\nu_k^{\rmd}(\tilde{x})$ is
  the distribution of $H(\tilde{x}, Z_{n+1})$, $\nu_k^{\rmcc}(\tilde{x})$ is the
  distribution of
  $\nabla f(\tilde{x}) + \gua^{-1/2} \Sigma^{1/2}(\tilde{x}) \int_{k \gua}^{(k+1)
    \gua}(\gua + s)^{-\alpha} \rmd \bfB_s$.

\begin{lemma_colt}
  \label{lemma:sgd_diff}
  Assume \rref{assum:f_lip} and \rref{assum:lip_sigma}. Let $\bgamma >0$ and
  $\alpha \in \coint{0,1}$.  Then for any $T \geq 0$, there exists
  $\cttrois \geq 0$ such that for any $\gamma \in \ocint{0, \bgamma}$,
  $k \in \nset$ with $(k+1)\gua \leq T$ and $X_0 \in \rset^{\dim}$ we have
  \begin{equation}
    \expe{\| \bfXt_{(k+1)\gua}^k - \tilde{X}_{k+1} \|^{2}} \leq \cttrois \gamma^{2} (k+1)^{-2\alpha} \vareps^2  \eqsp ,
  \end{equation}  
  where for any $k \in \nset$ and $t \in \ccint{k \gua, (k+1) \gua}$
  \begin{align}
    &\tilde{X}_{k+1} = \bfX_{k \gua} - \gamma (k+1)^{-\alpha} H(\bfX_{k \gua}, Z_{k+1}) \eqsp , \\
    &\bfXt_{t} = \bfX_{k \gua} -  \gua^{-1}\gamma(1+k)^{-\alpha} \defEns{(t - k \gua) \nabla f(\bfX_{k \gua}) + \gua^{1/2} \Sigma^{1/2}(\bfX_{k \gua}) \int_{k \gua}^t \rmd \bfB_s}   \eqsp ,
  \end{align}
  We recall that for any $\tilde{x} \in \rset^d$, $\nu_k^{\rmd}(\tilde{x})$ is
  the distribution of $H(\tilde{x}, Z_{n+1})$, $\nu_k^{\rmcc}(\tilde{x})$ is the
  distribution of
  $\nabla f(\tilde{x}) + \gua^{-1/2} \Sigma^{1/2}(\tilde{x}) \int_{k \gua}^{(k+1)
    \gua}(\gua + s)^{-\alpha} \rmd \bfB_s$.
\end{lemma_colt}
\begin{proof}
  Let $\alpha \in \coint{0,1}$, $\bgamma > 0$, $\gamma \in \ocint{0, \bgamma}$,
  $k \in \nset$, $t \in \ccintLigne{k\gua, (k+1)\gua}$ and
  $X_0 \in \rset^{\dim}$. Using \Cref{thm:coupling} we
  have
  \begin{align}
    &\expe{\| \bfXd_{(k+1)\gua} - X_{k+1} \|^{2}} \\ & \qquad = \gamma^{2}(k+1)^{-2\alpha } \expe{\norm{\nabla f(\bfX_{k \gua}) + \Sigma^{1/2}(X_0)G_k - H(\bfX_{k \gua}, Z_k)}^{2}} \\
    & \qquad \leq 2\gamma^{2}(k+1)^{-2\alpha } \expe{\wassersteinD[2]^2(\nu_k^{\rmd}(\bfX_{k \gua}), \nu_k^{\rmcc}(\bfX_{k \gua}))}  \eqsp .
  \end{align}
which concludes the proof upon using \eqref{eq:vareps_def}.
\end{proof}

\begin{lemma_colt}
  \label{lemma:mean_square_diff}
  Let $\bgamma >0$ and $\alpha \in \coint{0,1}$. Assume
  \tup{\Cref{assum:f_lip}} and
  \tup{\Cref{assum:lip_sigma}}.  Then for any $T \geq 0$, there exists
  $\ctquatre \geq 0$ such that for any $\gamma \in \ocint{0, \bgamma}$,
  $k \in \nset$ with $(k+1)\gua \leq T$ and $X_0 \in \rset^{\dim}$ we have
  \begin{equation}
    \CPE{\| \bfX_{(k+1)\gua} - \bfXt_{(k+1)\gua} \|^{2}}{\mcf_{k \gua}} \leq \ctquatre \defEns{\gamma^{4}(k+1)^{-4\alpha } + \gamma^{2 } (k+1)^{-2(1+\alpha)}} (1 + \norm{\bfX_{k \gua}}^{2}) \eqsp ,
  \end{equation}
  where $(\bfX_t)_{t \geq 0}$ be the solution of \eqref{eq:sde} and for any
  $t \in \ccint{k \gua, (k+1) \gua}$ we have 
  \begin{equation}
    \label{eq:bfxt}
    \bfXt_{t} = \bfX_{k \gua} -  \gua^{-1}\gamma(1+k)^{-\alpha} \defEns{(t - k \gua) \nabla f(\bfX_{k \gua}) + \gua^{1/2} \Sigma^{1/2}(\bfX_{k \gua}) \int_{k \gua}^t \rmd \bfB_s}   \eqsp .
  \end{equation}
\end{lemma_colt}

\begin{proof}
  Let $\alpha \in \coint{0,1}$, $\bgamma > 0$, $\gamma \in \ocint{0, \bgamma}$,
  $k \in \nset$ and $t \in \ccintLigne{k\gua, (k+1)\gua}$.  Let
  $(\bfX_t)_{t \geq 0}$ is the solution of \eqref{eq:sde} and
  $(\bfXt_t)_{t \in \ccint{k \gua, (k+1) \gua}}$ given by \eqref{eq:bfxt}. Using
  Jensen's inequality and that $\gua \gamma^{-1} = \gua^{\alpha}$ we have
  \begin{align}
    \label{eq:strong_uno}
    &\CPE{\norm{\bfX_{(k+1)\gua} - \bfXt_{(k+1)\gua}}^{2}}{\mcf_{k\gua}} \\ 
    & \eqsp \leq \mathbb{E}\left[ \left\|-\intk (\gua + s)^{-\alpha} \nabla f (\bfX_s) \rmd s - \gua^{1/2} \intk (\gua + s)^{-\alpha}\Sigma(\bfX_s)^{1/2} \rmd \bfB_s \right. \right. \\ 
    & \left. \left. \qquad \qquad \left. + \gamma(k+1)^{-\alpha} \nabla f (\bfX_{k \gua}) + \gamma \gua^{-1/2} (k+1)^{-\alpha} \Sigma(\bfX_{k \gua})^{1/2}\intk \rmd \bfB_s \right\|^{2} \middle| \mcf_{k \gua} \right. \right] \\
    & \eqsp \leq 2\CPE{ \norm{-\gua^{-\alpha}\intk (1 + \gua^{-1}s)^{-\alpha} \nabla f (\bfX_s) \rmd s + \gamma(k+1)^{-\alpha} \nabla f (\bfX_{k \gua})}^{2}}{\mcf_{k \gua}} \\ 
    & \eqsp \qquad  +  2\mathbb{E}\left[ \left\|- \gua^{1/2 - \alpha} \intk (1 + \gua^{-1}s)^{-\alpha} \Sigma(\bfX_s)^{1/2} \rmd \bfB_s  \right. \right. \\ 
    &\eqsp \qquad  \qquad \left. \left. \left. + \gamma \gua^{-1/2} (k+1)^{-\alpha} \Sigma(\bfX_{k \gua})^{1/2}\intk \rmd \bfB_s \right\|^{2} \middle| \mcf_{k \gua} \right.  \right] \\
    & \eqsp \leq 2\gua^{-2\alpha}\CPE{ \norm{\intk \defEns{(k+1)^{-\alpha} \nabla f (\bfX_{k \gua})-(1 + \gua^{-1}s)^{-\alpha} \nabla f (\bfX_s)}\rmd s }^{2}}{\mcf_{k \gua}} \\ & \quad  +  2\gua^{1-2\alpha }\CPE{\norm{ \intk \defEns{ (k+1)^{-\alpha}\Sigma(\bfX_{k \gua})^{1/2} - (1 + \gua^{-1}s)^{-\alpha}\Sigma(\bfX_s)^{1/2}} \rmd \bfB_s }^{2} }{\mcf_{k \gua}} \eqsp .
  \end{align}
  We now treat each term separately.

Using Jensen's inequality, Itô isometry, Fubini-Tonelli's theorem, \Cref{assum:f_lip}, \Cref{assum:lip_sigma} and \Cref{lemma:mean_square_one} we have

\begin{align}
  \label{eq:strong_trio}
  &\CPE{\norm{ \intk \defEns{ (k+1)^{-\alpha}\Sigma(\bfX_{k \gua})^{1/2} - (1 + \gua^{-1}s)^{-\alpha}\Sigma(\bfX_s)^{1/2}} \rmd \bfB_s }^{2}}{\mcf_{k \gua}}  \\
  & \qquad \leq 2 \left[(k+1)^{-2\alpha }\abs{\intk \CPE{\| \Sigma(\bfX_{k \gua})^{1/2} - \Sigma(\bfX_s)^{1/2}\|^{2}}{\mcf_{k \gua}} \rmd s } \right. \\ 
  &\qquad \left.+ \eta \abs{\intk \{ (k+1)^{-\alpha} - (1+\gua^{-1}s)\}^2 \rmd s } \right] \\
  & \qquad \leq 2 \gua \parentheseDeux{(k+1)^{-2\alpha }\Mtt^2 \sup_{s \in [k\gua, (k+1)\gua]} \CPE{\| \bfX_s - \bfX_{k \gua} \|^{2}}{\mcf_{k \gua}} + \eta \alpha^{2} (k+1)^{-2(1+\alpha)} } \\
  & \qquad \leq 2\gua \parentheseDeux{(k+1)^{-4\alpha }\Mtt^2 \ctdeux \gamma^{2} + \eta \alpha^{2} (k+1)^{-2(1+\alpha)} } (1+\|\bfX_{k \gua}\|^{2}) \eqsp .
\end{align}
  
Using Jensen's inequality, Fubini-Tonelli's theorem, the fact that for any $u >0$, $u^{-\alpha} -(u+1)^{-\alpha} \leq \alpha u^{-(\alpha +1)}$, \Cref{assum:f_lip} and \Cref{lemma:mean_square_one} we get that
  \begin{align}
    \label{eq:strong_duo}
    &\CPE{ \norm{\intk \defEns{(k+1)^{-\alpha} \nabla f (\bfX_{k \gua})-(1 + \gua^{-1}s)^{-\alpha} \nabla f (\bfX_s)}\rmd s }^2}{\mcf_{k \gua}}
      \\ & \qquad \leq \gua^{2} \sup_{s \in [k\gua, (k+1)\gua]} \defEns{\expe{\norm{(k+1)^{-\alpha} \nabla f (\bfX_{k \gua})-(1 + \gua^{-1}s)^{-\alpha} \nabla f (\bfX_s)}^2}{\mcf_{k \gua}}\rmd s} \\
    & \qquad \leq 2\gua^{2} \sup_{s \in [k\gua, (k+1)\gua]} \left\lbrace\| \nabla f(\bfX_{k \gua}) \|^{2} |(k+1)^{-\alpha} - (1+\gua^{-1}s)^{-\alpha}|^{2} \right. \\
    &\qquad \qquad \left.+ (1+\gua s^{-1})^{-2\alpha }\expe{\| \nabla f(\bfX_s) - \nabla f(\bfX_{k \gua}) \|^{2}}{\mcf_{k \gua}}\right\rbrace \\
    & \qquad \leq 2\gua^{2} \left( \alpha^{2} \| \nabla f(\bfX_{k \gua}) \|^{2}(k+1)^{-2(1+\alpha)} \right. \\
    & \qquad \qquad \left. + (k+1)^{-2\alpha } \Lip^{2}\sup_{s \in [k\gua, (k+1)\gua]} \CPE{\| \bfX_s - \bfX_{k \gua} \|^{2}}{\mcf_{k \gua}}\right) \\
    & \qquad  \leq 2\gua^{2} \parentheseDeux{\alpha^{2} \| \nabla f(\bfX_{k \gua}) \|^{2}(k+1)^{-2(1+\alpha)} + (k+1)^{-4\alpha } \Lip^{2} \ctdeux \gamma^{2} (1+\norm{\bfX_{k \gua}}^{2})} \\
    & \qquad  \leq 2\gua^{2} \parentheseDeux{\alpha^{2} (\| \nabla f(0)\|^{2} + \Lip^{2}) (k+1)^{-2(1+\alpha)} + (k+1)^{-4\alpha } \Lip^{2} \ctdeux \gamma^{2} } (1+\norm{\bfX_{k \gua}}^{2}) \eqsp .
  \end{align}

Combining \eqref{eq:strong_uno}, \eqref{eq:strong_duo} and \eqref{eq:strong_trio} concludes the proof upon setting
  \begin{equation}
    \ctquatre = 4  \parentheseDeux{\Mtt^2 \ctdeux  + \eta \alpha^{2} + \alpha^{2} (\| \nabla f(0)\|^{2} + \Lip^{2}) + \Lip^{2} \ctdeux  } \eqsp .
  \end{equation}

\end{proof}

\begin{proposition_colt}
  \label{prop:mean_square_diff_fin}
  Let $\bgamma >0$ and $\alpha \in \coint{0,1}$. Assume
  \tup{\Cref{assum:f_lip}}, \tup{\Cref{assum:grad_sto}-\ref{item:ml}} and
  \tup{\Cref{assum:lip_sigma}}.  Then for any $T \geq 0$, there exists
  $\ctcinq \geq 0$ such that for any $\gamma \in \ocint{0, \bgamma}$,
  $k \in \nset$ with $(k+1)\gua \leq T$ and $X_0 \in \rset^{\dim}$ we have
  \begin{equation}
    \CPE{\| \bfX_{(k+1)\gua} - \tilde{X}_{k+1} \|^{2}}{\mcf_{k \gua}} \leq \ctcinq \defEns{\gamma^{4}(k+1)^{-4\alpha } + \gamma^{2}(k+1)^{-2\alpha} \vareps^2} (1 + \norm{\bfX_{k \gua}}^{2}) \eqsp ,
  \end{equation}
  where $(\bfX_t)_{t \geq 0}$ is the solution of \eqref{eq:sde} and for any $k \in \nset$
  \begin{equation}
    \tilde{X}_{k+1} = \bfX_{k \gua} - \gamma (k+1)^{-\alpha} H(\bfX_{k \gua}, Z_{k+1}) \eqsp .
  \end{equation}

\end{proposition_colt}

\begin{proof}
  The proof is straightforward upon combining \Cref{lemma:sgd_diff} and \Cref{lemma:mean_square_diff}.
\end{proof}

\begin{proposition_colt}
  \label{prop:strong_approx_appendix}
  Let $\bgamma >0$, $\alpha \in \coint{0,1}$ and
  $\gamma \in \ocint{0, \bgamma}$. Assume \tup{\Cref{assum:f_lip}},
  \tup{\Cref{assum:grad_sto}-\ref{item:ml}} and \tup{\Cref{assum:lip_sigma}}.
  Then there exists a coupling $((\bfB_t)_{t \geq 0}, (Z_n)_{n \in \nset})$.
  such that the following hold:
  \begin{enumerate}[wide, labelwidth=!, labelindent=0pt, label=(\alph*)]
  \item $(Z_{n})_{n \in \nset}$ is a sequence of independent random
    variables such that for any $n \in \nset$, $Z_n$ is distributed according to $\pi^Z$.
  \item For any
        $T \geq 0$, there exists $C \geq 0$ such that for any
        $\gamma \in \ocint{0, \bgamma}$, $n \in \nset$ with
        $\gua = \gamma^{1/(1-\alpha)}$, $n\gua \leq T$ we have
  \begin{equation}
    \label{eq:strong_approx}
      \expesq{\| \bfX_{n\gua} - X_{n} \|^{2}} \leq C (\gamma^{\delta} \vareps + \gamma) (1 + \log(\gamma^{-1}))  \eqsp , \quad \text{with $\delta = \min(1, (2-2\alpha)^{-1})$,}
    \end{equation}
    where $(\bfX_{t})_{t \geq 0}$ is solution of \eqref{eq:sde}, 
    $(X_n)_{n \in \nset}$ is defined by \eqref{eq:sgd} with
    $\bfX_0 = X_0 \in \rset^d$ and
    \begin{equation}
  \label{eq:vareps_def}
  \vareps^2 = \sup_{k \in \{0, \dots, n_T\}} \expe{\wassersteinD[2]^2(\nu_k^{\rmd}(\bfX_{k \gua}), \nu_k^{\rmcc}(\bfX_{k \gua}))} \eqsp , 
\end{equation}
where for any $\tilde{x} \in \rset^d$, $\nu_k^{\rmd}(\tilde{x})$ is the
distribution of $H(\tilde{x}, Z_{n+1})$, $\nu_k^{\rmcc}(\tilde{x})$ is the
distribution of
$\nabla f(\tilde{x}) + \gua^{-1/2} \Sigma^{1/2}(\tilde{x}) \int_{k \gua}^{(k+1)
  \gua}(\gua + s)^{-\alpha} \rmd \bfB_s$.
  \end{enumerate}  
\end{proposition_colt}

\begin{proof}
  Let $\alpha \in \coint{0,1}$, $\bgamma > 0$, $\gamma \in \ocint{0, \bgamma}$,
  $k \in \nset$, and $X_0 \in \rset^{\dim}$.  The first part of the proof is a
  direct consequence of \Cref{thm:coupling}. We now turn to the second part of
  the proof.  Let $(E_k)_{k \in \nset}$ such that for any $k \in \nset$,
  $E_k = \expeLigne{\normLigne{\bfX_{k\gua} - X_k}^2}$. Note that $E_0= 0$. Let
  $\tilde{X}_{k+1} = \bfX_{k \gua} - \gamma(k+1)^{-\alpha} H(\bfX_{k \gua},
  Z_{k+1})$. We have
\begin{align}
  \label{eq:recurrence}
  E_{k+1} &= \expe{\norm{\bfX_{(k+1)\gua} - X_{k+1}}^2} \\
          &= \expe{\norm{\bfX_{(k+1)\gua} - \tilde{X}_{k+1} + \tilde{X}_{k+1} - X_{k+1}}^2} \\
          &= \expe{\norm{\bfX_{(k+1)\gua} - \tilde{X}_{k+1}}^2}   +
            2 \expe{\langle \bfX_{(k+1)\gua} - \tilde{X}_{k+1} , \tilde{X}_{k+1} - X_{k+1} \rangle } \\ &\qquad   + \expe{\norm{\tilde{X}_{k+1} - X_{k+1}}^2} \\
          &= \expe{\norm{\bfX_{(k+1)\gua} - \tilde{X}_{k+1}}^2}   +  \expe{\norm{\tilde{X}_{k+1} - X_{k+1}}^2}
  \\ &\qquad + 2 \expe{\langle \bfX_{(k+1)\gua} - \tilde{X}_{k+1} , \bfX_{k \gua} - X_k }
  \\ & \qquad + 2 \gamma(k+1)^{-\alpha} \expe{\langle \bfX_{(k+1)\gua} - \tilde{X}_{k+1}, H(X_k, Z_{k+1}) - H(\bfX_{k\gua}, Z_{k+1})\rangle}  \eqsp .
\end{align}
Let $a_k = \gamma^{4}(k+1)^{-4\alpha} + \gamma^{2} (k+1)^{-2\alpha}$ and $a_k^{\vareps} = \gamma^{4}(k+1)^{-4\alpha} + \vareps^2 \gamma^{2} (k+1)^{-2\alpha}$.
We now bound each of the four terms appearing in \eqref{eq:recurrence}
\begin{enumerate}[wide, labelwidth=!, labelindent=0pt, label=(\alph*)]
\item First, using \Cref{prop:mean_square_diff_fin} and \Cref{lemma:bound_moments} we have 
\begin{align}
  \label{eq:recu_un}
  &\expe{\norm{\bfX_{(k+1)\gua} - \tilde{X}_{k+1}}^2} = \expe{\CPE{\norm{\bfX_{(k+1)\gua} - \tilde{X}_{k+1}}^2}{\mcf_{k \gua}}} \\
  & \qquad \leq \expe{\ctcinq (\gamma^{4}(k+1)^{-4\alpha } + \vareps^2 \gamma^{2} (k+1)^{-2\alpha}) \parenthese{1 + \norm{\bfX_{k \gua}}^{2}}} \\
  & \qquad \leq \ctun \ctcinq (\gamma^{4}(k+1)^{-4\alpha} + \vareps^2 \gamma^{2} (k+1)^{-2\alpha}) \parenthese{1 + \norm{X_0}^{2}} \leq \ctsix^{(a)} a_k^\vareps \eqsp ,
\end{align}
with $\ctsix^{(a)} \geq 0$ which does not depend on $\gamma$ and $k$.
\item Second, using \Cref{assum:f_lip},
  \Cref{assum:grad_sto}-\ref{item:approx_sto} and that for any $a,b \geq 0$,
  $(a+b)^2 \leq 2a^2 + 2b^2$ we have
  \begin{align}
  \expe{\norm{\tilde{X}_{k+1} - X_{k+1}}^2} &= \expe{\norm{\bfX_{k \gua} - X_k - \gamma(k+1)^{-\alpha}(H(\bfX_{k \gua}, Z_{k+1}) - H(X_k, Z_{k+1})) }^2} \\
                                             &\leq (1 + \gamma \Lip (k+1)^{-\alpha})^2 \expeLigne{\normLigne{\bfX_{k\gua} - X_k}^2} 
  \\
                                             & \leq (1 + 2\gamma \Lip(k+1)^{-\alpha} + \gamma^2 \Lip^2(k+1)^{-2\alpha}) E_k 
  \leq (1 + \ctsix^{(b)} a_k^{1/2}) E_k 
  \eqsp ,     \label{eq:recu_deux}
\end{align}
with $\ctsix^{(b)} \geq 0$ which does not depend on $\gamma$ and $k$.
\item In what follows, let 
  $\bfXt_{(k+1)\gua} = \bfX_{k \gua} - \gamma (k+1)^{-\alpha} \defEns{\nabla
    f(\bfX_{k \gua}) + \Sigma(\bfX_{k \gua})^{1/2}G_k}$, with
  $G_k = \gua^{-1/2} \int_{k \gua}^{(k+1)\gua} \rmd \bfB_s$.  Let
  $b_k = \gamma^3(k+1)^{-3\alpha} + \gamma(k+1)^{-2(1 + \alpha/2)}$.
  
  \noindent Using
  \Cref{assum:grad_sto} we have
  $\CPELigne{\bfXt_{(k+1)\gua}}{\mck_k} =
  \CPELigne{\tilde{X}_{k+1}}{\mck_k}$.  Combining this result, the
  Cauchy-Schwarz inequality, \Cref{lemma:mean_square_diff},
  \Cref{lemma:bound_moments} and that for any $a, b \geq 0$,
  $(a+b)^{1/2} \leq a^{1/2} + b^{1/2}$ and $2ab \leq a^2 + b^2$ we obtain
  \begin{align}
    &\expe{\langle \bfX_{(k+1)\gua} - \tilde{X}_{k+1} , \bfX_{k \gua} - X_k \rangle} \\
    &\quad= \expe{\langle \CPELigne{\bfX_{(k+1)\gua} - \tilde{X}_{k+1}}{\mck_k} , \bfX_{k \gua} - X_k \rangle } \\
                                                                              & \quad = \expe{\langle \CPELigne{\bfX_{(k+1)\gua} - \bfXt_{(k+1)\gua}}{\mck_k} , \bfX_{k \gua} - X_k \rangle} \\
                                                                              & \quad \leq  \expe{ \CPEsq{\norm{\bfX_{(k+1)\gua} - \bfXt_{(k+1)\gua}}^2}{\mck_k}  \norm{\bfX_{k \gua} - X_k} } \\
                                                                              & \quad \leq  \expesq{\norm{\bfX_{(k+1)\gua} - \bfXt_{(k+1)\gua}}^2}  \expesq{\norm{\bfX_{k \gua} - X_k}^2} \\
    & \quad \leq \ctun^{1/2} \ctquatre^{1/2} \defEns{\gamma^{4}(k+1)^{-4\alpha } + \gamma^{2 } (k+1)^{-2(1+\alpha)}}^{1/2}
      (1+ \norm{X_0}^{2}) E_k^{1/2} \\
    & \quad \leq \ctun^{1/2} \ctquatre^{1/2} \defEns{\gamma^{3/2}(k+1)^{-3\alpha/2} + \gamma^{1/2} (k+1)^{-(1+\alpha/2)}}
      (1+ \norm{X_0}^{2}) \gamma^{1/2}(k+1)^{-\alpha/2}E_k^{1/2} \\
    & \quad \leq \ctsix^{(c)} \defEns{\gamma^{3}(k+1)^{-3\alpha} + \gamma (k+1)^{-2(1+\alpha/2)}} / 2
      +  a_k^{1/2}E_k/2 \leq \ctsix^{(c)} b_k 
      +  a_k^{1/2}E_k \eqsp .     \label{eq:recu_trois}
  \end{align}
with $\ctsix^{(c)} \geq 0$ which does not depend on $\gamma$ and $k$.

\item Finally, using the Cauchy-Schwarz inequality, \eqref{eq:recu_un}, \Cref{assum:grad_sto} and \Cref{assum:f_lip} and that for any $a, b \geq 0$, $(a+b)^{1/2} \leq a^{1/2} + b^{1/2}$, we have
  \begin{align}
    \label{eq:recu_quatre}
    &\gamma(k+1)^{-\alpha} \expe{\langle \bfX_{(k+1)\gua} - \tilde{X}_{k+1}, H(X_k, Z_{k+1}) - H(\bfX_{k\gua}, Z_{k+1})\rangle} \\ & \ \leq \gamma(k+1)^{-\alpha} \expesq{\norm{\bfX_{(k+1)\gua} - \tilde{X}_{k+1}}^2} \expesq{\norm{H(X_k, Z_{k+1}) - H(\bfX_{k\gua}, Z_{k+1})}^2} \\
    & \ \leq (\ctsix^{(a)})^{1/2} \gamma(k+1)^{-\alpha} a_k^{1/2} \Lip  E_k \leq \ctsix^{(d)}  a_k E_k \eqsp . 
  \end{align}
with $\ctsix^{(d)} \geq 0$ which does not depend on $\gamma$ and $k$.
\end{enumerate}
Let $\bctsix = \ctsix^{(a)} + \ctsix^{(b)} + \ctsix^{(c)} + \ctsix^{(d)}$.
Finally, we have using \eqref{eq:recu_un}, \eqref{eq:recu_deux},
\eqref{eq:recu_trois} and \eqref{eq:recu_quatre} in \eqref{eq:recurrence}
\begin{equation}
  \label{eq:pre_gronwall}
  E_{k+1} \leq \bctsix (a_k + a_k^{1/2}) E_k + \bctsix (a_k^\vareps + b_k) \eqsp .
\end{equation}
We denote $v_k = \bctsix (a_k^{1/2} + a_k)$ and
  $w_k = \bctsix (a_k^\vareps + b_k)$.
  Using \Cref{lemma:sum} and that $a_k^{1/2} \leq \gamma(k+1)^{-\alpha} + \gamma^2(k+1)^{-2\alpha}$, there exists $\ctsix^{(e)} \geq 0$ which does not depend on $\gamma$ and $k$ such that
  \begin{equation}
    \label{eq:bound_one}
     \sum_{k=0}^{N-1} v_k \leq \ctsix^{(e)} \eqsp .
   \end{equation}
   In addition, we have that for any $k \in \nset$,
   \begin{equation}
     v_k \leq \bctsix (\gamma^2(k+1)^{-2\alpha} \vareps^2 + \gamma^3(k+1)^{-3\alpha} + \gamma^4(k+1)^{-4\alpha} + \gamma(k+1)^{-2(1+ \alpha/2)}) \eqsp .
   \end{equation}
   Using that $\gamma \gua^{\alpha} = \gua$ and \Cref{lemma:sum} there exists $\ctsix^{(f)} \geq 0$ which does not depend on $\gamma$ and $k$ such that
\begin{equation}
  \label{eq:bound_two}
  \sum_{k=0}^{N-1} v_k \leq \left\lbrace
    \begin{aligned}
       &\ctsix^{(f)} \gamma^2 (1 + \log(\gamma^{-1})) & \text{if } \alpha \geq 1/2 \eqsp ,\\
       &\ctsix^{(f)} \gua \vareps^2  & \text{if } \alpha < 1/2 \eqsp . \\
    \end{aligned}
    \right.
  \end{equation}

  Using \eqref{eq:pre_gronwall} and \Cref{lemma:gronwall} we obtain that
\begin{align}
  E_k &\leq \sum_{k=0}^{N-1} w_k  + \exp \parentheseDeux{\sum_{k=0}^{N-1} v_k} \sum_{k=0}^{N-1}v_k w_k \\
  &\leq \sum_{k=0}^{N-1} w_k  + \exp \parentheseDeux{\sum_{k=0}^{N-1} v_k} \parenthese{\sum_{k=0}^{N-1}v_k}\parenthese{ \sum_{k=0}^{N-1} w_k} \eqsp .   \label{eq:gronwall}
\end{align}
Combining \eqref{eq:bound_one}, \eqref{eq:bound_two} and \eqref{eq:gronwall} concludes the first part of the proof.
\end{proof}

\subsection{The case of batch noise}
\label{sec:case-batch-noise}

In this section, we refine our results in the specific case of a batch noise. We
recall our main result in this setting in \Cref{coro:batch_noise_supp}. The
proof is based on quantitative bounds in the CLT w.r.t. to $\wassersteinD[2]$,
see \cite{bonis2020stein}. In \Cref{prop:grad_flow}, we show that contrary to
the SDE setting the gradient flow has an error of order at least
$\bigO(M^{-1})$.

\begin{corollary_colt}
  \label{coro:batch_noise_supp}
  Let $\bgamma >0$ and $\alpha \in \coint{0,1}$. Assume
  \tup{\Cref{assum:f_lip}}, \tup{\rref{assum:grad_sto}-\ref{item:ml}} and
  \tup{\Cref{assum:lip_sigma}} (with respect to $(\msy, \mcy, \pi)$). Let $H$ be
  given by \eqref{eq:H-def}.  Assume that there exists
  $x^\star \in \rset^d$, $C, p \geq 0$ such that for any $x \in \rset^d$ and
  $y \in \msy$
  \begin{equation}
    \textstyle{\int_{\msy} \normLigne{\nabla \tilde{f}(x^\star, y)}^4 \rmd \pi(y) <
  +\infty \eqsp, \quad \normLigne{\Sigma_f(x)^{-1/2}}\leq C (1 + \norm{x}^p) \eqsp .}
\end{equation}
Then, there exists a random variable
  $((\bfB_t)_{t \geq 0}, (Z_n)_{n \in \nset})$ such that for any $T \geq 0$,
  there exists $C \geq 0$ such that for any $\gamma \in \ocint{0, \bgamma}$,
  $n \in \nset$ with $n\gua \leq T$ $\gua = \gamma^{1/(1-\alpha)}$ we have
  \begin{equation}
      \expesq{\| \bfX_{n\gua} - X_{n} \|^{2}} \leq C (\gamma^{\delta}M^{-1}  + \gamma) (1 + \log(\gamma^{-1}))  \eqsp , \quad \text{with $\delta = \min(1, (2-2\alpha)^{-1})$ .}
    \end{equation}    
  \end{corollary_colt}

  \begin{proof}
    Let $\bgamma >0$, $\alpha \in \coint{0,1}$ and $M \in \nset$. Applying \Cref{prop:strong_approx}, there exists a random variable
  $((\bfB_t)_{t \geq 0}, (Z_n)_{n \in \nset})$ such that or any $T \geq 0$,
  there exists $C \geq 0$ such that for any $\gamma \in \ocint{0, \bgamma}$,
  $n \in \nset$ with $n\gua \leq T$ $\gua = \gamma^{1/(1-\alpha)}$ we have
  \begin{equation}
    \expesq{\| \bfX_{n\gua} - X_{n} \|^{2}} \leq C (\gamma^{\delta}\vareps  + \gamma) (1 + \log(\gamma^{-1}))  \eqsp , \quad \text{with $\delta = \min(1, (2-2\alpha)^{-1})$ ,}
  \end{equation}
    where $(\bfX_{t})_{t \geq 0}$ is solution of \eqref{eq:sde}, 
    $(X_n)_{n \in \nset}$ is defined by \eqref{eq:sgd} with
    $\bfX_0 = X_0 \in \rset^d$ and
    \begin{equation}
  \label{eq:vareps_def}
  \vareps^2 = \sup_{k \in \{0, \dots, n_T\}} \expe{\wassersteinD[2]^2(\nu^{\rmd}(\bfX_{k \gua}), \nu^{\rmcc}(\bfX_{k \gua}))} \eqsp , 
\end{equation}
where for any $\tilde{x} \in \rset^d$, $\nu^{\rmd}(\tilde{x})$ is the
distribution of $H(\tilde{x}, Z_{n+1})$ and $\nu^{\rmcc}(\tilde{x})$ is the
distribution of
$\nabla f(\tilde{x}) + \gua^{-1/2} \Sigma^{1/2}(\tilde{x}) \int_{k \gua}^{(k+1)
  \gua}(\gua + s)^{-\alpha} \rmd \bfB_s$. Our goal is now to control $\vareps$
in this specific setting. Let $x \in \rset^d$, $k \in \nset$ and
$(X_1^x, X_2^x)$ be an optimal coupling between $\nu_k^{\rmd}$ and
$\nu_k^{\rmcc}$. Note that $X_2^x$ is a Gaussian random variable with mean
$\nabla f(x)$ and covariance matrix $\Sigma(x)$, where
$\Sigma(x) = (1/M) \Sigma_f(x)$ with
$\Sigma_f = \pi[(\tilde{\nabla} f(x, \cdot) - \nabla f (x))(\tilde{\nabla}f (x,
\cdot) - \nabla f (x))^\top]$. In particular, we get that
\begin{equation}
  \label{eq:upper_wass}
  \wassersteinD[2]^2(\nu^{\rmd}(x), \nu^{\rmcc}(x)) \leq M^{-1} \normLigne{\Sigma_f^{1/2}(x)}^2 \wassersteinD[2]^2(\tilde{\nu}(x), \nu(x)) \eqsp ,
\end{equation}
where $\tilde{\nu}(x)$ is the distribution of
$M^{-1} \sum_{k=1}^M \Sigma_f(x)^{-1/2}\{ \nabla \tilde{f}(x, Y_i) - \nabla
f(x)\}$ with $\{Y_i\}_{i=1}^M$ distributed according to $\pi^{\otimes M}$ and
$\nu$ the distribution of a Gaussian random variable with zero mean and identity
covariance matrix. Denote
$\{Y_i\}_{i=1}^M = \{\Sigma_f(x)^{-1/2}(\nabla \tilde{f}(x, Y_i) - \nabla
f(x))\}_{i=1}^M$. The random variables $\{Y_i\}_{i=1}^M$ are \iid ,
$\expeLigne{Y_i} = 0$ and $\expeLigne{Y_i Y_i^\top} = \Id$. In addition, using
\Cref{assum:f_lip} and \Cref{assum:grad_sto}-\ref{item:ml} we have that
\begin{align}
  \norm{Y_1} &\leq 8 \normLigne{\Sigma_f(x)^{-1/2}}^4 (\norm{\nabla f(x, U_1)}^4 + \norm{\nabla f(x)}^4) \\
  &\leq 216  C^4 (1 + \norm{x}^p)^4 (\norm{\nabla f(x^\star, U_1)}^4 + \Lip^4 \norm{x}^4 + \Lip^4 \norm{x^\star}^4 + \norm{\nabla f(0)}^4 + \Lip^4 \norm{x}^4)\eqsp . 
\end{align}
Combining this result and the fact
$\int_{\msz} \normLigne{\nabla \tilde{f}(x^\star, y)}^4 \rmd \pi(y) <
+\infty$, there exists $q \in \nset$ and $C \geq 0$ such that
\begin{equation}
  \normLigne{\Sigma_f^{1/2}(x)}^2 \expeLigne{\norm{Y_1}^4} \leq C (1 + \norm{x}^{2q}) \eqsp . 
\end{equation}
Therefore combining \cite[Theorem 1]{bonis2020stein} and \eqref{eq:upper_wass}, there
exists $C \geq 0$ such that for any $x \in \rset^d$
\begin{equation}
  \wassersteinD[2]^2(\nu_k^{\rmd}(x), \nu_k^{\rmcc}(x)) \leq C M^{-1/2} (1 + \norm{x}^{2q}) \eqsp . 
\end{equation}
Using \Cref{lemma:bound_moments}, there exists $C \geq 0$ such that
$\vareps \leq C M^{-1}$, which concludes the proof.
  \end{proof}

  \begin{proposition_colt}
    \label{prop:grad_flow}
    Let $\alpha \in \coint{0, 1/2}$, $T \geq 0$ and $\bgamma > 0$ such that for
    any $\gamma \in \ocint{0, \bgamma}$,
    $(T - \gua)^{1 - 2\alpha} - \gua^{1- 2 \alpha} \geq (T/2)^{1 - 2\alpha}$.  Let
    $f = 0$ and $\tilde{f}: \ \rset^d \to \rset$ such that for any
    $x \in \rset^d$ and $z \in \rset^d$,
    $\tilde{f}(x, z) = \langle x, z \rangle$,
    $(\msz, \mcz) = ((\rset^d)^M, \mcb{\rset^d}^{\otimes M})$,
    $\pi^Z = \pi^{\otimes M}$ with $M \in \nset$, $\pi$ a Gaussian distribution
    with zero mean and identity covariance matrix. In this case, for any
    $\gamma \in \ocint{0, \bgamma}$ we have for $n = \floor{T / \gua}$
  \begin{equation}
    \label{eq:strong_approx_lower}
      \expesq{\| \bfX_{n\gua} - X_{n} \|^{2}} \geq  M^{-1/2} \gamma^{\delta} (1 - 2 \alpha)^{-1/2} (T/2)^{1/2 - \alpha} \eqsp , \quad \text{with $\delta = \min(1, (2-2\alpha)^{-1})$ ,}
    \end{equation}
    where $(\bfX_t)_{t \geq 0}$ is the solution of
    $\rmd \bfX_t = -\nabla f(\bfX_t) \rmd t$ and $(X_n)_{n \in \nset}$ is a
    solution of \eqref{eq:sgd}.
  \end{proposition_colt}

  \begin{proof}
    Note that for any $t \geq 0$, $\bfX_t = 0$.  In addition, for any
    $n \in \nset$, we have $X_n = (1/M) \gamma \sum_{k=0}^{n-1} (k+1)^{-\alpha} \sum_{m=1}^M Z^{k,m}$,
    where $\ensemble{Z^{k,m}}{k,m \in \nset}$ is a collection of independent
    Gaussian random variables with zero mean and identity covariance matrix.
    Therefore we get that
    \begin{align}
      \expe{\norm{X_n}^2} &= (1/M) \gamma^2 \sum_{k=0}^{n-1} (k+1)^{-2\alpha} \\
      &\geq (1/M) \gamma^2 \int_1^n t^{-2 \alpha} \rmd t \geq M^{-1} \gamma^2 \gua^{2 \alpha -1} (1 - 2 \alpha)^{-1/2} ((T - \gua)^{1 - 2\alpha} - \gua^{1 - 2\alpha}) \eqsp ,
    \end{align}
which concludes the proof.
  \end{proof}
  
\subsection{Weak approximation}
\label{sec:weak-approximation}

We also derive weak approximation estimates of order 1. Note that in the case
where $\alpha \geq 1/2$, these weak results are a direct consequence of
\Cref{prop:strong_approx}. Denote by $\gfun_{p, k}$ the set of $k$-times
continuously differentiable functions $g$ such that there exists $\Ktt \geq 0$
such that for any $x \in \rset^{\dim}$,
$\max(\normLigne{\nabla g(x)}, \dots, \normLigne{\nabla^kg(x)}) \leq \Ktt (1 +
\norm{x}^{p})$. We state our main result in \Cref{prop:weak_approx}.

\begin{proposition_colt}
  \label{prop:weak_approx}
  Let $\bgamma >0$, $\alpha \in \coint{0,1}$ and $p \in \nset$. Assume that
  $f \in \gfun_{p, 4}$, $\Sigma^{1/2} \in \gfun_{p, 3}$,
  \tup{\Cref{assum:f_lip}}, \tup{\Cref{assum:grad_sto}-\ref{item:ml}} and
  \tup{\Cref{assum:lip_sigma}}.  Let $g \in \gfun_{p,2}$. In addition, assume
  that for any $m \in \nset$ there exists $x^\star \in \rset^d$ such that
  $\int_{\msz} \norm{H(x^\star, z)}^{2m} \rmd \pi^Z(z) < +\infty$.  Then for any
  $T \geq 0$, there exists $C \geq 0$ such that for any
  $\gamma \in \ocint{0, \bgamma}$, $n \in \nset$ with
  $\gua = \gamma^{1/(1-\alpha)}$, $n\gua \leq T$ we have
  \begin{equation}
    \label{eq:weak_approx}
      |\expe{g(\bfX_{n\gua}) - g(X_n)}| \leq C \gamma (1 + \log(\gamma^{-1}))  \eqsp .
      \end{equation}
\end{proposition_colt}
These results extend \citep[Theorem 1.1 (a)]{li2017sme} to the
non-increasing stepsize case. Once again, the result obtained in
\Cref{prop:weak_approx} must be compared to similar weak error
controls for SDEs. For example, under appropriate conditions,
\citep{talay1990expansion} shows that the EM discretization
$ Y_{n+1} = Y_n + \gamma \rmb(Y_n) + \sqrt{\gamma} \upsigma(Y_n)
G_{n+1}$ is a weak approximation of order $1$ of
\eqref{eq:sde_homo_y}.

We now turn to the proof of \Cref{prop:weak_approx}. We start with a useful
technical lemma in \Cref{lemma:g_fun_ineq}. Then, before giving the proof of
\Cref{prop:weak_approx}, we highlight that the result is straightforward for
$\alpha \in \coint{1/2, 1}$ in \Cref{prop:weak_error}. We provide a one-step
approximation error bound in \Cref{prop:weak_one_step} and conclude in
\Cref{prop:weak_approx_appendix}.  We recall that $\gfun_p$ is the set of twice
continuously differentiable functions from $\rset^{\dim}$ to $\rset$ such that
for any $g \in \gfun_p$, there exists $\Ktt \geq 0$ such that for any
$x \in \rset^{\dim}$
\begin{equation}
  \label{eq:g_fun}
  \max\defEns{\norm{\nabla g(x)}, \norm{\nabla^2 g(x)}} \leq \Ktt (1 + \norm{x}^{p}) \eqsp ,
\end{equation}
with $p \in \nset$.

\begin{lemma_colt}
  \label{lemma:g_fun_ineq}
  Let $p \in \nset$, $g \in \gfun_p$ and let $\Ktt \geq 0$ as in \eqref{eq:g_fun}. Then, for any $x,y \in \rset^{\dim}$
  \begin{equation}
    |g(y)- g(x) - \langle \nabla g(x), y - x \rangle | \leq \Ktt (1 + \norm{x}^p + \norm{y}^p) \norm{x-y}^2 \eqsp .
  \end{equation}
\end{lemma_colt}

\begin{proof}
      Using that for any $x \mapsto \| x \|^p$ is convex, and Cauchy-Schwarz inequality  we get for any $x,y \in \rset^{\dim}$
\begin{align}
  |g(x) - g(y) - \langle \nabla g(x), y- x \rangle | &\leq \int_0^1 |\nabla^2 g(x + t(y-x)) (y-x)^{\otimes 2} | \rmd t \\
                                                     &\leq \norm{x-y}^2 \int_0^1 |\nabla^2 g(x + t(y-x)) (y-x)^{\otimes 2} | \rmd t \\
  &\leq \Ktt (1 + \norm{x}^p + \norm{y}^p) \norm{x-y}^2 \eqsp .
\end{align}
\end{proof}

\begin{proposition_colt}
  \label{prop:weak_error}
  Let $\bgamma >0$ and $\alpha \in \coint{1/2,1}$ and $p \in \nset$. Assume
  \tup{\Cref{assum:f_lip}}, \tup{\Cref{assum:grad_sto}-\ref{item:ml}} and
  \tup{\Cref{assum:lip_sigma}}.  In addition, assume
  that for any $m \in \nset$ there exists $x^\star \in \rset^d$ such that
  $\int_{\msz} \norm{H(x^\star, z)}^{2m} \rmd \pi^Z(z) < +\infty$. Then for any $T \geq 0$ and $g \in \gfun_p$, there exists
  $\ctsept \geq 0$ such that for any $\gamma \in \ocint{0, \bgamma}$,
  $k \in \nset$ with $k\gua \leq T$ and $X_0 \in \rset^{\dim}$ we have
    \begin{equation}
      \CPE{|g(\bfX_{k \gua}) - g(X_k)|}{\mcf_k} \leq \ctsept \gamma (1 + \log(\gamma^{-1}))  \eqsp ,
      \end{equation}
  where $(X_k)_{k \in \nset}$ satisfies the recursion \eqref{eq:sgd} and $(\bfX_t)_{t \geq 0}$ is the solution of \eqref{eq:sde} with $\bfX_{0} = X_0$
\end{proposition_colt}

\begin{proof}
  Let $p \in \nset$, $g \in \gfun_p$, $\alpha \in \coint{1/2,1}$, $\bgamma > 0$,
  $\gamma \in \ocint{0, \bgamma}$, $k \in \nset$, and $X_0 \in \rset^{\dim}$.
    Using that for any $x \mapsto \| x \|^p$ is convex, for any $x, y \in \rset^{\dim}$ we get
\begin{align}
  |g(x) - g(y)| &\leq \int_{0}^1 |\langle \nabla g(x + t(y-x)), y - x\rangle | \rmd t
                \leq \| y - x \| \int_0^1 \| \nabla g(x + t(y-x)) \| \rmd t  \\
  &\leq \| y - x \| \Ktt (1 + \norm{x}^{p} + \norm{y}^{p}) \eqsp .
\end{align}
  Combining this result, \Cref{prop:strong_approx_appendix},
  \Cref{lemma:bound_moments} and the Cauchy-Schwarz inequality we get that
\begin{equation}
  \expe{|g(\bfX_{k \gua}) - g(X_k)|} \leq \Ktt \ctsix \gamma (1 + \log(\gamma^{-1})) (\ctun + \cttun)^{1/2} (1 + \norm{X_0}^{2p})^{1/2} \eqsp ,
\end{equation}
which concludes the proof.
\end{proof}

\begin{proposition_colt}
  \label{prop:weak_one_step}
  Let $p \in \nset$ and $g \in \gfun_p$. Let $\bgamma >0$ and
  $\alpha \in \coint{0,1}$. Assume \tup{\Cref{assum:f_lip}},
  \tup{\Cref{assum:grad_sto}-\ref{item:ml}}, \tup{\Cref{assum:lip_sigma}} and that for any $m \in \nset$ there exists $x^\star \in \rset^d$ such that
  $\int_{\msz} \norm{H(x^\star, z)}^{2m} \rmd \pi^Z(z) < +\infty$.  Then for any $T \geq 0$, there exists $\cthuit \geq 0$ such
  that for any $\gamma \in \ocint{0, \bgamma}$, $k \in \nset$ with
  $(k+1)\gua \leq T$ and $X_0 \in \rset^{\dim}$ we have
  \begin{equation}
    \abs{\CPE{g(\bfX_{(k+1)\gua}) - g(X_{k+1})}{\mcg_k}} \leq \cthuit  \defEns{\gamma^2(k+1)^{-2\alpha} + \gamma (k+1)^{-(1+\alpha)}} (1 + \norm{\bfX_{k \gua}}^{p+2}) \eqsp ,
  \end{equation}
  where $(X_k)$ is the solution of \eqref{eq:sgd} and $\bfX_t$ is the solution of
  \begin{equation}
    \bfX_t = X_k - \int_{k \gua}^t (s + \gua)^{-\alpha} \nabla f (\bfX_s) \rmd s + \gua^{1/2} \int_{k \gua}^t(s + \gua)^{-\alpha} \Sigma(\bfX_s)^{1/2} \rmd \bfB_s \eqsp .
  \end{equation}
\end{proposition_colt}

\begin{proof}
  Let
  $\bfXd_{(k+1)\gua} = X_k - \gamma (k+1)^{-\alpha} \defEns{\nabla f(\bfX_{k
      \gua}) + \Sigma(X_k)^{1/2}G_k}$, with
  $G_k = \gua^{-1/2} \int_{k \gua}^{(k+1)\gua} \rmd \bfB_s$.  Using
  \Cref{assum:grad_sto} we have $\CPE{\bfXd_{(k+1)\gua}}{\mcg_k} = \CPE{X_{k+1}}{\mcg_k}$.
  Using \Cref{lemma:bound_moments}, \Cref{lemma:mean_square_one}, \Cref{lemma:mean_square_diff}, \Cref{lemma:g_fun_ineq} and the Cauchy-Schwarz inequality we have
  \begin{align}
    &\abs{\CPE{g(\bfX_{(k+1)\gua}) - g(X_{k+1})}{\mcg_k}} \\ & \quad \leq \abs{\CPE{\langle \nabla g(X_k), \bfX_{(k+1)\gua} - X_{k+1} \rangle}{\mcg_k}} \\
    & \quad \quad + \Ktt \CPE{\norm{\bfX_{(k+1)\gua} - X_k}^2(1 + \norm{X_k}^{p} + \norm{\bfX_{(k+1)\gua}}^{p})}{\mcg_{k}} \\
    & \quad \quad  + \Ktt \CPE{\norm{X_{k+1} - X_k}^2(1 + \norm{X_k}^{p} + \norm{X_{k+1}}^{p})}{\mcg_k} \\
    &\quad \leq \abs{\langle \nabla g(X_k), \CPE{\bfX_{(k+1)\gua} - \bfXd_{k+1} } \rangle}{\mcg_k} \\ & \quad \quad+ 3^{1/2}\Ktt\CPE{\norm{\bfX_{(k+1)\gua} - X_k}^4}{\mcg_k}^{1/2} \CPE{(1 + \norm{X_k}^{2p} + \norm{X_{k+1}}^{2p})}{\mcg_k}^{1/2}  \\ & \quad \quad + 3^{1/2}\Ktt\CPE{\norm{X_{k+1} - X_k}^4}{\mcg_k}^{1/2}\CPE{(1 + \norm{X_k}^{2p} + \norm{\bfX_{(k+1)\gua}}^{2p})}{\mcg_k}^{1/2} \\
    &\quad \leq \Ktt(1 + \norm{X_k}^p) \CPE{\norm{\bfX_{(k+1)\gua} - \bfXd_{k+1}}^2}{\mcg_k}^{1/2}  \\ & \quad \quad+ 3^{1/2}\Ktt\CPE{\norm{\bfX_{(k+1)\gua} - X_k}^4}{\mcg_k}^{1/2} (1 + \ctun)^{1/2}(1 + \norm{X_k}^{p})  \\ & \quad \quad + 3^{1/2}\Ktt\CPE{\norm{X_{k+1} - X_k}^4}{\mcg_k}^{1/2}(1 + \cttun)^{1/2}(1 + \norm{X_k}^{p}) \\
    &\quad \leq \Ktt(1 + \norm{X_k}^p) \ctquatre^{1/2} \defEns{\gamma^2 (k+1)^{-2\alpha} + \gamma(k+1)^{-(1+\alpha)}}(1 + \norm{X_k})  \\ & \quad \quad + 3^{1/2}\Ktt\ctdeux^{1/2}\gamma^2(k+1)^{-2\alpha}(1 + \norm{X_k}^2) (1 + \ctun)^{1/2}(1 + \norm{X_k}^{p})  \\ & \quad \quad + 3^{1/2}\Ktt\ctdeux^{1/2}\gamma^2(k+1)^{-2\alpha}(1 + \norm{X_k}^2)(1 + \cttun)^{1/2}(1 + \norm{X_k}^{p}) \eqsp ,
  \end{align}
  which concludes the proof.
\end{proof}

\begin{proposition_colt}
  \label{prop:weak_approx_appendix}
  Let $\bgamma >0$ and $\alpha \in \coint{0,1}$. Assume that
  $f \in \gfun_{p, 4}$, $\Sigma^{1/2} \in \gfun_{p, 3}$
  \tup{\Cref{assum:f_lip}}, \tup{\Cref{assum:grad_sto}-\ref{item:ml}} and
  \tup{\Cref{assum:lip_sigma}}.  Let $p \in \nset$ and $g \in \gfun_{p,2}$. In
  addition, assume that for any $m \in \nset$ there exists $x^\star \in \rset^d$
  such that $\int_{\msz} \norm{H(x^\star, z)}^{2m} \rmd \pi^Z(z) < +\infty$.
  Then for any $T \geq 0$, there exists $\ctneuf \geq 0$ such that for any
  $\gamma \in \ocint{0, \bgamma}$, $k \in \nset$ with $k\gua \leq T$ and
  $X_0 \in \rset^{\dim}$ we have
  \begin{equation}
    \label{eq:weak_approx_appendix}
      |\expe{g(\bfX_{k\gua}) - g(X_k)}| \leq \ctneuf \gamma (1 + \log(\gamma^{-1}))  \eqsp ,
      \end{equation}
\end{proposition_colt}
where $(X_k)_{k \in \nset}$ satisfies the recursion \eqref{eq:sgd} and $(\bfX_t)_{t \geq 0}$ is the solution of \eqref{eq:sde} with $\bfX_{0} = X_0$.

\begin{proof}
  For any $k \in \nset$ with $k \gua \leq T$, let $g_k(x) = \expe{g(\bfX_{k \gua})}$ with $\bfX_0 = x$.
  Since $f \in \gfun_{p, 4}$, $\Sigma^{1/2} \in \gfun_{p, 3}$ and $g \in \gfun_{p, 2}$ one can show, see \citep{blago1961some} or \citep[Proposition 2.1]{kunita1981decomposition}, that there exists $m \in \nset$ and $\Ktt \geq 0$ such that for any $k \in \nset$ $g_k \in \rmc^m(\rset^{\dim}, \rset)$ and
  \begin{equation}
    \max \defEns{\norm{g_k(x)}, \dots ,\norm{\nabla^m g_k(x)}} \leq \Ktt (1 + \norm{x}^p) \eqsp .
  \end{equation}
  Therefore, $g_k \in \gfun_{p, m}$ with constants uniform in $k \in \nset$.
  In addition, for any $k \in \nset$ with $k \gua \leq T$, let $h^{(1)}_{k}(x) = \expe{g_k(X_{k+1})}$ with $X_k = x$ and
  $h^{(2)}_k(x) = \expe{g_k(\bfX_{(k+1)\gua})}$ with $\bfX_{k \gua} = x$. Using \Cref{prop:weak_one_step} we have for any $k \in \nset$, $k \gua \leq T$
  \begin{equation}
   \abs{h_{k}^{(1)}(x) - h_k^{(2)}(x)} \leq \cthuit  \defEns{\gamma^2(k+1)^{-2\alpha} + \gamma (k+1)^{-(1+\alpha)}} (1 + \norm{x}^{m+2}) \eqsp .
 \end{equation}
Therefore, using \Cref{lemma:bound_moments} we have for any $k \in \nset$ with $k \gua \leq T$ and $j \leq k$,
\begin{equation}
  \label{eq:h_diff}
  \abs{\expe{h_{k-j-1}^{(1)}(X_j) - h_{k-j-1}^{(2)}(X_j)}} \leq \cttun \cthuit  \defEns{\gamma^2(k+1)^{-2\alpha} + \gamma (k+1)^{-(1+\alpha)}} (1 + \norm{X_0}^{m+2}) \eqsp .
\end{equation}

  Now, let $k \in \nset$ with $k \gua \leq T$ and consider the family
  $\ensembleLigne{(X_{\ell}^j)_{\ell \in \nset}}{j =0, \dots, N}$, defined by the following recursion: for any
  $j \in \{0, \dots, N\}$ $X_0^j = X_0$ and for any $\ell \in \nset$:
  \begin{enumerate}[wide, labelwidth=!, labelindent=0pt, label=(\alph*)]
  \item if $\ell \geq  j$,
    \begin{equation}
      X_{\ell+1}^j = X_{\ell}^j - \gamma (k+1)^{-\alpha} H(X_{\ell}^j, Z_{\ell+1}) \eqsp ,
    \end{equation}
  \item if $\ell < j$, $X_{\ell+1}^j = \bfX_{(\ell +1)\gua}^j$, where $\bfX_{\ell \gua}^j = X_{\ell}^j$ and for any $t \in \ccint{\ell \gua, (\ell +1)\gua}$ we have
    \begin{equation}
      \bfX_{t}^j = X_{\ell}^j - \int_{\ell \gua}^t (\gua + s)^{-\alpha} \nabla f(\bfX_s^j) \rmd s - \gua^{1/2} \int_{\ell \gua}^t (\gua + s)^{-\alpha} \Sigma^{1/2}(\bfX_s^j) \rmd \bfB_s \eqsp .
    \end{equation}
  \end{enumerate}
  We have
  \begin{equation}
    \label{eq:summation}
    \abs{\expe{g(\bfX_{k \gua}) - g(X_k)}} = \abs{\expe{g(X_k^k) - g(X_k^{0})}} = \sum_{j=0}^{k-1} \abs{\expe{g(X_k^{j+1}) - g(X_k^{j})}} \eqsp .
  \end{equation}
Using \eqref{eq:h_diff} we get
\begin{align}
  \label{eq:cond_exp}
  \abs{\expe{g(X_k^{j+1}) - g(X_k^{j})}} &= \abs{\expe{\CPE{g(X_k^j) - g(X_k^{j+1})}{X_k^j}}} \\
                                         &= \abs{\expe{h_{k-j-1}^{(1)}(X_j) - h_{k-j-1}^{(2)}(X_j)}} \\
                                         &\leq \cttun \cthuit  \defEns{\gamma^2(k+1)^{-2\alpha} + \gamma (k+1)^{-(1+\alpha)}} (1 + \norm{X_0}^{m+2}) \\
  &\leq \ctneuf^{(a)}\gamma^2(k+1)^{-2\alpha} + \gamma (k+1)^{-(1+\alpha)} \eqsp ,
\end{align}
with $\ctneuf^{(a)} \geq 0$ which does not depend on $k$ or $\gamma$
In addition, using \Cref{lemma:sum} there exists $\ctneuf^{(b)} \geq 0$ such that
\begin{equation}
  \sum_{k=0}^{N-1} \defEns{\gamma^2(k+1)^{-2\alpha} + \gamma (k+1)^{-(1+\alpha)}} \leq \ctneuf^{(b)} \gamma \eqsp .
\end{equation}
Combining these last two results concludes the proof.
\end{proof}

%% file: appendix_strongly.tex
\section{Strongly-Convex case (under \Cref{assum:grad_sto}-\ref{item:approx_sto})}
\label{sec:strongly-convex_appendix}

In this section, we gather the proofs for the study of the long-time behavior of
SGD in the strongly convex case. Note that all of our proofs are derived under
\Cref{assum:grad_sto}-\ref{item:approx_sto}. We refer to \Cref{app:fz_strongly}
for similar results under \Cref{assum:grad_sto}-\ref{item:ml}. First, we start
by deriving and proving our main results in the strongly convex case both for
the continuous-time and the discrete-time dynamics in
\Cref{sec:conv-strongly-conv}. Then, we refine our study to explicit the
dependency of the constant w.r.t. to the parameters of the problem in
\Cref{sec:dependency-constant}. Finally, we show that our results can be extended
to cover the case where the strongly convex assumption is replaced by a weaker
Kurdyka-\LL condition, in \Cref{sec:conv-results-under}.

\subsection{Convergence results in the strongly convex case}
\label{sec:conv-strongly-conv}

First, we begin by deriving \Cref{thm:strong_continuous_f} which is a
consequence of \Cref{thm:strong_continuous} and provides convergence rates for
$(\expe{f(\bfX_t)}-\min_{\rset^d} f)_{t \geq 0}$. Then, we turn to the study of
the discrete-time setting. We start by giving the proof of
\Cref{lemma:ode_discrete}. The discrete analogous of
\Cref{thm:strong_continuous} is given in \Cref{thm:strongly_discrete}. Similarly
the discrete-time counterpart to \Cref{thm:strong_continuous_f} is given in
\Cref{thm:strongly_discrete_f}.

\begin{corollary_colt}
	\label{thm:strong_continuous_f}
	Let $\alpha, \gamma \in \ooint{0,1}$ and $(\bfX_t)_{t\geq0}$ be given by
        \eqref{eq:sde}.  Assume
        \rref{assum:f_lip},\tup{\rref{assum:grad_sto}-\ref{item:approx_sto}},
        \rref{assum:lip_sigma} and
        \tup{\rref{assum:strongly_cvx}-\ref{item:f_str_conv}}. Then there exists
        $C\geq 0$ such that for any $T>0$,
        $\expe{f(\bfX_T)}-\min_{\rset^d} f \leq C T^{-\alpha}$.
\end{corollary_colt}
\begin{proof}
  The proof is a direct consequence of \rref{assum:f_lip}, \citep[Lemma
  1.2.3]{nesterov2004introductory} and \Cref{thm:strong_continuous}.
\end{proof}
\begin{proof}[Proof of \Cref{lemma:ode_discrete}]
    Assume that there exists $n \in \bN$ such that $u_n>B$, and let
  $n_1= \inf \ensemble{n \geq 0 }{ u_n > B }$.  By definition of $B$ we have
  $n_1 \geq n_0+1$. Moreover we have $u_{n_1}-u_{n_1-1}\leq
  F(n_1-1,u_{n_1-1})$. Since $n_1-1 \geq n_0$ we get that
  $u_{n_1}-u_{n_1-1} \leq A_2$ and $u_{n_1-1} \geq u_{n_1}-A_2 \geq A_1$.
  Consequently, $F(n_1-1, u_{n_1-1})<0$ and $u_{n_1}<u_{n_1-1}$, which is a
  contradiction.
\end{proof}
We state a discrete analogous of \Cref{thm:strong_continuous}. Note that the
proof is considerably simpler than the one of~\citep{bachmoulines2011}.
\begin{theorem}
	\label{thm:strongly_discrete}
	Let $\gamma \in \ooint{0,1}$ and $\alpha \in \ocint{0,1}$. Let
        $(X_n)_{n\geq 0}$ be given by \eqref{eq:sgd}.
        Assume \tup{\rref{assum:grad_sto}-\ref{item:approx_sto}}
        and \tup{\rref{assum:strongly_cvx}-\ref{item:f_str_conv}}. Then there
        exists $\Cstrongdisc>0$ such that for all $N\geq 1$,
	\begin{equation}
		\expeLigne{\norm{X_N-x\st}^2} \leq \Cstrongdisc N^{-\alpha} \eqsp .
	\end{equation}
	In the case where $\alpha=1$ we have to assume additionally that $\gamma>1/(2\mu)$.
\end{theorem}
\begin{proof}
	Let $\gamma \in \ooint{0,1}$ and $\alpha \in \ocint{0,1}$. Let $(X_n)_{n\geq 0}$ be given by \eqref{eq:sgd}.
	Using \rref{assum:strongly_cvx}-\ref{item:f_str_conv} we get for all $n \geq 0$,
	\begin{align}		
		\expen{\norm{X_{n+1}-x\st}^2} &= \expe{\norm{X_n-x\st-\gamma(n+1)^{-\alpha}H(X_n,Z_{n+1})}^2} \\
		&= \norm{X_n-x\st}^2+\gamma^2(n+1)^{-2\alpha}\expen{\norm{H(X_n,Z_{n+1})}^2} \\
		&\quad-2\gamma(n+1)^{-\alpha}\expen{\la X_n-x\st, H(X_n, Z_{n+1}) \ra} \\
		&\leq \norm{X_n-x\st}^2+\gamma^2(n+1)^{-2\alpha}\parentheseDeux{\eta+\norm{\nabla f(X_n)}^2}\\
                                              &\quad -2\gamma(n+1)^{-\alpha}\la X_n-x\st, \nabla f(X_n) \ra \eqsp .
        \end{align}
        Therefore, we have
        \begin{multline}
		\expe{\norm{X_{n+1}-x\st}^2} \\ \leq \expe{\norm{X_n-x\st}^2}\parentheseDeux{1-2\gamma (n+1)^{-\alpha}\mu + \gamma^2(n+1)^{-2\alpha}\Lip^2} + \eta \gamma^2(n+1)^{-2\alpha} \eqsp . \label{eq:strong_cvx_disc_base}
	\end{multline}
	We note now $u_n \doteq \expe{\norm{X_n-x\st}^2}$ and
        $v_n \doteq n^{\alpha} u_n$. Using \eqref{eq:strong_cvx_disc_base} and
        Bernoulli's inequality we have, for all $n \geq 0$
	\begin{align}
		v_{n+1}-v_n &= (n+1)^{\alpha} u\nn - n^{\alpha}u_n \\
		&= (n+1)^{\alpha} (u\nn-u_n))+u_n((n+1)\pal - n\pal) \\
		&\leq \parentheseDeux{-2\gamma \mu + \gamma^2\Lip^2 (n+1)\pmal}u_n+\eta\gamma^2(n+1)\pmal+u_n n\pal \parentheseDeux{(1+1/n)\pal - 1} \\
		&\leq \parentheseDeux{-2\gamma \mu + \gamma^2\Lip^2 (n+1)\pmal+\alpha n^{\alpha-1}}u_n+\eta\gamma^2(n+1)\pmal \eqsp .
	\end{align}
	Therefore, in the case where $\alpha<1$, there exists $n_0 \geq 0$ such that for all $n \geq n_0$,
	\begin{align}
		v\nn-v_n &\leq -\gamma\mu u_n+\eta\gamma^2(n+1)\pmal \\
		&\leq -\gamma\mu n\pmal v_n+\eta\gamma^2(n+1)\pmal\leq (n+1)\pmal(-\gamma\mu v_n +\eta\gamma^2) \eqsp .
	\end{align}
	And in the case where $\alpha=1$, if $\gamma>1/(2\mu)$ we have the existence of $n_1 \geq 0$ such that for all $n \geq n_1$,
	\begin{equation}
		v_{n+1}-v_n\leq \parentheseDeux{(1/2-\gamma \mu) + \gamma^2\Lip^2 (n+1)\pmal+\alpha n^{\alpha-1}}u_n+\eta\gamma^2(n+1)\pmal \eqsp .
	\end{equation}
	Using \Cref{lemma:ode_discrete} this shows that, for $\alpha \in \ocint{0,1}$, there exists a constant $\Cstrongdisc > 0$ such that for all $n \geq 0$,
	$v_n \leq \Cstrongdisc$. This proves the result.
\end{proof}
Using \rref{assum:f_lip} and the descent lemma~\citep[Lemma
  1.2.3]{nesterov2004introductory} we have the immediate corollary

\begin{corollary_colt}
	\label{thm:strongly_discrete_f}
	Let $\alpha \in \ocint{0,1}$ and $\gamma \in \ooint{0,1}$. Let
        $(X_n)_{n\geq 0}$ be given by \eqref{eq:sgd}.  Assume
        \rref{assum:f_lip},\tup{\rref{assum:grad_sto}-\ref{item:approx_sto}} and
        \tup{\rref{assum:strongly_cvx}-\ref{item:f_str_conv}}. Then there exists
        $\Cstrongdiscf>0$ such that for all $N\geq 1$,
	\begin{equation}
		\expe{f(X_N)-f\st} \leq \Cstrongdiscf N^{-\alpha} \eqsp .
	\end{equation}
	If $\alpha=1$ we have also assumed that $\gamma>1/(2\mu)$.
      \end{corollary_colt}

\subsection{Quantitative constants in the strongly convex setting}
\label{sec:dependency-constant}

We first state in \Cref{lemma:ode_plus} a specific version of \Cref{lemma:ode}
in the case where there exists $t_0 > 0$ such that for any $t \geq 0$ and
$F(t, x) \geq -\rmf(x) \rmg(t)$ with $\rmf$ superlinear. In particular, this
lemma allows to obtain (i) an exponential forgetting of the initial conditions,
(ii) a more explicit expression of the constant appearing in
\Cref{lemma:ode}. The improved version of \Cref{thm:strong_continuous} with
explicit constants is stated \Cref{thm:explicit_bounds}.

\begin{lemma_colt}
  \label{lemma:ode_plus}
  Let $F \in \rmc^1(\bR_+\times \bR, \bR)$ and
  $v \in \rmc^1(\bR_+,\bR_+)$ such that for all $t \geq 0$,
  $\rmd v(t)/\rmd t \leq F(t,v(t))$. Assume that there exist
  $\rmf: \ \rset \to \rset$, 
  $\rmg \in \rmc(\rset_+, \rset_+)$, $t_0 >0$,
  $A \geq 0$ and $\upbeta >0$ such that the following conditions hold.
  \begin{enumerate}[label=(\alph*)]
  \item \label{item:a_opt} For any $t \geq t_0$, $r \in \ocint{0, 1}$ and $u \geq 0$, $r F(t,u) \leq F(t, ru)$.
  \item \label{item:b_opt} For any $t \geq t_0$ and $u \geq 0$, $F(t,u) \leq -\rmf(u) \rmg(t)$.
  \item \label{item:c_opt} For any $u \geq A$, $\rmf(u) > \upbeta u$.
  \end{enumerate}
  Then, for any $t \geq 0$,
  \begin{equation}
    v(t) \leq \max\defEnsLigne{A, \exp[\upbeta(G(t_0) -G(t))]\max_{s \in \ccint{0, t_0}}v(s)} \eqsp,
    \end{equation}
    with $G(t) = \int_0^t \rmg(s) \rmd s$.
\end{lemma_colt}

\begin{proof}
  Let $T \geq 0$ and $v_T(t) = v(t) \exp[\upbeta(G(t) - G(T))]$. Using
  condition \ref{item:a_opt} and that $G$ is
  non-decreasing since for any $t \geq 0$, $\rmg(t) \geq 0$, we have
  for any $t \in \ocint{0,T}$
  \begin{equation}
    \rmd v_T(t) / \rmd t \leq \exp[\upbeta(G(t) - G(T))] F(t, v(t)) + \upbeta \rmg(t) v_T(t) \leq F(t, v_T(t)) + \upbeta \rmg(t) v_T(t) \eqsp .
  \end{equation}
  Using this result and conditions \ref{item:b_opt}-\ref{item:c_opt},
  we have for any $t \geq t_0$ such that $v_T(t) \geq A$
  \begin{equation}
    \label{eq:ineq_yT}
     \rmd v_T(t) / \rmd t \leq -\rmf(v_T(t)) \rmg(t) + \upbeta v_T(t) \rmg(t) < 0 \eqsp .
   \end{equation}
   Let $B = \max(A, \max_{s \in \ccint{0,t_0}} v_T(s))$. Assume that
   $\msa = \ensembleLigne{t \in \ccint{0,T}}{v_T(t) > B} \neq \emptyset$
   and let $t_1 = \inf \msa$. Note that $t_1 \geq t_0$ and
   $v_T(t_1) \geq A$. Therefore, using \eqref{eq:ineq_yT} we have
   $ \rmd v_T(t_1) / \rmd t < 0$ and therefore, there exists
   $0 < t_2 < t_1$ such that $v_T(t_2) > v_T(t_1)$ but then
   $t_2 \in \msa$ and $t_2 < \inf \msa$. Hence, $\msa = \emptyset$ and
   we get that for any $t \in \ccint{0,T}$, $v_T(t) \leq B$.
   Therefore, we get that for any $t \geq 0$,
   \begin{equation}
     v(t) = v_t(t) \leq \max\defEnsLigne{A, \exp[\upbeta(G(t_0) - G(t))] \max_{s \in \ccint{0,t_0}} v(s)} \eqsp ,
   \end{equation}
   which concludes the proof.
\end{proof}

\begin{theorem}
  \label{thm:explicit_bounds}
	\label{prop:constante}
	Let $\alpha, \gamma \in \ooint{0,1}$ and $(\bfX_t)_{t\geq0}$ be given by
        \eqref{eq:sde}.  Assume \rref{assum:f_lip},
        \tup{\rref{assum:grad_sto}-\ref{item:approx_sto}},
        \rref{assum:lip_sigma} and
        \tup{\rref{assum:strongly_cvx}-\ref{item:approx_sto}}. Then for any
        $T \geq 0$,
        \begin{equation}
          \expeLigne{\norm{\bfX_T-x\st}^2} \leq \max \defEns{4 \gua \eta / \mu, C \expeLigne{\normLigne{\bfX_0 - x^{\star}}^2} \exp[-\mu (\gua + T)^{1-\alpha}/(2-2\alpha)]} ( \gua + T)^{-\alpha} \eqsp ,
        \end{equation}
        with
        \begin{equation}
          C = (1 + \eta \Psi(\alpha, t_0)) \exp[\mu (\gua +
t_0)^{1-\alpha}/(2-2\alpha)](\gua + t_0)^{\alpha} \eqsp .
        \end{equation}

\end{theorem}

\begin{proof}
  Let $\alpha, \gamma \in \ocint{0,1}$ and consider $\mce : \rset_+ \to \rset_+$
  defined for $t \geq 0$ by
  $\mce(t) = \expeLigne{(t+\gua)^{\alpha} \normLigne{\bfX_t - x^{\star}}^2}$,
  with $\gua = \gamma^{1/(1-\alpha)}$.  Using Dynkin's formula, see
  \Cref{lemma:dynkin}, we have for any $t\geq 0$,
\begin{align}
  \mce(t)=\mce(0) + \alpha \int_0^t  \frac{\mce(s)}{s+\gua}\dd s   + \int_0^t \gua \frac{\expe{\trace (\Sigma(\bfX_s))}}{ (s+\gua)^{\alpha}} \dd s -2 \int_0^t \expe{\la \nabla f(\bfX_s), \bfX_s-x\st \ra} \dd s
           \eqsp . \label{eq:strong_ito}
\end{align}
We now differentiate this expression with  respect to $t$ and
using \Cref{assum:strongly_cvx} and \Cref{assum:grad_sto}, we get for any $t > 0$,
\begin{align}
\rmd   \mce(t)/\rmd t &=\alpha \mce(t)(t+\gua)^{-1} -2 \expe{\la \nabla f(\bfX_t), \bfX_t-x\st \ra} + \gua \expe{\trace (\Sigma(\bfX_t))} (t+\gua)^{-\alpha} \\
  &\leq \alpha \mce(t)/(t+\gua) -2\mu \expeLigne{\normLigne{\bfX_t-x\st}^2} + \gua \eta/ (t+\gua)^{\alpha} \\
  &\leq F(t,\mce(t)) = \alpha \mce(t)(t+\gua)^{-1} -2\mu \mce(t)(t+\gua)^{-\alpha} + \gua \eta (t+\gua)^{-\alpha} \eqsp,
\end{align}
where we have used in the penultimate line that
$\trace(\Sigma(x)) \leq \eta$ for any $x \in \rset^d$ by
\Cref{assum:grad_sto}. Let
$t_0 = \max((\alpha/\mu)^{1/(1-\alpha)}-\gua, \gua)$. We have for any $t \geq t_0$, and $u \geq 0$
\begin{equation}
  F(t,u) \leq -\rmf(u) \rmg(t) \eqsp, \qquad \rmg(t) = (t+ \gua)^{-\alpha} \eqsp , \qquad \rmf(u) =\mu u - \gua \eta \eqsp .
\end{equation}
Hence the conditions \ref{item:a_opt} and \ref{item:b_opt} of
\Cref{lemma:ode_plus} are satisfied.  Let $\upbeta = \mu/2$ and
$A = 4 \gua \eta / \mu$. We obtain that for any $t \geq t_0$ and
$u \geq A$, $\rmf(u) > \mu u / 2$ and therefore condition
\ref{item:c_opt} of \Cref{lemma:ode_plus} is satisfied. Applying
\Cref{lemma:ode_plus}, we obtain that for any $t \geq 0$
\begin{equation}
  \mce(t) \leq \max( 4 \gua \eta / \mu, \exp[-\mu (\gua + t)^{1-\alpha}/(2-2\alpha)] B) \eqsp ,
\end{equation}
with
$B = \exp[\mu (\gua + t_0)^{1-\alpha}/(2-2\alpha)] \max_{s \in
  \ccint{0,t_0}} \mce(s)$.  We have that
$\max_{s \in \ccint{0,t_0}} \mce(s) \leq (t_0 + \gua)^{\alpha} \max_{s
  \in \ccint{0,t_0}} \expeLigne{\normLigne{\bfX_s -
    x^{\star}}}^2$. Using Dynkin's formula, see \Cref{lemma:dynkin},
we have for any $t\geq 0$,
\begin{equation}
  \expe{\norm{\bfX_t- x^{\star}}}^2 \leq \expe{\norm{\bfX_0- x^{\star}}}^2 + \eta \Psi(\alpha, t_0) \eqsp ,
\end{equation}
with
\begin{equation}
  \Psi(\alpha, t_0) =  \left \lbrace
  \begin{aligned}
    & \gamma^2 / (2\alpha - 1) & \text{if $2\alpha >1$ } \eqsp , \\
    & \gua \log(\gua^{-1} (t_0 + \gua))& \text{if $2\alpha = 1$ } \eqsp , \\
    & \gua (t_0 + \gua)^{1-2\alpha}/(1 - 2\alpha) & \text{otherwise} \eqsp .
  \end{aligned}
  \right.
\end{equation}
We conclude the proof upon setting
$C = (1 + \eta \Psi(\alpha, t_0)) \exp[\mu (\gua +
t_0)^{1-\alpha}/(2-2\alpha)](\gua + t_0)^{\alpha}$.
\end{proof}

      \subsection{Convergence results under Kurdyka-\LL \ conditions}
\label{sec:conv-results-under}

We state now an equivalent result of \Cref{thm:strong_continuous_f} under weaker
assumptions, namely the \LL inequality with $r=2$, that we restate as it is
usually given, with $c>0$, \ie, for any $x \in \rset^d$,
\begin{equation}
	\label{eq:loj2}
	 f(x)-f(x\st) \leq c\normLigne{\nabla f(x)}^2 \eqsp .
\end{equation}
Note that \eqref{eq:loj2} is verified for all strongly convex
functions~\citep{karimi}. The equivalent of \Cref{thm:strong_continuous_f} is
stated in \Cref{prop:kl2} (for the continuous-time process). The equivalent of
\Cref{thm:strongly_discrete_f} is given in \Cref{prop:kl2_discrete} (for the
discrete-time process).

\begin{proposition_colt}
	\label{prop:kl2}
	Let $\alpha, \gamma \in \ooint{0,1}$ and $(\bfX_t)_{t\geq0}$ be given by
        \eqref{eq:sde}.  Assume \rref{assum:f_lip}, \tup{\rref{assum:grad_sto}-\ref{item:approx_sto}}, \rref{assum:lip_sigma} and that $f$ verifies \eqref{eq:loj2}. Then there exists
        $\Cstrongloj>0$ such that for any $T > 0$,
        \begin{equation}
        	\expe{f(\bfX_T)-f\st} \leq \Cstrongloj T^{-\alpha} \eqsp .
        \end{equation}
\end{proposition_colt}

\begin{proof}
Let $\alpha, \gamma \in \ooint{0,1}$ and $(\bfX_t)_{t\geq0}$ be given by \eqref{eq:sde}. Without loss of generality we can assume that $f\st=\min_{x\in\bRd}f(x)=0$.
We note $\mce(t)=(t+\gua)^{\alpha}\expe{f(\bfX_t)}$
and we apply Lemma~\ref{lemma:dynkin} to the stochastic process
$((t+\gua)^{\alpha}f(\bfX_t))_{t \geq 0}$, and using \rref{assum:f_lip}, \rref{assum:grad_sto}-\ref{item:approx_sto}, \rref{assum:lip_sigma}, \eqref{eq:loj2} and Lemma~\ref{lemma:bound_scal} this gives, for all $t>0$,
\begin{align}
	\mce(t)-\mce(0) &= \int_0^t \alpha (s+\gua)^{\alpha-1} \expe{f(\bfX_s)}\dd s - \int_0^t \expe{\norm{\nabla f(\bfX_s)}^2} \dd s \\
	&\quad+ (\gua/2)\int_0^t (s+\gua)^{-\alpha}\expe{\trace(\nabla^2f(\bfX_s)\Sigma(\bfX_s)) }\dd s \\
	\dd \mce(t)/\dd t&\leq\alpha \mce(t)(t+\gua)^{-1}-(1/c)\mce(t)(t+\gua)^{-\alpha} + (\gua/2)\boundLSig (t+\gua)^{-\alpha} \eqsp .
\end{align}
We can now apply Lemma~\ref{lemma:ode} to $F(t,x)=\alpha x(t+\gua)^{-1}-(1/c)x(t+\gua)^{-\alpha} + (\gua/2)\boundLSig (t+\gua)^{-\alpha}$ with $t_0=(2c\alpha)^{1/(1-\alpha)}$ and $A=2\gua c\boundLSig$, which shows the existence of $\Cstrongloj>0$ such that for all $t>0$, $\mce(t)\leq \Cstrongloj$, concluding the proof.
\end{proof}

And we now state its discrete counterpart, which is an equivalent of \Cref{thm:strongly_discrete_f}.
\begin{proposition_colt}
	\label{prop:kl2_discrete}
	Let $\alpha \in \ocint{0,1}$ and $\gamma \in \ooint{0,1}$. Let
        $(X_n)_{n\geq 0}$ be given by \eqref{eq:sgd}.  Assume
        \rref{assum:f_lip}, \tup{\rref{assum:grad_sto}-\ref{item:approx_sto}}
        and that $f$ verifies \eqref{eq:loj2}. Then there exists
        $\Cstronglojdisc>0$ such that for all $N\geq 1$,
	\begin{equation}
		\expe{f(X_N)-f\st} \leq \Cstronglojdisc N^{-\alpha} \eqsp .
	\end{equation}
	In the case where $\alpha=1$ we have to assume additionally that $\gamma>2/c$.
\end{proposition_colt}

\begin{proof}
  Let $\alpha \in \ocint{0,1}$ and $\gamma \in \ooint{0,1}$. Let
  $(X_n)_{n\geq 0}$ be given by \eqref{eq:sgd}. Let $n\geq 0$. Applying the
  descent lemma \citep[Lemma 1.2.3]{nesterov2004introductory} (using
  \rref{assum:f_lip}) we get
	\begin{align}
		\expen{f(X\nn)}&=\expen{f(X_n-\gamma/(n+1)\pal H(X_n,Z\nn))} \\
		&\leq f(X_n)-\gamma/(n+1)\pal\expen{\la \nabla f(X_n), H(X_n,Z\nn) \ra} \\
		&\quad+ \gamma^2/(n+1)^{2\alpha} (\Lip/2)\expenLigne{\norm{H(X_n,Z\nn)}^2} \\
		&\leq f(X_n)-\gamma/(n+1)\pal\norm{\nabla f(X_n)}^2+(\Lip\gamma^2/2)(n+1)^{-2\alpha}\parentheseDeux{\eta+\norm{\nabla f(X_n)}^2} \\
		\expe{f(X\nn)}-f\st &\leq \expe{f(X_n)}-f\st + \gamma(n+1)\pmal\expeLigne{\normLigne{\nabla f(X_n)}^2}\parentheseDeux{-1 + (\Lip\gamma/2)(n+1)^{-\alpha}} \\
		&\quad+(\Lip\gamma^2/2)(n+1)^{-2\alpha} \eta \eqsp .
	\end{align}
	This shows the existence of $n_2 \geq 0$ such that using \eqref{eq:loj2} we have for all $n \geq n_2$,
	\begin{align}
		\expe{f(X\nn)}-f\st &\leq \expe{f(X_n)}-f\st - (\gamma/2)(n+1)\pmal\expe{\norm{\nabla f(X_n)}^2} +(\Lip\gamma^2/2)(n+1)^{-2\alpha} \eta \\
		&\leq (\expe{f(X_n)}-f\st)\parentheseDeux{1 - (\gamma c\pinv/2)(n+1)\pmal}+(\Lip\gamma^2/2)(n+1)^{-2\alpha} \eta \eqsp .
	\end{align}
	We note now for all $n \geq 0$, $u_n=\expe{f(X_n)}-f\st$ and $v_n=n\pal u_n$. We have
	\begin{align}
		v_{n+1}-v_n &= (n+1)^{\alpha} u\nn - n^{\alpha}u_n \\
		&= (n+1)^{\alpha} (u\nn-u_n))+u_n((n+1)\pal - n\pal) \\
		&\leq -(\gamma c\pinv/2) u_n +(\Lip\gamma^2\eta/2)(n+1)\pmal +u_n n\pal \parentheseDeux{(1+1/n)\pal - 1} \\
		&\leq u_n(-(\gamma c\pinv/2) + \alpha n^{\alpha-1})+(\Lip\gamma^2\eta/2)(n+1)\pmal \eqsp .
	\end{align}
	If $\alpha <1$, or if $1-\gamma c\pinv/2<0$ we have the existence of $n_3 \geq n_2$ and $\Cstrongtilde>0$ such that for all $n \geq n_3$,
	\begin{align}
		v\nn-v_n &\leq -\Cstrongtilde u_n + (\Lip\gamma^2\eta/2)(n+1)\pmal \\
		&\leq \defEns{-\Cstrongtilde v_n + (\Lip\gamma^2\eta/2)}(n+1)\pmal
	\end{align}
	This proves the existence of $\Cstronglojdisc > 0$ such that for all
        $n \geq 0$, $v_n \leq \Cstronglojdisc$, which concludes the proof.
\end{proof}


%% file: appendix_fz_strongly.tex
\section{Strongly convex case (under \Cref{assum:grad_sto}-\ref{item:ml})}
\label{app:fz_strongly}

This section gather the proofs for the study of the strongly convex case under
\Cref{assum:grad_sto}-\ref{item:ml}. It is the counterpart of
\Cref{sec:strongly-convex_appendix}. We start by establishing useful lemmas
under \Cref{assum:grad_sto}-\ref{item:ml} in \Cref{sec:technical-results}. Then
we present the counterpart of the results obtained in
\Cref{sec:conv-strongly-conv} in \Cref{sec:equiv-crefs-results}.

\subsection{Technical results}
\label{sec:technical-results}

We begin by several lemmas to control $\trace \Sigma$ and $\expeLigne{\normLigne{\nabla \f}^2}$.
We will note $\trcst = 6\Lip^2+4\Lip+3\eta$.

\begin{lemma_colt}
	\label{lemma:gradfz_norm}
	Assume \tup{\rref{assum:grad_sto}-\ref{item:ml}}. Then, for all $x \in
        \bR^d$, we have
	\begin{equation}
		\int_{\msz} \normLigne{\nabla \f(x,z)}^2 \dd \muz(z) \leq \trcst(\norm{x-x\st}^2 + 1) \eqsp .
	\end{equation}
\end{lemma_colt}

\begin{proof}
  Let $x\in \bR^d$. Using \tup{\rref{assum:grad_sto}-\ref{item:ml}} we have 
\begin{align}
	\int_{\msz} \normLigne{\nabla \f(x,z)}^2 \dd \muz(z) &= \int_{\msz} \normLigne{\nabla \f(x,z) - \nabla \f(x\st,z)+\nabla \f(x\st,z)}^2 \dd \muz(z) \\
	& \leq 2\int_{\msz} \normLigne{\nabla \f(x,z)-\nabla \f(x\st,z)}^2 + \int_{\msz} \normLigne{\nabla \f(x\st,z)}^2 \dd \muz(z) \dd \muz(z) \\
	& \leq 2 \Lip^2 \normLigne{x-x\st}^2 + 2\eta \eqsp ,
\end{align}
which concludes the proof.
\end{proof}

\begin{lemma_colt}
	\label{lemma:trace_norm}
	Assume \tup{\rref{assum:grad_sto}-\ref{item:ml}}. Then, for all $x \in
        \bR^d$, we have
	\begin{equation}
		\trace(\Sigma(x)) \leq \trcst \parenthese{1+\norm{x-x\st}^2} \eqsp .
	\end{equation}
\end{lemma_colt}

\begin{proof}
Let $x \in \bR^d$. Using \tup{\rref{assum:grad_sto}-\ref{item:ml}}  we have
\begin{align}
  \trace(\Sigma(x))&=\trace \parenthese{\int_{\msz} (\nabla \f(x,z)-\nabla f(x))(\nabla \f(x,z)-\nabla f(x))\T \dd \muz(z)} \\
                   &=\int_{\msz} \normLigne{\nabla \f(x,z)-\nabla f(x)}^2 \dd \muz(z) \\
                   &=\int_{\msz} \normLigne{\nabla \f(x,z)-\nabla \f(x\st,z)+\nabla \f(x\st,z)-\nabla f(x)}^2 \dd \muz(z) \\
                   &\leq 3\int_{\msz} \parenthese{\normLigne{\nabla \f(x,z)-\nabla \f(x\st,z)}^2 + \normLigne{\nabla \f(x\st,z)}^2 + \normLigne{\nabla f(x)}^2} \dd \muz(z) \\
                   &\leq 6 \Lip^2 \normLigne{x-x\st}^2+3\eta \eqsp \leq \trcst \parenthese{1+\normLigne{x-x\st}^2} \eqsp ,
\end{align}
which concludes the proof.
\end{proof}

\begin{lemma_colt}
	\label{lemma:ab}
	Let $a,b \in \bR^d$. Then $\norm{a+b}^2 \geq \norm{a}^2/2-\norm{b}^2 \eqsp .$
\end{lemma_colt}
\begin{proof}
Let $a,b \in \bR^d$. Using the fact that $2xy \leq 2y^2+x^2/2$ for all $x,y \in \bR$, we have
\begin{align}
  \norm{a+b}^2 &=\norm{a}^2+\norm{b}^2+2\la a,b \ra \\
               &\geq \norm{a}^2+\norm{b}^2-2\norm{a}\norm{b}\geq \norm{a}^2+\norm{b}^2 -2\norm{b}^2-\norm{a}^2/2 \geq \norm{a}^2/2-\norm{b}^2 \eqsp .
\end{align}
\end{proof}

\begin{lemma_colt}
  \label{lemma:kolmo}
  Let $f \in \rmc^1(\rset^d, \rset)$. Assume that there exists $\Lip \geq 0$ such that for any $x,y \in \rset^d$, $\nabla f$ is $\Lip$-Lipschitz. Then for any $x \in \rset^{\dim}$
  \begin{equation}
    \label{eq:majo}
    \norm{\nabla f(x)}^2 \leq 2 \Lip (f(x) - \inf_{\rset^d} f) \eqsp .
  \end{equation}
\end{lemma_colt}
\begin{proof}
  Using \cite[Lemma 1.2.3]{nesterov2004introductory}, we have for any $x,y \in \rset^d$
  \begin{equation}
    f(y) - f(x) \leq \langle \nabla f(x), y - x \rangle + (\Lip/2) \norm{y-x}^2 \eqsp .
  \end{equation}
  We obtain \eqref{eq:majo} by minimizing both side of the previous inequality
  w.r.t. $y$.
\end{proof}

\begin{lemma_colt}
	\label{lemma:gradfz}
	Assume \tup{\rref{assum:grad_sto}-\ref{item:ml}}. In addition, assume
        that $\f(\cdot,z)$ is convex for all $z\in \msz$.  For all
        $x \in \bR^d$, we have
	\begin{equation}
\int_{\msz} \normLigne{\nabla \f(x,z)}^2 \dd \muz(z) \leq \trcst \parenthese{f(x)-f(x\st) + 1} \eqsp .
\end{equation}
\end{lemma_colt}

\begin{proof}
Let $x\in \bR^d$ and $z\ \in \msz$. Using the smoothness and convexity of $\f(\cdot,z)$, taking the expectation and using \Cref{lemma:ab} we have
\begin{align}
	\f(x,z)-\f(x\st,z) &\geq \la \nabla \f(x\st,z), x-x\st \ra + (1/2\Lip)\normLigne{\nabla \f(x,z)-\nabla \f(x\st,z)}^2 \\
	f(x)-f(x\st) &\geq (1/2\Lip) \expeLigne{\normLigne{\nabla \f(x,z)-\nabla \f(x\st,z)}^2} \\
	& \geq (1/4\Lip)\expeLigne{\normLigne{\nabla \f(x,z)}^2}-(1/2\Lip) \expeLigne{\normLigne{\nabla f(x)}^2} \eqsp ,
\end{align}
We conclude upon combining this result with \Cref{lemma:kolmo}.
\end{proof}

\begin{lemma_colt}
	\label{lemma:trace_f}
	Assume \rref{assum:f_lip}, \tup{\rref{assum:grad_sto}-\ref{item:ml}} and
        \rref{assum:lip_sigma}. Assume additionally that $\f(\cdot,z)$ is convex
        for all $z\in \msz$.  For all $x \in \bR^d$, we have
	\begin{equation}
		\trace(\Sigma(x)) \leq  \trcst \parenthese{f(x)-f(x\st) + 1} \eqsp .
	\end{equation}
\end{lemma_colt}

\begin{proof}
Let $x\in \bR^d$. Then using \Cref{assum:grad_sto}-\ref{item:ml} and \Cref{lemma:gradfz} we have
\begin{align}
  \trace\parenthese{\Sigma(x)} &=\int_{\msz} \normLigne{\nabla \f(x,z)-\nabla f(x)}^2 \dd \muz(z) \\
                               &=\int_{\msz}  \normLigne{\nabla \f(x,z)}^2 + \normLigne{\nabla f(x)}^2 -2\la \nabla \f(x,z), \nabla f(x) \ra \dd \muz(z) \\
                               &=\int_{\msz}  \normLigne{\nabla \f(x,z)}^2 \dd \muz(z) - \normLigne{\nabla f(x)}^2 \\
  &\leq \int_{\msz}  \normLigne{\nabla \f(x,z)}^2 \dd \muz(z) \leq \trcst \parenthese{f(x)-f(x\st) + 1} \eqsp ,
\end{align}
which concludes the proof.
\end{proof}

\subsection{Equivalent to \Cref{sec:conv-strongly-conv}}
\label{sec:equiv-crefs-results}

The equivalent of \Cref{thm:strong_continuous} and \Cref{thm:strongly_discrete} are given in
\Cref{thm:strongly_b_2} and \Cref{thm:strongly_b_1} respectively.

\begin{theorem}
  \label{thm:strongly_b_2}
  Let $\alpha, \gamma \in \ooint{0,1}$ and $(\bfX_t)_{t\geq0}$ be given by
  \eqref{eq:sde}.  Assume \rref{assum:f_lip},
  \tup{\rref{assum:grad_sto}-\ref{item:ml}}, \rref{assum:lip_sigma} and
  \tup{\rref{assum:strongly_cvx}-\ref{item:ftilde_str_conv}}. Then there exists $C\geq0$ (explicit in the proof)
  such that for any $T \geq 1$,
  $ \expeLigne{\norm{\bfX_T-x\st}^2} \leq C T^{-\alpha}$.
\end{theorem}

\begin{proof}
  Let $\alpha, \gamma \in \ocint{0,1}$ and consider $\mce : \rset_+ \to \rset_+$
  defined for $t \geq 0$ by
  $\mce(t) = \expeLigne{(t+\gua)^{\alpha} \normLigne{\bfX_t - x^{\star}}^2}$,
  with $\gua = \gamma^{1/(1-\alpha)}$.  Using Dynkin's formula, see
  \Cref{lemma:dynkin}, we have for any $t\geq 0$,
\begin{align}
  \mce(t)=\mce(0) + \alpha \int_0^t  \frac{\mce(s)}{s+\gua}\dd s   + \int_0^t \gua \frac{\expe{\trace (\Sigma(\bfX_s))}}{ (s+\gua)^{\alpha}} \dd s -2 \int_0^t \expe{\la \nabla f(\bfX_s), \bfX_s-x\st \ra} \dd s
           \eqsp .
\end{align}
We now differentiate this expression with  respect to $t$ and
using \Cref{assum:strongly_cvx}, \Cref{assum:grad_sto} and \Cref{lemma:trace_norm}, we get for any $t > 0$,
\begin{align}
&\rmd   \mce(t)/\rmd t =\alpha \mce(t)(t+\gua)^{-1} -2 \expe{\la \nabla f(\bfX_t), \bfX_t-x\st \ra} + \gua \expe{\trace (\Sigma(\bfX_t))} (t+\gua)^{-\alpha} \\
  & \eqsp \leq \alpha \mce(t)/(t+\gua) -2\mu \expeLigne{\normLigne{\bfX_t-x\st}^2} + \gua \trcst/ (t+\gua)^{\alpha} + \gua \trcst \expeLigne{\normLigne{\bfX_t-x\st}^2}  (t+\gua)^{-\alpha} \\
  & \eqsp \leq \alpha \mce(t)(t+\gua)^{-1} -2\mu \mce(t)(t+\gua)^{-\alpha} + \gua \trcst (t+\gua)^{-\alpha} + \gua\trcst \mce(t)(t+\gua)^{-2\alpha} \eqsp .
\end{align}
Hence, using \Cref{lemma:ode} we get, for any $t \geq 0$, $\mce(t) \leq B$, which concludes the proof.
\end{proof}

\begin{theorem}
  \label{thm:strongly_b_1}
  Let $\gamma \in \ooint{0,1}$ and $\alpha \in \ocint{0,1}$. Let
  $(X_n)_{n\geq 0}$ be given by \eqref{eq:sgd}.  Assume
  \tup{\rref{assum:grad_sto}-\ref{item:ml}} and
  \tup{\rref{assum:strongly_cvx}-\ref{item:ftilde_str_conv}}. Then there exists
  $\Cstrongdisc>0$ such that for all $N\geq 1$,
	\begin{equation}
		\expeLigne{\normLigne{X_N-x\st}^2} \leq \Cstrongdisc N^{-\alpha} \eqsp .
	\end{equation}
	In the case where $\alpha=1$ we have to assume additionally that $\gamma>1/(2\mu)$.
\end{theorem}

\begin{proof}
	Let $\gamma \in \ooint{0,1}$ and $\alpha \in \ocint{0,1}$. Let $(X_n)_{n\geq 0}$ be given by \eqref{eq:sgd}.
	Using \rref{assum:strongly_cvx}-\ref{item:ftilde_str_conv} and \Cref{lemma:gradfz_norm} we get for all $n \geq 0$,
	\begin{align}		
		\expen{\norm{X_{n+1}-x\st}^2} &= \expen{\norm{X_n-x\st-\gamma(n+1)^{-\alpha}\nabla \f(X_n,Z_{n+1})}^2} \\
		&= \norm{X_n-x\st}^2+\gamma^2(n+1)^{-2\alpha}\expen{\norm{\nabla \f(X_n,Z_{n+1})}^2} \\
		&\quad-2\gamma(n+1)^{-\alpha}\expen{\la X_n-x\st, \nabla \f(X_n, Z_{n+1}) \ra} \\
		&\leq \norm{X_n-x\st}^2+\trcst\gamma^2(n+1)^{-2\alpha}\parentheseDeux{1+\norm{X_n-x\st}^2}\\
    &\quad -2\gamma(n+1)^{-\alpha}\la X_n-x\st, \nabla f(X_n) \ra \eqsp .
   \end{align}
        Therefore, we have
        \begin{multline}
		\expe{\norm{X_{n+1}-x\st}^2} \\ \leq \expe{\norm{X_n-x\st}^2}\parentheseDeux{1-2\gamma (n+1)^{-\alpha}\mu + \gamma^2(n+1)^{-2\alpha}\trcst} + \trcst \gamma^2(n+1)^{-2\alpha} \eqsp . \label{eq:strong_cvx_disc_base_fz}
	\end{multline}
	We note now $u_n \doteq \expe{\norm{X_n-x\st}^2}$ and $v_n \doteq n^{\alpha} u_n$. Using \eqref{eq:strong_cvx_disc_base_fz} and Bernoulli's inequality we have, for all $n \geq 0$
	\begin{align}
		v_{n+1}-v_n &= (n+1)^{\alpha} u\nn - n^{\alpha}u_n \\
		&= (n+1)^{\alpha} (u\nn-u_n))+u_n((n+1)\pal - n\pal) \\
		&\leq \parentheseDeux{-2\gamma \mu + \gamma^2\trcst (n+1)\pmal}u_n+\trcst\gamma^2(n+1)\pmal+u_n n\pal \parentheseDeux{(1+1/n)\pal - 1} \\
		&\leq \parentheseDeux{-2\gamma \mu + \gamma^2\trcst (n+1)\pmal+\alpha n^{\alpha-1}}u_n+\trcst\gamma^2(n+1)\pmal \eqsp .
	\end{align}
	Therefore, in the case where $\alpha<1$, there exists $n_0 \geq 0$ such that for all $n \geq n_0$,
	\begin{align}
		v\nn-v_n &\leq -\gamma\mu u_n+\trcst\gamma^2(n+1)\pmal \\
		&\leq -\gamma\mu n\pmal v_n+\trcst\gamma^2(n+1)\pmal \\
		&\leq (n+1)\pmal(-\gamma\mu v_n +\trcst\gamma^2) \eqsp .
	\end{align}
	And in the case where $\alpha=1$, if $\gamma>1/(2\mu)$ we have the existence of $n_1 \geq 0$ such that for all $n \geq n_1$,
	\begin{equation}
		v_{n+1}-v_n\leq \parentheseDeux{(1/2-\gamma \mu) + \gamma^2\trcst (n+1)\pmal+\alpha n^{\alpha-1}}u_n+\trcst\gamma^2(n+1)\pmal \eqsp .
	\end{equation}
	Using \Cref{lemma:ode_discrete} this shows that, for $\alpha \in \ocint{0,1}$, there exists a constant $\Cstrongdisc > 0$ such that for all $n \geq 0$,
	$v_n \leq \Cstrongdisc$. This proves the result.
\end{proof}

\subsection{Equivalent to 
  \Cref{sec:conv-results-under}}
\label{sec:equiv-crefs-results}

The equivalent of \Cref{prop:kl2} and \Cref{prop:kl2_discrete} are given in
\Cref{thm:strongly_b_3} and \Cref{thm:strongly_b_4} respectively.

\begin{proposition_colt}
  \label{thm:strongly_b_3}
  Let $\alpha, \gamma \in \ooint{0,1}$ and $(\bfX_t)_{t\geq0}$ be given by
  \eqref{eq:sde}.  Assume \rref{assum:f_lip},
  \tup{\rref{assum:grad_sto}-\ref{item:ml}}, \rref{assum:lip_sigma} and that $f$
  verifies \eqref{eq:loj2}. Then there exists $\Cstrongloj>0$ such that for any
  $T > 0$,
        \begin{equation}
        	\expe{f(\bfX_T)-f\st} \leq \Cstrongloj T^{-\alpha} \eqsp .
        \end{equation}
\end{proposition_colt}

\begin{proof}
  Let $\alpha, \gamma \in \ooint{0,1}$ and $(\bfX_t)_{t\geq0}$ be given by
  \eqref{eq:sde}. Without loss of generality we can assume that
  $f\st=\min_{x\in\bRd}f(x)=0$.  We note
  $\mce(t)=(t+\gua)^{\alpha}\expe{f(\bfX_t)}$ and we apply
  Lemma~\ref{lemma:dynkin} to the stochastic process
  $((t+\gua)^{\alpha}f(\bfX_t))_{t \geq 0}$, and using \rref{assum:f_lip},
  \rref{assum:grad_sto}-\ref{item:ml}, \eqref{eq:loj2} and
  Lemma~\ref{lemma:trace_f} this gives, for all $t>0$,
\begin{align}
	\mce(t)-\mce(0) &= \int_0^t \alpha (s+\gua)^{\alpha-1} \expe{f(\bfX_s)}\dd s - \int_0^t \expe{\norm{\nabla f(\bfX_s)}^2} \dd s \\
	&\quad+ (\gua/2)\int_0^t (s+\gua)^{-\alpha}\expe{\trace(\nabla^2f(\bfX_s)\Sigma(\bfX_s)) }\dd s \\
	\dd \mce(t)/\dd t&\leq\alpha \mce(t)(t+\gua)^{-1}-(1/c)\mce(t)(t+\gua)^{-\alpha} + (\gua/2)\Lip \trcst\parenthese{1+\expe{f(\bfX_t)}} (t+\gua)^{-\alpha} \\
	& \leq \alpha \mce(t)(t+\gua)^{-1}-(1/c)\mce(t)(t+\gua)^{-\alpha} + (\gua/2)\Lip\trcst \mce(t)(t+\gua)^{-2\alpha} \\
	&\qquad + (\gua/2)\Lip \trcst(t+\gua)\pmal \eqsp .
\end{align}
We can now apply Lemma~\ref{lemma:ode}, concluding the proof.
\end{proof}
\begin{proposition_colt}
  \label{thm:strongly_b_4}
  Let $\alpha \in \ocint{0,1}$ and $\gamma \in \ooint{0,1}$. Let
  $(X_n)_{n\geq 0}$ be given by \eqref{eq:sgd}.  Assume \rref{assum:f_lip},
  \tup{\rref{assum:grad_sto}-\ref{item:ml}} and that $f$ verifies
  \eqref{eq:loj2}. Then there exists $\Cstronglojdisc>0$ such that for all
  $N\geq 1$,
	\begin{equation}
		\expe{f(X_N)-f\st} \leq \Cstronglojdisc N^{-\alpha} \eqsp .
	\end{equation}
	In the case where $\alpha=1$ we have to assume additionally that $\gamma>2/c$.
\end{proposition_colt}

\begin{proof}
	Let $\alpha \in \ocint{0,1}$ and $\gamma \in \ooint{0,1}$. Let $(X_n)_{n\geq 0}$ be given by \eqref{eq:sgd}. Let $n\geq 0$. Applying the descent lemma (using \rref{assum:f_lip}) and \Cref{lemma:gradfz} gives
	\begin{align}
		\expen{f(X\nn)}&=\expen{f(X_n-\gamma/(n+1)\pal \nabla \f(X_n,Z\nn)} \\
		&\leq f(X_n)-\gamma/(n+1)\pal\expen{\la \nabla f(X_n), \nabla \f(X_n,Z\nn) \ra} \\
		&\quad+ \gamma^2/(n+1)^{2\alpha} (\Lip/2)\expen{\norm{\nabla \f(X_n,Z\nn)}^2} \\
		&\leq f(X_n)-\gamma/(n+1)\pal\norm{\nabla f(X_n)}^2+(\Lip\gamma^2/2)(n+1)^{-2\alpha}\trcst\parentheseDeux{1+f(X_n)} \\
		\expe{f(X\nn)}-f\st &\leq \expe{f(X_n)}-f\st - \gamma(n+1)\pmal\expe{\norm{\nabla f(X_n)}^2}\\
		&\quad+(\Lip\trcst\gamma^2/2)(n+1)^{-2\alpha} + (\Lip\trcst\gamma^2/2)(n+1)^{-2\alpha} \parenthese{\expe{f(X_n)}-f\st} \eqsp .
	\end{align}
	This shows the existence of $n_2 \geq 0$ such that using \eqref{eq:loj2} we have for all $n \geq n_2$,
	\begin{align}
		\expe{f(X\nn)}-f\st &\leq (\expe{f(X_n)}-f\st)\parentheseDeux{1 - (\gamma c\pinv/2)(n+1)\pmal}+(\Lip\trcst\gamma^2/2)(n+1)^{-2\alpha} \eta \eqsp .
	\end{align}
	We note now for all $n \geq 0$, $u_n=\expe{f(X_n)}-f\st$ and $v_n=n\pal u_n$. We have
	\begin{align}
		v_{n+1}-v_n &= (n+1)^{\alpha} u\nn - n^{\alpha}u_n \\
		&= (n+1)^{\alpha} (u\nn-u_n))+u_n((n+1)\pal - n\pal) \\
		&\leq -(\gamma c\pinv/2) u_n +(\Lip\gamma^2\trcst/2)(n+1)\pmal +u_n n\pal \parentheseDeux{(1+1/n)\pal - 1} \\
		&\leq u_n(-(\gamma c\pinv/2) + \alpha n^{\alpha-1})+(\Lip\gamma^2\trcst/2)(n+1)\pmal \eqsp .
	\end{align}
	If $\alpha <1$, or if $1-\gamma c\pinv/2<0$ we have the existence of $n_3 \geq n_2$ and $\Cstrongtilde>0$ such that for all $n \geq n_3$,
	\begin{align}
		v\nn-v_n &\leq -\Cstrongtilde u_n + (\Lip\gamma^2\trcst/2)(n+1)\pmal \\
		&\leq \defEns{-\Cstrongtilde v_n + (\Lip\gamma^2\trcst/2)}(n+1)\pmal
	\end{align}
	This proves the existence of $\Cstronglojdisc > 0$ such that for all $n \geq 0$,
	\begin{equation}
	v_n \leq \Cstronglojdisc \eqsp ,
	\end{equation}
	concluding the proof.
\end{proof}


%% file: appendix_tech.tex

\section{Convex case (under \Cref{assum:grad_sto}-\ref{item:approx_sto})}
\label{sec:convex-case_appendix}

In this section we gather our results about the long-time behavior of SGD and
its continuous-time counterpart in
\Cref{assum:grad_sto}-\ref{item:approx_sto}. In \Cref{sec:technical-results-2},
we derive technical results. In \Cref{app:shamir_cont} we provide the proof of
\Cref{thm:shamir_continuous} (continuous-time setting). In
\Cref{app:shamir_discrete}, we give the proof of \Cref{thm:shamir_discrete}
(discrete-time setting).

\subsection{Technical Results}
\label{sec:technical-results-2}

\begin{lemma_colt}
  \label{lemma:bound_scal}
  Let $f \in \rmc^2(\rset^{\dim}, \rset)$. Assume \rref{assum:f_lip} and
  \tup{\rref{assum:grad_sto}-\ref{item:approx_sto}}. Then for any
  $x \in \rset^{\dim}$ we have
  \begin{equation}
    \abs{\langle \nabla^2 f(x) , \Sigma(x) \rangle} \leq \boundLSig \eqsp , \qquad \absLigne{\langle \nabla f(x) \nabla f(x)^{\top}, \Sigma(x) \rangle} \leq \eta^2 \norm{\nabla f(x) }^2 \eqsp .
  \end{equation}
  Similarly, assume \rref{assum:f_lip} and
  \tup{\rref{assum:grad_sto}-\ref{item:ml}}. Then there exists $C \geq 0$ such tha for any $x \in \rset^{\dim}$
  we have
  \begin{equation}
    \abs{\langle \nabla^2 f(x) , \Sigma(x) \rangle} C (1 + \norm{x}^2) \eqsp , \qquad \absLigne{\langle \nabla f(x) \nabla f(x)^{\top}, \Sigma(x) \rangle} \leq C \norm{\nabla f(x) }^2(1 + \norm{x}^2) \eqsp .
  \end{equation}
\end{lemma_colt}

\begin{proof}
  Let $x\in \rset^{\dim}$. Using Cauchy-Schwarz's inequality, we have
  $\abs{\langle \nabla^2 f(x), \Sigma(x) \rangle} \leq \normLigne{\nabla^2f(x)}
  \normLigne{\Sigma(x)}_{\ast}$, where $\norm{\cdot}$ is the operator norm and
  $\norm{\cdot}_{\ast}$ is the nuclear norm.  Using \rref{assum:f_lip} we have
  $\normLigne{\nabla^2f(x)}\leq \Lip$ for all $x\in \bRd$. In addition, denoting
  $(\lambda_i)_{i \in \{1, \dots, d\}}$ the eigenvalues of $\Sigma(x)$, using
  that $\Sigma$ is positive semi-definite and \Cref{assum:grad_sto} we have
  \begin{equation}
    \textstyle{
      \norm{\Sigma(x)}_{\ast} = \sum_{i=1}^{d} \abs{\lambda_i} = \sum_{i=1}^{d} \lambda_i = \trace(\Sigma(x)) \leq \eta \eqsp .
      }
  \end{equation}
This concludes the first part of the proof. For the second part we have
\begin{equation}
  \abs{\langle \nabla f(x) \nabla f(x)^{\top}, \Sigma(x) \rangle} \leq \sup_{i \in \{1, \dots, d\}} \lambda_i \norm{\nabla f(x)}^2 \leq \eta^2 \norm{\nabla f(x)}^2 \eqsp ,  
\end{equation}
which concludes the first part of the proof. The last part of the proof is an
immediate consequence of \Cref{lemma:bound_f_sig}.
\end{proof}
The following lemma consists into taking the expectation in Itô's formula.
\begin{lemma_colt}
  \label{lemma:dynkin}
  Let $\alpha \in \coint{0, 1}$ and $\gamma >0$. Assume
  $f, g \in \rmc^2(\rset^{\dim}, \rset)$, \rref{assum:f_lip},
  \tup{\rref{assum:grad_sto}} and \rref{assum:lip_sigma}
  and let $(\bfX_t)_{t \geq 0}$ solution of \eqref{eq:sde}.  Then for any
  $\varphi \in \rmc^1(\coint{0, +\infty}, \rset)$, $Y \in \mcf_0$ and
  $\expe{\normLigne{Y}^2+ \abs{g(Y)}} < +\infty$, we have the following results:
  \begin{enumerate}[wide, labelwidth=!, labelindent=0pt, label=(\alph*)]
  \item For any $t \geq 0$, \begin{align} &\expe{\norm{\bfX_t - Y}^2\varphi(t)}
      = \expe{\norm{\bfX_0 - Y}^2\varphi(0)} \\ & \qquad- 2 \int_0^t (\gua +
      s)^{-\alpha} \varphi(s) \expe{\langle \nabla f(\bfX_s), \bfX_s - Y
        \rangle} \rmd s \\ & \qquad + \gua \int_0^t (\gua + s)^{-2\alpha}
      \varphi(s) \expe{ \trace(\Sigma(\bfX_s))} \rmd s + \int_0^t \varphi'(s)
      \expeLigne{\normLigne{\bfX_s - Y}^2} \rmd s \eqsp .
                                                                                                                                                                                                                          \label{eq:norme_carre}
  \end{align}
  \item  For any $t \geq 0$   \begin{align}
    &\expe{(f(\bfX_t) - g(Y))\varphi(t)} = \expe{(f(\bfX_0) - g(Y))\varphi(0)} \\ & \quad - \int_0^t (\gua + s)^{-\alpha} \varphi(s) \expeLigne{\normLigne{\nabla f(\bfX_s)}^2} \rmd s  \\ & \quad +  (\gua/2) \int_0^t (\gua + s)^{-2\alpha} \varphi(s) \expe{ \langle \nabla^2 f(\bfX_s) , \Sigma(\bfX_s) \rangle} \rmd s  \\ & \quad + \int_0^t  \varphi'(s) \expe{f(\bfX_s) - g(Y)} \rmd s  \eqsp .
      \end{align}
  \item If $\expeLigne{\normLigne{Y}^{2p}}< +\infty$, then for any $t \geq 0$
      \begin{align}
        &\expe{\norm{\bfX_t - Y}^{2p}\varphi(t)} = \expe{\norm{\bfX_0 - Y}^{2p}\varphi(0)} \\ & \ - 2p \int_0^t (\gua + s)^{-\alpha} \varphi(s) \expe{\langle \nabla f(\bfX_s), \bfX_s - Y \rangle \norm{\bfX_t - Y}^{2(p-1)}} \rmd s   \\ & \ +  \gua p  \int_0^t (\gua + s)^{-2\alpha} \varphi(s) \expe{ \trace(\Sigma(\bfX_s)) \norm{\bfX_s-Y}^{2(p-1)}} \rmd s \\ & \  +
\gua 2p(p-1)  \int_0^t (\gua + s)^{-2\alpha} \varphi(s) \expe{ \langle \Sigma(\bfX_s) , (\bfX_t-Y)(\bfX_t - Y)^{\top} \rangle  \norm{\bfX_s-Y}^{2(p-2)}} \rmd s  \\ & \ + \int_0^t \varphi'(s) \expe{(f(\bfX_s) - g(Y))^{2p}} \rmd s  \eqsp .
      \end{align}
  \item If $\expeLigne{\abs{g(Y)}^{p}}< +\infty$, then for any $t \geq 0$
      \begin{align}
      &\expe{(f(\bfX_t) - g(Y))^{p}\varphi(t)} = \expe{(f(\bfX_0) - g(Y))^{p}\varphi(0)} \\ & \ - p \int_0^t (\gua + s)^{-\alpha} \varphi(s) \expe{\norm{\nabla f(\bfX_s)}^2(f(\bfX_s)-g(Y))^{p-1}} \rmd s   \\ & \ +  \gua (p/2)  \int_0^t (\gua + s)^{-2\alpha} \varphi(s) \expe{ \langle \nabla^2 f(\bfX_s) , \Sigma(\bfX_s) \rangle (f(\bfX_s)-g(Y))^{p-2}} \rmd s \\ & \  + \gua p(p-1)/2 \int_0^t (\gua + s)^{-2\alpha}\varphi(s) \expe{ \langle \nabla f(\bfX_s) \nabla f(\bfX_s)^{\top}, \Sigma(\bfX_s) \rangle (f(\bfX_s)-g(Y))^{p-2}} \\ & \ + \int_0^t \varphi'(s) \expe{(f(\bfX_s) - g(Y))^{p}} \rmd s  \eqsp .
      \end{align}
  \end{enumerate}
    \end{lemma_colt}

    \begin{proof}
      Let $\alpha \in \coint{0,1}$, $\gamma >0$ and $(\bfX_t)_{t \geq 0}$ the solution of \eqref{eq:sde}.
      Note that for any $t \geq 0$, we have
      \begin{equation}
        \label{eq:mean_quad}
 \langle \bfX \rangle_t = \gua \int_0^t (\gua + s)^{-2\alpha} \trace(\Sigma(\bfX_s))\rmd s \eqsp .
      \end{equation}
      We divide the rest of the proof into our parts.
      \begin{enumerate}[wide, labelwidth=!, labelindent=0pt, label=(\alph*)]
      \item First, let $y \in \rset^{\dim}$ and
        $F_y: \coint{0, +\infty} \times \rset^{\dim}$ such that for any
        $t \in \coint{0, +\infty}$, $x \in \rset^{\dim}$,
        $F_y(t, x) = \varphi(t) \normLigne{x -y}^2$. Since
        $(\bfX_t)_{t \geq 0}$ is a strong solution of \eqref{eq:sde} we have
        that $(\bfX_t)_{t \geq 0}$ is a continuous semi-martingale. Using this
        result, the fact that $F \in \rmc^{1,2}(\coint{0, +\infty}, \rset^{d})$
        and Itô's lemma \citep[Chapter 3, Theorem 3.6]{karatzas1991brownian} we
        obtain that for any $t \geq 0$ almost surely
        \begin{align}
          F_y(t, \bfX_t) &= F_y(0, \bfX_0) + \int_0^t \partial_1 F_y(s, \bfX_s) \rmd s + \int_0^t \langle \partial_2 F_y(s, \bfX_s), \rmd \bfX_s \rangle \\
          &\quad+ (1/2) \int_0^t \langle \partial_{2, 2} F_y(s, \bfX_s), \rmd \langle \bfX \rangle_s \rangle  \\
                       & = F_y(0, \bfX_0) + \int_0^t \varphi'(s) \norm{\bfX_s - y}^2 \rmd s + \int_0^t \langle \partial_2 F_y(s, \bfX_s),  \rmd \bfX_s \rangle   \\
                       &\quad+ (1/2) \int_0^t \langle \partial_{2, 2} F_y(s, \bfX_s),  \rmd \langle \bfX \rangle_s \rangle \\
                         & = F_y(0, \bfX_0) + \int_0^t \varphi'(s) \norm{\bfX_s - y}^2 \rmd s - 2 \int_0^t  (\gua + s)^{-\alpha}\varphi(s)  \langle \nabla f(\bfX_s), \bfX_s - y \rangle \rmd s \\ & \qquad + 2 \gua^{1/2}\int_0^t(\gua +s)^{-\alpha} \varphi(s) \langle \bfX_s -y , \Sigma(\bfX_s)^{1/2} \rmd \bfB_s \rangle  \\
          & \qquad + \gua  \int_0^t (\gua + s)^{-2\alpha}\varphi(s)\trace(\Sigma(\bfX_s)) \rmd s \eqsp .          \label{eq:ito_carre}
        \end{align}
        Using \Cref{assum:f_lip} have for any $x \in \rset^{\dim}$,
        \begin{equation}
          \abs{\langle \nabla f(x), x -y \rangle} \leq \norm{\nabla f(0)}\norm{x-y} + \Lip\norm{x}\norm{x-y} \eqsp .
        \end{equation}
Therefore, using this result \Cref{lemma:bound_moments}, Cauchy-Schwarz's inequality and that $\expeLigne{\normLigne{Y}^2} < +\infty$, we obtain that for any $t \geq 0$ there exists $\btta \geq 0$ such that
\begin{equation}
  \label{eq:bound_sup}
  \sup_{s \in \ccint{0, t}} \expeLigne{\normLigne{\bfX_s - Y}^2} \leq \btta \eqsp , \qquad \sup_{s \in \ccint{0, t}} \expe{\abs{\langle \nabla f(\bfX_s), \bfX_s -Y \rangle}} \leq \btta \eqsp .
\end{equation}
In addition, we have using \Cref{lemma:bound_f_sig} that for any $t \geq 0$,
$\expeLigne{|\trace(\Sigma(\bfX_s))|} = \expeLigne{\trace(\Sigma(\bfX_s))} \leq
C (1 + \btta)$ if \Cref{assum:grad_sto}-\ref{item:ml} holds or
$\expeLigne{|\trace(\Sigma(\bfX_s))|} = \expeLigne{\trace(\Sigma(\bfX_s))} \leq
\eta$ if \Cref{assum:grad_sto}-\ref{item:approx_sto} holds. Combining these
results, \eqref{eq:bound_sup}, \eqref{eq:ito_carre}, that
$(\int_0^t(\gua + t)^{-\alpha}\varphi(t)\langle \bfX_t- Y, \Sigma(\bfX_t)^{1/2}
\rmd \bfB_t \rangle)_{t \geq 0}$ is a martingale and Fubini-Lebesgue's theorem
we obtain for any $t \geq 0$
\begin{align}
  \expe{\varphi(t) \norm{\bfX_t -Y}^2} &= \expe{\CPE{F_Y(t, \bfX_t)}{\mcf_0}} \\
  &=\expe{\varphi(0) \norm{\bfX_0 -Y}^2} + \int_0^t \varphi'(s) \expe{\norm{\bfX_s - Y}^2} \rmd s \\ &\quad - 2 \int_0^t  (\gua + s)^{-\alpha}\varphi(s)  \expe{\langle \nabla f(\bfX_s), \bfX_s - Y \rangle} \rmd s  \\
  &\quad+ \gua  \int_0^t (\gua + s)^{-2\alpha}\varphi(s)\expe{\trace(\Sigma(\bfX_s))} \rmd s \eqsp ,
\end{align}
which concludes the proof of \eqref{eq:norme_carre}.
      \item Second,  let $y \in \rset^{\dim}$ and $F: \coint{0, +\infty} \times \rset^{\dim} $ such that for any $t \in \coint{0, +\infty}$, $x \in \rset^{\dim}$, $F_y(t, x) = \varphi(t) (f(x) - g(y))$.  Using that $(\bfX_t)_{t \geq 0}$ is a continuous semi-martingale, the fact that $F \in \rmc^{1,2}(\coint{0, +\infty}, \rset^{d})$ and Itô's lemma \citep[Chapter 3, Theorem 3.6]{karatzas1991brownian} we obtain that for any $t \geq 0$ almost surely
        \begin{align}
          \label{eq:ito_f}
          F_y(t, \bfX_t) &= F_y(0, \bfX_0) + \int_0^t \partial_1 F_y(s, \bfX_s) \rmd s + \int_0^t \langle \partial_2 F_y(s, \bfX_s), \rmd \bfX_s \rangle \\ & \quad + (1/2) \int_0^t \langle \partial_{2, 2} F_y(s, \bfX_s), \rmd \langle \bfX \rangle_s \rangle  \\
                       & = F_y(0, \bfX_0) + \int_0^t \varphi'(s) (f(\bfX_s) - g(y)) \rmd s + \int_0^t \langle \partial_2 F_y(s, \bfX_s),  \rmd \bfX_s \rangle   \\ 
                       &\quad + (1/2) \int_0^t \langle \partial_{2, 2} F_y(s, \bfX_s),  \rmd \langle \bfX \rangle_s \rangle \\
          & = F_y(0, \bfX_0) + \int_0^t \varphi'(s) (f(\bfX_s) - g(y)) \rmd s - \int_0^t  (\gua + s)^{-\alpha}\varphi(s)   \norm{\nabla f (\bfX_s)}^2 \rmd s \\ 
          & \qquad + \gua^{1/2}\int_0^t (\gua + s)^{-\alpha} \varphi(s) \langle \nabla f(\bfX_s) , \Sigma(\bfX_s)^{1/2} \rmd \bfB_s \rangle  \\
          &\quad + (\gua/2)  \int_0^t (\gua + s)^{-2\alpha}\varphi(s)\langle \nabla^2 f(\bfX_s) , \Sigma(\bfX_s) \rangle \rmd s \eqsp .
        \end{align}
        Using \Cref{assum:f_lip} and that for any $a, b \geq 0$,
        $(a+b)^2 \leq 2(a^2+b^2)$ we have for any $x,y \in \rset^{\dim}$,
        \begin{equation}
          |f(x) -g(y)| \leq |f(0)| + \|  \nabla f(0)\| \| x \| + (\Lip/2) \| x \|^2 + |g(y)| \eqsp , \quad \norm{\nabla f(x)}^2 \leq 2 \norm{\nabla f(0)}^2 + 2 \Lip^2 \norm{x}^2 \eqsp .
        \end{equation}
Therefore, using this result \Cref{lemma:bound_moments}, Cauchy-Schwarz's inequality and that $\expeLigne{g(Y)^2} < +\infty$, we obtain that for any $t \geq 0$ there exists $\btta \geq 0$ such that
\begin{equation}
  \sup_{s \in \ccint{0, t}} \expe{\abs{f(\bfX_s) - g(Y)}} \leq \btta \eqsp , \quad \sup_{s \in \ccint{0, t}} \expeLigne{\normLigne{\nabla f(\bfX_s)}^2} \leq \btta \eqsp , \quad \sup_{s \in \ccint{0,t}} \expeLigne{\abs{\langle \nabla^2 f(\bfX_s), \Sigma(\bfX_s) \rangle}} \leq \btta \eqsp .
\end{equation}
Combining this result, \Cref{lemma:bound_scal}, the fact that $(\int_0^t\varphi(s) \langle \nabla f(\bfX_s) , \Sigma(\bfX_s)^{1/2} \rmd \bfB_s \rangle)_{t \geq 0}$ is a martingale and Fubini-Lebesgue's theorem we obtain that for any $t \geq 0$
\begin{align}
  \expe{F_y(t, \bfX_t) } &= \expe{\CPE{F_Y(t, \bfX_t)}{\mcf_0}} \\
                         &= \expe{\varphi(0)(f(\bfX_0) - g(Y))} + \int_0^t \varphi'(s) \expe{(f(\bfX_s) - g(Y))} \rmd s \\
                         &\quad- \int_0^t  (\gua + s)^{-\alpha}\varphi(s)   \expe{\norm{\nabla f (\bfX_s)}^2} \rmd s \\
                          & \qquad   + (\gua/2)  \int_0^t (\gua + s)^{-2\alpha}\varphi(s) \expe{ \langle \nabla^2 f(\bfX_s) , \Sigma(\bfX_s) \rangle } \rmd s \eqsp .
\end{align}
\item Let $y \in \rset^{\dim}$ and $F_y: \coint{0, +\infty} \times \rset^{\dim} $ such that for any $t \in \coint{0, +\infty}$, $x, y\in \rset^{\dim}$, $F_y(t, x) = \varphi(t) \norm{x - y}^{2p}$.  Using that $(\bfX_t)_{t \geq 0}$ is a continuous semi-martingale, that $F_y \in \rmc^{1,2}(\coint{0, +\infty}, \rset^{d})$ and Itô's lemma \citep[Chapter 3, Theorem 3.6]{karatzas1991brownian} we obtain that for any $t \geq 0$ almost surely
        \begin{align}
          \label{eq:ito_fp}
          F_y(t, \bfX_t) &=F_y(0, \bfX_0) + \int_0^t \partial_1 F_y(s, \bfX_s) \rmd s + \int_0^t \langle \partial_2 F_y(s, \bfX_s), \rmd \bfX_s \rangle \\ & \quad + (1/2) \int_0^t \langle \partial_{2, 2} F_y(s, \bfX_s), \rmd \langle \bfX \rangle_s \rangle  \\
                       & = F_y(0, \bfX_0) + \int_0^t \varphi'(s) \norm{\bfX_s - y}^{2p} \rmd s + \int_0^t \langle \partial_2 F_y(s, \bfX_s),  \rmd \bfX_s \rangle   \\
                       &\quad +(1/2) \int_0^t \langle \partial_{2, 2} F_y(s, \bfX_s),  \rmd \langle \bfX \rangle_s \rangle \\
                       & = F_y(0, \bfX_0) + \int_0^t \varphi'(s) \norm{\bfX_s - y}^{2p} \rmd s \\
                       &\quad - 2p \int_0^t  (\gua + s)^{-\alpha}\varphi(s)   \langle \nabla f(\bfX_s), \bfX_s - y \rangle \norm{\bfX_s) - y}^{2(p-1)} \rmd s \\ 
                       & \quad + 2p \gua^{1/2}\int_0^t (\gua + s)^{-\alpha} \varphi(s) \langle \bfX_s - y , \Sigma(\bfX_s)^{1/2} \norm{\bfX_s - y}^{2(p-1)} \rmd \bfB_s \rangle   \\ & \quad + p \gua  \int_0^t (\gua + s)^{-2\alpha}\varphi(s)\trace(\Sigma(\bfX_s)) \norm{\bfX_s - y}^{2(p-1)} \rmd s \\
          & \quad + 2p(p-1)  \int_0^t (\gua + s)^{-2\alpha}\varphi(s)\langle (\bfX_s - y) (\bfX_s -y)^{\top}, \Sigma(\bfX_s)  \rangle \norm{\bfX_s - y}^{2(p-2)} \rmd s \eqsp .
        \end{align}
        Using \Cref{assum:f_lip} and that for any $a, b \geq 0$,
        $(a+b)^2 \leq 2(a^2+b^2)$ we have for any $x,y \in \rset^{\dim}$,
        Therefore, using this result \Cref{lemma:bound_moments}, Cauchy-Schwarz's inequality and that $\expeLigne{\normLigne{Y}^2} < +\infty$, we obtain that for any $t \geq 0$ there exists $\btta \geq 0$ such that
\begin{equation}
  \label{eq:bound_sup2}
  \sup_{s \in \ccint{0, t}} \expe{\norm{\bfX_s - Y}^{2p}} \leq \btta \eqsp , \qquad \sup_{s \in \ccint{0, t}} \expe{\abs{\langle \nabla f(\bfX_s), \bfX_s -Y \rangle\norm{\bfX_s - Y}^{2(p-1)}}} \leq \btta \eqsp ,
\end{equation}
and
\begin{equation}
  \label{eq:bound_sup3}
  \sup_{s \in \ccint{0, t}} \expe{\normLigne{\Sigma^{1/2}(\bfX_s)}\norm{\bfX_s - Y}^{2p-1}} \leq \btta \eqsp , \quad \sup_{s \in \ccint{0, t}} \expe{\norm{\Sigma(\bfX_s)}\norm{\bfX_s - y}^{2(p-1)}} \leq \btta \eqsp .
\end{equation}
Combining these results, \Cref{lemma:bound_scal}, that $(\int_0^t\varphi(s) \langle \nabla f(\bfX_s) , \Sigma(\bfX_s)^{1/2} (f(\bfX_s) - g(Y))^{p-1}\rmd \bfB_s \rangle)_{t \geq 0}$ is a martingale and Fubini-Lebesgue's theorem we obtain that for any $t \geq 0$
\begin{align}
  &\expe{F_y(t, \bfX_t) } = \expe{\CPE{F_Y(t, \bfX_t)}{\mcf_0}} \\
                         &\quad = \expe{\varphi(0)\norm{\bfX_0 - Y}^{2p}} + \int_0^t \varphi'(s) \expe{\norm{\bfX_s - Y}^{2p}} \rmd s \\ & \quad - 2p \int_0^t  (\gua + s)^{-\alpha}\varphi(s)   \expe{\langle \nabla f(\bfX_s), \bfX_s - y \rangle \norm{\bfX_s) - y}^{2(p-1)}} \rmd s \\ & \quad   + \gua p  \int_0^t (\gua + s)^{-2\alpha}\varphi(s) \expe{\trace(\Sigma(\bfX_s)) \norm{\bfX_s - y}^{2(p-1)} } \rmd s \\
                         & \quad + 2\gua p(p-1)  \int_0^t (\gua + s)^{-2\alpha}\varphi(s) \expe{ \langle (\bfX_s - y) \nabla (\bfX_s -y)^{\top}, \Sigma(\bfX_s)  \rangle \norm{\bfX_s - y}^{2(p-2)}} \rmd s \eqsp .
\end{align}
      \item Let $y \in \rset^{\dim}$ and $F: \coint{0, +\infty} \times \rset^{\dim} $ such that for any $t \in \coint{0, +\infty}$, $x, y\in \rset^{\dim}$, $F_y(t, x) = \varphi(t) (f(x) - g(y))^{2p}$.  Using that $(\bfX_t)_{t \geq 0}$ is a continuous semi-martingale, the fact that $F \in \rmc^{1,2}(\coint{0, +\infty}, \rset^{d})$ and Itô's lemma \citep[Chapter 3, Theorem 3.6]{karatzas1991brownian} we obtain that for any $t \geq 0$ almost surely
        \begin{align}
          \label{eq:ito_fp}
          &F_y(t, \bfX_t) =F_y(0, \bfX_0) + \int_0^t \partial_1 F_y(s, \bfX_s) \rmd s + \int_0^t \langle \partial_2 F_y(s, \bfX_s), \rmd \bfX_s \rangle \\ & \ + (1/2) \int_0^t \langle \partial_{2, 2} F_y(s, \bfX_s), \rmd \langle \bfX \rangle_s \rangle  \\
                       & = F_y(0, \bfX_0) + \int_0^t \varphi'(s) (f(\bfX_s) - g(y))^{2p} \rmd s \\
                       & \ + \int_0^t \langle \partial_2 F_y(s, \bfX_s),  \rmd \bfX_s \rangle   + (1/2) \int_0^t \langle \partial_{2, 2} F_y(s, \bfX_s),  \rmd \langle \bfX \rangle_s \rangle \\
                       & = F_y(0, \bfX_0) + \int_0^t \varphi'(s) (f(\bfX_s) - g(y))^{2p} \rmd s \\
                       & \ - 2p \int_0^t  (\gua + s)^{-\alpha}\varphi(s)   \norm{\nabla f (\bfX_s)}^2(f(\bfX_s) - g(y))^{2(p-1)} \rmd s \\ & \ + 2p \gua^{1/2}\int_0^t (\gua + s)^{-\alpha} \varphi(s) \langle \nabla f(\bfX_s) , \Sigma(\bfX_s)^{1/2} (f(\bfX_s) - g(y))^{2(p-1)} \rmd \bfB_s \rangle   \\ & \ + p \gua  \int_0^t (\gua + s)^{-2\alpha}\varphi(s)\langle \nabla^2 f(\bfX_s) , \Sigma(\bfX_s)  \rangle (f(\bfX_s) - g(y))^{2(p-1)} \rmd s \\
          & \ + 2p(p-1)  \int_0^t (\gua + s)^{-2\alpha}\varphi(s)\langle \nabla f(\bfX_s) \nabla f(\bfX_s)^{\top}, \Sigma(\bfX_s)  \rangle (f(\bfX_s) - g(y))^{2(p-2)} \rmd s
        \end{align}
        Using \Cref{assum:f_lip} and that for any $a, b \geq 0$,
        $(a+b)^2 \leq 2(a^2+b^2)$ we have for any $x,y \in \rset^{\dim}$,
        \begin{equation}
          \begin{aligned}
            & |f(x) -g(y)|^{2p} \leq 4^{2p-1}|f(0)|^{2p} + 4^{2p-1} \| \nabla f(0)\|^{2p} \| x \|^{2p} + (4^{2p-1}\Lip/2) \| x \|^{4p} + 4^{2p-1}|g(y)|^{2p} \eqsp , \\
            & \norm{\nabla f(x)}^2 \leq 2 \norm{\nabla f(0)}^2 + 2 \Lip^2 \norm{x}^2 \eqsp .
          \end{aligned}
        \end{equation}
Therefore, using this result \Cref{lemma:bound_moments}, \Cref{lemma:bound_scal}, Hölder's inequality and that $\expeLigne{g(Y)^2} < +\infty$, we obtain that for any $t \geq 0$ there exists $\btta \geq 0$ such that
\begin{equation}
  \begin{aligned}
    &\sup_{s \in \ccint{0, t}} \expe{\abs{f(\bfX_s) - g(Y)}^{2p}} \leq \btta \eqsp , \qquad \sup_{s \in \ccint{0, t}} \expe{\norm{\nabla f(\bfX_s)}^2\abs{f(\bfX_s) - g(Y)}^{2(p-1)}} \leq \btta \eqsp , \\
    &\sup_{s \in \ccint{0, t}} \expe{\abs{\langle \nabla f(\bfX_s) \nabla f(\bfX_s)^{\top}, \Sigma(\bfX_s)  \rangle (f(\bfX_s) - g(Y))^{2(p-2)}}} \leq \btta \eqsp .
    \end{aligned}
\end{equation}
Combining this result, \Cref{lemma:bound_scal}, that $(\int_0^t\varphi(s) \langle \nabla f(\bfX_s) , \Sigma(\bfX_s)^{1/2} (f(\bfX_s) - g(Y))^{p-1}\rmd \bfB_s \rangle)_{t \geq 0}$ is a martingale and Fubini-Lebesgue's theorem we obtain that for any $t \geq 0$
\begin{align}
  &\expe{F_y(t, \bfX_t) } = \expe{\CPE{F_Y(t, \bfX_t)}{\mcf_0}} \\
  &= \expe{\varphi(0)(f(\bfX_0) - g(Y))^{2p}} + \int_0^t \varphi'(s) \expe{(f(\bfX_s) - g(Y))^{2p}} \rmd s \\ & \quad - 2p \int_0^t  (\gua + s)^{-\alpha}\varphi(s)   \expe{\norm{\nabla f (\bfX_s)}^2(f(\bfX_s) - g(y))^{2(p-1)}} \rmd s \\ & \quad   + \gua p  \int_0^t (\gua + s)^{-2\alpha}\varphi(s) \expe{ \langle \nabla^2 f(\bfX_s) , \Sigma(\bfX_s) \rangle (f(\bfX_s) - g(Y))^{2(p-1)}} \rmd s \\
  & \quad + 2\gua p(p-1)  \int_0^t (\gua + s)^{-2\alpha}\varphi(s) \expe{ \langle \nabla f(\bfX_s) \nabla f(\bfX_s)^{\top}, \Sigma(\bfX_s) \rangle (f(\bfX_s) - g(Y))^{2(p-2)}} \rmd s \eqsp .
\end{align}
      \end{enumerate}
    \end{proof}
    The following lemma is a useful tool that converts results on $\rmc^2$
    functions to $\rmc^1$ functions.
\begin{lemma_colt}
	\label{lemma:eps}
	Assume \rref{assum:f_lip}, \tup{\rref{assum:f}-\ref{item:approx_sto}},
        \rref{assum:lip_sigma} and that $\argmin_{x \in \rset^{\dim}} f$ is
        bounded. Then there exists $(\feps)_{\veps>0}$ such that for any
        $\vareps >0$, $\feps$ is convex, $\rmc^2$ with $\Lip$-Lipschitz
        continuous gradient. In addition, there exists $\Db \geq 0$ such that
        the following properties are satisfied.
	\begin{enumerate}[wide, labelwidth=!, labelindent=0pt, label=(\alph*)]
        \item For all $\veps>0$, $\feps$ admits a minimizer $x\st_{\veps}$ and
          $\limsup_{\veps\to0}\feps(x_{\veps}\st)\leq f(x\st)$.
        \item $\liminf_{\vareps \to 0} \| x_{\vareps}^{\star} \| \leq \Db$.
        \item for any $T \geq 0$,
          $\lim_{\vareps \to 0}  \expe{|\feps(\bfX_{T, \vareps}) - f(\bfX_T)|}  = 0$ , where
          $(\bfX_{t, \vareps})_{t \geq 0}$ is the solution of \eqref{eq:sde}
          replacing $f$ by $f_{\vareps}$.
	\end{enumerate}
\end{lemma_colt}

\begin{proof}
Let $\varphi \in \rmc^{\infty}_c(\bR^d,\bR_+)$ be an even compactly-supported function such that $\int_{\rset^{\dim}} \varphi(z) \rmd z = 1$. For any $\varepsilon>0$ and $x\in \bRd$, let $\phe(x)=\varepsilon^{-\dim}\varphi(x/\varepsilon)$ and $\feps=\phe\ast f$. Since $\varphi \in \rmc^{\infty}_c(\bR^d,\bR_+)$ and is compactly-supported, we have $\feps \in \rmc^{\infty}(\rset^{\dim}, \rset)$. In addition, we have for any $\vareps >0$, $(\nabla f)_{\vareps} = \nabla \feps$.

First, we show that for any $\vareps$, $f_{\vareps}$ is convex and
$\nabla \feps$ is $\Lip$-Lipschitz continuous.  Let $\vareps >0$, $x,y \in \bRd$
and $t \in \ccint{0,1}$. Using \tup{\Cref{assum:f}-\ref{item:approx_sto}} we
have
\begin{align}
	\feps(tx+(1-t)y)&=\intd f(tx+(1-t)y-z)\phe(z) \dd z
	\\ &\leq \intd \defEns{tf(x-z)+(1-t)f(y-z)}\phe(z) \dd z \\
	& \leq t\feps(x)+(1-t)\feps(y) \eqsp .
\end{align}
Hence, $\feps$ is convex. In addition, using \Cref{assum:f_lip} and that $\int_{\rset^{\dim}} \phe(z) \rmd z = 1$ we have
\begin{align}
	\norm{\nabla \feps(x)-\nabla \feps(y)}
	\leq \intd \norm{\nabla f(x-z)-\nabla f(y-z)}\phe(z) \dd z
	\leq \Lip \norm{x-y} \eqsp,
\end{align}
which proves that $\nfeps$ is $\Lip$-Lipschitz continuous.

Second we show that $\feps$ and $\nfeps$ converge uniformly towards $f$ and $\nabla f$. Let $\veps >0$, $x \in \bRd$. Using the convexity of $f$ and that $\phe$ is even, we get
\begin{align}
	\feps(x)-f(x)&=\intd (f(x-z)-f(x))\phe(z) \dd z \\
	&\geq -\intd \la \nabla f(x),z\ra \phe(z) \dd z \\
	&\geq -\la \nabla f(x), \intd z\phe(z) \dd z \ra
	\geq 0 \eqsp , 	\label{eq:tech_geq}
\end{align}
Conversely, using the descent lemma \citep[Lemma 1.2.3]{nesterov2004introductory} and that $\phe$ is even, we have
\begin{align}
	\feps(x)-f(x)&=\intd (f(x-z)-f(x))\phe(z) \dd z \\
	&\leq \intd \parenthese{- \la \nabla f(x), z \ra + (\Lip/2)\norm{z}^2}\phe(z)  \dd z \\
	&\leq (\Lip/2)\intd \veps^2\norm{z/\veps}^2\veps^{-d}\varphi(z/\veps)\dd z \leq (\Lip/2)\veps^2\intd \norm{u}^2\varphi(u)\dd u \eqsp . \label{eq:tech_leq}
\end{align}
Combining \eqref{eq:tech_geq} and \eqref{eq:tech_leq} we get that $\lim_{\vareps \to 0} \| f - \feps \|_{\infty} =0$.
Using \Cref{assum:f_lip} we have for any $x \in \rset^{\dim}$
\begin{multline}
  \norm{\nfeps(x) - \nabla f(x)} \leq \norm{(\nabla f)_{\vareps}(x) - \nabla f(x)} \\ \leq \int_{\rset^{\dim}} \| \nabla f(x -z ) - \nabla f(x) \| \phe(z) \rmd z \leq \Lip \vareps \int_{\rset^{\dim}} \|z \| \varphi(z) \rmd z \eqsp ,
\end{multline}
Hence, we obtain that $\lim_{\vareps \to 0} \| \nabla \feps - \nabla f \|_{\infty} = 0$.
Finally, since $f$ is coercive \citep[Proposition B.9]{bertsekas} and $(f_{\vareps})_{\vareps >0}$ converges uniformly towards $f$ we have that for any $\vareps >0$, $\feps$ is coercive.

We divide the rest of the proof into three parts.
\begin{enumerate}[wide, labelwidth=!, labelindent=0pt, label=(\alph*)]
\item Let $\vareps >0$. Since $\feps$ is coercive and continuous it admits a minimizer $x_{\vareps}^{\star}$.
  In addition, we have
  \begin{equation}
    \label{eq:tech_ineq}
    \feps(x_{\vareps}^{\star}) \leq \feps(x^{\star}) \leq f(x^{\star}) + \| \feps - f \|_{\infty} \eqsp .
  \end{equation}
  Therefore, $\limsup_{\vareps \to 0} \feps(x_{\vareps}^{\star}) \leq f(x^{\star})$.
\item Let $\vareps \in \ocint{0,1}$. Using \eqref{eq:tech_ineq}, we obtain that $| \feps(x^{\star}) | \leq |f(x^{\star})| + \sup_{\vareps \in \ocint{0,1}} \| \feps - f \|_{\infty}$. Since $f$ is coercive, we obtain that $(x_{\vareps}^{\star})_{\vareps \in \ocint{0,1}}$ is bounded and therefore there exists $\Db \geq 0$ such that $\liminf_{\vareps \to 0} \| x_{\vareps}^{\star} \| \leq \Db$.
\item Let $\vareps >0$, $T \geq 0$ and $(\bfX_{t, \vareps})_{t \geq 0}$ be the solution of
  \eqref{eq:sde} replacing $f$ by $f_{\vareps}$. Using \eqref{eq:sde}, the fact that $\lim_{\vareps \to 0} \| \nabla f -  \nabla f_{\vareps} \|_{\infty} = 0$, \Cref{assum:f_lip} and Grönwall's inequality \citep[Theorem 1.2.2]{pachpatte1998ineq} we have
  \begin{align}
    \expe{\norm{\bfX_{T, \vareps} - \bfX_T}^2} &\leq \expe{\norm{ \int_{0}^T  (\gua +s)^{-\alpha} \defEns{-\nabla \feps(\bfX_{t, \vareps}) + \nabla f(\bfX_t)}\rmd t }^2} \\
                                               &\leq 2 \gua^{-2\alpha} T \int_{0}^T \expe{\norm{\nabla f(\bfX_{t, \vareps}) - \nabla f(\bfX_t) }^2} \rmd t + 2\gua^{-2\alpha} T^2 \| \nabla f - \nabla \feps \|_{\infty}^2 \\
                                               &\leq 2 \Lip \gua^{-2\alpha} T \int_{0}^T \expe{\norm{\bfX_{t, \vareps} - \bfX_t }^2} \rmd t + 2\gua^{-2\alpha} T^2 \| \nabla f - \nabla \feps \|_{\infty}^2 \\
    &\leq 2\gua^{-2\alpha} T^2 \| \nabla f - \nabla \feps \|_{\infty}^2 \exp\parentheseDeux{2\Lip\gua^{-2\alpha} T^2} \eqsp .     \label{eq:bound_norm}
  \end{align}
  Therefore $\lim_{\vareps \to 0} \expe{\norm{\bfX_{T, \vareps} - \bfX_T}^2} = 0$. In addition, using the Cauchy-Schwarz inequality, \rref{assum:f_lip} and \Cref{lemma:bound_moments} we have
  \begin{align}
    \expe{| f(\bfX_{T, \vareps}) - f(\bfX_T) |} &\leq \expe{\int_{0}^1 \| \nabla f(\bfX_{T} + t(\bfX_{T, \vareps} - \bfX_T)) \| \| \bfX_{T, \vareps} - \bfX_T \| \rmd t } \\
                                                  &\leq \expe{ (\|\bfX_{T, \vareps} \| + \| \bfX_T \| + \| x^{\star} \|) \| \bfX_{T, \vareps} - \bfX_T \| } \\
                                                  &\leq 3^{1/2} \parenthese{\|x^{\star}\|^2 + \expe{\norm{\bfX_T}^2} + \expe{\norm{\bfX_{T, \vareps}}^2}}^{1/2} \expe{\norm{\bfX_{T, \vareps} - \bfX_T}^2}^{1/2} \\
    &\leq 3^{1/2} (\norm{x^{\star}} + 2 \ctun)^{1/2} (1 + \norm{x_0}^2)^{1/2} \expe{\norm{\bfX_{T, \vareps} - \bfX_T}^2}^{1/2} \eqsp .     \label{eq:bound_fun}
  \end{align}
Therefore, using \eqref{eq:bound_norm}, \eqref{eq:bound_fun} and the fact that $\lim_{\vareps \to 0} \| f - \feps \|_{\infty} = 0$ we obtain that
\begin{equation}
  \lim_{\vareps \to 0}  \expe{|\feps(\bfX_{T, \vareps}) - f(\bfX_T)|} \leq  \lim_{\vareps \to 0}  \expe{|f(\bfX_{T, \vareps}) - f(\bfX_T)|} + \lim_{\vareps \to 0} \|f - f_{\vareps} \|_{\infty} = 0 \eqsp ,
\end{equation}
which concludes the proof.
\end{enumerate}
\end{proof}

\begin{lemma_colt}
	\label{lemma:alpha}
	Let $x,y \geq 1$. Let $\alpha \in \ocint{0,1/2}$.
	If $y<x$ then $x^{\alpha}-y^{\alpha} \leq x^{1-\alpha} - y^{1-\alpha}$.
\end{lemma_colt}

\begin{proof}
Let $\lambda \in \ooint{0,1}$ such that $y=\lambda x$. Then $x^{\alpha}-y^{\alpha}=x^{\alpha}(1-\lambda^{\alpha}) \leq x^{1-\alpha}(1-\lambda^{1-\alpha})=x^{1-\alpha}-y^{1-\alpha}$ because $x>1$, $\lambda <1$ and $\alpha \leq 1-\alpha$.
\end{proof}


%% file: appendix_convex.tex

\label{app:shamir}
\subsection{Proof of \Cref{thm:shamir_continuous}}
\label{app:shamir_cont}

In this section we prove \Cref{thm:shamir_continuous}.  We begin with
\Cref{lemma:D2} which is a useful result to bound
$\expeLigne{\normLigne{\bfX_t-x\st}^2}$. Then, we introduce the averaging
process in \eqref{eq:S_def}. The study of this process is central in our
proof. First we establish \Cref{lemma:control_derivative} which allows to
control the time-derivative of the process $S$. We show that the difference
$\expeLigne{f(\bfX_T)} -f^\star$ can be rewritten as the sum of three terms
involving $S$. We bound each one of these three terms in \Cref{lemma:s1_sT},
\Cref{lemma:st-fstar} and \Cref{lemma:s0-s1}, concluding the proof of
\Cref{thm:shamir_continuous}. We finish this section with a proof of
\Cref{thm:shamir_continuous_C1} which extends our result to the case where
$f \in \rmc^1(\rset^d, \rset)$.

\begin{lemma_colt}
	\label{lemma:D2}
        Assume \tup{\rref{assum:f}-\ref{item:f_conv}}. Let $(\bfX_t)_{t \geq 0}$
        be given by \eqref{eq:sde}. Then, for any
        $\alpha, \gamma \in \ooint{0,1}$, there exists $\Cconvcontun \geq0$ and
        $\Cconvcontdeux\geq0$ and a function $\Cconvcont:\bR_+ \to \bR_+$ such
        that, for any $t \geq 0$,
\begin{equation}
	\expeLigne{\normLigne{\bfX_t-x\st}^2}\leq \Cconvcontun\Cconvcont(t+\gua)+\Cconvcontdeux  \eqsp .
\end{equation}
And we have
\begin{equation}
	\Cconvcont(t)=
	\begin{cases}
		t^{1-2\alpha}&\mbox{if }\alpha<1/2 \eqsp ,\\
		\log(t)&\mbox{if }\alpha=1/2 \eqsp ,\\
		0&\mbox{if }\alpha>1/2 \eqsp .
	\end{cases}
\end{equation}
The values of the constants are given by
\begin{align}
	\Cconvcontun&=
		\begin{cases}
			\gua\sigb(1-2\alpha)\pinv&\mbox{if }\alpha<1/2 \eqsp, \\
			\gua\sigb&\mbox{if }\alpha=1/2 \eqsp ,\\
			0&\mbox{if }\alpha>1/2 \eqsp .
		\end{cases}\\
	\Cconvcontdeux&=
	\begin{cases}
		\norm{X_0-x\st}^2&\mbox{if }\alpha<1/2 \eqsp, \\
		\norm{X_0-x\st}^2-\gua\sigb\log(\gua)&\mbox{if }\alpha=1/2 \eqsp ,\\
		\norm{X_0-x\st}^2+(2\alpha-1)\pinv\gua^{2-2\alpha}\sigb&\mbox{if }\alpha>1/2 \eqsp ,
	\end{cases}
\end{align}
\end{lemma_colt}
\begin{proof}
	Let $\alpha, \gamma \in \ooint{0,1}$ and $t \geq 0$. Let $(\bfX_t)_{t \geq 0}$ be given by \eqref{eq:sde}. We consider the function $F:\bR \times \bRd \to \bR_+$ defined as follows
	\begin{equation}
		\forall (t,x)\in\bR\times\bRd, \, F(t,x)=\norm{x-x\st}^2 \eqsp .
	\end{equation}
Applying Lemma~\ref{lemma:dynkin} to the stochastic process $(F(t,\bfX_t))_{t\geq 0}$ and using \tup{\rref{assum:f}-\ref{item:f_conv}} and \tup{\rref{assum:grad_sto}-\ref{item:approx_sto}} gives that for all $t\geq 0$,
	\begin{align}
          &\expe{\norm{\bfX_t-x\st}^2}-\expe{\norm{\bfX_0-x\st}^2} \\
          & \qquad =-2\int_0^T (t+\gua)^{-\alpha}\la \bfX_t-x\st, \nabla f(\bfX_t)\ra \dd t + \int_0^T \gua(t+\gua)^{-2\alpha}\trace(\Sigma(\bfX_t)) \dd t \\
          &\qquad \leq \gua\sigb\int_0^T (t+\gua)^{-2\alpha} \dd t \eqsp .
	\end{align}

	We now distinguish three cases:
	\begin{enumerate}[wide, labelwidth=!, labelindent=0pt, label=(\alph*)]
		\item If $\alpha < 1/2$, then we have
			\begin{align}
				\expe{\norm{\bfX_t-x\st}^2} &\leq \norm{X_0-x\st}^2 +\gua\sigb(1-2\alpha)\pinv((T+\gua)^{1-2\alpha}-\gua^{1-2\alpha}) \\
				&\leq \norm{X_0-x\st}^2+\gua\sigb(1-2\alpha)\pinv(T+\gua)^{1-2\alpha} \eqsp .
			\end{align}
		\item If $\alpha = 1/2$, then we have 
			\begin{align}
				\expe{\norm{\bfX_t-x\st}^2} &\leq \norm{X_0-x\st}^2 + \gua\sigb(\log(T+\gua)-\log(\gua)) \\
				&\leq \gua \sigb \log(T+\gua)+ \norm{X_0-x\st}^2-\gua\sigb\log(\gua) \eqsp .
			\end{align}
		\item If $\alpha > 1/2$, then we have 
			\begin{align}
				\expe{\norm{\bfX_t-x\st}^2} &\leq \norm{X_0-x\st}^2 +\gua\sigb(1-2\alpha)\pinv((T+\gua)^{1-2\alpha}-\gua^{1-2\alpha}) \\
				&\leq \norm{X_0-x\st}^2 + (2\alpha-1)\pinv \gua^{2-2\alpha}\sigb \eqsp .
			\end{align}
	\end{enumerate}

\end{proof}
We now turn to the proof of \Cref{thm:shamir_continuous}.  Let
$f\in \rmc^2(\R^{\dim},\R)$.  Let $\gamma \in \ooint{0,1}$ and
$\alpha \in \ocint{0,1/2}$ and $T \geq 1$. Let $(\bfX_t)_{t \geq 0}$ be given by
\eqref{eq:sde}.

Let $S: \ \ccint{0,T} \to \coint{0,+\infty}$ defined by
\begin{equation}
  \label{eq:S_def}
  \left\lbrace
  \begin{aligned}
    &S(t) = \textstyle{t^{-1} \int_{T-t}^T \defEns{\expe{f(\bfX_s)} -f^{\star}} \rmd s \eqsp , \qquad \text{if } t > 0 \eqsp ,} \\
    &S(0) = \expe{f(\bfX_T)} \eqsp , \qquad \qquad  \qquad \qquad \quad \eqsp \eqsp \text{otherwise.}
  \end{aligned}
  \right.
\end{equation}
With this notation we have
\begin{equation}
  \expe{f(\bfX_T)} - f^{\star} = S(0) - S(1) + S(1) - S(T) + S(T) - f^{\star} \eqsp .
\end{equation}
We are now going to control each one of the three terms $(S(0) - S(1))$,
$(S(1) - S(T))$, $(S(T) - f^\star)$ as follows:
\begin{enumerate}[wide, labelwidth=!, labelindent=0pt, label=(\alph*)]
\item Case $S(1) - S(T)$ (\Cref{lemma:s1_sT}): we adapt the idea of suffix
  averaging of \cite{shamirzhang} to the continuous-time setting. In particular,
  we control the time-derivative of $S$ in \Cref{lemma:control_derivative}.
\item Case $S(T) - f^\star$ (\Cref{lemma:st-fstar}): this result is known and
  corresponds to the optimal convergence rate of the averaged sequence towards
  the minimum of $f$. We provide its proof for completeness.
\item Case $S(0) - S(1)$ (\Cref{lemma:s0-s1}): this last term is specific to the
  continuous-time setting and is a necessary modification to the classic
  averaging control of $S(\vareps) - S(T)$, established in \Cref{lemma:s1_sT} for
  $\vareps = 1$, which diverges for $\vareps$ close to $0$.
\end{enumerate}
Before controlling each one of these terms we state the following useful lemma,
which will allow us to control the derivative of $S$.

\begin{lemma_colt}
  \label{lemma:control_derivative}
  Assume \rref{assum:f_lip}, \tup{\rref{assum:grad_sto}-\ref{item:approx_sto}},
  \rref{assum:lip_sigma}, and \tup{\rref{assum:f}-\ref{item:f_conv}}. Then, for
  any $\alpha, \gamma \in \ooint{0,1}$, $T \geq 0$, $u \in \ccint{0,T}$ and $Y$
  any $\rset^d$-valued random variable such that
  $\expeLigne{\normLigne{Y -x^\star}^2} \leq \Cconvcontun \Cconvcont
  (T+\gua)+\Cconvcontdeux$ with $\Cconvcontun$ and $\Cconvcontdeux$ given in
  \Cref{lemma:D2}, we have
  \begin{align}
  \label{eq:shamir_master_cons}
	\begin{split}
		\int_{T-u}^T \expe{f(\bfX_t)-f(\xo)}\dd t&\leq (\Csham/2) \parenthese{(T+\gua)^{\alpha}-(T-u+\gua)^{\alpha}}\\
		&\quad+(1/2) (T-u+\gua)^{\alpha} \expeLigne{\normLigne{\bfX_{T-u}-\xo}^2}\\
		&\quad+(\Csham/2)\parenthese{(T+\gua)^{1-\alpha}-(T-u+\gua)^{1-\alpha}}\log(T+\gua) \eqsp ,
	\end{split}
\end{align}
with
$\Csham = \max(4 \Cconvcontdeux,(\gua \sigb+4\alpha\Cconvcontun)
(1-\alpha)\pinv)$, with $\Cconvcontun$ and $\Cconvcontdeux$ given in
\Cref{lemma:D2}.
\end{lemma_colt}

\begin{proof}
  For any $\xoo \in \rset^{\dim}$ we define the function
  $F_{\xoo}:\bR_+ \times \bR^{\dim} \to \bR$ by
  \begin{equation}
    \label{eq:fun_x0}
		F_{\xoo}(t,x)=(t+\gua)^{\alpha}\norm{x- \xoo}^2 \eqsp .
	\end{equation}
        Using \Cref{lemma:D2} and that for any $a,b \geq 0$,
        $(a+b)^2 \leq 2a^2 + 2b^2$, we have
	\begin{align}
		\expe{\norm{\bfX_t-\xo}^2}&=\expe{\norm{(\bfX_t-x\st)+(x\st-\xo)}^2}\\
		&\leq 2\expe{\norm{\bfX_t-x\st}^2}+2\expe{\norm{\xo-x\st}^2}\\
                                          &\leq 2\Cconvcontun \Cconvcont
                                            (t+\gua)+4\Cconvcontdeux+2\Cconvcontun \Cconvcont(T+\gua) \\
          &\leq 2\Cconvcontun \Cconvcont
                                            (t+\gua) + 2\Cconvcontun \Cconvcont(T+\gua) + \Cconvconttrois \eqsp .
	\end{align}
	with $\Cconvconttrois = 4\Cconvcontdeux$. This gives in particular, for every $t\in\ccint{0,T}$,
	\begin{align}
		\label{eq:shamir_bound_esp}
		(t+\gua)^{\alpha-1}\expe{\norm{\bfX_t-\xo}^2} &\leq \defEns{\Cconvconttrois+2\Cconvcontun(T+\gua)^{1-2\alpha}\log(T+\gua)}(t+\gua)^{\alpha-1}\\
		&\quad+2\Cconvcontun\log(T+\gua)(t+\gua)^{-\alpha} \eqsp ,
	\end{align}
with $\Cconvcontun=0$ if $\alpha>1/2$. Notice that the additional $\log(T+\gua)$ term is only needed in the case where $\alpha=1/2$. For any $(t,x)\in \bR_+ \times \bRd$, we have
	\begin{align}
		&\partial_t F_{\xoo}(t,x)=\alpha(t+\gua)^{\alpha-1} \norm{x-\xoo}^2 \eqsp, \\
		&\partial_x F_{\xoo}(t,x)=2(t+\gua)^{\alpha} (x-\xoo) \eqsp, \quad
		\partial_{xx} F_{\xoo}(t,x)=2(t+\gua)^{\alpha} \eqsp .
	\end{align}
	Using Lemma~\ref{lemma:dynkin} on the stochastic process
        $(F_{\xo}(t,\bfX_t))_{t \geq 0}$, we have that for any $u\in\ccint{0,T}$
\begin{align}
	\label{eq:shamir_master}
	\begin{split}
          &\expe{F_{\xo}(T,\bfX_T)}-\expe{F_{\xo}(T-u,\bfX_{T-u})}\\
          & \qquad =\int_{T-u}^T\alpha(t+\gua)^{\alpha-1}\expe{\norm{\bfX_t-\xo}^2}\dd t -2\int_{T-u}^T\expe{\la \bfX_t-\xo, \nabla f(\bfX_t) \ra} \dd t \\
		&\quad \qquad + \int_{T-u}^T\gua(t+\gua)^{-\alpha}\expe{\trace (\Sigma(\bfX_t))}\dd t \eqsp .
              \end{split}
\end{align}
Combining this result, \tup{\rref{assum:f}-\ref{item:f_conv}},
\tup{\rref{assum:grad_sto}-\ref{item:approx_sto}}, \eqref{eq:fun_x0},
\eqref{eq:shamir_bound_esp} and~\eqref{eq:shamir_master} we obtain for any
$u \in \ccint{0,T}$
\begin{align}
		&-(T-u+\gua)^{\alpha}\expe{\norm{\bfX_{T-u}-\xo}^2} \\
		&\quad \leq \Cconvconttrois \int_{T-u}^T\alpha(t+\gua)^{\alpha-1} \dd t + \sigb\gua \int_{T-u}^T (t+\gua)^{-\alpha} \dd t\\
		&\qquad  + 2\alpha\Cconvcontun\log(T+\gua)\defEns{\int_{T-u}^T(t+\gua)^{-\alpha} \dd t+(T+\gua)^{1-2\alpha} \int_{T-u}^T(t+\gua)^{\alpha-1} \dd t} \\
		&\qquad- 2\int_{T-u}^T \expe{f(\bfX_t)-f(\xo)}\dd t \\
		&\quad \leq \Cconvconttrois\parenthese{(T+\gua)^{\alpha}-(T-u+\gua)^{\alpha}}-2\int_{T-u}^T \expe{f(\bfX_t)-f(\xo)}\dd t\\
		&\qquad+(\gua \sigb+2\alpha\Cconvcontun) (1-\alpha)\pinv \parenthese{(T+\gua)^{1-\alpha}-(T-u+\gua)^{1-\alpha}}\log(T+\gua) \\
		&\qquad +2\Cconvcontun\log(T+\gua)\defEns{(T+\gua)^{\alpha}-(T-u+\gua)^{\alpha}}(T+\gua)^{1-2\alpha} \eqsp .
\end{align}
Therefore, we get for any $u \in \ccint{0, T}$
\begin{align}
  \label{eq:shamir_master_cons}
	\begin{split}
		\int_{T-u}^T \expe{f(\bfX_t)-f(\xo)}\dd t&\leq (\Csham/2) \parenthese{(T+\gua)^{\alpha}-(T-u+\gua)^{\alpha}}\\
		&\quad+(1/2) (T-u+\gua)^{\alpha} \expe{\norm{\bfX_{T-u}-\xo}^2}\\
		&\quad+(\Csham/2)\parenthese{(T+\gua)^{1-\alpha}-(T-u+\gua)^{1-\alpha}}\log(T+\gua) \eqsp ,
	\end{split}
\end{align}
with $\Csham = \max(\Cconvconttrois,(\gua \sigb+4\alpha\Cconvcontun) (1-\alpha)\pinv)$.
\end{proof}

\begin{lemma_colt}
  \label{lemma:s1_sT}
    Assume \rref{assum:f_lip},
    \tup{\rref{assum:grad_sto}-\ref{item:approx_sto}}, \rref{assum:lip_sigma},
    and \tup{\rref{assum:f}-\ref{item:f_conv}}. Then, for any
    $\alpha, \gamma \in \ooint{0,1}$ and $T \geq 0$ we have
    \begin{equation}
    S(1)-S(T) \leq 2\Csham\log(T+\gua) \log(1 + T)(T+\gua)^{-\min(\alpha,1-\alpha)} \eqsp , 
  \end{equation}
  with $S$ given in \eqref{eq:S_def}.
  \end{lemma_colt}

\begin{proof}  
In the case where $\alpha \leq 1/2$, \Cref{lemma:alpha} gives that for all $u \in \ccint{0,T}$:
\begin{equation}
	\parenthese{(T+\gua)^{\alpha}-(T-u+\gua)^{\alpha}} \leq \parenthese{(T+\gua)^{1-\alpha}-(T-u+\gua)^{1-\alpha}} \eqsp ,
\end{equation}
and we also have, for all $u \in \ccint{0,T}$:
\begin{align}
  &(T+\gua)^{1-\alpha}-(T+\gua-u)^{1-\alpha}\\
  &\quad= \parenthese{((T+\gua)^{1-\alpha}-(T+\gua-u)^{1-\alpha})((T+\gua)^{\alpha} +(T+\gua-u)^{\alpha})} \\ & \qquad \times \parenthese{(T+\gua)^{\alpha}+(T+\gua-u)^{-\alpha}}^{-1}  \\
  &\quad\leq \left((T+\gua)-(T+\gua-u)+(T+\gua)^{1-\alpha}(T+\gua-u)^{\alpha}  \right. \\ & \qquad \left. -(T+\gua)^{\alpha}(T+\gua-u)^{1-\alpha} \right)  \times (T+\gua)^{-\alpha} \leq 2u/(T+\gua)^{\alpha} \eqsp . 	\label{eq:shamir_qte_conj}
\end{align}

And in the case where $\alpha > 1/2$, for all $u \in \ccint{0,T}$:
\begin{equation}
	\parenthese{(T+\gua)^{1-\alpha}-(T-u+\gua)^{1-\alpha}} \leq \parenthese{(T+\gua)^{\alpha}-(T-u+\gua)^{\alpha}} \eqsp ,
\end{equation}
and we also have, for all $u \in \ccint{0,T}$:
\begin{align}
	\label{eq:shamir_qte_conj_2}
  &(T+\gua)^{\alpha}-(T+\gua-u)^{\alpha}\\
  &\quad= \parenthese{((T+\gua)^{\alpha}-(T+\gua-u)^{\alpha})((T+\gua)^{1-\alpha}+(T+\gua-u)^{1-\alpha})} \\
  & \qquad \parenthese{(T+\gua)^{1-\alpha}+(T+\gua-u)^{1-\alpha}}^{-1} \\
  &\quad\leq \left( (T+\gua)-(T+\gua-u)+(T+\gua)^{\alpha}(T+\gua-u)^{1-\alpha}\right. \\ & \qquad\left. -(T+\gua)^{1-\alpha}(T+\gua-u)^{\alpha}\right)  \times (T+\gua)^{-1+\alpha}\leq 2u/(T+\gua)^{1-\alpha} \eqsp .
\end{align}
Now, using \Cref{lemma:control_derivative} with $\xo = \bfX_{T-u}$ we obtain,
for all $u \in \ccint{0,T}$:
\begin{equation}
	\label{eq:shamir_int}
	\expe{\int_{T-u}^T f(\bfX_t)-f(\bfX_{T-u}) \dd t} \leq 2\Csham\log(T+\gua) (T+\gua)^{-\min(\alpha,1-\alpha)}u \eqsp .
\end{equation}
Since $S$ is a differentiable function and using \eqref{eq:shamir_int}, we have for all $u \in (0,T)$,
\begin{equation}
	\label{eq:shamir_derivative}
	S'(u)=-u^{-2}\int_{T-u}^T \expe{f(\bfX_t)}\dd t  + u\pinv \expe{f(\bfX_{T-u})}
  =-u\pinv(S(u)-\expe{f(\bfX_{T-u})}) \eqsp .
\end{equation}
This last result implies $-S'(u)\leq 2\Csham\log(T+\gua) /(T+\gua)^{-\min(\alpha,1-\alpha)} u\pinv$ and integrating we get
\begin{equation}
  S(1)-S(T) \leq 2\Csham\log(T+\gua) \log(T)(T+\gua)^{-\min(\alpha,1-\alpha)} \eqsp .
\end{equation}
\end{proof}

\begin{lemma_colt}
  \label{lemma:st-fstar}
    Assume \rref{assum:f_lip},
    \tup{\rref{assum:grad_sto}-\ref{item:approx_sto}}, \rref{assum:lip_sigma},
    and \tup{\rref{assum:f}-\ref{item:f_conv}}. Then, for any
    $\alpha, \gamma \in \ooint{0,1}$ and $T \geq 0$ we have
    \begin{equation}
  S(T)-f\st 
   \leq 2\Csham
              T^{-\min(\alpha,
              1-\alpha)}\log(T+\gua)\eqsp ,
\end{equation}
with $S$ given in \eqref{eq:S_def}.
  \end{lemma_colt}

\begin{proof}
  Using \Cref{lemma:control_derivative}, with $u=T$ and $\xo=x\st$, and
  $\normLigne{\bfX_0-x\st} \leq \Csham$ we obtain
\begin{align}
	\label{eq:shamir_ST}
	\begin{split}
		\int_0^T \expe{f(X_s)}\dd s - Tf\st &\leq (\Csham/2)\parenthese{(T+\gua)^\alpha - \gua^{\alpha}+\defEns{(T+\gua)^{1-\alpha}-\gamma}\log(T+\gua)} \\ & \qquad +(1/2)\gua \expe{\norm{\bfX_0-x\st}^2} \eqsp .
              \end{split}
\end{align}
Using this result we have
\begin{align}
  S(T)-f\st &\leq T^{-1}\Csham(T+\gua)^{\max(1-\alpha, \alpha)}\log(T+\gua) \\
  & \qquad + \Csham \gua T^{-1}  /2
   \leq 2\Csham
              T^{-\min(\alpha,
              1-\alpha)}\log(T+\gua)\eqsp .
\end{align}
\end{proof}

\begin{lemma_colt}
  \label{lemma:s0-s1}
    Assume \rref{assum:f_lip},
    \tup{\rref{assum:grad_sto}-\ref{item:approx_sto}}, \rref{assum:lip_sigma},
    and \tup{\rref{assum:f}-\ref{item:f_conv}}. Then, for any
    $\alpha, \gamma \in \ooint{0,1}$ and $T \geq 0$ we have
    \begin{equation}
      S(0) - S(1) \leq \Csham\Lip (T-1)^{-2\alpha} \eqsp ,
\end{equation}
with $S$ given in \eqref{eq:S_def}.
\end{lemma_colt}
\begin{proof}
We have
\begin{align}
	\label{eq:shamir_S1}
	S(0)-S(1)&=\expe{f(\bfX_T)}-S(1)=\int_{T-1}^T \parenthese{\expe{f(\bfX_T)}-\expe{f(\bfX_s)}} \dd s \eqsp .
\end{align}
Using Lemma~\ref{lemma:dynkin} on the stochastic process $f(\bfX_t)_{t\geq 0}$ and \rref{assum:f_lip}, we have for all $s\in[T-1,T]$
\begin{align}
	&\expe{f(\bfX_T)}-\expe{f(\bfX_s)}\\ & \qquad = -\int_s^T (\gua + t)^{-\alpha} \expeLigne{\norm{\nabla f(\bfX_t)}^2} \dd t + (\Lip/2)\gua \int_s^T (t+\gua)^{-2\alpha}\expe{\trace({\Sigma(\bfX_t)})}\dd t \\
	& \qquad \leq (\sigb\Lip/2)\gua \int_{s}^T (t+\gua)^{-2\alpha}\dd t  \leq (\Csham\Lip/2) (s+\gua)^{-2\alpha}(T-s)\eqsp .
\end{align}
Plugging this result into \eqref{eq:shamir_S1} yields
\begin{align}
	\label{eq:shamir_second}
	S(0)-S(1)&\leq (\Csham\Lip/2)\int_{T-1}^T (T-s)(s+\gua)^{-2\alpha} \dd s
	\leq \Csham\Lip (T-1+\gua)^{-2\alpha} \leq  \Csham\Lip (T-1)^{-2\alpha} \eqsp .
\end{align}
\end{proof}
We now give the extension of \Cref{thm:shamir_continuous} to the case where the function $f$ is only continuously differentiable and such that $\argmin_{\rset^{\dim}} f$ is bounded, see \Cref{thm:shamir_continuous_C1}.

\begin{proof}
  Let $\alpha, \gamma \in \ocint{0,1}$ and $T \geq 0$. $(\feps)_{\vareps > 0}$ be given by \Cref{lemma:eps}. Let $\delta=\min(\alpha, 1-\alpha)$. We can apply, \Cref{thm:shamir_continuous} to $f_{\vareps}$ for each $\vareps >0$. Therefore there exists $\Ccont_{\vareps}$ such that
  \begin{equation}
    \label{eq:eps_thm}
    \expe{f(\bfX_{T, \vareps})}-f(x_{\vareps}^{\star}) \leq \Ccont_{\vareps} \parentheseDeux{\log(T)^2T^{-\delta}+\log(T)T^{-\delta}+T^{-\delta}+(T-1)^{-2\alpha}}  \eqsp ,
  \end{equation}
  where $(\bfX_{t, \vareps})_{t \geq 0}$ is given by \eqref{eq:sde} with $\bfX_t = x_0$ (upon replacing $f$ by $\feps$) and
  \begin{equation}
    \Ccont_{\vareps} = 4 \max(2\Cconvcontdeux+2\norm{x_0 -x_{\vareps}^{\star}}^2 ,(\gua \sigb+2\alpha\Cconvcontun) (1-\alpha)\pinv) \eqsp .
  \end{equation}
  Using \eqref{eq:eps_thm} and \Cref{lemma:eps} we have
  \begin{align}
    \expe{f(\bfX_T)} - f\st &\leq \liminf_{\vareps \to 0} \expe{\feps(\bfX_{t, \vareps})} - \limsup_{\vareps \to 0} \feps(x_{\vareps}^{\star}) \\
                            &\leq \liminf_{\vareps \to 0} \defEns{\expe{\feps(\bfX_{t, \vareps})} -  \feps(x_{\vareps}^{\star})} \\
                            &\leq \liminf_{\vareps \to 0} \Ccont_{\vareps} \parentheseDeux{\log(T)^2T^{-\delta}+\log(T)T^{-\delta}+T^{-\delta}+(T-1)^{-2\alpha}} \\
    &\leq \Ccont_1 \parentheseDeux{\log(T)^2T^{-\delta}+\log(T)T^{-\delta}+T^{-\delta}+(T-1)^{-2\alpha}} \eqsp ,
  \end{align}
  with $\Ccont_1 = 3 \max(2\Cconvcontdeux+4\norm{x_0}^2+4\maxnorm^2 ,(\gua \sigb+2\Cconvcontun) (1-\alpha)\pinv)$, where $\maxnorm=\max_{y \in \argmin_{\bRd} f} \norm{y}$.
\end{proof}

\subsection{Proof of \Cref{thm:shamir_discrete}}
\label{app:shamir_discrete}

In this section we prove \Cref{thm:shamir_discrete}.  The proof is clearly more
involved than the one of \Cref{thm:shamir_continuous}. We will follow a similar
way as in the proof of \Cref{thm:shamir_continuous}, with more
technicalities. Again, one of the main argument of the proof is the suffix
averaging technique that was introduced in~\citep{shamirzhang}. We begin by the
discrete counterpart of \Cref{lemma:D2} in
\Cref{lemma:D2_discrete}. \Cref{prop:shamir_discrete_intermediate} is a first
step towards proving \Cref{thm:shamir_discrete}. It provides suboptimal bounds
for $\expeLigne{f(X_n)} -f^\star$. In order to prove this proposition, as in the
continuous-time case, we introduce the averaged process in
\eqref{eq:S_def_disc}. First, we control its derivative in
\Cref{lemma:control_derivative_discrete} (which is the discrete-time counterpart
of \Cref{lemma:control_derivative}). Then, we rewrite
$\expeLigne{f(X_n)} -f^\star$ as a sum of two terms involving $S$, which we
bound in \Cref{lemma:s0-st-discrete} (discrete counterpart of \Cref{lemma:s1_sT} and
\Cref{lemma:s0-s1}) and \Cref{lemma:st-fstar-discrete} (discrete counterpart of
\Cref{lemma:st-fstar}). This concludes the proof of \Cref{thm:shamir_discrete}
using our original bootstrapping technique. Finally, we conclude this section
with an extension of our result to the case where $\nabla f$ is bounded and no
longer Lipschitz continuous in \Cref{cor:shamir_bounded}.

\begin{lemma_colt}
	\label{lemma:D2_discrete}
	Assume \rref{assum:f_lip}, \tup{\rref{assum:f}-\ref{item:f_conv}}, \tup{\rref{assum:grad_sto}-\ref{item:approx_sto}}. Then for any $\alpha, \gamma \in \ooint{0,1}$, there exists $\Cconvdiscun\geq 0$, $\Cconvdiscdeux\geq 0$ and a function $\Cconvdisc:\bR_+ \to \bR_+$ such that, for any $n \geq 0$,
\begin{equation}
	\expe{\norm{X_n-x\st}^2}\leq \Cconvdiscun \Cconvdisc(n+1)+\Cconvdiscdeux \eqsp .
\end{equation}
And we have
\begin{equation}
	\Cconvdisc(t)=
	\begin{cases}
		t^{1-2\alpha}&\mbox{if }\alpha<1/2 \eqsp ,\\
		\log(t)&\mbox{if }\alpha=1/2 \eqsp ,\\
		0&\mbox{if }\alpha>1/2 \eqsp .
	\end{cases}
\end{equation}
The values of the constants are given by
\begin{align}
	\Cconvdiscun&=
	\begin{cases}
		2\gamma^2\eta(1-2\alpha)\pinv&\mbox{if }\alpha<1/2 \eqsp, \\
		\gamma^2\eta&\mbox{if }\alpha=1/2 \eqsp ,\\
		0&\mbox{if }\alpha>1/2 \eqsp .
	\end{cases}\\
	\Cconvdiscdeux&=
	\begin{cases}
		2\max_{k \leq (\gamma \Lip/2)^{1/\alpha}} \expe{\norm{X_k-x\st}^2}&\mbox{if }\alpha<1/2 \eqsp, \\
		2\max_{k \leq (\gamma \Lip/2)^{1/\alpha}} \expe{\norm{X_k-x\st}^2}+2\gamma^2\eta&\mbox{if }\alpha=1/2 \eqsp ,\\
		2\max_{k \leq (\gamma \Lip/2)^{1/\alpha}} \expe{\norm{X_k-x\st}^2}+\gamma^2\eta(2\alpha-1)\pinv&\mbox{if }\alpha>1/2 \eqsp ,
	\end{cases}
\end{align}
\end{lemma_colt}

\begin{proof}
	Let $f:\bRd \to \bR$ verifying assumptions \rref{assum:f_lip} and \tup{\rref{assum:f}-\ref{item:f_conv}}. We consider $(X_n)_{n \geq 0}$ satisfying \eqref{eq:sgd}. Let $x\st \in \bRd$ be given by \tup{\rref{assum:f}-\ref{item:f_conv}}. We have, using \eqref{eq:sgd} and \tup{\rref{assum:grad_sto}-\ref{item:approx_sto}} that for all $n\geq (\gamma \Lip/2)^{1/\alpha}$,
	\begin{align}
		&\expec{\norm{X_{n+1}-x\st}^2}{\cF_n}=\expec{\norm{X_n-x\st-\gamma(n+1)^{-\alpha}H(X_n,Z_{n+1})}^2}{\cF_n}\\
		& \qquad =\norm{X_n-x\st}^2-2\gamma/(n+1)^{\alpha}\la X_n-x\st, \expec{H(X_n,Z_{n+1})}{\cF_n}\ra \\
		& \qquad \quad +\gamma^2(n+1)^{-2\alpha}\expec{\norm{H(X_n,Z_{n+1})}^2}{\cF_n} \\
			& \qquad =\norm{X_n-x\st}^2-2\gamma/(n+1)^{\alpha}\la X_n-x\st, \nabla f(X_n)\ra\\
			& \qquad \quad+\gamma^2(n+1)^{-2\alpha}\expec{\norm{H(X_n,Z_{n+1})-\nabla f(X_n)+\nabla f(X_n)}^2}{\cF_n} \\
			& \qquad =\norm{X_n-x\st}^2-2\gamma/(n+1)^{\alpha}\la X_n-x\st, \nabla f(X_n)\ra \\
			& \qquad \quad+\gamma^2(n+1)^{-2\alpha}\expec{\norm{H(X_n,Z_{n+1})-\nabla f(X_n)}^2}{\cF_n}\\
			& \qquad \quad+\gamma^2(n+1)^{-2\alpha}\Big(\expec{\norm{\nabla f(X_n)}^2}{\cF_n} \\
			& \qquad \qquad + 2\expec{\la H(X_n,Z_{n+1})-\nabla f(X_n),\nabla f(X_n) \ra}{\cF_n}\Big)\\
		& \qquad =\norm{X_n-x\st}^2-2\gamma/(n+1)^{\alpha}\la X_n-x\st, \nabla f(X_n)\ra+\gamma^2\eta(n+1)^{-2\alpha} \\
		& \qquad \quad+\gamma^2(n+1)^{-2\alpha}\norm{\nabla f(X_n)}^2 \\
		& \qquad \leq\norm{X_n-x\st}^2-2\gamma/\Lip(n+1)^{-\alpha}\norm{\nabla f(X_n)}^2+\gamma^2\eta(n+1)^{-2\alpha}\\
		& \qquad \quad+\gamma^2(n+1)^{-2\alpha}\norm{\nabla f(X_n)}^2 \\
		& \qquad \leq\norm{X_n-x\st}^2+\gamma/(n+1)^{\alpha}\norm{\nabla f(X_n)}^2\parentheseDeux{\gamma/(n+1)^{\alpha}-2/\Lip}+\gamma^2\eta(n+1)^{-2\alpha} \\
		& \qquad \leq\norm{X_n-x\st}^2+\gamma^2\eta(n+1)^{-2\alpha} \\
		&  \qquad \leq \expe{\norm{X_n-x\st}^2}+\gamma^2\eta(n+1)^{-2\alpha} \eqsp , 		\label{eq:D2_discrete_main}
	\end{align}
where we used the co-coercivity of $f$.
Summing the previous inequality leads to
\begin{equation}
	\expe{\norm{X_n-x\st}^2}-\expe{\norm{X_0-x\st}^2}\leq \gamma^2\eta \sum_{k=1}^n k^{-2\alpha} \eqsp .
\end{equation}
As in the previous proof we now distinguish three cases:

\begin{enumerate}[wide, labelwidth=!, labelindent=0pt, label=(\alph*)]
		\item If $\alpha < 1/2$, we have 
			\begin{align}
                          \expe{\norm{X_n-x\st}^2}&\leq \norm{X_0-x\st}^2+\gamma^2\eta(1-2\alpha)\pinv(n+1)^{1-2\alpha} \\ &\leq \norm{X_0-x\st}^2+2\gamma^2\eta(1-2\alpha)\pinv n^{1-2\alpha} \eqsp .
			\end{align}
		\item If $\alpha = 1/2$, we have 
			\begin{align}
				\expe{\norm{X_n-x\st}^2}\leq \norm{X_0-x\st}^2+\gamma^2\eta(\log(n)+2) \eqsp .
			\end{align}
                      \item If $\alpha > 1/2$, we have
			\begin{align}
				\expe{\norm{X_n-x\st}^2}\leq \norm{X_0-x\st}^2+\gamma^2\eta(2\alpha-1)\pinv \eqsp .
			\end{align}
	\end{enumerate}
\end{proof}

We now turn to the proof of \Cref{thm:shamir_discrete} by stating an
intermediate result where we assume a condition bounding
$\expeLigne{\normLigne{\nabla f(X_n)}^2}$. This proposition provides non-optimal
convergence rates for SGD but will be used as a central tool to improve them via
a bootstrapping technique and obtain optimal convergence rates.

\begin{proposition_colt}
	\label{prop:shamir_discrete_intermediate}
	Let $\gamma, \alpha \in \ooint{0,1}$ and $x_0 \in \rset^{\dim}$ and
        $(X_n)_{n\geq 0}$ be given by \eqref{eq:sgd}. Assume \rref{assum:f_lip},
        \tup{\rref{assum:f}-\ref{item:f_conv}},
        \tup{\rref{assum:grad_sto}-\ref{item:approx_sto}}.  Suppose additionally
        that there exists $\alpha\st\in\ccint{0,1/2}$, $\beta>0$ and
        $\Cshamdisc\geq 0$ such that for all $n \in \nset$
        \begin{equation}
        	\label{eq:assum_shamir_int}
    				\expeLigne{\normLigne{\nabla f(X_n)}^2} \leq
    				\begin{cases}
    					\Cshamdisc (n+1)^{\beta}\log(n+1)&\mbox{if }\alpha \leq \alpha\st \eqsp ,\\
    					\Cshamdisc &\mbox{if }\alpha>\alpha\st \eqsp .
    				\end{cases}
        \end{equation}

        Then there exists $\Cshamt \geq 0$
        such that, for all $N\geq 1$,
		\begin{equation}
			\expe{f(X_N)}-f\st \leq \Cshamt \defEns{ (1+\log(N+1))^2 /(N+1)^{\min(\alpha,1-\alpha)}\Psial(N+1) +1/(N+1)} \eqsp ,
	\end{equation}
	where for any $n \in \nset$
	\begin{equation}
 		\Psial(n)=
 		\begin{cases}
 			n^{\beta}(1+\log(n))&\mbox{if }\alpha \leq \alpha\st \eqsp , \\
 			1 &\mbox{if }\alpha>\alpha\st \eqsp .
 		\end{cases}
 \end{equation}
\end{proposition_colt}

\begin{proof}
  Let $\alpha, \gamma \in \ooint{0,1}$ and $N \geq 1$. Let $(X_n)_{n \geq 0}$ be given by \eqref{eq:sgd}.
  The proof is a straightforward application of \Cref{lemma:s0-st-discrete} and \Cref{lemma:st-fstar-discrete} below with 
  $\Cshamt = 2\max((2\gamma)\pinv\norm{X_0-x\st}^2,2\Cdisc)$.
\end{proof}
Let $(S_k)_{k \in \defEns{0,\cdots,N}}$ be given for any $k \in \nset$ by
\begin{equation}
  \label{eq:S_def_disc}
	S_k = (k+1)\pinv \sum_{t=N-k}^N \expe{f(X_t)} \eqsp .
      \end{equation}
      Note that $\expeLigne{f(X_N)} - f^\star = (S_N - S_0) + (S_0 - f^\star)$.
      We are now going to control each one of the two terms $(S_0 - S_N)$ and
      $(S_N - f^\star)$ as follows:
\begin{enumerate}[wide, labelwidth=!, labelindent=0pt, label=(\alph*)]
\item Case $S_n -S_0$ (\Cref{lemma:s0-st-discrete}): this is an adaption of the
  idea of suffix averaging of \cite{shamirzhang} to our setting (one of crucial
  difference lies into the control of the sequence
  $(\expeLigne{\nabla f(X_n)}^2)_{n \in \nset}$ which is assumed to be uniformly
  bounded in \cite{shamirzhang}). In particular, we control the (discrete)
  time-derivative of $S$ in \Cref{lemma:control_derivative_discrete}.
\item Case $S(T) - f^\star$ (\Cref{lemma:st-fstar-discrete}): this result is known and
  corresponds to the optimal convergence rate of the averaged sequence towards
  the minimum of $f$. We provide its proof for completeness.
\end{enumerate}
Before controlling each one of these terms we state the following useful lemma,
which will allow us to control the derivative of $S$.

\begin{lemma_colt}
  \label{lemma:control_derivative_discrete}
  Assume \rref{assum:f_lip}, \tup{\rref{assum:grad_sto}-\ref{item:approx_sto}},
  and \tup{\rref{assum:f}-\ref{item:f_conv}}. In addition, assume that
  \eqref{eq:assum_shamir_int} holds. Then, for any
  $\alpha, \gamma \in \ooint{0,1}$, $N \in \nset$, $u \in \{0, \dots, N\}$ and
  $Y$ any $\rset^d$-valued random variable such that
  $\expeLigne{\normLigne{Y -x^\star}^2} \leq
  \Cconvdiscun(N+1)^{1-2\alpha}\log(N+1)+4\Cconvdiscdeux$ with $\Cconvdiscun$
  and $\Cconvdiscdeux$ given in \Cref{lemma:D2_discrete}, we have
\begin{align}
	&\expe{\sum_{k=N-u}^N f(X_k)-f(X_{N-u})} \\ & \qquad \leq 2\Cdisc (u+1)/(N+1)^{\min(\alpha, 1-\alpha)} (1+\log(N+1))\Psial(N+1) \\
                                                   &\qquad \qquad +(2\gamma)\pinv (N-u+1)^{\alpha} \expeLigne{\normLigne{X_{N-u}-\xo}^2} \eqsp ,
                                                     	\label{eq:shamir_discrete_master}
\end{align}
with
$\Cdisc\doteq 4\Cconvddeux + (2\gamma)\pinv( 4\Cconvdiscdeux+4\Cconvdiscun)$
with $\Cconvdiscun$ and $\Cconvdiscun$ given in \Cref{lemma:D2_discrete} and
 \begin{equation}
 		\Psial(n)=
 		\begin{cases}
 			n^{\beta}(1+\log(n))&\mbox{if }\alpha \leq \alpha\st \eqsp , \\
 			1 &\mbox{if }\alpha>\alpha\st \eqsp .
 		\end{cases}
 \end{equation}  
\end{lemma_colt}

\begin{proof}
Let $\ell \in \defEns{0,\cdots,N}$,
let $k \geq \ell$, let $\xo \in \cF_{\ell}$. Using \tup{\rref{assum:f}-\ref{item:f_conv}} we have
\begin{align}
	\expek{\norm{X_{k+1}-\xo}^2}&=\expek{\norm{X_k-\xo-\gamma (k+1)^{-\alpha} H(X_k, Z\kk)}^2}\\
	&=\norm{X_k-\xo}^2+\gamma^2(k+1)^{-2\alpha}\expek{\norm{H(X_k,Z_{k+1})}^2}\\
	&\quad-2\gamma(k+1)^{-\alpha}\la X_k-\xo, \nabla f(X_k)\ra\\
		\expe{f(X_k)-f(\xo)} &\leq (2\gamma)\pinv(k+1)^{\alpha}\parenthese{\expe{\norm{X_k-\xo}^2}-\expe{\norm{X_{k+1}-\xo}^2}}\\
		&\quad+(\gamma/2)(k+1)^{-\alpha}\expe{\expek{\norm{H(X_k,Z_{k+1})}^2}}\\
		\expe{f(X_k)-f(\xo)} &\leq (2\gamma)\pinv(k+1)^{\alpha}\parenthese{\expe{\norm{X_k-\xo}^2}-\expe{\norm{X_{k+1}-\xo}^2}}\\
		&\quad+(\gamma/2)(k+1)^{-\alpha}\parenthese{\eta+\expe{\norm{\nabla f(X_k)}^2}} \eqsp . 	\label{eq:shamir_discrete_before_sum}
\end{align}
Let $u \in \defEns{0,\cdots,N}$. Summing now \eqref{eq:shamir_discrete_before_sum} between $k=N-u$ and $k=N$ gives
\begin{align}	
		\expe{\sum_{k=N-u}^N f(X_k)-f(\xo)} &\leq \Cconvdun \sum_{k=N-u}^N (k+1)^{-\alpha} \\
		&\quad+ (2\gamma)\pinv \sum_{k=N-u+1}^N \expe{\norm{X_k-\xo}^2}\parenthese{(k+1)^{\alpha}-k^{\alpha}} \\
	&\quad+\Cconvddeux  \sum_{k=N-u}^N \expe{\norm{\nabla f(X_k)}^2} (k+1)^{-\alpha}\\
	&\quad+(2\gamma)\pinv (N-u+1)^{\alpha} \expe{\norm{X_{N-u}-\xo}^2}\eqsp . \label{eq:shamir_discrete_sum}
\end{align}
We now have to conduct separate analyses depending on the value of $\alpha$.

\begin{enumerate}[wide, labelwidth=!, labelindent=0pt, label=(\alph*)]
\item First assume that $\alpha \leq \alpha^\star$.  In that case
  \eqref{eq:assum_shamir_int} gives that
		\begin{equation}
			\expe{\norm{\nabla f(X_k)}^2} \leq \Cshamdisc (N+1)^{\beta}\log(N+1),
		\end{equation}
		and Lemma~\ref{lemma:D2_discrete} gives that for all $k \in \defEns{0,\dots,N}$,
		\begin{align}
			\expe{\norm{X_k-\xo}^2}&\leq 2\expe{\norm{X_k-x\st}^2}+2\expe{\norm{\xo-x\st}^2} \\
			&\leq 2\Cconvdiscun (k+1)^{1-2\alpha}\log(k+1)+2\Cconvdiscun(N+1)^{1-2\alpha}\log(N+1)+4\Cconvdiscdeux \\
			&\leq 4\Cconvdiscun (N+1)^{1-2\alpha}\log(N+1)+4\Cconvdiscdeux \eqsp .
		\end{align}
		We note $\Cconvdisctrois \doteq 4\Cconvdiscdeux$.  Combining
                \eqref{eq:shamir_discrete_sum} and
                $\Cconvdtrois=(\Cconvdun+\Cconvddeux\Cshamdisc)(1-\alpha)\pinv$
                we get that
		\begin{align}
				&\expe{\sum_{k=N-u}^N f(X_k)-f(\xo)} \leq  \Cconvdun (1-\alpha)\pinv \parenthese{(N+1)^{1-\alpha}-(N-u)^{1-\alpha}} \\
				&\quad +(2\gamma)\pinv (N-u+1)^{\alpha} \expe{\norm{X_{N-u}-\xo}^2} \\
				&\quad+ (2\gamma)\pinv\parenthese{\Cconvdisctrois+4\Cconvdiscun(N+1)^{1-2\alpha}\log(N+1)}\parenthese{(N+1)^{\alpha}-(N-u+1)^{\alpha}} \\
				&\quad + \Cconvddeux\Cshamdisc(N+1)^{\beta}\log(N+1) (1-\alpha)\pinv \parenthese{(N+1)^{1-\alpha}-(N-u)^{1-\alpha}} \\
				&\quad \leq  \Cconvdtrois (N+1)^{\beta}(1+\log(N+1))^2  \parenthese{(N+1)^{1-\alpha}-(N-u)^{1-\alpha}} \\
				&\quad +(2\gamma)\pinv (N-u+1)^{\alpha} \expe{\norm{X_{N-u}-\xo}^2} \\
				&\quad+ (2\gamma)\pinv\Cconvdisctrois\parenthese{(N+1)^{\alpha}-(N-u)^{\alpha}}\\
				&\quad+(2\gamma)\pinv 4\Cconvdiscun\parenthese{(N+1)^{1-\alpha}-(N-u)^{1-\alpha}} \\
			& \quad\leq \Cdisc (N+1)^{\beta}(1+\log(N+1))^2\parenthese{(N+1)^{1-\alpha}-(N-u)^{1-\alpha}} \\
			&\quad +(2\gamma)\pinv (N-u+1)^{\alpha} \expe{\norm{X_{N-u}-\xo}^2} \eqsp , 			\label{eq:shamir_discrete_leq}
		\end{align}
		where we used Lemma~\ref{lemma:alpha}.

	Notice now that, similarly to \eqref{eq:shamir_qte_conj} we have
	\begin{align}
    \label{eq:alpha_conj}
		&(N+1)^{1-\alpha}-(N-u)^{1-\alpha} \\
		&\qquad= \defEns{\parenthese{(N+1)^{1-\alpha}-(N-u)^{1-\alpha}} \parenthese{(N+1)^{\alpha}+(N-u)^{\alpha}}}\parenthese{(N+1)^{\alpha}+(N-u)^{\alpha}}\pinv \\
		&\qquad\leq 2(u+1)/(N+1)^{\alpha} \eqsp .
	\end{align}

      \item Second, assume that $\alpha \in \ocint{\alpha^\star, 1/2}$. Using
        Lemma~\ref{lemma:D2_discrete}, we have for all
        $k \in \defEns{0,\dots,N}$,
		\begin{align}
			\expe{\norm{X_k-\xo}^2}&\leq 2\expe{\norm{X_k-x\st}^2}+2\expe{\norm{\xo-x\st}^2} \\
			&\leq 2\Cconvdiscun (k+1)^{1-2\alpha}\log(k+1)+2\Cconvdiscun(N+1)^{1-2\alpha}\log(N+1)+4\Cconvdiscdeux \\
			&\leq 4\Cconvdiscun (N+1)^{1-2\alpha}\log(N+1)+4\Cconvdiscdeux \eqsp .
		\end{align}
		Using \eqref{eq:assum_shamir_int}, \eqref{eq:shamir_discrete_sum} rewrites
		\begin{align}
			\label{eq:shamir_discrete_eq}
				&\expe{\sum_{k=N-u}^N f(X_k)-f(\xo)} \leq \Cconvdun (1-\alpha)\pinv \parenthese{(N+1)^{1-\alpha}-(N-u)^{1-\alpha}} \\
				&\quad+(2\gamma)\pinv (N-u+1)^{\alpha} \expe{\norm{X_{N-u}-\xo}^2} \\
				&\quad+ (2\gamma)\pinv\parenthese{\Cconvdisctrois+4\Cconvdiscun\log(N+1)(N+1)^{1-2\alpha}}\parenthese{(N+1)^{\alpha}-(N-u+1)^{\alpha}} \\
				&\quad + \Cconvddeux\Cshamdisc(1-\alpha)\pinv \parenthese{(N+1)^{1-\alpha}-(N-u)^{1-\alpha}} \\
				&\leq  \Cconvdtrois \parenthese{(N+1)^{1-\alpha}-(N-u)^{1-\alpha}} +(2\gamma)\pinv (N-u+1)^{\alpha} \expe{\norm{X_{N-u}-\xo}^2} \\
				&\quad+ (2\gamma)\pinv\parenthese{\Cconvdisctrois+4\Cconvdiscun}(1+\log(N+1))\parenthese{(N+1)^{\alpha}-(N-u)^{\alpha}} \\
			&\leq\Cdisc (1+ \log(N+1)) \parenthese{(N+1)^{1-\alpha}-(N-u)^{1-\alpha}} \\
			&\quad+(2\gamma)\pinv (N-u+1)^{\alpha} \expe{\norm{X_{N-u}-\xo}^2} \eqsp .
		\end{align}
              \item Finally, assume that $\alpha > 1/2$.  In that case,
                $\alpha>\alpha\st$ and Lemma~\ref{lemma:D2_discrete} gives
		\begin{equation}
			\forall k \in \defEns{0,\dots,N}, \, \expe{\norm{X_k-\xo}^2}\leq 2\expe{\norm{X_k-x\st}^2}+2\expe{\norm{\xo-x\st}^2} \leq 4\Cconvdiscdeux = \Cconvdisctrois \eqsp .
		\end{equation}
		Using Lemma~\ref{lemma:alpha} and \eqref{eq:assum_shamir_int} we rewrite \eqref{eq:shamir_discrete_sum} as
		\begin{align}
			\label{eq:shamir_discrete_geq}
			\begin{split}
				&\expe{\sum_{k=N-u}^N f(X_k)-f(\xo)} \leq  (\Cconvdun+\gamma\Cshamdisc/2) (1-\alpha)\pinv \parenthese{(N+1)^{1-\alpha}-(N-u)^{1-\alpha}} \\
				&\quad+(2\gamma)\pinv (N-u+1)^{\alpha} \expe{\norm{X_{N-u}-\xo}^2} \\
				&\quad+ (2\gamma)\pinv\Cconvdisctrois\parenthese{(N+1)^{\alpha}-(N-u+1)^{\alpha}}
			\end{split} \\
			\begin{split}
				&\leq  \Cconvdtrois \parenthese{(N+1)^{1-\alpha}-(N-u)^{1-\alpha}} +(2\gamma)\pinv (N-u+1)^{\alpha} \expe{\norm{X_{N-u}-\xo}^2} \\
				&\quad+ (2\gamma)\pinv\Cconvdisctrois\parenthese{(N+1)^{\alpha}-(N-u)^{\alpha}}
			\end{split} \\
			&\leq\Cdisc \parenthese{(N+1)^{\alpha}-(N-u)^{\alpha}} +(2\gamma)\pinv (N-u+1)^{\alpha} \expe{\norm{X_{N-u}-\xo}^2} \eqsp .
		\end{align}

	Notice now that, similarly to \eqref{eq:shamir_qte_conj} we have
	\begin{align}
		&(N+1)^{\alpha}-(N-u)^{\alpha} \\
		&\qquad= \defEns{\parenthese{(N+1)^{\alpha}-(N-u)^{\alpha}}  \parenthese{(N+1)^{1-\alpha}+(N-u)^{1-\alpha}}}\\ & \qquad \quad \times \parenthese{(N+1)^{1-\alpha}+(N-u)^{1-\alpha}}\pinv \leq 2(u+1)/(N+1)^{1-\alpha} \eqsp .
	\end{align}

\end{enumerate}

Finally, putting the three cases above together we obtain

\begin{align}
	&\expe{\sum_{k=N-u}^N f(X_k)-f(X_{N-u})} \\ & \qquad \leq 2\Cdisc (u+1)/(N+1)^{\min(\alpha, 1-\alpha)} (1+\log(N+1))\Psial(N+1) \\
                                                   &\qquad \qquad +(2\gamma)\pinv (N-u+1)^{\alpha} \expe{\norm{X_{N-u}-\xo}^2} \eqsp ,
                                                     	\label{eq:shamir_discrete_master}
\end{align}
with
 \begin{equation}
 		\Psial(n)=
 		\begin{cases}
 			n^{\beta}(1+\log(n))&\mbox{if }\alpha \leq \alpha\st \eqsp , \\
 			1 &\mbox{if }\alpha>\alpha\st \eqsp .
 		\end{cases}
              \end{equation}
Note that the additional $\log(N+1)$ factor can be removed if $\alpha \neq 1/2$.
              \end{proof}

\begin{lemma_colt}
  \label{lemma:s0-st-discrete}
  Assume \rref{assum:f_lip}, \tup{\rref{assum:grad_sto}-\ref{item:approx_sto}}
  and \tup{\rref{assum:f}-\ref{item:f_conv}}. In addition, assume that
  \eqref{eq:assum_shamir_int} holds. Then, for any
  $\alpha, \gamma \in \ooint{0,1}$ and $N \in \nset$ we have
    \begin{equation}
      S_0-S_N \leq 2\Cdisc (N+1)^{-\min(\alpha, 1-\alpha)} (1+\log(N+1))^2 \Psial(N+1) \eqsp .
  \end{equation}
  with $S$ given in \eqref{eq:S_def_disc}.
  \end{lemma_colt}

  \begin{proof}
    Let $u \in \defEns{0,\dots, N}$. Using
    \Cref{lemma:control_derivative_discrete} with the choice $\xo=X_{N-u}$ gives
\begin{align}
	\label{eq:shamir_discrete_sum2}
	\expe{\sum_{k=N-u}^N f(X_k)-f(X_{N-u})}\leq 2\Cdisc (u+1)/(N+1)^{\min(\alpha, 1-\alpha)} (1+\log(N+1)) \Psial(N+1) \eqsp .
\end{align}

		And then,
		\begin{align}			
			S_u&=(u+1)\pinv \sum_{k=N-u}^N \expe{f(X_k)} \\
			&\leq 2\Cdisc (N+1)^{-\min(\alpha, 1-\alpha)} (1+\log(N+1))\Psial(N+1) + \expe{f(X_{N-u})} \eqsp . \label{eq:shamir_discrete_Su}
		\end{align}
		We have now, using \eqref{eq:shamir_discrete_Su},
		\begin{align}
			uS_{u-1}&=(u+1)S_u-\expe{f(X_{N-u})} \\
			&=uS_u+S_u-\expe{f(X_{N-u})}\\
			&\leq uS_u+2\Cdisc (N+1)^{-\min(\alpha, 1-\alpha)} (1+\log(N+1)) \Psial(N+1) \\
			S_{u-1}-S_u &\leq  2\Cdisc u\pinv (N+1)^{-\min(\alpha, 1-\alpha)} \log(N+1) \\
			S_0-S_N &\leq 2\Cdisc (N+1)^{-\min(\alpha, 1-\alpha)} (1+\log(N+1)) \Psial(N+1) \sum_{u=1}^N(1/u)  \\
                  S_0-S_N &\leq 2\Cdisc (N+1)^{-\min(\alpha, 1-\alpha)} (1+\log(N+1))^2 \Psial(N+1) \eqsp .
                            			\label{eq:shamir_discrete_a}
		\end{align}
              \end{proof}

\begin{lemma_colt}
  \label{lemma:st-fstar-discrete}
  Assume \rref{assum:f_lip}, \tup{\rref{assum:grad_sto}-\ref{item:approx_sto}}
  and \tup{\rref{assum:f}-\ref{item:f_conv}}. In addition, assume that
  \eqref{eq:assum_shamir_int} holds. Then, for any
  $\alpha, \gamma \in \ooint{0,1}$ and $N \in \nset$ we have
  \begin{align}
    S_N-f\st&\leq 2\Cdisc (1+\log(N+1))^2 (N+1)^{-\min(\alpha,1-\alpha)}\Psial(N+1)\\
		&\quad+(2\gamma)\pinv(N+1)\pinv\norm{X_0-x\st}^2 \eqsp .
  \end{align}
  with $S$ given in \eqref{eq:S_def_disc}.
  \end{lemma_colt}
              \begin{proof}
                Using \Cref{lemma:control_derivative_discrete} with the choice
                $\xo=x\st$ and $u=N$ gives
	\begin{align}
		(N+1)\pinv\expe{\sum_{k=0}^N f(X_k)-f(x\st)} &\leq 2\Cdisc (1+\log(N+1)) (N+1)^{-\min(\alpha,1-\alpha)}\Psial(N+1)\\
                                                             &\quad+(2\gamma)\pinv(N+1)\pinv\norm{X_0-x\st}^2
        \end{align}
        Therefore,
        \begin{align}
		S_N-f\st&\leq 2\Cdisc (1+\log(N+1))^2 (N+1)^{-\min(\alpha,1-\alpha)}\Psial(N+1)\\
		&\quad+(2\gamma)\pinv(N+1)\pinv\norm{X_0-x\st}^2 \eqsp .		\label{eq:shamir_discrete_b}
	\end{align}
        \end{proof}
We can finally conclude the proof of \Cref{thm:shamir_discrete}.
\begin{proof}
  We begin by proving by induction over $m \in \bN^*$ that the following
  assertion \Cref{assum:recu_conv}($m$) is true.
  \begin{assumptionH}[$m$]
    \label{assum:recu_conv}
    For any $\alpha > 1/(m+1)$, there exists $\Cshamaplus>0$ such that for all
    $n\in\nset,\ \expeLigne{\normLigne{\nabla f(X_n)}^2} \leq \Cshamaplus$. In
    addition, for any $\alpha \leq 1/(m+1)$, there exists $\Cshamamoins>0$ such
    that for all
    $n\in\nset,\ \expeLigne{\normLigne{\nabla f(X_n)}^2} \leq \Cshamamoins
    n^{1-(m+1)\alpha}(1+\log(n))^3$.
\end{assumptionH}
For $m=1$, \Cref{assum:recu_conv}$(1)$ is an immediate consequence of
\rref{assum:f_lip} and Lemma~\ref{lemma:D2_discrete}, with
$\Cshamaplus=\Lip^2\Cconvdiscdeux$ and
$\Cshamamoins=\Lip^2\max(\Cconvdiscun, \Cconvdiscdeux)$.  Now, let $m\in \bN^*$
and suppose that \Cref{assum:recu_conv}$(m)$ holds.  Let
$\alpha \in \ooint{0,1}$. Setting $\alpha\st=1/(m+1)$ we have that
\eqref{eq:assum_shamir_int} is verified with $\beta=1-(m+1)\alpha$.
Consequently, using \rref{assum:f_lip}, \tup{\rref{assum:f}-\ref{item:f_conv}}
and \tup{\rref{assum:grad_sto}-\ref{item:approx_sto}} we can apply
Proposition~\ref{prop:shamir_discrete_intermediate} and for any
$\alpha \leq 1/(m+1)$ we have
\begin{align}
	\expe{f(X_N)}-f\st &\leq \Cshamt \defEns{ (1+\log(N+1))^2 /(N+1)^{\min(\alpha,1-\alpha)}\Psial(N+1) +1/(N+1)} \\
	&\leq \Cshamt \defEns{ (1+\log(N+1))^3 (N+1)^{-\alpha}(N+1)^{1-(m+1)\alpha} +1/(N+1)} \\
	&\leq \Cshamt \defEns{ (1+\log(N+1))^3 (N+1)^{1-(m+2)\alpha} +1/(N+1)}\eqsp . 	\label{eq:shamir_rec}
\end{align}
In particular, if $\alpha >1/(m+2)$ we have the existence of
$\bar{\Cconv}_{\alpha}>0$ such that for all $n \in \nset$,
$\expeLigne{f(X_n)}-f\st \leq \bar{\Cconv}_{\alpha}$. And using
\rref{assum:f_lip} and Lemma~\ref{lemma:kolmo} we get that, for all
$n \in \nset$
\begin{equation}
\expeLigne{\normLigne{\nabla f(X_n)}^2} \leq 2\Lip \expe{f(X_n)-f\st} \leq 2\Lip \bar{\Cconv}_{\alpha} \eqsp ,
\end{equation}
Combining this result with \eqref{eq:shamir_rec}, we get that
\Cref{assum:recu_conv}$(m+1)$ holds with
$\Cshamaplus=2\Lip \bar{\Cconv}_{\alpha}$ and $\Cshamamoins=2\Cshamt$. We
conclude by recursion.

Now, let $\alpha \in \ooint{0,1}$. Since $\rset$ is archimedean, there exists
$m \in \bN^*$ such that $\alpha>1/(m+1)$ and therefore
\Cref{assum:recu_conv}($m$) shows the existence of $\Cshamdisc>0$ such that
$\expeLigne{\normLigne{\nabla f(X_n)}^2} \leq \Cshamdisc$ for all $n \in
\bN^*$. Applying Proposition~\ref{prop:shamir_discrete_intermediate} gives the
existence of $\Cdisc >0$ such that for all $N\geq 1$
\begin{equation}
	\expe{f(X_N)}-f\st \leq \Cdisc (1+\log(N+1))^2/(N+1)^{\min(\alpha, 1-\alpha)} \eqsp ,
\end{equation}
with $\Cdisc = 2 \Cshamt$, concluding the proof.
\end{proof}

We present now a corollary of the previous theorem under a different setting. Let us assume, as in~\citep{shamirzhang}, that $\nabla f$ is not Lipschitz-continuous but bounded instead.

\begin{corollary_colt}
	\label{cor:shamir_bounded}
	Let $\gamma, \alpha \in \ooint{0,1}$ and $x_0 \in \rset^{\dim}$ and
        $(X_n)_{n\geq 0}$ be given by \eqref{eq:sgd}. Assume \tup{\rref{assum:f}-\ref{item:f_conv}}, \tup{\rref{assum:grad_sto}-\ref{item:approx_sto}} and $\nabla f$ bounded. Then there exists $\Cdisc_b \geq 0$
        such that, for all $N\geq 1$,
		\begin{equation}
       \expe{f(X_N)}-f\st \leq \Cdisc_b (1+\log(N+1))^2/(N+1)^{\min(\alpha,1-\alpha)} \eqsp .
     \end{equation}
\end{corollary_colt}

\begin{proof}
	The proof follows the same lines as the ones of Lemma~\ref{lemma:D2_discrete} and Proposition~\ref{prop:shamir_discrete_intermediate}. We show that both conclusions hold under the assumption that $\nabla f$ is bounded instead of being Lipschitz-continuous.

	In order to prove that Lemma~\ref{lemma:D2_discrete} still holds, let us do the following computation. We consider $(X_n)_{n \geq 0}$ satisfying \eqref{eq:sgd}.
	We have, using \eqref{eq:sgd}, \tup{\rref{assum:f}-\ref{item:f_conv}} and \tup{\rref{assum:grad_sto}-\ref{item:approx_sto}} that for all $n\geq 0$,
	\begin{align}
		\expecLigne{\norm{X_{n+1}-x\st}^2}{\cF_n}&=\expec{\norm{X_n-x\st-\gamma(n+1)^{-\alpha}H(X_n,Z_{n+1})}^2}{\cF_n}\\
		&=\norm{X_n-x\st}^2-2\gamma/(n+1)^{\alpha}\la X_n-x\st, \expec{H(X_n,Z_{n+1})}{\cF_n}\ra \\
		&\quad + \gamma^2(n+1)^{-2\alpha}\expec{\norm{H(X_n,Z_{n+1})}^2}{\cF_n} \\
		&=\norm{X_n-x\st}^2-2\gamma/(n+1)^{\alpha}\la X_n-x\st, \nabla f(X_n)\ra\\
		&\quad+\gamma^2\eta(n+1)^{-2\alpha}+\gamma^2(n+1)^{-2\alpha}\norm{\nabla f(X_n)}^2 \\
		\expeLigne{\normLigne{X\nn-x\st}^2} &\leq \expeLigne{\normLigne{X_n-x\st}^2}+\gamma^2(\eta+\norm{\nabla f}_{\infty})(n+1)^{-2\alpha} \eqsp .
	\end{align}
        And we obtain the same equation as in \eqref{eq:D2_discrete_main}, with
        a different constant before the asymptotic term
        $(n+1)^{-2\alpha}$. Hence the conclusions of
        Lemma~\ref{lemma:D2_discrete} still hold, because \rref{assum:f_lip} is
        never used in the remaining of the proof. We can now safely apply
        Proposition~\ref{prop:shamir_discrete_intermediate} (since
        \rref{assum:f_lip} is only used to use Lemma~\ref{lemma:D2_discrete})
        with $\alpha\st=0$. This concludes the proof.
\end{proof}


%% file: appendix_fz.tex
\section{Convex case (under \Cref{assum:grad_sto}-\ref{item:ml})}
\label{app:fz_convex}

In this section, we prove similar results to the ones of
\Cref{sec:convex-case_appendix} under \Cref{assum:grad_sto}-\ref{item:ml}.  In
\Cref{sec:equiv-cref_cont} we prove the equivalent to \Cref{app:shamir_cont} in
this setting (in particular we recover the optimal rate in the convex setting
under \Cref{assum:grad_sto}-\ref{item:ml} for continuous SGD). Similarly, in
\Cref{sec:equiv-cref_disc} we prove the equivalent to \Cref{app:shamir_discrete} in
this setting (in particular we recover the optimal rate in the convex setting
under \Cref{assum:grad_sto}-\ref{item:ml} for SGD).

\subsection{Equivalent to \Cref{app:shamir_cont}}
\label{sec:equiv-cref_cont}

First, we start with \Cref{lemma:D2_fz} which is an equivalent of
\Cref{lemma:D2}. The discussion conducted at the begin of \Cref{app:shamir_cont}
is still valid here. However, similarly to the discrete-case under
\Cref{assum:grad_sto}-\ref{item:approx_sto} we have to rely on some
bootstrapping technique to conclude. The equivalent to
\Cref{lemma:control_derivative} is given in \Cref{lemma:control_derivative}.
The intermediate result needed to apply our bootstrapping procedure is stated in
\Cref{prop:shamir_cont_int_fz}. \Cref{lemma:s1_sT_fz}, \Cref{lemma:st-fstar_fz}
and \Cref{lemma:s0-s1_fz} are the counterparts to \Cref{lemma:s1_sT},
\Cref{lemma:st-fstar} and \Cref{lemma:s0-s1} respectively. We state and prove
our main result in \Cref{thm:shamir_continuous_fz}.

\begin{lemma_colt}
	\label{lemma:D2_fz}
        Assume \rref{assum:f_lip}, \tup{\rref{assum:grad_sto}-\ref{item:ml}},
        \rref{assum:lip_sigma} and
        \tup{\rref{assum:f}-\ref{item:ftilde_conv}}. Let $(\bfX_t)_{t \geq 0}$
        be given by \eqref{eq:sde}. Then, for any
        $\alpha, \gamma \in \ooint{0,1}$, there exists $\Cconvcontun \geq0$ and
        $\Cconvcontdeux\geq0$ and a function $\Cconvcont:\bR_+ \to \bR_+$ such
        that, for any $t \geq 0$,
\begin{equation}
	\expe{\norm{\bfX_t-x\st}^2}\leq \Cconvcontun\Cconvcont(t+\gua)+\Cconvcontdeux  \eqsp .
\end{equation}
And we have
\begin{equation}
	\Cconvcont(t)=
	\begin{cases}
		t^{1-2\alpha}&\mbox{if }\alpha<1/2 \eqsp ,\\
		\log(t)&\mbox{if }\alpha=1/2 \eqsp ,\\
		0&\mbox{if }\alpha>1/2 \eqsp .
	\end{cases}
\end{equation}
The values of the constants are given by
\begin{align}
	\Cconvcontun&=
		\begin{cases}
			\gua\trcst(1-2\alpha)\pinv&\mbox{if }\alpha<1/2 \eqsp, \\
			\gua\trcst&\mbox{if }\alpha=1/2 \eqsp ,\\
			0&\mbox{if }\alpha>1/2 \eqsp .
		\end{cases}\\
	\Cconvcontdeux&=
	\begin{cases}
		2 \max_{t \leq (\gua\trcst)^{1/\alpha}} \expe{\norm{X_t-x\st}^2}&\mbox{if }\alpha<1/2 \eqsp, \\
		2 \max_{t \leq (\gua\trcst)^{1/\alpha}} \expe{\norm{X_t-x\st}^2}-\gua\trcst\log(\gua)&\mbox{if }\alpha=1/2 \eqsp ,\\
		2 \max_{t \leq (\gua\trcst)^{1/\alpha}} \expe{\norm{X_t-x\st}^2}+(2\alpha-1)\pinv\gua^{2-2\alpha}\trcst&\mbox{if }\alpha>1/2 \eqsp ,
	\end{cases}
\end{align}
\end{lemma_colt}

\begin{proof}
  Let $\alpha, \gamma \in \ooint{0,1}$ and $t \geq 0$. Let $(\bfX_t)_{t \geq 0}$
  be given by \eqref{eq:sde}. We consider the function
  $F:\bR \times \bRd \to \bR_+$ given for any $(t,x)\in\bR\times\bRd$ by
  $F(t,x)=\norm{x-x\st}^2$.  Applying Lemma~\ref{lemma:dynkin} to the stochastic
  process $(F(t,\bfX_t))_{t\geq 0}$ and using \Cref{lemma:trace_f} and
  \tup{\rref{assum:f}-\ref{item:ftilde_conv}} gives that for all
  $t\geq (\gua\trcst)^{1/\alpha}$,
	\begin{align}
		&\expe{\norm{\bfX_t-x\st}^2}-\expe{\norm{\bfX_0-x\st}^2} \\ &\qquad =-2\int_0^T (t+\gua)^{-\alpha}\expe{\la \bfX_t-x\st, \nabla f(\bfX_t)\ra} \dd t+ \int_0^T \gua(t+\gua)^{-2\alpha}\expe{\trace(\Sigma(\bfX_t))} \dd t\\
		&\qquad\leq -2\int_0^T (t+\gua)^{-\alpha}\expe{\la \bfX_t-x\st, \nabla f(\bfX_t)\ra} \dd t \\
          & \qquad \qquad + \int_0^T \gua(t+\gua)^{-2\alpha}\trcst\parentheseDeux{1+\expe{f(\bfX_t)-f(x\st)}} \dd t \\
		&\qquad \leq -2\int_0^T (t+\gua)^{-\alpha}\expe{\la \bfX_t-x\st, \nabla f(\bfX_t)\ra} \dd t \\
          &\qquad \qquad + \int_0^T \gua(t+\gua)^{-2\alpha}\trcst\parentheseDeux{1+\expe{\la \nabla f(\bfX_t), \bfX_t - x\st \ra}} \dd t \\
		&\qquad\leq \gua\trcst\int_0^T (t+\gua)^{-2\alpha} \dd t + \int_0^T (t+\gua)^{-\alpha} \expe{\la \nabla f(\bfX_t), \bfX_t - x\st \ra} \defEns{-2 + \gua\trcst(t+\gua)^{-\alpha}}  \dd t \\
		&\qquad \leq \gua\trcst\int_0^T (t+\gua)^{-2\alpha} \dd t \eqsp .
	\end{align}

	We now distinguish three cases:
	\begin{enumerate}[label= (\alph*),  wide, labelwidth=!, labelindent=0pt]
		\item If $\alpha < 1/2$, then we have 
			\begin{align}
				\expe{\norm{\bfX_t-x\st}^2} &\leq \norm{X_0-x\st}^2 +\gua\trcst(1-2\alpha)\pinv((T+\gua)^{1-2\alpha}-\gua^{1-2\alpha}) \\
				&\leq \norm{X_0-x\st}^2+\gua\trcst(1-2\alpha)\pinv(T+\gua)^{1-2\alpha} \eqsp .
			\end{align}
		\item If $\alpha = 1/2$, then we have 
			\begin{align}
				\expe{\norm{\bfX_t-x\st}^2} &\leq \norm{X_0-x\st}^2 + \gua\trcst(\log(T+\gua)-\log(\gua)) \\
				&\leq \gua \trcst \log(T+\gua)+ \norm{X_0-x\st}^2-\gua\trcst\log(\gua) \eqsp .
			\end{align}
		\item If $\alpha > 1/2$, then we have
			\begin{align}
				\expe{\norm{\bfX_t-x\st}^2} &\leq \norm{X_0-x\st}^2 +\gua\trcst(1-2\alpha)\pinv((T+\gua)^{1-2\alpha}-\gua^{1-2\alpha}) \\
				&\leq \norm{X_0-x\st}^2 + (2\alpha-1)\pinv \gua^{2-2\alpha}\trcst \eqsp .
			\end{align}
	\end{enumerate}

\end{proof}


\begin{proposition_colt}
	\label{prop:shamir_cont_int_fz}
	Let $\gamma, \alpha \in \ooint{0,1}$ and $x_0 \in \rset^{\dim}$
        $(\bfX_t)_{t \geq 0}$ be given by \eqref{eq:sde}. Assume
        \rref{assum:f_lip}, \tup{\rref{assum:grad_sto}-\ref{item:ml}},
        \tup{\rref{assum:f}-\ref{item:ftilde_conv}} and \rref{assum:lip_sigma}.
        In addition, assume that there exists $\alpha\st\in\ccint{0,1/2}$,
        $\beta>0$ and $\Cshamdisc\geq 0$ such that for all $t \in [0,T]$
        \begin{equation}
        	\label{eq:assum_shamir_cont_int_fz}
    				\expe{\trace \Sigma(\bfX_t)} \leq
    				\begin{cases}
    					\Cshamdisc (t+\gua)^{\beta}\log(t+\gua)&\mbox{if }\alpha \leq \alpha\st \eqsp ,\\
    					\Cshamdisc &\mbox{if }\alpha>\alpha\st \eqsp .
    				\end{cases}
        \end{equation}

        Then there exists $\Cshamt \geq 0$
        such that, for all $T\geq 1$,
		\begin{equation}
		\expe{f(\bfX_T)}-f\st \leq \Cshamt\parentheseDeux{\log(T+\gua)^2(T+\gua)^{-\min(\alpha,1-\alpha)}}(1+\Psial(T+\gua)) \eqsp ,
	\end{equation}
	with
	\begin{equation}
 		\Psial(t)=
 		\begin{cases}
 			t^{\beta}&\mbox{if }\alpha \leq \alpha\st \eqsp , \\
 			0 &\mbox{if }\alpha>\alpha\st \eqsp .
 		\end{cases}
 \end{equation}
\end{proposition_colt}

\begin{proof}
  Let $f\in \rmc^2(\R^{\dim},\R)$.  Let $\gamma \in \ooint{0,1}$ and
  $\alpha \in \ocint{0,1/2}$ and $T \geq 1$. Let $(\bfX_t)_{t \geq 0}$ be given
  by \eqref{eq:sde}. Combining \Cref{lemma:s1_sT_fz}, \Cref{lemma:st-fstar_fz}
  and \Cref{lemma:s0-s1_fz} gives the desired result
\begin{equation}
	\expe{f(\bfX_T)}-f\st \leq \Ccontz\parentheseDeux{\log(T+\gua)^2(T+\gua)^{-\min(\alpha,1-\alpha)}}(1+\Psial(T+\gua)) \eqsp ,
\end{equation}
with $\Ccontz=(4/\gua^2)(4\Csham + 2\Cshamz+2\Csham\Lip)$.
        \end{proof}

Let $S: \ \ccint{0,T} \to \coint{0,+\infty}$ defined by
\begin{equation}
  \label{eq:S_def_fz}
  \left\lbrace
  \begin{aligned}
    &S(t) = \textstyle{t^{-1} \int_{T-t}^T \defEns{\expe{f(\bfX_s)} -f^{\star}} \rmd s \eqsp , \qquad \text{if } t > 0 \eqsp ,} \\
    &S(0) = \expe{f(\bfX_T)} \eqsp .
  \end{aligned}
  \right.
\end{equation}
With this notation we have
\begin{equation}
  \expe{f(\bfX_T)} - f^{\star} = S(0) - S(1) + S(1) - S(T) + S(T) - f^{\star} \eqsp .
\end{equation}
We are now going to control each one of the three terms $(S(0) - S(1))$,
$(S(1) - S(T))$, $(S(T) - f^\star)$ as follows:
\begin{enumerate}[wide, labelwidth=!, labelindent=0pt, label=(\alph*)]
\item Case $S(1) - S(T)$ (\Cref{lemma:s1_sT_fz}): we adapt the idea of suffix
  averaging of \cite{shamirzhang} to the continuous-time setting. In particular,
  we control the time-derivative of $S$ in \Cref{lemma:control_derivative_fz}
  (counterpart of \Cref{lemma:control_derivative}). Note that
  \Cref{lemma:s1_sT_fz} is the counterpart to \Cref{lemma:s1_sT}.
\item Case $S(T) - f^\star$ (\Cref{lemma:st-fstar_fz}): this result is known and
  corresponds to the optimal convergence rate of the averaged sequence towards
  the minimum of $f$. We provide its proof for completeness. Note that
  \Cref{lemma:st-fstar_fz} is the counterpart to \Cref{lemma:st-fstar}.
\item Case $S(0) - S(1)$ (\Cref{lemma:s0-s1_fz}): this last term is specific to
  the continuous-time setting and is a necessary modification to the classic
  averaging control of $S(\vareps) - S(T)$, established in \Cref{lemma:s1_sT_fz}
  for $\vareps = 1$, which diverges for $\vareps$ close to $0$. Note that
  \Cref{lemma:s0-s1_fz} is the counterpart to \Cref{lemma:s0-s1}.
\end{enumerate}
Before controlling each one of these terms we state the following useful lemma,
which will allow us to control the derivative of $S$.

\begin{lemma_colt}
  \label{lemma:control_derivative_fz}
  Assume \rref{assum:f_lip}, \tup{\rref{assum:grad_sto}-\ref{item:ml}},
  \rref{assum:lip_sigma}, and \tup{\rref{assum:f}-\ref{item:ftilde_conv}}. In
  addition, assume that \eqref{eq:assum_shamir_cont_int_fz} holds. Then, for any
  $\alpha, \gamma \in \ooint{0,1}$, $T \geq 0$, $u \in \ccint{0,T}$ and $Y$ any
  $\rset^d$-valued random variable such that
  $\expeLigne{\normLigne{Y -x^\star}^2} \leq \Cconvcontun \Cconvcont
  (T+\gua)+\Cconvcontdeux$ with $\Cconvcontun$ and $\Cconvcontdeux$ given in
  \Cref{lemma:D2_fz}, we have
\begin{align}
  \label{eq:shamir_master_cons_fz}
  &\int_{T-u}^T \expe{f(\bfX_t)-f(\yo)}\dd t\\
  & \qquad \leq \Csham \parenthese{(T+\gua)^{\alpha}-(T-u+\gua)^{\alpha}}\\
		&\qquad \qquad+(1/2) (T-u+\gua)^{\alpha} \expe{\norm{\bfX_{T-u}-\yo}^2}\\
		&\qquad \qquad+\Csham\log(T+\gua)\defEns{(T+\gua)^{\alpha}-(T-u+\gua)^{\alpha}}(T+\gua)^{1-2\alpha}\\
		&\qquad \qquad +\Csham\parenthese{(T+\gua)^{1-\alpha}-(T-u+\gua)^{1-\alpha}}\log(T+\gua)(1+\Psial(T+\gua)) \eqsp ,
\end{align}
with
$\Csham = \max(4\Cconvcontdeux/2,(\gua \Cshamdisc/2+\Cconvcontun)
(1-\alpha)\pinv)$ and
$\Psial(t)=0 \mbox{ if }\alpha>\alpha\st\mbox{ and }t^{\beta}\mbox{ if }\alpha
\leq \alpha\st$, with $\Cconvcontun$ and $\Cconvcontdeux$ given in
\Cref{lemma:D2_fz}.
\end{lemma_colt}

\begin{proof}For any $\yo \in \rset^{\dim}$ we define the function
  $F_{\yoo}:\bR_+ \times \bR^{\dim} \to \bR$ by
	\begin{equation}
		F_{\yoo}(t,x)=(t+\gua)^{\alpha}\norm{x- \yoo}^2 \eqsp .
	\end{equation}
        Using Lemma~\ref{lemma:D2}, that $\Cconvcont$ is non-decreasing and that
        for any $a,b \geq 0$, $(a+b)^2 \leq 2a^2 + 2b^2$, we have
	\begin{align}
		\expe{\norm{\bfX_t-\yo}^2}&=\expe{\norm{(\bfX_t-x\st)+(x\st-\yo)}^2}\\
		&\leq 2\expe{\norm{\bfX_t-x\st}^2}+2\expe{\norm{\yo-x\st}^2}\\
                                          &\leq 2\Cconvcontun \Cconvcont
                                            (t+\gua)+4\Cconvcontdeux+2\Cconvcontun \Cconvcont(T+\gua) \\
          &\leq 2\Cconvcontun \Cconvcont
                                            (t+\gua) + 2\Cconvcontun \Cconvcont(T+\gua) + \Cconvconttrois \eqsp .
	\end{align}
	with $\Cconvconttrois = 4\Cconvcontdeux$. This gives in particular, for every $t\in\ccint{0,T}$,
	\begin{align}
		\label{eq:shamir_bound_esp_fz}
		(t+\gua)^{\alpha-1}\expe{\norm{\bfX_t-\yo}^2} &\leq \parentheseDeux{\Cconvconttrois+2\Cconvcontun(T+\gua)^{1-2\alpha}\log(T+\gua)}(t+\gua)^{\alpha-1}\\
		&\quad+2\Cconvcontun\log(T+\gua)(t+\gua)^{-\alpha} \eqsp ,
	\end{align}
with $\Cconvcontun=0$ if $\alpha>1/2$. Notice that the additional $\log(T+\gua)$ term is only needed in the case where $\alpha=1/2$. For any $(t,x)\in \bR_+ \times \bRd$, we have
	\begin{align}
		&\partial_t F_{\yoo}(t,x)=\alpha(t+\gua)^{\alpha-1} \norm{x-\yoo}^2 \eqsp, \\
		&\partial_x F_{\yoo}(t,x)=2(t+\gua)^{\alpha} (x-\yoo) \eqsp, \quad
		\partial_{xx} F_{\yoo}(t,x)=2(t+\gua)^{\alpha} \eqsp .
	\end{align}
	Using Lemma~\ref{lemma:dynkin} on the stochastic process
        $(F_{\yo}(t,\bfX_t))_{t \geq 0}$, we have that for any $u\in\ccint{0,T}$
\begin{align}
	\label{eq:shamir_master_fz}
	\begin{split}
		\expe{F_{\yo}(T,\bfX_T)}-\expe{F_{\yo}(T-u,\bfX_{T-u})}&=\int_{T-u}^T\alpha(t+\gua)^{\alpha-1}\expe{\norm{\bfX_t-\yo}^2}\dd t\\
		&\quad -2\int_{T-u}^T\expe{\la \bfX_t-\yo, \nabla f(\bfX_t) \ra} \dd t \\
		&\quad+ \int_{T-u}^T\gua(t+\gua)^{-\alpha}\expe{\trace (\Sigma(\bfX_t))}\dd t \eqsp .
              \end{split}
\end{align}

We distinguish now several cases depending on the value of $\alpha$.

\begin{enumerate}[label= (\alph*),  wide, labelwidth=!, labelindent=0pt]
\item If $\alpha \leq \alpha^\star$, then combining
  \tup{\rref{assum:f}-\ref{item:ftilde_conv}},
  \rref{assum:grad_sto}-\ref{item:ml}, \eqref{eq:assum_shamir_cont_int_fz},
  \eqref{eq:shamir_bound_esp_fz} and~\eqref{eq:shamir_master_fz} we obtain for
  any $u \in \ccint{0,T}$
\begin{align}
		&-(T-u+\gua)^{\alpha}\expe{\norm{\bfX_{T-u}-\yo}^2} \\
		&\quad \leq \Cconvconttrois \int_{T-u}^T\alpha(t+\gua)^{\alpha-1} \dd t + \Cshamdisc \gua \int_{T-u}^T (t+\gua)^{-\alpha}(t+\gua)^{\beta}\log(t+\gua) \dd t\\
		&\qquad  + 2\alpha\Cconvcontun\log(T+\gua)\defEns{\int_{T-u}^T(t+\gua)^{-\alpha} \dd t+(T+\gua)^{1-2\alpha} \int_{T-u}^T(t+\gua)^{\alpha-1} \dd t} \\
		&\qquad- 2\int_{T-u}^T \expe{f(\bfX_t)-f(\yo)}\dd t \\
		&\quad \leq \Cconvconttrois\parenthese{(T+\gua)^{\alpha}-(T-u+\gua)^{\alpha}}-2\int_{T-u}^T \expe{f(\bfX_t)-f(\yo)}\dd t\\
		&\qquad+(\gua \Cshamdisc(T+\gua)^{\beta}+2\alpha\Cconvcontun) (1-\alpha)\pinv \parenthese{(T+\gua)^{1-\alpha}-(T-u+\gua)^{1-\alpha}}\log(T+\gua) \\
		&\qquad +2\Cconvcontun\log(T+\gua)\defEns{(T+\gua)^{\alpha}-(T-u+\gua)^{\alpha}}(T+\gua)^{1-2\alpha} \eqsp .
\end{align}
\item If $\alpha > \alpha^\star$, then combining
  \tup{\rref{assum:f}-\ref{item:ftilde_conv}},
  \rref{assum:grad_sto}-\ref{item:ml}, \eqref{eq:assum_shamir_cont_int_fz}
  \eqref{eq:shamir_bound_esp_fz} and~\eqref{eq:shamir_master_fz} we obtain for
  any $u \in \ccint{0,T}$
\begin{align}
		&-(T-u+\gua)^{\alpha}\expe{\norm{\bfX_{T-u}-\yo}^2} \\
		&\quad \leq \Cconvconttrois \int_{T-u}^T\alpha(t+\gua)^{\alpha-1} \dd t + \Cshamdisc\gua \int_{T-u}^T (t+\gua)^{-\alpha} \dd t\\
		&\qquad  + 2\alpha\Cconvcontun\log(T+\gua)\defEns{\int_{T-u}^T(t+\gua)^{-\alpha} \dd t+(T+\gua)^{1-2\alpha} \int_{T-u}^T(t+\gua)^{\alpha-1} \dd t} \\
		&\qquad- 2\int_{T-u}^T \expe{f(\bfX_t)-f(\yo)}\dd t \\
		&\quad \leq \Cconvconttrois\parenthese{(T+\gua)^{\alpha}-(T-u+\gua)^{\alpha}}-2\int_{T-u}^T \expe{f(\bfX_t)-f(\yo)}\dd t\\
		&\qquad+(\gua \Cshamdisc+2\alpha\Cconvcontun) (1-\alpha)\pinv \parenthese{(T+\gua)^{1-\alpha}-(T-u+\gua)^{1-\alpha}}\log(T+\gua) \\
		&\qquad +2\Cconvcontun\log(T+\gua)\defEns{(T+\gua)^{\alpha}-(T-u+\gua)^{\alpha}}(T+\gua)^{1-2\alpha} \eqsp .
\end{align}
\end{enumerate}
Putting this together we get for any $u \in \ccint{0, T}$
\begin{align}
  \label{eq:shamir_master_cons_fz}
  &\int_{T-u}^T \expe{f(\bfX_t)-f(\yo)}\dd t\\
  & \qquad \leq \Csham \parenthese{(T+\gua)^{\alpha}-(T-u+\gua)^{\alpha}}\\
		&\qquad \qquad+(1/2) (T-u+\gua)^{\alpha} \expe{\norm{\bfX_{T-u}-\yo}^2}\\
		&\qquad \qquad+\Csham\log(T+\gua)\defEns{(T+\gua)^{\alpha}-(T-u+\gua)^{\alpha}}(T+\gua)^{1-2\alpha}\\
		&\qquad \qquad +\Csham\parenthese{(T+\gua)^{1-\alpha}-(T-u+\gua)^{1-\alpha}}\log(T+\gua)(1+\Psial(T+\gua)) \eqsp ,
\end{align}
with $\Csham = \max(\Cconvconttrois/2,(\gua \Cshamdisc/2+\Cconvcontun) (1-\alpha)\pinv)$ and $\Psial(t)=0 \mbox{ if }\alpha>\alpha\st\mbox{ and }t^{\beta}\mbox{ if }\alpha \leq \alpha\st$.
\end{proof}
We divide the rest of the proof into three parts, to bound the quantities $S(1)-S(T)$, $S(T)-f\st$ and $S(0)-S(1)$.

\begin{lemma_colt}
  \label{lemma:s1_sT_fz}
  Assume \rref{assum:f_lip}, \tup{\rref{assum:grad_sto}-\ref{item:ml}},
  \rref{assum:lip_sigma}, and \tup{\rref{assum:f}-\ref{item:ftilde_conv}}. In
  addition, assume that \eqref{eq:assum_shamir_cont_int_fz} holds. Then, for any
  $\alpha, \gamma \in \ooint{0,1}$ and $T \geq 0$ we have
      \begin{equation}
\label{eq:shamir_derivative_fz}
  S(1)-S(T) \leq 2\Cshamz\log(T+\gua) \log(1 + T)(T+\gua)^{-\min(\alpha,1-\alpha)}(1+\Psial(T+\gua)) \eqsp ,
\end{equation}
  with $S$ given in \eqref{eq:S_def_fz}.
  \end{lemma_colt}

  \begin{proof}
In the case where $\alpha \leq 1/2$, \Cref{lemma:alpha} gives that for all $u \in \ccint{0,T}$:
\begin{equation}
	\parenthese{(T+\gua)^{\alpha}-(T-u+\gua)^{\alpha}} \leq \parenthese{(T+\gua)^{1-\alpha}-(T-u+\gua)^{1-\alpha}} \eqsp ,
\end{equation}
and we also have, for all $u \in \ccint{0,T}$:
\begin{align}
  &(T+\gua)^{1-\alpha}-(T+\gua-u)^{1-\alpha}\\
  &\quad= \parenthese{((T+\gua)^{1-\alpha}-(T+\gua-u)^{1-\alpha})((T+\gua)^{\alpha} +(T+\gua-u)^{\alpha})} \\ & \qquad \times \parenthese{(T+\gua)^{\alpha}+(T+\gua-u)^{-\alpha}}^{-1}  \\
  &\quad\leq \parenthese{(T+\gua)-(T+\gua-u)+(T+\gua)^{1-\alpha}(T+\gua-u)^{\alpha}  -(T+\gua)^{\alpha}(T+\gua-u)^{1-\alpha}} \\ & \qquad \times (T+\gua)^{-\alpha} \leq 2u/(T+\gua)^{\alpha} \eqsp . 	\label{eq:shamir_qte_conj_fz}
\end{align}
And in the case where $\alpha > 1/2$, for all $u \in \ccint{0,T}$:
\begin{equation}
	\parenthese{(T+\gua)^{1-\alpha}-(T-u+\gua)^{1-\alpha}} \leq \parenthese{(T+\gua)^{\alpha}-(T-u+\gua)^{\alpha}} \eqsp ,
\end{equation}
and we also have, for all $u \in \ccint{0,T}$:
\begin{align}
	\label{eq:shamir_qte_conj_2_fz}
	&(T+\gua)^{\alpha}-(T+\gua-u)^{\alpha}\\
	&\quad= \parenthese{((T+\gua)^{\alpha}-(T+\gua-u)^{\alpha})((T+\gua)^{1-\alpha}+(T+\gua-u)^{1-\alpha})} \\
  & \qquad \parenthese{(T+\gua)^{1-\alpha}+(T+\gua-u)^{1-\alpha}}^{-1} \\
	&\quad\leq \parenthese{(T+\gua)-(T+\gua-u)+(T+\gua)^{\alpha}(T+\gua-u)^{1-\alpha}-(T+\gua)^{1-\alpha}(T+\gua-u)^{\alpha}} \\ & \qquad \times (T+\gua)^{-1+\alpha} \leq 2u/(T+\gua)^{1-\alpha} \eqsp .
\end{align}
Now, plugging $\yo = \bfX_{T-u}$ in \Cref{lemma:control_derivative_fz} we obtain, for all $u \in \ccint{0,T}$:
\begin{equation}
	\label{eq:shamir_int_fz}
	\expe{\int_{T-u}^T f(\bfX_t)-f(\bfX_{T-u}) \dd t} \leq 2\Cshamz\log(T+\gua) (T+\gua)^{-\min(\alpha,1-\alpha)}(1+\Psial(T+\gua))u \eqsp ,
\end{equation}
with $\Cshamz=(2\Csham+\Csham(1+\gua^{1-2\alpha}))$.

Since $S$ is a differentiable function and using \eqref{eq:shamir_int_fz}, we have for all $u \in (0,T)$,
\begin{equation}
	S'(u)=-u^{-2}\int_{T-u}^T \expe{f(\bfX_t)}\dd t  + u\pinv \expe{f(\bfX_{T-u})}
  =-u\pinv(S(u)-\expe{f(\bfX_{T-u})}) \eqsp .
\end{equation}
This last result implies $-S'(u)\leq 2\Cshamz\log(T+\gua) /(T+\gua)^{\min(\alpha,1-\alpha)} (1+\Psial(T+\gua)) u\pinv$ and integrating we get
\begin{equation}
\label{eq:shamir_derivative_fz}
  S(1)-S(T) \leq 2\Cshamz\log(T+\gua) \log(T)(T+\gua)^{-\min(\alpha,1-\alpha)}(1+\Psial(T+\gua)) \eqsp .
\end{equation}
\end{proof}

\begin{lemma_colt}
  \label{lemma:st-fstar_fz}
  Assume \rref{assum:f_lip}, \tup{\rref{assum:grad_sto}-\ref{item:ml}},
  \rref{assum:lip_sigma}, and \tup{\rref{assum:f}-\ref{item:ftilde_conv}}. In
  addition, assume that \eqref{eq:assum_shamir_cont_int_fz} holds. Then, for any
  $\alpha, \gamma \in \ooint{0,1}$ and $T \geq 0$ we have
  \begin{equation}
  S(T)-f\st \leq 4\Csham T^{-\min(\alpha,1-\alpha)}(1+\log(T+\gua))(1+\Psial(T+\gua)) \eqsp ,
   \label{eq:shamir_ST_fz}
\end{equation}  
  with $S$ given in \eqref{eq:S_def_fz}.
  \end{lemma_colt}

\begin{proof}
  Using \Cref{lemma:control_derivative_fz} with $u=T$ and $\yo=x\st$, and
  $\normLigne{\bfX_0-x\st} \leq \Csham$ we obtain
\begin{align}
		\int_0^T \expe{f(X_s)}\dd s - Tf\st &\leq (\Csham/2)\parenthese{(T+\gua)^\alpha - \gua^{\alpha}} +(1/2)\gua^{\alpha} \expe{\norm{\bfX_0-x\st}^2} \\
		& \quad +(\Csham/2)\parentheseDeux{(T+\gua)^{\alpha}-\gua^{\alpha}}(T+\gua)^{1-2\alpha}\log(T+\gua) \\ 
		& \quad + (\Csham/2) \log(T+\gua) \parentheseDeux{(T+\gua)^{1-\alpha}-\gua^{1-\alpha}}(1+\Psial(T+\gua)) \\
		&\leq (\Csham/2)(T+\gua)^\alpha +(\Csham/2)\gua^{\alpha} \\
		& \quad+(\Csham/2)(T+\gua)^{1-\alpha}\log(T+\gua) \\ 
		& \quad+ (\Csham/2) \log(T+\gua) (T+\gua)^{1-\alpha}(1+\Psial(T+\gua))
		\eqsp .
\end{align}
Using this result we have
\begin{align}
  S(T)-f\st &\leq T^{-1}2\Csham(T+\gua)^{\max(1-\alpha, \alpha)}(1+\log(T+\gua))(1+\Psial(T+\gua)) + 2\Csham \gua^{\alpha} T^{-1}  /2 \\
   &\leq 4\Csham T^{-\min(\alpha,1-\alpha)}(1+\log(T+\gua))(1+\Psial(T+\gua)) \eqsp .
   \label{eq:shamir_ST_fz}
\end{align}
\end{proof}

\begin{lemma_colt}
  \label{lemma:s0-s1_fz}
  Assume \rref{assum:f_lip}, \tup{\rref{assum:grad_sto}-\ref{item:ml}},
  \rref{assum:lip_sigma}, and \tup{\rref{assum:f}-\ref{item:ftilde_conv}}. In
  addition, assume that \eqref{eq:assum_shamir_cont_int_fz} holds. Then, for any
  $\alpha, \gamma \in \ooint{0,1}$ and $T \geq 0$ we have
\begin{equation}
	\label{eq:shamir_second_fz}
	S(0)-S(1) \leq  \Csham\Lip \log(T+\gua)(1+\Psial(T+\gua)) (T-1)^{-2\alpha} \eqsp ,
\end{equation}  
  with $S$ given in \eqref{eq:S_def_fz}.
  \end{lemma_colt}

  \begin{proof}
We have
\begin{align}
	\label{eq:shamir_S1_fz}
	S(0)-S(1)&=\expe{f(\bfX_T)}-S(1)=\int_{T-1}^T \parenthese{\expe{f(\bfX_T)}-\expe{f(\bfX_s)}} \dd s \eqsp .
\end{align}
Using Lemma~\ref{lemma:dynkin} on the stochastic process $f(\bfX_t)_{t\geq 0}$, \rref{assum:f_lip} and \eqref{eq:assum_shamir_cont_int_fz}, we have for all $s\in[T-1,T]$
\begin{align}
  \expe{f(\bfX_T)}-\expe{f(\bfX_s)} &= -\int_s^T (\gua + t)^{-\alpha} \expeLigne{\normLigne{\nabla f(\bfX_t)}^2}\dd t \\
  & \qquad + (\Lip/2)\gua \int_s^T (t+\gua)^{-2\alpha}\expe{\trace({\Sigma(\bfX_t)})}\dd t \\
	& \leq (\Lip/2)\gua \Cshamdisc \log(T+\gua)(1+\Psial(T+\gua)) \int_{s}^T (t+\gua)^{-2\alpha}\dd t \\
	& \leq \Csham\Lip \log(T+\gua)(1+\Psial(T+\gua)) (s+\gua)^{-2\alpha}(T-s)\eqsp .
\end{align}
Plugging this result into \eqref{eq:shamir_S1_fz} yields
\begin{align}
	\label{eq:shamir_second_fz}
	S(0)-S(1)&\leq \Csham\Lip \log(T+\gua)(1+\Psial(T+\gua))\int_{T-1}^T (T-s)(s+\gua)^{-2\alpha} \dd s \\
	&\leq \Csham\Lip \log(T+\gua)(1+\Psial(T+\gua))(T-1+\gua)^{-2\alpha} \\
	&\leq  \Csham\Lip \log(T+\gua)(1+\Psial(T+\gua)) (T-1)^{-2\alpha} \eqsp .
\end{align}
\end{proof}

\begin{theorem}
	\label{thm:shamir_continuous_fz}
        Let $\alpha, \gamma \in \ooint{0,1}$ and $(\bfX_t)_{t \geq 0}$ be given by \eqref{eq:sde}. Assume $f \in \rmc^2(\rset^{\dim}, \rset)$, \rref{assum:f_lip}, \tup{\rref{assum:grad_sto}-\ref{item:ml}},
         \rref{assum:lip_sigma} and \tup{\rref{assum:f}-\ref{item:ftilde_conv}}. Then, there exists $C \geq 0$ (explicit and given in the proof) such that for any $T \geq 1$
  \begin{equation}
    \expe{f(\bfX_T)}-{ \textstyle\min_{\rset^d}} f
    \leq C(1+\log(T))^2/T^{\alpha\wedge(1-\alpha)}
		\eqsp .
  \end{equation}
\end{theorem}

\begin{proof}
  We begin by proving by induction over $m \in \bN^*$ that the following
  assertion \Cref{assum:recu_conv_fz}($m$) is true.
  \begin{assumptionH}[$m$]
    \label{assum:recu_conv_fz}
    For any $\alpha > 1/(m+1)$, there exists $\Cshamaplus>0$ such that for all
    $t \geq 0,\ \expe{\trace(\Sigma(X_t))} \leq \Cshamaplus$. In
    addition, for any $\alpha \leq 1/(m+1)$, there exists $\Cshamamoins>0$ such
    that for all
    $t \geq 0 ,\ \expe{\trace(\Sigma(X_t))} \leq \Cshamamoins
    (t+\gua)^{1-(m+1)\alpha}\log(t+\gua)^2$.
\end{assumptionH}

For $m=1$, \Cref{assum:recu_conv_fz}($1$) is an immediate consequence of
\rref{assum:f_lip} and Lemma~\ref{lemma:D2_fz}.  Now, let $m\in \bN^*$ and
suppose that \Cref{assum:recu_conv_fz}($m$) holds.  Let
$\alpha \in \ooint{0,1}$. Setting $\alpha\st=1/(m+1)$ we see that
\eqref{eq:assum_shamir_cont_int_fz} is verified with $\beta=1-(m+1)\alpha$.
Consequently, using \rref{assum:f_lip},
\tup{\rref{assum:f}-\ref{item:ftilde_conv}}, \rref{assum:grad_sto}-\ref{item:ml}
we can apply Proposition~\ref{prop:shamir_cont_int_fz} which shows that, for
$\alpha\leq 1/(m+1)$, there exists $ \Ccontz >0$ such that for all $T\geq 1$,
\begin{align}
	\expe{f(\bfX_T)}-f\st &\leq \Ccontz \defEns{ \log(T+\gua))^2 /(T+\gua)^{\min(\alpha,1-\alpha)}\Psial(T+\gua)} \\
	&\leq \Ccontz \defEns{  \log(T+\gua))^2 (T+\gua)^{-\alpha}(T+\gua)^{1-(m+1)\alpha}} \\
	&\leq \Ccontz \defEns{ \log(T+\gua))^2  (T+\gua)^{1-(m+2)\alpha}}\eqsp . 	\label{eq:shamir_rec_fz}
\end{align}
In particular, if $\alpha >1/(m+2)$ we have the existence of $\bar{\Cconv}_{\alpha}>0$ such that for all $n \in \defEns{0,\cdots, N}$,
$\expe{f(X_n)}-f\st \leq \bar{\Cconv}_{\alpha}$. And using \rref{assum:f_lip} and Lemma~\ref{lemma:trace_f} we get that, for all $t \in \ccint{0,T}$,
\begin{equation}
\expe{\trace(\Sigma(\bfX_t))} \leq \trcst(1+\expe{f(\bfX_t)-f\st}) \leq \trcst (1 + \bar{\Cconv}_{\alpha}),
\end{equation}
Combining this result with \eqref{eq:shamir_rec_fz}, we get that
\Cref{assum:recu_conv_fz}($m+1$) holds with
$\Cshamaplus=\trcst (1 + \bar{\Cconv}_{\alpha})$ and $\Cshamamoins=\Ccontz$. We
conclude by recursion,

Now, let $\alpha \in \ooint{0,1}$. Since $\bR$ is archimedean, there exists
$m \in \bN^*$ such that $\alpha>1/(m+1)$ and therefore
\Cref{assum:recu_conv_fz}($m$) shows the existence of $\Cshamdisc>0$ such that
$\expe{\trace(\Sigma(\bfX_t))} \leq \Cshamdisc$ for all $t \in \ccint{0,T}$.
Applying Proposition~\ref{prop:shamir_cont_int_fz} gives the existence of
$\Ccont >0$ such that for all $T\geq 1$
\begin{equation}
	\expe{f(\bfX_T)}-f\st \leq \Ccont (1+\log(T))^2/T^{\min(\alpha, 1-\alpha)} \eqsp ,
\end{equation}
concluding the proof.
\end{proof}

\subsection{Equivalent to \Cref{app:shamir_discrete}}
\label{sec:equiv-cref_disc}

First, we start with \Cref{lemma:D2_discrete_fz} which is an equivalent to
\Cref{lemma:D2_discrete}. The discussion conducted at the begin of
\Cref{app:shamir_discrete} is still valid here (with changes in the
bootstrapping used). \Cref{prop:shamir_discrete_intermediate_fz},
\Cref{lemma:control_derivative_discrete_fz}, \Cref{lemma:s0-st-discrete_fz} and
\Cref{lemma:st-fstar-discrete_fz} are the counterparts of
\Cref{prop:shamir_discrete_intermediate},
\Cref{lemma:control_derivative_discrete}, \Cref{lemma:s0-st-discrete} and
\Cref{lemma:st-fstar-discrete} respectively. Finally, our main result is stated
and proven in \Cref{thm:shamir_discrete_fz}.

\begin{lemma_colt}
	\label{lemma:D2_discrete_fz}
	Assume \rref{assum:f_lip}, \tup{\rref{assum:f}-\ref{item:ftilde_conv}},
        \tup{\rref{assum:grad_sto}-\ref{item:ml}}. Then for any
        $\alpha, \gamma \in \ooint{0,1}$, there exists $\Cconvdiscun\geq 0$,
        $\Cconvdiscdeux\geq 0$ and a function $\Cconvdisc:\bR_+ \to \bR_+$ such
        that, for any $n \geq 0$,
\begin{equation}
	\expe{\norm{X_n-x\st}^2}\leq \Cconvdiscun \Cconvdisc(n+1)+\Cconvdiscdeux \eqsp .
\end{equation}
And we have
\begin{equation}
	\Cconvdisc(t)=
	\begin{cases}
		t^{1-2\alpha}&\mbox{if }\alpha<1/2 \eqsp ,\\
		\log(t)&\mbox{if }\alpha=1/2 \eqsp ,\\
		0&\mbox{if }\alpha>1/2 \eqsp .
	\end{cases}
\end{equation}
The values of the constants are given by
\begin{align}
	\Cconvdiscun&=
	\begin{cases}
		2\gamma^2\eta(1-2\alpha)\pinv&\mbox{if }\alpha<1/2 \eqsp, \\
		\gamma^2\eta&\mbox{if }\alpha=1/2 \eqsp ,\\
		0&\mbox{if }\alpha>1/2 \eqsp .
	\end{cases}\\
	\Cconvdiscdeux&=
	\begin{cases}
		2\max_{k \leq (\gamma \Lip/2)^{1/\alpha}} \expe{\norm{X_k-x\st}^2}&\mbox{if }\alpha<1/2 \eqsp, \\
		2\max_{k \leq (\gamma \Lip/2)^{1/\alpha}} \expe{\norm{X_k-x\st}^2}+2\gamma^2\eta&\mbox{if }\alpha=1/2 \eqsp ,\\
		2\max_{k \leq (\gamma \Lip/2)^{1/\alpha}} \expe{\norm{X_k-x\st}^2}+\gamma^2\eta(2\alpha-1)\pinv&\mbox{if }\alpha>1/2 \eqsp ,
	\end{cases}
\end{align}
\end{lemma_colt}

\begin{proof}
	Let $f:\bRd \to \bR$ verifying assumptions \rref{assum:f_lip} and \tup{\rref{assum:f}-\ref{item:ftilde_conv}}. We consider $(X_n)_{n \geq 0}$ satisfying \eqref{eq:sgd}. Let $x\st \in \bRd$. We have, using \eqref{eq:sgd}, \Cref{lemma:gradfz} and \tup{\rref{assum:f}-\ref{item:ftilde_conv}} that for all $n\geq (\gamma \trcst)^{1/\alpha}$,
	\begin{align}
		&\expecLigne{\normLigne{X_{n+1}-x\st}^2}{\cF_n}=\expecLigne{\normLigne{X_n-x\st-\gamma(n+1)^{-\alpha}\nabla \f(X_n,Z_{n+1})}^2}{\cF_n}\\
		& \qquad =\norm{X_n-x\st}^2-2\gamma/(n+1)^{\alpha}\la X_n-x\st, \expecLigne{\nabla \f(X_n,Z_{n+1})}{\cF_n}\ra \\
		& \qquad \quad +\gamma^2(n+1)^{-2\alpha}\expecLigne{\normLigne{\nabla \f(X_n,Z_{n+1})}^2}{\cF_n} \\
			& \qquad \leq \norm{X_n-x\st}^2-2\gamma/(n+1)^{\alpha}\la X_n-x\st, \nabla f(X_n)\ra\\
			& \qquad \quad+\trcst\gamma^2(n+1)^{-2\alpha}\parenthese{f(X_n)-f(x\st)+1} \\
			& \qquad \leq \norm{X_n-x\st}^2-2\gamma/(n+1)^{\alpha}\la X_n-x\st, \nabla f(X_n)\ra\\
			& \qquad \quad+\trcst\gamma^2(n+1)^{-2\alpha}\la \nabla f(X_n), X_n-x\st \ra +\trcst\gamma^2(n+1)^{-2\alpha}  \\
		& \qquad \leq\norm{X_n-x\st}^2+\gamma/(n+1)^{\alpha}\la \nabla f(X_n), X_n-x\st \ra\parentheseDeux{\trcst\gamma/(n+1)^{\alpha}-2}+\gamma^2\trcst(n+1)^{-2\alpha} \\
		& \qquad \leq\norm{X_n-x\st}^2+\gamma^2\trcst(n+1)^{-2\alpha} \\
		&\expeLigne{\norm{X_{n+1}-x\st}^2} \leq \expeLigne{\norm{X_n-x\st}^2}+\gamma^2\trcst(n+1)^{-2\alpha} \eqsp .	\label{eq:D2_discrete_main_fz}
	\end{align}
Summing the previous inequality leads to
\begin{equation}
	\expe{\norm{X_n-x\st}^2}-\expe{\norm{X_0-x\st}^2}\leq \gamma^2\trcst \sum_{k=1}^n k^{-2\alpha} \eqsp .
\end{equation}
As in the previous proof we now distinguish three cases:

\begin{enumerate}[label= (\alph*),  wide, labelwidth=!, labelindent=0pt]
		\item If $\alpha < 1/2$, we have 
			\begin{align}
				\expe{\norm{X_n-x\st}^2}&\leq \norm{X_0-x\st}^2+\gamma^2\trcst(1-2\alpha)\pinv(n+1)^{1-2\alpha}\\
				& \leq \norm{X_0-x\st}^2+2\gamma^2\trcst(1-2\alpha)\pinv n^{1-2\alpha} \eqsp .
			\end{align}
		\item If $\alpha = 1/2$, we have $\expeLigne{\normLigne{X_n-x\st}^2}\leq \norm{X_0-x\st}^2+\gamma^2\trcst(\log(n)+2)$.
		\item If $\alpha > 1/2$, we have $\expeLigne{\normLigne{X_n-x\st}^2}\leq \norm{X_0-x\st}^2+\gamma^2\trcst(2\alpha-1)\pinv$.
	\end{enumerate}
\end{proof}


In order to prove the theorem we will need an intermediate proposition.
\begin{proposition_colt}
	\label{prop:shamir_discrete_intermediate_fz}
	Let $\gamma, \alpha \in \ooint{0,1}$ and $x_0 \in \rset^{\dim}$ and
        $(X_n)_{n\geq 0}$ be given by \eqref{eq:sgd}. Assume \rref{assum:f_lip},
        \tup{\rref{assum:f}-\ref{item:ftilde_conv}}, \rref{assum:grad_sto}-\ref{item:ml}.  In addition,
        assume that there exists $\alpha\st\in\ccint{0,1/2}$, $\beta>0$ and
        $\Cshamdisc\geq 0$ such that for all $n \in \defEns{0,\cdots, N}$
        \begin{equation}
        	\label{eq:assum_shamir_int_fz}
    				\expeLigne{\normLigne{\nabla \f(X_n, Z)}^2} \leq
    				\begin{cases}
    					\Cshamdisc (n+1)^{\beta}\log(n+1)&\mbox{if }\alpha \leq \alpha\st \eqsp ,\\
    					\Cshamdisc &\mbox{if }\alpha>\alpha\st \eqsp .
    				\end{cases}
        \end{equation}

        Then there exists $\Cshamt \geq 0$
        such that, for all $N\geq 1$,
		\begin{equation}
			\expe{f(X_N)}-f\st \leq \Cshamt \defEns{ (1+\log(N+1))^2 /(N+1)^{\min(\alpha,1-\alpha)}\Psial(N+1) +1/(N+1)} \eqsp ,
	\end{equation}
	with
	\begin{equation}
 		\Psial(n)=
 		\begin{cases}
 			n^{\beta}&\mbox{if }\alpha \leq \alpha\st \eqsp , \\
 			1 &\mbox{if }\alpha>\alpha\st \eqsp .
 		\end{cases}
 \end{equation}
\end{proposition_colt}

\begin{proof}
  Let $\alpha, \gamma \in \ooint{0,1}$ and $N \geq 1$. Let $(X_n)_{n \geq 0}$ be
  given by \eqref{eq:sgd}.  And finally, combining
  \eqref{eq:shamir_discrete_a_fz} and~\eqref{eq:shamir_discrete_b_fz} together
  gives, for $\Cshamt \doteq 2\max((2\gamma)\pinv\norm{X_0-x\st}^2,2\Cdisc)$,

  \end{proof}

  Let $(S_k)_{k \in \defEns{0,\cdots,N}}$ defined for any
  $k \in \{0, \dots, N\}$ by
  \begin{equation}
    \label{eq:S_def_disc_fz}
	 S_k = (k+1)\pinv \sum_{t=N-k}^N \expe{f(X_t)} \eqsp .
\end{equation}
      Note that $\expeLigne{f(X_N)} - f^\star = (S_N - S_0) + (S_0 - f^\star)$.
      We are now going to control each one of the two terms $(S_0 - S_N)$ and
      $(S_N - f^\star)$ as follows:
\begin{enumerate}[wide, labelwidth=!, labelindent=0pt, label=(\alph*)]
\item Case $S_n -S_0$ (\Cref{lemma:s0-st-discrete_fz}): this is an adaption of
  the idea of suffix averaging of \cite{shamirzhang} to our setting (one of
  crucial difference lies into the control of the sequence
  $(\expeLigne{\nabla f(X_n)}^2)_{n \in \nset}$ which is assumed to be uniformly
  bounded in \cite{shamirzhang}). In particular, we control the (discrete)
  time-derivative of $S$ in \Cref{lemma:control_derivative_discrete_fz}
  (counterpart to \Cref{lemma:control_derivative_discrete}). Note that
  \Cref{lemma:s0-st-discrete_fz} is the counterpart to
  \Cref{lemma:s0-st-discrete}.
\item Case $S(T) - f^\star$ (\Cref{lemma:st-fstar-discrete_fz}): this result is
  known and corresponds to the optimal convergence rate of the averaged sequence
  towards the minimum of $f$. We provide its proof for completeness. Note that
  \Cref{lemma:st-fstar-discrete_fz} is the counterpart to
  \Cref{lemma:st-fstar-discrete}.
\end{enumerate}
Before controlling each one of these terms we state the following useful lemma,
which will allow us to control the derivative of $S$.

\begin{lemma_colt}
  \label{lemma:control_derivative_discrete_fz}
  Assume \rref{assum:f_lip}, \tup{\rref{assum:grad_sto}-\ref{item:ml}},
  and \tup{\rref{assum:f}-\ref{item:ftilde_conv}}. In addition, assume that
  \eqref{eq:assum_shamir_int_fz} holds. Then, for any
  $\alpha, \gamma \in \ooint{0,1}$, $N \in \nset$, $u \in \{0, \dots, N\}$ and
  $Y$ any $\rset^d$-valued random variable such that
  $\expeLigne{\normLigne{Y -x^\star}^2} \leq
  \Cconvdiscun(N+1)^{1-2\alpha}\log(N+1)+4\Cconvdiscdeux$ with $\Cconvdiscun$
  and $\Cconvdiscdeux$ given in \Cref{lemma:D2_discrete_fz}, we have
\begin{align}
	\expe{\sum_{k=N-u}^N f(X_k)-f(Y)} & \leq 2\Cdisc (u+1)/(N+1)^{\min(\alpha, 1-\alpha)} (1+\log(N+1))\Psial(N+1) \\
                                                   &\quad +(2\gamma)\pinv (N-u+1)^{\alpha} \expeLigne{\normLigne{X_{N-u}-\yo}^2} \eqsp ,
                                                     	\label{eq:shamir_discrete_master_fz}
\end{align}
with
 \begin{equation}
 		\Psial(n)=
 		\begin{cases}
 			n^{\beta}&\mbox{if }\alpha \leq \alpha\st \eqsp , \\
 			1 &\mbox{if }\alpha>\alpha\st \eqsp ,
 		\end{cases}
              \end{equation}
              and $\Cdisc\doteq 4 (\Cconvddeux\Cshamdisc)(1-\alpha)\pinv + (2\gamma)\pinv( \Cconvdiscdeux+4\Cconvdiscun)$.
\end{lemma_colt}

\begin{proof}Let $\ell \in \defEns{0,\cdots,N}$,
let $k \geq \ell$, let $\yo \in \cF_{\ell}$. Using \tup{\rref{assum:f}-\ref{item:ftilde_conv}} we have
\begin{align}
	\expek{\norm{X_{k+1}-\yo}^2}&=\expek{\norm{X_k-\yo-\gamma (k+1)^{-\alpha} \nabla\f(X_k, Z\kk)}^2}\\
	&=\norm{X_k-\yo}^2+\gamma^2(k+1)^{-2\alpha}\expek{\norm{\nabla\f(X_k,Z_{k+1})}^2}\\
	&\quad-2\gamma(k+1)^{-\alpha}\la X_k-\yo, \nabla f(X_k)\ra\\
		\expe{f(X_k)-f(\yo)} &\leq (2\gamma)\pinv(k+1)^{\alpha}\parenthese{\expe{\norm{X_k-\yo}^2}-\expe{\norm{X_{k+1}-\yo}^2}}\\
		&\quad+(\gamma/2)(k+1)^{-\alpha}\expe{\expek{\norm{\nabla\f(X_k,Z_{k+1})}^2}}\\
		\expe{f(X_k)-f(\yo)} &\leq (2\gamma)\pinv(k+1)^{\alpha}\parenthese{\expe{\norm{X_k-\yo}^2}-\expe{\norm{X_{k+1}-\yo}^2}}\\
		&\quad+(\gamma/2)(k+1)^{-\alpha}\expe{\norm{\nabla\f(X_k,Z_{k+1})}^2} \eqsp . 	\label{eq:shamir_discrete_before_sum_fz}
\end{align}
Let $u \in \defEns{0,\cdots,N}$. Summing now \eqref{eq:shamir_discrete_before_sum_fz} between $k=N-u$ and $k=N$ gives
\begin{align}	
		\expe{\sum_{k=N-u}^N f(X_k)-f(\yo)} &\leq (2\gamma)\pinv \sum_{k=N-u+1}^N \expe{\norm{X_k-\yo}^2}\parenthese{(k+1)^{\alpha}-k^{\alpha}} \\
	&\quad+\Cconvddeux  \sum_{k=N-u}^N \expe{\norm{\nabla \f(X_k,Z_{k+1})}^2} (k+1)^{-\alpha}\\
	&\quad+(2\gamma)\pinv (N-u+1)^{\alpha} \expe{\norm{X_{N-u}-\yo}^2}\eqsp . \label{eq:shamir_discrete_sum_fz}
\end{align}
In the following we will take for $\yo$ either $x\st$ or $X_m$ for $m \in \ccint{0,N}$.
We now have to run separate analyses depending on the value of $\alpha$.

\begin{enumerate}[label= (\alph*),  wide, labelwidth=!, labelindent=0pt]
	\item If $\alpha \leq \alpha^\star$, then \eqref{eq:assum_shamir_int_fz} gives that
		\begin{equation}
			\expe{\norm{\nabla f(X_k, Z_{k+1})}^2} \leq \Cshamdisc (N+1)^{\beta}\log(N+1),
		\end{equation}
		and Lemma~\ref{lemma:D2_discrete_fz} gives that for all $k \in \defEns{0,\dots,N}$,
		\begin{align}
			\expe{\norm{X_k-\yo}^2}&\leq 2\expe{\norm{X_k-x\st}^2}+2\expe{\norm{\yo-x\st}^2} \\
			&\leq 2\Cconvdiscun (k+1)^{1-2\alpha}\log(k+1)+2\Cconvdiscun(N+1)^{1-2\alpha}\log(N+1)+4\Cconvdiscdeux \\
			&\leq 4\Cconvdiscun (N+1)^{1-2\alpha}\log(N+1)+4\Cconvdiscdeux \eqsp .
		\end{align}
		We define $\Cconvdisctrois \doteq 4\Cconvdiscdeux$.  Using
                \eqref{eq:shamir_discrete_sum_fz} with
                $\Cconvdtrois=(\Cconvddeux\Cshamdisc)(1-\alpha)\pinv$, we get
		\begin{align}
				&\expe{\sum_{k=N-u}^N f(X_k)-f(\yo)} \leq  (2\gamma)\pinv (N-u+1)^{\alpha} \expe{\norm{X_{N-u}-\yo}^2} \\
				&\qquad+ (2\gamma)\pinv\parenthese{\Cconvdisctrois+4\Cconvdiscun(N+1)^{1-2\alpha}\log(N+1)}\parenthese{(N+1)^{\alpha}-(N-u+1)^{\alpha}} \\
				&\qquad + \Cconvddeux\Cshamdisc(N+1)^{\beta}\log(N+1) (1-\alpha)\pinv \parenthese{(N+1)^{1-\alpha}-(N-u)^{1-\alpha}} \\
				&\quad \leq  \Cconvdtrois (N+1)^{\beta}\log(N+1) \parenthese{(N+1)^{1-\alpha}-(N-u)^{1-\alpha}} \\
				&\qquad +(2\gamma)\pinv (N-u+1)^{\alpha} \expe{\norm{X_{N-u}-\yo}^2} \\
				&\qquad+ (2\gamma)\pinv\Cconvdisctrois\parenthese{(N+1)^{\alpha}-(N-u)^{\alpha}}\\
				&\qquad+(2\gamma)\pinv 4\Cconvdiscun\parenthese{(N+1)^{1-\alpha}-(N-u)^{1-\alpha}}\log(N+1) \\
			& \quad\leq \Cdisc (N+1)^{\beta}(1+\log(N+1))\parenthese{(N+1)^{1-\alpha}-(N-u)^{1-\alpha}} \\
			&\qquad +(2\gamma)\pinv (N-u+1)^{\alpha} \expe{\norm{X_{N-u}-\yo}^2} \eqsp , 			\label{eq:shamir_discrete_leq_fz}
		\end{align}
		where we used Lemma~\ref{lemma:alpha}. Similarly to
                \eqref{eq:shamir_qte_conj_fz} we have
	\begin{align}
    \label{eq:alpha_conj_fz}
		&(N+1)^{1-\alpha}-(N-u)^{1-\alpha} \\
		&\qquad= \defEns{\parenthese{(N+1)^{1-\alpha}-(N-u)^{1-\alpha}} \parenthese{(N+1)^{\alpha}+(N-u)^{\alpha}}}\parenthese{(N+1)^{\alpha}+(N-u)^{\alpha}}\pinv \\
		&\qquad\leq 2(u+1)/(N+1)^{\alpha} \eqsp .
	\end{align}

      \item If $\alpha \in \ocint{\alpha^\star, 1/2}$, then
        Lemma~\ref{lemma:D2_discrete_fz} gives that for all
        $k \in \defEns{0,\dots,N}$,
		\begin{align}
			\expe{\norm{X_k-\yo}^2}&\leq 2\expe{\norm{X_k-x\st}^2}+2\expe{\norm{\yo-x\st}^2} \\
			&\leq 2\Cconvdiscun (k+1)^{1-2\alpha}\log(k+1)+2\Cconvdiscun(N+1)^{1-2\alpha}\log(N+1)+4\Cconvdiscdeux \\
			&\leq 4\Cconvdiscun (N+1)^{1-2\alpha}\log(N+1)+4\Cconvdiscdeux \eqsp .
		\end{align}
		Combining \eqref{eq:assum_shamir_int_fz} and \eqref{eq:shamir_discrete_sum_fz} we have
		\begin{align}
			\label{eq:shamir_discrete_eq_fz}
				&\expe{\sum_{k=N-u}^N f(X_k)-f(\yo)} \leq (2\gamma)\pinv (N-u+1)^{\alpha} \expe{\norm{X_{N-u}-\yo}^2} \\
				&\quad+ (2\gamma)\pinv\parenthese{\Cconvdisctrois+4\Cconvdiscun\log(N+1)(N+1)^{1-2\alpha}}\parenthese{(N+1)^{\alpha}-(N-u+1)^{\alpha}} \\
				&\quad + \Cconvddeux\Cshamdisc(1-\alpha)\pinv \parenthese{(N+1)^{1-\alpha}-(N-u)^{1-\alpha}} \\
				&\leq  \Cconvdtrois \parenthese{(N+1)^{1-\alpha}-(N-u)^{1-\alpha}} +(2\gamma)\pinv (N-u+1)^{\alpha} \expe{\norm{X_{N-u}-\yo}^2} \\
				&\quad+ (2\gamma)\pinv\parenthese{\Cconvdisctrois+4\Cconvdiscun}(1+\log(N+1))\parenthese{(N+1)^{\alpha}-(N-u)^{\alpha}} \\
			&\leq\Cdisc (1+ \log(N+1)) \parenthese{(N+1)^{1-\alpha}-(N-u)^{1-\alpha}} \\
			&\quad+(2\gamma)\pinv (N-u+1)^{\alpha} \expe{\norm{X_{N-u}-\yo}^2} \eqsp .
		\end{align}
              \item If $\alpha > 1/2$, then $\alpha>\alpha\st$ and
                Lemma~\ref{lemma:D2_discrete_fz} gives
		\begin{equation}
			\forall k \in \defEns{0,\dots,N}, \, \expe{\norm{X_k-\yo}^2}\leq 2\expe{\norm{X_k-x\st}^2}+2\expe{\norm{\yo-x\st}^2} \leq 4\Cconvdiscdeux = \Cconvdisctrois \eqsp .
		\end{equation}
		Using \Cref{lemma:alpha}, \eqref{eq:assum_shamir_int_fz} and
                \eqref{eq:shamir_discrete_sum_fz} we have
		\begin{align}
			\label{eq:shamir_discrete_geq_fz}
			\begin{split}
				&\expe{\sum_{k=N-u}^N f(X_k)-f(\yo)} \leq  (\gamma\Cshamdisc/2) (1-\alpha)\pinv \parenthese{(N+1)^{1-\alpha}-(N-u)^{1-\alpha}} \\
				&\quad+(2\gamma)\pinv (N-u+1)^{\alpha} \expe{\norm{X_{N-u}-\yo}^2} \\
				&\quad+ (2\gamma)\pinv\Cconvdisctrois\parenthese{(N+1)^{\alpha}-(N-u+1)^{\alpha}}
			\end{split} \\
			\begin{split}
				&\leq  \Cconvdtrois \parenthese{(N+1)^{1-\alpha}-(N-u)^{1-\alpha}} +(2\gamma)\pinv (N-u+1)^{\alpha} \expe{\norm{X_{N-u}-\yo}^2} \\
				&\quad+ (2\gamma)\pinv\Cconvdisctrois\parenthese{(N+1)^{\alpha}-(N-u)^{\alpha}}
			\end{split} \\
			&\leq\Cdisc \parenthese{(N+1)^{\alpha}-(N-u)^{\alpha}} +(2\gamma)\pinv (N-u+1)^{\alpha} \expe{\norm{X_{N-u}-\yo}^2} \eqsp .
		\end{align}

	Similarly to \eqref{eq:shamir_qte_conj_fz} we have
	\begin{align}
		(N+1)^{\alpha}-(N-u)^{\alpha} &= \defEns{\parenthese{(N+1)^{\alpha}-(N-u)^{\alpha}}  \parenthese{(N+1)^{1-\alpha}+(N-u)^{1-\alpha}}}\\ & \qquad \quad \times \parenthese{(N+1)^{1-\alpha}+(N-u)^{1-\alpha}}\pinv \\
		&\leq 2(u+1)/(N+1)^{1-\alpha} \eqsp .
	\end{align}
\end{enumerate}
Finally, putting the three cases above together we obtain
\begin{align}
	\expe{\sum_{k=N-u}^N f(X_k)-f(Y)} & \leq 2\Cdisc (u+1)/(N+1)^{\min(\alpha, 1-\alpha)} (1+\log(N+1))\Psial(N+1) \\
                                                   &\quad +(2\gamma)\pinv (N-u+1)^{\alpha} \expe{\norm{X_{N-u}-\yo}^2} \eqsp ,
                                                     	\label{eq:shamir_discrete_master_fz}
\end{align}
with
 \begin{equation}
 		\Psial(n)=
 		\begin{cases}
 			n^{\beta}&\mbox{if }\alpha \leq \alpha\st \eqsp , \\
 			1 &\mbox{if }\alpha>\alpha\st \eqsp .
 		\end{cases}
 \end{equation}
 Note that the additional $\log(N+1)$ factor can be removed if $\alpha \neq 1/2$.
\end{proof}

  \begin{lemma_colt}
  \label{lemma:s0-st-discrete_fz}
  Assume \rref{assum:f_lip}, \tup{\rref{assum:grad_sto}-\ref{item:ml}}
  and \tup{\rref{assum:f}-\ref{item:ftilde_conv}}. In addition, assume that
  \eqref{eq:assum_shamir_int_fz} holds. Then, for any
  $\alpha, \gamma \in \ooint{0,1}$ and $N \in \nset$ we have
  \begin{equation}
    S_0-S_N \leq 2\Cdisc (N+1)^{-\min(\alpha, 1-\alpha)} (1+\log(N+1))^2 \Psial(N+1) \eqsp .
  \end{equation}
  with $S$ given in \eqref{eq:S_def_disc_fz}.  
\end{lemma_colt}
\begin{proof}
  Let $u \in \defEns{0,\dots, N}$. Using
  \Cref{lemma:control_derivative_discrete_fz} with the choice $\yo=X_{N-u}$
  gives
\begin{align}
	\label{eq:shamir_discrete_sum2_fz}
	\expe{\sum_{k=N-u}^N f(X_k)-f(X_{N-u})}\leq 2\Cdisc (u+1)/(N+1)^{\min(\alpha, 1-\alpha)} (1+\log(N+1)) \Psial(N+1) \eqsp .
\end{align}

		And then,
		\begin{align}			
			S_u&=(u+1)\pinv \sum_{k=N-u}^N \expe{f(X_k)} \\
			&\leq 2\Cdisc (N+1)^{-\min(\alpha, 1-\alpha)} (1+\log(N+1))\Psial(N+1) + \expe{f(X_{N-u})} \eqsp . \label{eq:shamir_discrete_Su_fz}
		\end{align}
		We have now, using \eqref{eq:shamir_discrete_Su_fz},
		\begin{align}
			uS_{u-1}&=(u+1)S_u-\expe{f(X_{N-u})} \\
			&=uS_u+S_u-\expe{f(X_{N-u})}\\
			&\leq uS_u+2\Cdisc (N+1)^{-\min(\alpha, 1-\alpha)} (1+\log(N+1)) \Psial(N+1) \\
			S_{u-1}-S_u &\leq  2\Cdisc u\pinv (N+1)^{-\min(\alpha, 1-\alpha)} \log(N+1) \\
			S_0-S_N &\leq 2\Cdisc (N+1)^{-\min(\alpha, 1-\alpha)} (1+\log(N+1)) \Psial(N+1) \sum_{u=1}^N(1/u)  \\
                  S_0-S_N &\leq 2\Cdisc (N+1)^{-\min(\alpha, 1-\alpha)} (1+\log(N+1))^2 \Psial(N+1) \eqsp .
                            			\label{eq:shamir_discrete_a_fz}
		\end{align}
                \end{proof}

\begin{lemma_colt}
  \label{lemma:st-fstar-discrete_fz}
  Assume \rref{assum:f_lip}, \tup{\rref{assum:grad_sto}-\ref{item:ml}}
  and \tup{\rref{assum:f}-\ref{item:ftilde_conv}}. In addition, assume that
  \eqref{eq:assum_shamir_int_fz} holds. Then, for any
  $\alpha, \gamma \in \ooint{0,1}$ and $N \in \nset$ we have
          \begin{align}
		S_N-f\st&\leq 2\Cdisc (1+\log(N+1))^2 (N+1)^{-\min(\alpha,1-\alpha)}\Psial(N+1)\\
		&\quad+(2\gamma)\pinv(N+1)\pinv\norm{X_0-x\st}^2 \eqsp .		\label{eq:shamir_discrete_b_fz}
	\end{align}
  with $S$ given in \eqref{eq:S_def_disc_fz}.
  \end{lemma_colt}
                
\begin{proof}
  Using \Cref{lemma:control_derivative_discrete_fz} with the choice $\yo=x\st$
  and $u=N$ gives
	\begin{align}
		(N+1)\pinv\expe{\sum_{k=0}^N f(X_k)-f(x\st)} &\leq 2\Cdisc (1+\log(N+1)) (N+1)^{-\min(\alpha,1-\alpha)}\Psial(N+1)\\
                                                             &\quad+(2\gamma)\pinv(N+1)\pinv\norm{X_0-x\st}^2
        \end{align}
        Therefore,
        \begin{align}
		S_N-f\st&\leq 2\Cdisc (1+\log(N+1))^2 (N+1)^{-\min(\alpha,1-\alpha)}\Psial(N+1)\\
		&\quad+(2\gamma)\pinv(N+1)\pinv\norm{X_0-x\st}^2 \eqsp .		\label{eq:shamir_discrete_b_fz}
	\end{align}
\end{proof}

\begin{theorem}
	\label{thm:shamir_discrete_fz}
	Let $\gamma, \alpha \in \ooint{0,1}$ and $(X_n)_{n\geq 0}$ be given by
        \eqref{eq:sgd}. Assume \rref{assum:f_lip},
        \tup{\rref{assum:grad_sto}-\ref{item:ml}} and \tup{\rref{assum:f}-\ref{item:ftilde_conv}}. Then, there
        exists $C \geq 0$ (explicit and given in the proof) such that for any
        $N \geq 1$,
		\begin{equation}
       \expe{f(X_N)}-{ \textstyle\min_{\rset^d}} f \leq C (1+\log(N+1))^2/(N+1)^{\alpha\wedge(1-\alpha)} \eqsp.
     \end{equation}
\end{theorem}

\begin{proof}
    We begin by proving by induction over $m \in \bN^*$ that the following
  assertion \Cref{assum:recu_conv_disc_fz}($m$) is true.
  \begin{assumptionH}[$m$]
    \label{assum:recu_conv_disc_fz}
  For any $\alpha > 1/(m+1)$, there exists $\Cshamaplus>0$ such that for all
  $n\in \nset,\ \expeLigne{\normLigne{\nabla \f(X_n, Z)}^2} \leq
  \Cshamaplus$. In addition, for any $\alpha \leq 1/(m+1)$, there exists
  $\Cshamamoins>0$ such that for all
  $n\in \nset,\ \expeLigne{\normLigne{\nabla \f(X_n, Z)}^2} \leq \Cshamamoins
  n^{1-(m+1)\alpha}(1+\log(n))^2$.
\end{assumptionH}

For $m=1$, \Cref{assum:recu_conv_disc_fz}($1$) is an immediate consequence of
\rref{assum:f_lip} and Lemma~\ref{lemma:D2_discrete_fz}, with
$\Cshamaplus=\Lip^2\Cconvdiscdeux$ and
$\Cshamamoins=\Lip^2\max(\Cconvdiscun, \Cconvdiscdeux)$.  Now, let $m\in \bN^*$
and assume that \Cref{assum:recu_conv_disc_fz}($m$) holds.  Let
$\alpha \in \ooint{0,1}$. Setting $\alpha\st=1/m+1$ we see that
\eqref{eq:assum_shamir_int_fz} is verified with $\beta=1-(m+1)\alpha$.
Consequently, using \rref{assum:f_lip},
\tup{\rref{assum:f}-\ref{item:ftilde_conv}}, \rref{assum:grad_sto}-\ref{item:ml}
we can apply Proposition~\ref{prop:shamir_discrete_intermediate_fz} which shows
that, for $\alpha\leq 1/(m+1)$
\begin{align}
	\expe{f(X_N)}-f\st &\leq \Cshamt \defEns{ (1+\log(N+1))^2 /(N+1)^{\min(\alpha,1-\alpha)}\Psial(N+1) +1/(N+1)} \\
	&\leq \Cshamt \defEns{ (1+\log(N+1))^2 (N+1)^{-\alpha}(N+1)^{1-(m+1)\alpha} +1/(N+1)} \\
	&\leq \Cshamt \defEns{ (1+\log(N+1))^2 (N+1)^{1-(m+2)\alpha} +1/(N+1)}\eqsp . 	\label{eq:shamir_rec_disc_fz}
\end{align}
In particular, if $\alpha >1/(m+2)$ we have the existence of $\bar{\Cconv}_{\alpha}>0$ such that for all $n \in \defEns{0,\cdots, N}$,
$\expe{f(X_n)}-f\st \leq \bar{\Cconv}_{\alpha}$. And using \rref{assum:f_lip} and Lemma~\ref{lemma:gradfz} we get that, for all $n \in \defEns{0,\cdots, N}$
\begin{equation}
\expeLigne{\normLigne{\nabla \f(X_n, Z)}^2} \leq \trcst(1+ \expe{f(X_n)-f\st}) \leq \trcst (1+\bar{\Cconv}_{\alpha}) \eqsp ,
\end{equation}
Combining this result with \eqref{eq:shamir_rec_disc_fz}, we get that
\Cref{assum:recu_conv_disc_fz}($m+1$) holds with
$\Cshamaplus=\trcst (1+\bar{\Cconv}_{\alpha})$ and $\Cshamamoins=2\Cshamt$.
Finally this proves that \Cref{assum:recu_conv_disc_fz}($m$) is true for any $n \geq 1$ by induction

Now, let $\alpha \in \ooint{0,1}$. Since $\bR$ is archimedean, there exists
$m \in \bN^*$ such that $\alpha>1/(m+1)$ and therefore
\Cref{assum:recu_conv_disc_fz}($m$) shows the existence of $\Cshamdisc>0$ such
that $\expeLigne{\normLigne{\nabla \f(X_n, Z)}^2} \leq \Cshamdisc$ for all $n \in \bN^*$.
Applying Proposition~\ref{prop:shamir_discrete_intermediate_fz} gives the
existence of $\Cdisc >0$ such that for all $N\geq 1$
\begin{equation}
	\expe{f(X_N)}-f\st \leq \Cdisc (1+\log(N+1))^2/(N+1)^{\min(\alpha, 1-\alpha)} \eqsp ,
\end{equation}
with $\Cdisc = 2 \Cshamt$, concluding the proof.
\end{proof}


%% file: appendix_weakly.tex

\section{Weakly Quasi-Convex Case}
\label{sec:weakly-convex_appendix}

In this section we give the proofs of the results presented in
\Cref{sec:non-convex-case}. We prove \Cref{thm:f_weak} in
\Cref{thm:f_weak:proof}. Technical lemmas are gathered in
\Cref{sec:technical-lemmas-1}. We control the norm of
$\expeLigne{\normLigne{\bfX_t - x^\star}^{2p}}^{1/p}$ in the convex framework in
\Cref{sec:control-norm-convex}. The proof of \Cref{sec:let-aplha-gamma:proof} is
presented in \Cref{sec:let-aplha-gamma:proof}. Its discrete counterpart is given
\Cref{sec:discr-count-crefc_alpha}. Finally, we conclude this section with the
proof of \Cref{thm:f_weak_discrete} in \Cref{thm:f_weak_discrete:proof}.

\subsection{Proof of \Cref{thm:f_weak}}
\label{thm:f_weak:proof}

  Without loss of generality, we assume that $f^{\star} = 0$.  Let
  $\alpha, \gamma \in \ooint{0,1}$, $x_0 \in \rset^{\dim}$, $a_t = \gua + t$,
  $\ell_t = 1 + \log(1 + \gua^{-1}t)$ for any $t \geq 0$ and
  $\delta = \min(\delta_1, \delta_2)$ with $\delta_1$ and $\delta_2$ given in
  \Cref{thm:f_weak}.  Using \Cref{lemma:dynkin}, we have for any $t \geq 0$
  \begin{align}
    &\expe{f(\bfX_t)a_t^{\delta}\ell_t^{-\varepsilon}} - f(x_0)\gua^{\delta}  \\ & \qquad = \int_0^t \left\lbrace -\ell_s^{-\varepsilon} a_s^{\delta-\alpha} \expeLigne{\norm{\nabla f(\bfX_s)}^2} \right.
     \left. + (\gua/2) \ell_s^{-\varepsilon} a_s^{\delta-2\alpha} \expe{\langle \nabla^2f(\bfX_s) , \Sigma(\bfX_s) \rangle } \right .\\
    & \qquad \qquad \left . + \delta \ell_s^{-\vareps} a_s^{\delta -1} \expe{f(\bfX_s)} -\varepsilon \ell_s^{-\varepsilon-1} a_s^{\delta} \expe{f(\bfX_s)}\right \rbrace \rmd s \eqsp .
  \end{align}
  Define for any $t \geq 0$,
  $\mce(t) = \expeLigne{f(\bfX_t)}a_t^{\delta}\ell_t^{-\vareps}$. $(t \mapsto \mce(t))$ is
  differentiable and using \Cref{assum:f_lip} and \Cref{assum:grad_sto} we have
  for any $t > 0$,
  \begin{equation}
    \dd\mce(t)/\dd t \leq -\ell_t^{-\vareps} a_t^{\delta - \alpha} \expe{\| \nabla f(\bfX_t) \|^2} + (\gua/2) \ell_t^{-\vareps} a_t^{\delta - 2\alpha} \boundLSig  + \delta a_t^{-1}\mce(t) \eqsp .
  \end{equation}
  Using, \Cref{assum:f_weak} and Hölder's inequality we have for any $t \geq 0$
  \begin{equation}
    \tau \expe{f(\bfX_t)} \leq \expe{\| \bfX_t - x^{\star}\|^{r_2 r_3}}^{r_3^{-1}} \expeLigne{\| \nabla f(\bfX_t) \|^2}^{r_1/2} \eqsp .
  \end{equation}
  Noting that $(r_3 r_1)^{-1} = r_1^{-1} - 1/2$, we get for any $t \geq 0$
  \begin{align}
    \expeLigne{\norm{\nabla f(\bfX_t)}^2} &\geq \tau^{2 r_1^{-1}} \expe{f(\bfX_t)}^{2r_1^{-1}} \expe{\| \bfX_t - x^{\star}\|^{r_2 r_3}}^{1 -2r_1^{-1}} \\
                                          &\geq \tau^{2r_1^{-1}} \Cbeta^{1-2r_1^{-1}} a_t^{\beta(1-2r_1^{-1})} \ell_t^{\explog(1-2r_1^{-1})} \expe{f(\bfX_t)}^{2r_1^{-1}} \\
    &\geq \tau^{2r_1^{-1}} \Cbeta^{1-2r_1^{-1}} a_t^{\beta(1-2r_1^{-1})- 2r_1^{-1}\delta} \ell_t^{\explog(1-2r_1^{-1}) - 2r_1^{-1}-\vareps} \mce(t)^{2r_1^{-1}} \eqsp .
  \end{align}
  Therefore, we have for any $t \geq 0$
\begin{align}
  \dd\mce(t)/\dd t &\leq - \tau^{2r_1^{-1}} \Cbeta^{1-2r_1^{-1}} a_t^{(1-2r_1^{-1})(\delta + \beta) - \alpha} \mce(t)^{2r_1^{-1}} + \gua \ell_t^{-\vareps} a_t^{\delta - 2\alpha} \Lip \eta  + \delta a_t^{-1} \mce(t)  \eqsp .
\end{align}
Let $\Cweakapp_3 = \max(\Cweakapp_1, \Cweakapp_2)$ with
\begin{equation}
  \begin{aligned}
    \Cweakapp_1 &= (|\delta|\Cbeta^{2r_1^{-1}-1}\tau^{-2r_1^{-1}}\gua^{(2r_1^{-1}-1)(\delta + \beta) +\alpha - 1})^{(2r_1^{-1}-1)^{-1}} \eqsp ,\\
    \Cweakapp_2 &= ((\Lip \eta/2) \Cbeta^{2r_1^{-1}-1}\tau^{-2r_1^{-1}}\gua^{(2r_1^{-1}-1)(\delta + \beta) + \delta - \alpha + 1})^{r_1/2} \eqsp .
  \end{aligned}
\end{equation}
If $\mce(t) \geq \Cweakapp_3$ then $\dd\mce(t)/\dd t \leq 0$. Let $\Cweakapp = \max(\Cweakapp_3, \mce(0))$, then for any $t \geq 0$, $\mce(t) \leq \Cweakapp$, which concludes the proof.

\subsection{Technical lemmas}
\label{sec:technical-lemmas-1}

\begin{lemma_colt}
  \label{lemma:f_beat_norme}
  Assume that $f$ is continuous, that
  $x^{\star} \in \argmin_{x\in\rset^{\dim}} f(x)$ and that there exist
  $c, R \geq 0$ such that for any $x \in \rset^{\dim}$ with
  $\normLigne{x- x^{\star}} \geq R$ we have
  $f(x) - f(x^{\star}) \geq c \normLigne{x - x^{\star}}$. Let $p \in \nset$, $X$
  a $d$-dimensional random variable and $\ttfun \geq 1$ such that
  $\expeLigne{(f(X) - f(x^{\star}))^{2p}} \leq \ttfun$. Then there exists
  $\ttfdeux \geq 0$ such that
  \begin{equation}
    \expe{\norm{X - x^{\star}}^{2p}} \leq \ttfdeux \ttfun \eqsp .
  \end{equation}
\end{lemma_colt}

\begin{proof}
  Since $f$ is continuous there exists $\ttamin \geq 0$ such that for any $x \in \rset^{\dim}$,
  $f(x) - f(x^{\star}) \geq c\normLigne{x - x^{\star}} - \ttamin$. Therefore, using Jensen's inequality and that $\ttfun \geq 1$ we have
  \begin{align}
    \expe{\norm{X - x^{\star}}^{2p}} &\leq c^{-2p} \sum_{k=0}^{2p} {k \choose 2p} \expe{(f(X) - f(x^{\star}))^k} \ttamin^{2p - k} \\
                                     &\leq c^{-2p}\sum_{k=0}^{2p} {k \choose 2p} \expe{(f(X) - f(x^{\star}))^{2p}}^{k/(2p)} \ttamin^{2p - k} \\
    &\leq c^{-2p} \sum_{k=0}^{2p} {k \choose 2p} \ttfun^{k/(2p)} \ttamin^{2p - k} \leq \ttfdeux \ttfun \eqsp ,
  \end{align}
  with $\ttfdeux = c^{-2p}\sum_{k=0}^{2p} {k \choose 2p} \ttamin^{2p-k}$.
\end{proof}

\begin{lemma_colt}
  \label{lemma:control_f}
  Assume \rref{assum:f_weak} with $r_1 = r_2 = 1$. Then for any $p \in \nset$ with $p \geq 2$ and $d$-dimensional random variable $X$ we have
\begin{equation}
  \label{eq:nabla_min}
   \expe{\norm{\nabla f(X)}^{2}(f(X) - f(x^{\star}))^{p-1}} \geq \expe{(f(X) - f(x^{\star}))^p}^{1+1/p} \expe{\norm{X - x^{\star}}^{2p}}^{-1/p} \eqsp ,
\end{equation}
\end{lemma_colt}

\begin{proof}
  Let $p \in \nset$ with $p\geq 2$ and let $\varpi = 2p/(p+1)$. Using \Cref{assum:f_weak}  we have for any $x \in \rset^{\dim}$
\begin{align}
  \norm{x - x^{\star}}^{\varpi} \norm{\nabla f(x)}^{\varpi} (f(x) - f(x^{\star}))^{\varpi(p-1) /2} \geq (f(x) - f(x^{\star}))^{\varpi (p+1)/2} \geq (f(x) - f(x^{\star}))^p \eqsp .
\end{align}
Let $\varsigma = 2\varpi^{-1} = 1+ p^{-1}$ and $\conj$ such that $\varsigma^{-1} + \conj^{-1} =1$. Using Hölder's inequality  the fact that $\conj \varpi = 2p$  we have
\begin{multline}
  \expe{\norm{X - x^{\star}}^{\varpi} \norm{\nabla f(X)}^{\varpi} (f(X) - f(x^{\star}))^{\varpi(p-1) /2}} \\ \leq \expe{\norm{\nabla f(X)}^{2}(f(X) - f(x^{\star}))^{p-1}}^{1/\varsigma} \expe{\norm{X - x^{\star}}^{2p}}^{1/\conj} \eqsp .
\end{multline}
Since, $\conj^{-1} = (1+p)^{-1}$ we have
\begin{equation}
   \expe{\norm{\nabla f(X)}^{2}(f(X) - f(x^{\star}))^{p-1}} \geq \expe{(f(X) - f(x^{\star}))^p}^{1+1/p} \expe{\norm{X - x^{\star}}^{2p}}^{-1/p} \eqsp ,
\end{equation}
which concludes the proof.
\end{proof}

\begin{lemma_colt}
  \label{lemma:bound_p}
  Let $\alpha, \gamma \in \ooint{0,1}$. Assume that \rref{assum:f_weak_quasi} holds then for any $p \in \nset$, there exists $\ttfpquatre \geq 0$ such that for any $t \geq 0$
  \begin{equation}
    \expe{\norm{\bfX_t - x^{\star}}^{2p}}^{1/p} \leq \ttfpquatre \defEns{1 + (\gua +t)^{1 - 2\alpha}} \eqsp .
  \end{equation}
\end{lemma_colt}

\begin{proof}
  Let $\alpha, \gamma \in \ooint{0,1}$ and $p \in \nset$. Let $\mce_{t,p} = \expe{\normLigne{\bfX_t-x^{\star}}^{2p}}$. Using
\Cref{lemma:bound_scal} and \Cref{lemma:dynkin} we have for any $t >0$
\begin{align}
  \dd\mce_{t,p}/\dd t &= - 2p(\gua + t)^{-\alpha} \expe{\langle \nabla f(\bfX_t), \bfX_t - x^{\star}\rangle \| \bfX_t - x^{\star} \|^{2(p -1)}} \\
  &\quad+ p\gua (\gua + t)^{-2\alpha} \left\lbrace \expe{\trace( \Sigma(\bfX_t)) \norm{\bfX_t - x^{\star}}^{2(p-1)}} \right. \\ & \qquad \left. + 2(p-1) \expe{\langle (\bfX_t-x^{\star})^{\top} (\bfX_t - x^{\star}), \Sigma(\bfX_t) \rangle \norm{\bfX_t - x^{\star}}^{2(p-2))}}\right\rbrace \\
  &\leq 2p\gua\eta(2p-1)(\gua + t)^{-2\alpha} \expe{\norm{\bfX_t - x^{\star}}^{2(p-1)}} \\
  &\leq p\gua\eta(2p-1)(\gua + t)^{-2\alpha} \mce_{t, (p-1)} \eqsp .   \label{eq:p_recu}
\end{align}
If $p=1$, the proposition holds and by recursion and using \eqref{eq:p_recu} we obtain the result for $p \in \nset$.
\end{proof}

\subsection{Control of the norm in the convex case}
\label{sec:control-norm-convex}

\begin{proposition_colt}
  \label{prop:recurrence}
  Let $\alpha, \gamma \in \ooint{0, 1}$. Let $m \in \ccint{0,2}$ and $\varphi > 0$ such that for any $p \in \nset$ there exists $\ttfpdeux \geq 0$  such that for any $t \geq 0$, $\expeLigne{\norm{\bfX_t - x^{\star}}^{2p}}^{1/p} \leq \ttfpun \defEnsLigne{1 + (\gua + t)^{m - \varphi \alpha}}$. Assume \rref{assum:f_lip} and \rref{assum:f_weak_quasi} and that there exist $R \geq 0$ and $c >0$ such that for any $x \in \rset^{\dim}$, with $\normLigne{x} \geq R$, $f(x) - f(x^{\star}) \geq c \norm{x - x^{\star}}$. Then, for any $p \in \nset$, there exists $\ttfpdeux \geq 0$ such that for any $t \geq 0$,
  \begin{equation}
    \expe{\norm{\bfX_t - x^{\star}}^{2p}}^{1/p}\leq \ttfpdeux \defEnsLigne{1 + (\gua + t)^{m - (1+\varphi) \alpha}} \eqsp .
  \end{equation}
\end{proposition_colt}

\begin{proof}
  If $\alpha \geq m/\varphi$ the proof is immediate since $\sup_{t \geq 0} \defEnsLigne{\expeLigne{\norm{\bfX_t - x^{\star}}^{2p}}^{1/p}} < +\infty$. Now assume that $\alpha < m/\varphi$. Let $p \in \nset$, $\delta_p = p(1+\varphi)\alpha - pm$ and $(t \mapsto \mce_{t,p})$ such that for any $t \geq 0$, $\mce_{t, p} = (f(\bfX_t) - f(x^{\star}))^{2p}(\gua + t)^{\delta_p}$ . Using \Cref{lemma:dynkin} we have for any $t > 0$
\begin{align}
  \label{eq:ito}
  \dd\mce_{t,p}/\dd t &= -2p(\gua + t)^{-\alpha + \delta_p} \expe{\norm{\nabla f(\bfX_t)}^{2}(f(\bfX_t) - f(x^{\star}))^{2p-1}} \\ & \quad + p\gua(\gua + t)^{-2\alpha+\delta_p} \left\lbrace \expe{\langle \nabla^2f(\bfX_t), \Sigma(\bfX_t) \rangle (f(\bfX_t) - f(x^{\star}))^{2p-1}} \right. \\
  & \quad + \left. (2p-1) \expe{ \langle \nabla f(\bfX_t) \nabla f(\bfX_t)^{\top}, \Sigma(\bfX_t) \rangle (f(\bfX_t) - f(x^{\star})^{2p-2})} \right\rbrace + \delta_p (\gua + t)^{-1} \mce_{t,p}\eqsp .
\end{align}
Combining \eqref{eq:ito}, \Cref{lemma:bound_scal}, \Cref{lemma:kolmo}, \Cref{lemma:control_f} and the fact that for any $t \geq 0$, $\expeLigne{\norm{\bfX_t - x^{\star}}^{4p}}^{1/(2p)} \leq \ttfpun \defEnsLigne{1 + (\gua + t)^{m - \varphi \alpha}}$ we get
\begin{align}
  \label{eq:res_final_pres}
  \dd\mce_{t,p}/\dd t  &\leq -2p(\gua + t)^{-\alpha + \delta_p} \expe{(f(\bfX_t) - f(x^{\star}))^{2p}}^{1+1/(2p)} \expe{\norm{\bfX_t - x^{\star}}^{4p}}^{-1/(2p)} \\
  & \quad + p\gua(\gua + t)^{-2\alpha+\delta_p} \left\lbrace \boundLSig \expe{(f(\bfX_t) - f(x^{\star}))^{2p-1}} \right.\\
  & \quad + \left. \Lip(2p-1)\eta^2 \expe{(f(\bfX_t) - f(x^{\star}))^{2p-1}} \right\rbrace + \delta_p (\gua + t)^{-1} \mce_{t,p}\\
  &\leq -2p(\gua + t)^{-\alpha - \delta_p/(2p)} \mce_{t,p}^{1 + 1/(2p)} \expe{\norm{\bfX_t - x^{\star}}^{2p}}^{-1/(2p)} \\
  & \quad + p \gua (d + 2p -1)\Lip \eta(1+\eta) (\gua + t)^{-2\alpha + \delta_p/(2p)} \mce_{t,p}^{1 - 1/(2p)}  + \delta_p (\gua + t)^{-1} \mce_{t,p}\\
  &\leq  -2p(\gua + t)^{-\alpha - \delta_p/(2p)} \mce_{t,p}^{1 + 1/(2p)} \ttfpun^{-1} \defEnsLigne{1 + (\gua + t)^{m - \varphi \alpha}}^{-1} \\
  & \quad + p \gua (d + 2p -1)\Lip \eta(1+\eta) (\gua + t)^{-2\alpha + \delta_p/p} \mce_{t,p}^{1 - 1/p} + \delta_p (\gua + t)^{-1} \mce_{t,p}\\
  &\leq  -2p(\gua + t)^{(\varphi-1)\alpha - \delta_p/(2p)- m} \mce_{t,p}^{1 + 1/(2p)} \ttfpun^{-1} \defEnsLigne{1 + (\gua + t)^{-m + \varphi \alpha}}^{-1}  \\
  & \quad + 2p \gua (d + 2p -1)\Lip \eta(1+\eta) (\gua + t)^{-2\alpha + \delta_p/(2p)} \mce_{t,p}^{1 - 1/(2p)} + \delta_p (\gua + t)^{-1} \mce_{t,p}\\
  &\leq  -p  \ttfpun^{-1} \defEnsLigne{1 + \gua^{-m + \varphi \alpha}}^{-1} (\gua + t)^{(\varphi-1)\alpha - \delta_p/(2p)- m} \mce_{t,p}^{1 + 1/(2p)}  \\
   & \quad + 2p \gua (d + 2p -1)\Lip \eta(1+\eta) (\gua + t)^{-2\alpha + \delta_p/(2p)} \mce_{t,p}^{1 - 1/(2p)} + \delta_p (\gua + t)^{-1} \mce_{t,p} \eqsp .
\end{align}
Since $m \in \ccint{0,2}$, we have that $1 - m + (\varphi - 1)\alpha \geq (1+ \varphi)\alpha/2 - m/2$. Hence,
\begin{equation}
  (1-\varphi)\alpha - \delta_p/(2p) - m \leq 2 \alpha + \delta_p/(2p) \eqsp , \qquad (1-\varphi)\alpha - \delta_p/(2p) - m \leq 1 \eqsp .
\end{equation}
Therefore, using \Cref{lemma:ode}, there exists $\ttfp^{(a)} \geq 1$ such that for any $t \geq 0$, $\mce_{t,p} \leq \ttfp^{(a)}$.
Hence, for any $t \geq 0$,
\begin{equation}
  \expe{(f(\bfX_t) - f(x^{\star}))^{2p}} \leq \ttfp^{(a)} (1 + (\gua + t)^{pm - p(1+\varphi)\alpha}) \eqsp .
\end{equation}
Using \Cref{lemma:f_beat_norme}, there exists $\ttfdeux \geq 0$ such that
\begin{equation}
  \expe{\norm{\bfX_t - x^{\star}}^{2p}} \leq \ttfdeux (1 + (\gua + t)^{pm - p(1+\varphi)\alpha}) \eqsp ,
\end{equation}
which concludes the proof upon using that for any $a, b \geq 0$, $(a+b)^{1/2} \leq a^{1/2} + b^{1/2}$.
\end{proof}

The following corollary is of independent interest.

\begin{corollary_colt}
  Let $\alpha, \gamma \in \ooint{0, 1}$. Assume \rref{assum:f} and that $\argmin_{\rset^{\dim}} f$ is bounded. Then, for any $p \geq 0$ and $t \geq 0$,
  \begin{equation}
    \expe{\norm{\bfX_t - x^{\star}}^{p}} < + \infty \eqsp .
  \end{equation}
\end{corollary_colt}

\begin{proof}
Without loss of generality we assume that $x^{\star} = 0$ and $f(x^{\star}) = 0$. First, since $\argmin_{\rset^{\dim}} f$ is bounded, there exists $\tilde{R} \geq 0$ such that for any $x \in \rset^{\dim}$ with $\norm{x} \geq \tilde{R}$, $f(x) > 0$. Let $\sphere = \{ x \in \rset^{\dim}, \normLigne{x} = 1 \}$ and consider $m: \sphere \to \ooint{0, +\infty}$ such that for any $\theta \in \sphere$, $m(\theta) = f(\tilde{R}\theta)$. $m$ is continuous since $f$ is convex and therefore it attains its minimum and there exists $m^{\star} >0$ such that for any $\theta \in \sphere$, $m(\theta) \geq m^{\star}$. Let $x \in \rset^{\dim}$ with $\norm{x} \geq 2\tilde{R}$. Since $f_x : \coint{0, +\infty} \to \rset$ such that $f_x(t) = f(tx)$ is convex we have
\begin{equation}
  (f(x) - f(\tilde{R}x/\norm{x}))(\norm{x} - \tilde{R})^{-1} \geq (f(\tilde{R}x/\norm{x}))\tilde{R}^{-1} \geq m^{\star}\tilde{R}^{-1} \eqsp .
\end{equation}
Therefore, there exists $c >0$ and $R \geq 0$ such that for any $x \in \rset^{\dim}$ with $\norm{x} \geq R$, $f(x) \geq c\normLigne{x}$. Let $p \in \nset$. Noticing that \Cref{assum:f} implies that \rref{assum:f_weak_quasi} holds we can apply \Cref{lemma:bound_p} and \Cref{prop:recurrence} with $m = 1$ and $\varphi=2$. Applying repeatedly \Cref{prop:recurrence} we obtain that there exists $\ttfp \geq 0$ such that
\begin{multline}
  \expe{\norm{\bfX_t - x^{\star}}^{2p}}^{1/p}\leq \ttfp \defEnsLigne{1 + (\gua + t)^{m - \ceil{\alpha^{-1}} \alpha}} \\ \leq \ttfp \defEnsLigne{1 + (\gua + t)^{m - \ceil{m/\alpha} \alpha}} \leq \ttfp \defEnsLigne{1 + \gua^{m - \ceil{m/\alpha} \alpha}} \eqsp ,
\end{multline}
which concludes the proof.
\end{proof}

\subsection{Proof of \Cref{cor:let-alpha-gamma}}
\label{sec:let-aplha-gamma:proof}

  Let $\alpha, \gamma \in \ooint{0,1}$ and $X_0 \in \rset^{\dim}$.
    Using \Cref{lemma:dynkin}, we have for any $t \geq 0$
    \begin{align}
      \expe{\norm{\bfX_t - x^{\star}}^2} &= \norm{X_0 - x^{\star}}^2 - \int (\gua + s)^{-\alpha} \langle f(\bfX_s), \bfX_s - x^{\star} \rangle \rmd s \\
      &\qquad + (\gua/2) \int (\gua + s)^{-2\alpha} \langle \Sigma(\bfX_s), \nabla^2 f(\bfX_s) \rangle \rmd s \eqsp .       \label{eq:integr_carre}
    \end{align}
    Let $\mce_t = \expeLigne{\normLigne{\bfX_t - x^{\star}}^2}$.
    Using, \eqref{eq:integr_carre} we have for any $t \geq 0$,
    \begin{equation}
      \label{eq:derivative}
      \mce_t' \leq - (\gua + t)^{-\alpha} \expe{\langle \nabla f (\bfX_s), \bfX_s - x^{\star} \rangle} + (\gua \boundLSig/2) (\gua + t)^{-2\alpha} \eqsp .
    \end{equation}
    We divide the proof into three parts.
    \begin{enumerate}[wide, labelwidth=!, labelindent=0pt, label=(\alph*)]
  \item First, assume that \rref{assum:f_weak_quasi} holds. Combining this result and \eqref{eq:derivative}, we get that for any $t \geq 0$, $\mce_t' \leq \gua \Lip \eta^2 d (\gua + t)^{-2\alpha}$. Therefore, there exist $\beta, \vareps \geq 0$ and $\Cbeta \geq 0$ such that $\expeLigne{\normLigne{\bfX_t - x^{\star}}^2} < \Cbeta (\gua + t)^{-\beta}(1+\log(1+\gua^{-1}t))^{\vareps}$ with $\beta = 0$ and $\vareps =0$ if $\alpha > 1/2$, $\beta = 1 -2 \alpha$ and $\vareps = 0$ if $\alpha < 1/2$ and $\beta = 0$ and $\vareps = 1$ if $\alpha = 1/2$. Combining this result and \Cref{thm:f_weak} concludes the proof.
  \item
    We can apply \Cref{lemma:bound_p} and \Cref{prop:recurrence} with $m = 1$ and $\varphi=2$. Applying repeatedly \Cref{prop:recurrence} we obtain that there exists $\ttfp \geq 0$ such that
\begin{multline}
  \expe{\norm{\bfX_t - x^{\star}}^{2p}}^{1/p} \\ \leq \ttfp \defEnsLigne{1 + (\gua + t)^{m - \ceil{\alpha^{-1}} \alpha}} \leq \ttfp \defEnsLigne{1 + (\gua + t)^{m - \ceil{m/\alpha} \alpha}} \leq \ttfp \defEnsLigne{1 + \gua^{m - \ceil{m/\alpha} \alpha}} \eqsp ,
\end{multline}
which concludes the proof.
  \item Finally, assume that there exists $R \geq 0$ such that for any
    $x \in \rset^{\dim}$ with $\| x \| \geq R$,
    $\langle \nabla f(x), x- x^{\star}\rangle \geq \mtt \norm{x-x^{\star}}^2$. Therefore, since $(x \mapsto \nabla f(x))$ is continuous, there exists $\mtta \geq 0$ such that for any $x \in \rset^{\dim}$, $\langle \nabla f(x), x- x^{\star}\rangle \geq \mtt \norm{x-x^{\star}}^2 - \mtta$. Combining this result and \eqref{eq:derivative}, we get that for any $t \geq 0$,
    \begin{equation}
      \mce_t' \leq - \mtt (\gua + t)^{-\alpha} \mce_t + (\gua + t)^{-\alpha} \mtta + \gua \Lip \eta (\gua + t)^{-2\alpha}
    \end{equation}
Hence, if $\mce_t \geq \max(\mtta / \mtt, \Lip \eta)$ we have that $\mce_t' \leq 0$ and for any $t \geq 0$, $\mce_t \leq \max(\mtta / \mtt, \Lip \eta, \mce_0)$ and is bounded. Therefore, there exist $\beta, \vareps \geq 0$ and $\Cbeta \geq 0$ such that $\expeLigne{\normLigne{\bfX_t - x^{\star}}^2} < \Cbeta (\gua + t)^{-\beta}(1+\log(1+\gua^{-1}t))^{\vareps}$ with $\beta = \vareps = 0$, which concludes the proof.
\end{enumerate}

\subsection{Discrete counterpart of \Cref{cor:let-alpha-gamma}}
\label{sec:discr-count-crefc_alpha}

\begin{corollary_colt}
  \label{cor:let-alpha-gamma_disc}
  Let $\alpha, \gamma \in \ooint{0,1}$ and $x_0 \in \rset^{\dim}$. Assume
  \rref{assum:f_lip}, \rref{assum:grad_sto}.
  Then we have:
  \begin{enumerate}[wide, labelwidth=!, labelindent=0pt, label=(\alph*)]
  \item \label{item:a_weak_disc} if \rref{assum:f_weak_quasi} holds then, there exists
    $\Cweakapp \geq 0$ such that for any $N \in \nsets$
  \begin{equation}
    \expe{f(X_N)}-f\st \leq \Cweakapp \parentheseDeux {N^{(1-3\alpha)/2} +N^{-\alpha/2} + N^{\alpha -1}} \eqsp ,
  \end{equation}
  \item \label{item:c_weak_disc} if \rref{assum:f_weak} holds and if there exists $R \geq 0$ such that
    for any $x \in \rset^{\dim}$ with $\| x \| \geq R$,
    $\langle \nabla f(x), x- x^{\star}\rangle \geq \mtt \norm{x-x^{\star}}^2$, then
    there exists $\Cweakapp \geq 0$ such that for any $N \in \nsets$
    \begin{equation}
          \expe{f(X_N)}-f\st \leq \Cweakapp \parentheseDeux{N^{-\alpha/2} + N^{\alpha -1}} \eqsp .
    \end{equation}
  \end{enumerate}
\end{corollary_colt}

\begin{proof}
  Let $\alpha, \gamma \in \ooint{0,1}$ and $x_0 \in \rset^{\dim}$.
  We have for any $n \in \nset$,
  \begin{align}
    \expe{\norm{X_{n+1} - x^{\star}}^2} &= \expe{\norm{X_n - x^{\star}}^2} + 2 \expe{\langle X_n - x^{\star}, X_{n+1} - X_n \rangle} + \expe{\norm{X_{n+1} - X_n}^2} \\
    &\leq \expe{\norm{X_n - x^{\star}}^2} - 2\gamma(n+1)^{-\alpha} \expe{\langle X_n - x^{\star}, \nabla f(X_n) \rangle} \\ & \qquad + 2\gamma^2(n+1)^{-2\alpha} \expe{\norm{\nabla f(X_n)}^2} + 2\gamma(n+1)^{-2\alpha} \eta \eqsp .     \label{eq:for_all}
  \end{align}
  We now divide the proof into two parts.
  \begin{enumerate}[wide, labelwidth=!, labelindent=0pt, label=(\alph*)]
  \item Using \rref{assum:f_weak_quasi} and \Cref{lemma:kolmo} we have for any $x \in \rset^{\dim}$,
    \begin{equation}
      \label{eq:weak_grad}
      \langle \nabla f(x), x - x^{\star} \rangle \geq \tau(f(x) - f(x^{\star})) \geq \tau \norm{\nabla f(x)}^2 / (2\Lip) \eqsp .
    \end{equation}
    Using \Cref{assum:f_lip}, \eqref{eq:for_all} and  \eqref{eq:weak_grad} we have for any $n \geq (4\gamma\Lip/\tau)^{1/\alpha}$
    \begin{align}
      \expe{\norm{X_{n+1} - x^{\star}}^2} &\leq  \expe{\norm{X_{n} - x^{\star}}^2}  + 2\gamma (n+1)^{-\alpha} (-\tau/(2\Lip) + \gamma (n+1)^{-\alpha})\expe{\norm{\nabla f(X_n)}^2} \\
      &\quad+ 2\gamma(n+1)^{-2\alpha} \eta \\
      &\leq \expe{\norm{X_{n} - x^{\star}}^2} + 2\gamma(n+1)^{-2\alpha} \eta \eqsp .
    \end{align}
    Therefore, there exist $\beta, \vareps \geq 0$ and $\Cbeta \geq 0$ such that $\expeLigne{\normLigne{\bfX_n - x^{\star}}^2} < \Cbeta (n+1)^{-\beta}(1+\log(1+n))^{\vareps}$ with $\beta = 0$ and $\vareps =0$ if $\alpha > 1/2$, $\beta = 1 -2 \alpha$ and $\vareps = 0$ if $\alpha < 1/2$ and $\beta = 0$ and $\vareps = 1$ if $\alpha = 1/2$. Combining this result and \Cref{thm:f_weak_discrete} concludes the proof.
  \item Finally, assume that there exists $R \geq 0$ such that for any
    $x \in \rset^{\dim}$ with $\| x \| \geq R$,
    $\langle \nabla f(x), x- x^{\star}\rangle \geq \mtt \norm{x-x^{\star}}^2$. Therefore, since $(x \mapsto \nabla f(x))$ is continuous, there exists $\mtta \geq 0$ such that for any $x \in \rset^{\dim}$, $\langle \nabla f(x), x- x^{\star}\rangle \geq \mtt \norm{x-x^{\star}}^2 - \mtta$. Combining this result and \eqref{eq:for_all} we get that for any $n \in \nset$ such that $n \geq (2/\gamma)^{\alpha^{-1}}$
    \begin{equation}
      \expe{\norm{X_{n+1} - x^{\star}}^2}  \leq (1 - \gamma(n+1)^{-\alpha}) \expe{\norm{X_{n} - x^{\star}}^2} + 2\gamma(n+1)^{-\alpha} \mtta + 2\gamma^2(n+1)^{-2\alpha} \eta \eqsp .
    \end{equation}
    Hence, if $n \geq (2/\gamma)^{-\alpha^{-1}}$ and $\expeLigne{\normLigne{X_n - x^{\star}}^2} \geq \max(2\mtta, 2\gamma \eta)$ then $\expeLigne{\normLigne{X_{n+1} - x^{\star}}^2} \leq \expeLigne{\normLigne{X_{n} - x^{\star}}^2}$. Therefore, we obtain by recursion that for any $n \in \nset$, that $(\expeLigne{\normLigne{X_{n} - x^{\star}}^2})_{n \in \nset}$ is bounded which concludes the proof by applying \Cref{thm:f_weak_discrete}.
  \end{enumerate}
\end{proof}

\subsection{Proof of \Cref{thm:f_weak_discrete}}
\label{thm:f_weak_discrete:proof}

  Without loss of generality, we assume that $f^{\star} = 0$. Let $\alpha, \gamma \in \ooint{0,1}$, $x_0 \in \rset^{\dim}$. Let $\delta = \min(\delta_1, \delta_2)$, with
  $\delta_1, \delta_2$ given in \Cref{thm:f_weak_discrete} and let $(E_k)_{k \in \nset}$ such that
  for any $k \in \nset$, $E_k = (k+1)^{\delta} \expe{f(X_k)} (1 + \log(k+1))^{-\vareps}$.
  There exists $c_{\delta} \in \rset$ such that for any $x \in \ccint{0,1}$, $(1 + x)^{\delta} \leq 1 + c_{\delta} x$.
  Hence, for any $n \in \nset$ we have
  \begin{equation}
    \label{eq:der_n}
    (n+2)^{\delta} - (n+1)^{\delta} \leq (n+1)^{\delta} \defEns{(1+(n+1)^{-1})^{\delta} - 1} \leq c_{\delta} (n+1)^{\delta - 1} \eqsp .
  \end{equation}
Using \citep[Lemma 1.2.3]{nesterov2004introductory} and \Cref{assum:grad_sto} we have for any $n \in \nset$ such that $n \geq (2\Lip\gamma)^{1/\alpha}$
\begin{align}
  &\expen{f(X_{n+1})} \leq f(X_n) - \gamma (n+1)^{-\alpha} \expen{\langle \nabla f(X_n), H(X_n, Z_{n+1}) \rangle} \\
  &\qquad+(\Lip/2) \gamma^2(n+1)^{-2\alpha} \expen{\norm{ H(X_n, Z_{n+1}) }^2} \\
     \expe{f(X_{n+1})} &\leq \expe{f(X_n)} - \gamma (n+1)^{-\alpha} \expe{\norm{\nabla f(X_n)}^2} \\
     &\qquad + \Lip\gamma^2(n+1)^{-2\alpha} \expe{\norm{ \nabla f(X_n) }^2} + \Lip\gamma^2(n+1)^{-2\alpha} \eta \\
                    &\leq \expe{f(X_n)} -\gamma (n+1)^{-\alpha} \defEns{1 - \Lip\gamma(n+1)^{-\alpha}} \expe{\norm{\nabla f(X_n)}^2}  + \Lip\gamma^2(n+1)^{-2\alpha} \eta \\
&\leq \expe{f(X_n)} -\gamma (n+1)^{-\alpha} \expe{\norm{\nabla f(X_n)}^2}/2  + \Lip\gamma^2(n+1)^{-2\alpha} \eta \eqsp .  \label{eq:der_f}
\end{align}
Combining \eqref{eq:der_n} and \eqref{eq:der_f} we get for any $n \in \nset$ such that $n \geq (2\Lip\gamma)^{1/2}$
\begin{align}
  E_{n+1} - E_{n} &= (n+2)^{\delta} \expe{f(X_{n+1})} (1 + \log(n+2))^{-\vareps } - (n+1)^{\delta} \expe{f(X_{n})} (1 + \log(n+1))^{-\vareps } \\
                    &\leq (1 + \log(n+1))^{-\vareps } \left[\defEns{(n+2)^{\delta} -(n+1)^{\delta}} (\expe{f(X_{n+1})}) \right. \\
                    &\qquad \left.+ (n+1)^{\delta} \defEns{\expe{f(X_{n+1})} - \expe{f(X_n)}}\right] \\
                    &\leq (1 + \log(n+1))^{-\vareps } \left[ \defEns{(n+2)^{\delta} -(n+1)^{\delta}} (\expe{f(X_{n})} + \Lip\gamma^2(n+1)^{-2\alpha} \eta) \right. \\
                    & \qquad \left. + (n+1)^{\delta} \defEns{-\gamma (n+1)^{-\alpha} \expe{\norm{\nabla f(X_n)}^2}/2  + \Lip\gamma^2(n+1)^{-2\alpha} \eta} \right] \\
                    &\leq (1 + \log(n+1))^{-\vareps } \left[ c_{\delta} (n+1)^{\delta-1} (\expe{f(X_{n})} + 2\gamma^2(n+1)^{-2\alpha} \eta) \right. \\
                    & \qquad \left. + (n+1)^{\delta} \defEns{-\gamma (n+1)^{-\alpha} \expe{\norm{\nabla f(X_n)}^2}/2  +\Lip\gamma^2(n+1)^{-2\alpha} \eta} \right] \\
                  &\leq c_{\delta} E_n+ 2\Lip\gamma^2(1+c_{\delta})(n+1)^{\delta - 2\alpha} (1 + \log(n+1))^{-\vareps } \eta \\
                  &\quad - \gamma (n+1)^{\delta - \alpha}(1 + \log(n+1))^{-\vareps } \expe{\| \nabla f(X_n) \|^2} / 2 \eqsp .   \label{eq:main_calc}
\end{align}
Using \eqref{assum:f_weak} and the fact that for any $k \in \nset$,
$\expe{\| X_k - x^{\star} \|^{r_2 r_3}} \leq \Cbeta(k+1)^{\beta}(1 + \log(1 +
k))^{\explog}$ and Hölder's inequality and that $r_1r_3 = 2(2r_1^{-1}-1)^{-1}$, we have for any $k \in \nset$
\begin{equation}
\label{eq:holder}
  \expe{\norm{\nabla f(X_k)}^2} \geq \expe{f(X_k)}^{2r_1^{-1}} \Cbeta^{-(2r_1^{-1}-1)^{-1}} \tau^{2r_1^{-1}} (k+1)^{-\beta(2r_1^{-1}-1)}(1 + \log(k+1))^{-\vareps(2r_1^{-1}-1)} \eqsp .
\end{equation}
Combining \eqref{eq:main_calc} and \eqref{eq:holder} we get that for any $n \in \nset$ with $n \geq (4 \gamma)^{1/\alpha}$
\begin{align}
  &E_{n+1} - E_n \leq c_{\delta} E_n+ 2\Lip\gamma^2(1+c_{\delta})(n+1)^{\delta - 2\alpha} (1 + \log(n+1))^{-\vareps } \eta \\ & \qquad - \gamma (n+1)^{\delta - \alpha -\beta(2r_1^{-1}-1)}\expe{f(X_n)}^{2r_1^{-1}} \Cbeta^{-(2r_1^{-1}-1)^{-1}} \tau^{2r_1^{-1}} (1 + \log(n+1))^{-\vareps 2r_1^{-1}} / 2 \\
  &\leq c_{\delta} E_n+ 2\Lip\gamma^2(1+c_{\delta})(n+1)^{\delta - 2\alpha} (1 + \log(n+1))^{-\vareps } \eta  \\
  &\quad - \gamma (n+1)^{\alpha -(\delta +\beta)(2r_1^{-1}-1)}E_n^{2r_1^{-1}} \Cbeta^{-(2r_1^{-1}-1)^{-1}} \tau^{2r_1^{-1}}  / 2 \eqsp .
\end{align}
Let $\Cweakapp_3 = \max(\Cweakapp_1, \Cweakapp_2)$ with
\begin{equation}
  \label{eq:5}
  \left\lbrace
  \begin{aligned}
    \Cweakapp_1 &= (2|c_{\delta}| \Cbeta^{2r_1^{-1} - 1} \tau^{-2r_1^{-1}})^{2r_1^{-1} -1} \eqsp , \\
    \Cweakapp_2 &=  (4\Lip\gamma^2(1+c_{\delta})\Cbeta^{2r_1^{-1} - 1} \tau^{-2r_1^{-1}})^{r_1/2} \eqsp .
  \end{aligned}
  \right.
\end{equation}
If $E_n \geq \Cweakapp_3$ and $n \geq (4\gamma)^{1/\alpha}$ then $E_{n+1} \leq E_n$.
Therefore, we obtain by recursion that $E_n \leq \Cweakapp$ with $\Cweakapp = \max(E_0, \dots, E_{\ceil{(2\Lip\gamma)^{1/\alpha}}}, \Cweakapp_3)$.
